\documentclass{illcdiss}

\usepackage{amsmath,amssymb,enumerate,graphicx, amsthm}
\usepackage{bussproofs}
\usepackage[applemac]{inputenc}
\usepackage[italian,spanish,english]{babel} 

\usepackage{xypic}
\usepackage{changebar} 
\usepackage[pdftex,bookmarks=true,bookmarksopen=true,bookmarksnumbered=true]{hyperref}

 \theoremstyle{theorem}  
 \newtheorem{thm}{Theorem}[section]
\newtheorem{cor}[thm]{Corollary}
\newtheorem{lem}[thm]{Lemma}
 \theoremstyle{theorem}
 \newtheorem{proposition}[thm]{Proposition}
 \theoremstyle{definition}
 \newtheorem{definition}[thm]{Definition}
 \newtheorem{ex}[thm]{Example}

\theoremstyle{definition}
\newtheorem{rem}[thm]{Remark}

\newcommand{\ba}{\backslash}
\newcommand{\h}{\vdash_{\scriptscriptstyle H_{\supset}}}

\newcommand{\ta}{\rightarrow}
\newcommand{\da}{\leftrightarrow}
\newcommand{\s}{\supset}
\newcommand{\cc}{\thicksim}

\renewcommand{\mathfrak}{\mathbf}

\newcommand{\tet}{\vartheta} 
\newcommand{\hi}{\vdash_{\scriptscriptstyle H}}
\newcommand{\lb}{\mathcal{LB}} 
\newcommand{\lbs}{\mathcal{LB}_\supset} 

\newcommand{\leibniz}{\mathbf{\Omega}} 
\newcommand{\Fi}{\mathcal{F}i}
\newcommand{\tarski}{\mathbf{\widetilde{\Omega}}}

\newcommand{\Al}[1][A]{\ensuremath{\mathbf{#1}} }
\newcommand{\false}{\mathsf{f}}
\newcommand{\true}{\mathsf{t}}
\newcommand{\Tr}{\mathsf{Tr}}
\newcommand{\FF}{\mathcal{FF}}

\newcommand{\LB}{\mathcal{LB}}
\newcommand{\Log}[1][LB]{\ensuremath{\vDash_{\mathcal{#1}}} }
\newcommand{\eqLog}{\mathrel{=\joinrel\mathrel\vert\mkern0mu\mathrel\vert\joinrel=_{\mathcal{LB}}}}

\newcommand{\Con}{\mathrm{Con}}
\newcommand{\Reg}{\mathrm{Reg}}
\newcommand{\reg}{\mathrm{reg}}

\newcommand{\PBL}{\mathsf{PreBiLat}}
\newcommand{\DPBL}{\mathsf{DPreBiLat}}
\newcommand{\IPBL}{\mathsf{IntPreBiLat}}
\newcommand{\BL}{\mathsf{BiLat}}
\newcommand{\DBL}{\mathsf{DBiLat}}
\newcommand{\IBL}{\mathsf{IntBiLat}}

\newcommand{\DBLC}{\mathsf{DBiLatCon}}
\newcommand{\IBLC}{\mathsf{IntBiLatCon}}
\newcommand{\IL}{\mathsf{InvLat}}
\newcommand{\DML}{\mathsf{DMLat}}
\newcommand{\KL}{\mathsf{KLat}}
\newcommand{\BOL}{\mathsf{BLat}}

\newcommand{\KB}{\mathsf{KBiLatCon}}
\newcommand{\CB}{\mathsf{CBiLatCon}}

\newcommand{\lat}{\mathsf{Lat}}
\newcommand{\dl}{\mathsf{DLat}}

\newcommand{\cl}{\mathsf{CILat}}
\newcommand{\ib}{\mathsf{ImpBiLat}}
\newcommand{\rd}{\mathsf{RDMLat}}
\newcommand{\ia}{\mathsf{IAlg}}

\newcommand{\Obj}{\textrm{Obj}}
\newcommand{\Mor}{\textrm{Mor}}

\newcommand{\co}{\mathop{\relbar}\nolimits}

\newcommand{\la}{\langle}
\newcommand{\ra}{\rangle}

\renewcommand{\phi}{\varphi}
\renewcommand{\mathfrak}{\mathbf}
\newcommand{\four}{\mathcal{FOUR}}
\newcommand{\fours}{\mathcal{FOUR}_{\supset}}

\newcommand{\five}{\mathcal{FIVE}}
\newcommand{\nine}{\mathcal{NINE}}
\newcommand{\default}{\mathcal{DEFAULT}}

\newcommand{\5}{\mathcal{FIVE}}

\newcommand{\7}{\mathcal{SEVEN}}

\newcommand{\seq}{\rhd}
\newcommand{\var}{\mathcal{V}ar}

\newcommand{\leib}[1][A]{\ensuremath{\boldsymbol{\Omega_{\Al[#1]}}} }
\newcommand{\alg}[1][]{\ensuremath{\mathbf{Alg#1}} }
\newcommand{\algstar}[1][]{\ensuremath{\mathbf{Alg^{\ast}#1}} }

\newcommand{\clb}{\ensuremath{\mathbf{C}_{\mathcal{LB}}}}
\newcommand{\RB}{\mathcal{R}eg}

\newcommand{\Gensinl}{ \thicksim  \joinrel  \mid } 

\newcommand{\Gensin}{\mid \joinrel \thicksim _{\mathcal{G_{LB}}} }
\newcommand{\Gens}{\mid \joinrel \thicksim _{\mathcal{G}} }

\newcommand{\g}{\mathcal{G}}

\begin{document}
\pagestyle{plain}
\pagenumbering{roman}


{\pagestyle{empty}
\newcommand{\printtitle}{%

{\Huge\bf An Algebraic Study \\[0.8cm]  of Bilattice-based Logics 
}}    

\begin{titlepage}
\par\vskip 2cm
\begin{center}
\printtitle

\par\vspace {4cm}
\illclogo{8cm}
\par\vspace {4cm}

\vfill
{\LARGE\bf Umberto Rivieccio}                           
\vskip 2cm
\end{center}
\end{titlepage}
%
%
\mbox{}\newpage
\setcounter{page}{1}




\clearpage
\par\vskip 2cm
\begin{center}
{\large
Tesi discussa per il conseguimento del titolo di \\
Dottore di ricerca in Filosofia \\
svolta presso la \\ Scuola di Dottorato in Scienze Umane \\
dell'Università degli Studi di Genova \\
\par\vspace {2cm}}

\printtitle
\par\vspace {4cm}
{\Large Umberto Rivieccio}                        
\par\vspace {4cm} 
{\Large Relatori: \\[0.4cm] 
Maria Luisa Montecucco \\ (Universit\`a di Genova) \\[0.8cm]
Ramon Jansana i Ferrer \\[0.2cm] (Universitat de Barcelona)} 
\end{center}

\clearpage
\par\vskip 2cm
\begin{center}
{\large
Departament de L\`ogica, Hist\`oria i Filosofia de la Ci\`encia \\
Facultat de Filosofia\\
Universitat de Barcelona \\
Programa de Doctorado: Ci\`encia Cognitiva i Llenguatge\\
\par\vspace {2cm}} 

\printtitle
\par\vspace {4cm}
{\Large Umberto Rivieccio}                        
\par\vspace {4cm} 
{\Large Directores: \\[0.4cm] Dra.\ Maria Luisa Montecucco \\ (Universit\`a di Genova) \\[0.8cm]
Dr.\ Ramon Jansana i Ferrer \\[0.2cm] (Universitat de Barcelona)} 
\end{center}

\clearpage
} 

\thispagestyle{plain}
\mbox{}
\vspace{2.5in}
\begin{center}
{\em Of all escapes from reality, mathematics is the most successful ever.}\\[1cm]
\end{center}
\begin{flushright}
       Giancarlo Rota
\end{flushright}

\tableofcontents
\acknowledgments

First of all, I would like to express my gratitude to my supervisors, Luisa Montecucco and Ramon Jansana, who helped and advised me in several ways during the three years of my PhD.

I also want to thank my professors at the University of Genoa, in particular: Dario Palladino, who followed my work and was very helpful to me since the days of my Laurea thesis; Carlo Penco, who helped me in several occasions, especially in organizing my stay in Barcelona; Angelo Campodonico, for his help as a coordinator of the Doctorate in Philosophy.

 I am grateful to many scholars I met during my stay in Barcelona: all the Logos people, in particular Manuel Garc\'ia-Carpintero, who invited me to Barcelona, and Jos\'e Mart\'inez, who had the patience to read some  paper of mine and introduced me to many interesting reading groups; Josep Maria Font, who gave me precious hints, among them the first idea of what was eventually to become this thesis; Joan Gispert and Antoni Torrens, who taught me logic, algebra and some Catalan; Llu\'is Godo and Francesc Esteva, for their extreme kindess; Enrico Marchioni, who gave me useful information on logic in Spain. I am particularly indebted to F\'elix Bou, my sensei and co-author of a long-expected paper: without his help this work would not have been possible. 

Finally, I want to mention some friends that helped me in various ways that they may not suspect. 

Those I met at the university: Luz Garc\'ia \'Avila, who taught me how to speak like a Mexican; Miguel \'Anguel Mota, who showed me how to drink like a Mexican; Daniel Palacín and his family (of sets); Sergi Oms, who introduced me to the Catalan literature; Chiara Panizza, who showed me how to survive any accident; Marco Cerami, who taught me how to change a tyre;  and Mirja P\'erez de Calleja, who taught me everything else. 

\ldots and outside the university: Umberto Marcacci and the elves, for countless hours of time lost; Caroline Bavay, for her welcome; Cristina Cervilla, for a scarf and a tie; Eva L\'opez, for being illogic and irrational; Silvia Izzi, for the sushi picnic by the lake; Beatriz Lara, for a coffee and a supper; Federica Q., for a surprise Roman holiday; Fiorella Aric\`o, just for being there always; and last but not least, Cimi, Pigi \& Gidio, for being my animal family.  

My last thanks go to my family, for whom no words would suffice\ldots \ and to all the people I forget.





\cleardoublepage
\pagestyle{headings}
\pagenumbering{arabic}





\chapter{Introduction and preliminaries} 
\label{ch:intro}

\section{Introduction and motivation}
\label{sec:intro}


The aim of this work is to develop a study from the perspective of Abstract Algebraic Logic of some bilattice-based logical systems introduced in the nineties by Ofer Arieli and Arnon Avron. The motivation for such an investigation has two main roots. 

On the one hand there is an interest in bilattices as an elegant formalism that gave rise in the last two decades to a variety of applications, especially in the field of Theoretical Computer Science and Artificial Intelligence. In this respect,  the present study aims to be a contribution to a better understanding 
of the mathematical and logical framework that underlie these applications.

On the other hand, our interest in bilattice-based logics comes from 
  Abstract Algebraic Logic. In very general terms, algebraic logic can be described as the study of the connections between algebra and logic. One of the main reasons that motivate this study is the possibility to treat logical problems with algebraic methods and viceversa: this is accomplished by associating to a logical system a class of algebraic models that can be regarded as the algebraic counterpart of that logic. Starting from the work of Tarski and his collaborators, the method of algebraizing logics has been increasingly developed and generalized. In the last two decades, algebraic logicians have focused their attention on the process of algebraization itself: this kind of investigation forms now a subfield of algebraic logic  known as Abstract Algebraic Logic (which we abbreviate AAL).

An important issue in AAL is the possibility to apply the methods of the general theory of the algebraization of logics to an increasingly wider range of logical systems. In this respect, some bilattice-based logics are particularly interesting as natural examples of so-called \emph{non protoalgebraic logics}, a class that includes the logical systems that are most difficult to treat with algebraic tools. 

Until recent years, relatively few non protoalgebraic logics had been studied. Possibly also because of this lack of examples, the general results available on this class of logics are still not comparable in number and depth with those that have been proved  for the logical systems that are, so to speak, well-behaved from the algebraic point of view, called \emph{protoalgebraic logics}. In this respect, the present work intends to be a contribution to the long-term goal of extending the general theory of the algebraization of logics beyond its present borders.  


Let us now introduce informally the main ideas that underlie the bilattice formalism and mention some of  their applications. 

Bilattices are algebraic structures proposed by  Matthew Ginsberg \cite{Gi88}  as a uniform framework for inference in Artificial Intelligence, in particular within the context of default and non-monotonic reasoning. In the last two decades the bilattice formalism has found interesting applications in  many fields, sometimes quite different from the original one, of which we shall cite just a few.

As observed by Ginsberg  \cite{Gi88},  many inference systems that are used in Artificial Intelligence can be unified within a many-valued framework whose space of truth values is a set endowed with a double lattice structure. The idea that truth values should be ordered is very common, indeed almost standard in many-valued logics: for instance, in fuzzy logics the values are (usually totally) ordered according to their ``degree of truth''. In this respect, Ginsberg's seminal idea was that, besides the order associated with the degree of truth, there is another ordering that is also natural to consider. This relation, which he called the  ``knowledge order'', is intended to reflect the degree of knowlegde or information associated with a sentence: for instance, in the context of automated reasoning, one can label a sentence as ``unknown'' when the epistemic agent has no information at all about the truth or falsity of that sentence. This idea, noted Ginsberg, was already present in the work of Belnap \cite{Be76}, \cite{Be77}, who proposed a similar interpretation for the well-known Belnap-Dunn four-valued logic. From a mathematical point of view, Ginsberg's main contribution was to develop a generalized framework that allows to handle arbitrary doubly ordered sets  of truth values.  

According to the notation introduced by Ginsberg, within the bilattice framework the two order relations are usually denoted by $\leq_t$ (where the $t$  is for ``truth'') and $\leq_k$ ($k$ for ``knowledge''). Concerning the usage of the term ``knowledge'', let us quote a remark due to Melvin Fitting ~\cite{Fi06b}: 

\begin{quotation}
The ordering $\leq_k$ should be thought of as ranking ``degree of
information''. Thus if $x \leq_k y$, $y$ gives us at least as much
information as $x$ (and possibly more). I suppose this really should be
written as $\leq_i$, using $i$ for information instead of $k$ for
knowledge. In some papers in the literature $i$ is used, but I have
always written $\leq_k$, and now I'm stuck with it.
\end{quotation}

We agree with Fitting's observation that using $\leq_i$ would be a better choice but, like himself, in the present work we will write  $\leq_k$, following a notation that has by now become standard.

After Ginsberg's initial work (besides \cite{Gi88},  see also \cite{Gi90a} 
and \cite{Gi95a}), bilattices were extensively investigated by Fitting, who considered applications to Logic Programming (\cite{Fi90}, \cite{Fi91}; on this topic see also \cite{Ko01a} 
 and \cite{LoSt04}
 ),
 to philosophical problems such as the theory of truth (\cite{Fi89}, \cite{Fi06b}) and studied their relationship with a family of many-valued systems generalizing Kleene's three-valued logics (\cite{Fi91b}, \cite{Fi94}). Other interesting applications include the analysis of entailment, implicature and presupposition in natural language \cite{Sc96}, the semantics of natural language questions \cite{NeFr02} 
and epistemic logic \cite{Si94}. 
  
In the nineties, bilattices were also investigated in depth by Arieli and Avron, both from an algebraic (\cite{Av95}, \cite{Av96a}) and from a logical point of view (
\cite{ArAv94}, \cite{ArAv98}). In order to deal with paraconsistency and non-monotonic reasoning in Artificial Intelligence, Arieli and Avron \cite{ArAv96} developed the first bilattices-based logical systems in the traditional sense. The simplest of these logics, which  we shall 
call $\mathcal{LB}$, is defined semantically from a class of matrices called \emph{logical bilattices}, and  is an expansion of the  aforementioned Belnap--Dunn  four--valued logic to the standard language of bilattices. In \cite{ArAv96} a Gentzen-style calculus is presented as a syntactic counterpart of $\mathcal{LB}$, and completeness and cut elimination are proved. In the same work, Arieli and Avron considered also an expansion of $\lb$, obtained by adding to it  two (interdefinable) implication connectives. This 
logic,  which we shall denote  by $\lbs$, is also introduced semantically using the notion of logical bilattice. In \cite{ArAv96} both a Gentzen- and a Hilbert-style presentation of $\lbs$ are given, and completeness and cut elimination  for the Gentzen calculus are proved.



 Our main concern in the present work  will be to investigate these two logical systems from the point of view of Abstract Algebraic Logic. This investigation will lead to interesting insights on both  logical and algebraic aspects of bilattices. 
 
The material is organized  as follows. The next section (\ref{sec:aal}) contains some notions of Abstract Algebraic Logic that will be needed in order to develop our approach to bilattice-based logics. In the following one (\ref{sec:bil})  we present the essential definitions and some known results on bilattices.

Chapter  \ref{ch:int} 
presents some new algebraic results 
that will be used to develop our treatment of  bilattice-based logics from the perspective of AAL: a generalization of the Represetation Theorem for bounded interlaced pre-bilattices and bilattices to the unbounded case (Sections \ref{sec:reppre} 
and \ref{sec:repbil}
), the study of filters and ideals in (pre-)bilattices (Section \ref{sec:bif}
) and 
a characterization of the variety of distributive bilattices (Section \ref{sec:dis}
).

In Chapter 
\ref{ch:log}
we study the (implicationless) logic of logical bilattices $\lb$, defined in Section 
\ref{sec:sem} both semantically and through the Gentzen-style presentation due to Arieli and Avron. 
In Section 
\ref{sec:hil} we introduce a Hilbert-style presentation for $\lb$ and prove completeness via a normal form theorem. In the following section (\ref{sec:tar}) we prove that $\lb$ has no consistent extensions and characterize this logic in terms of some metalogical properties of its associated consequence relation. Our Hilbert-style calculus is then used (Section \ref{sec:lb}) in order to study $\lb$ from the perspective of AAL, characterizing its algebraic models. In the last section of the chapter (\ref{sec:gentz}) we prove that the  Gentzen calculus introduced by Arieli and Avron is algebraizable in the sense of Rebagliato and Verd\'{u} \cite{ReV93b} and characterize its equivalent algebraic semantics.

In Chapter \ref{ch:add} 
we consider an expansion of $\lb$, also due to Arieli and Avron, obtained by adding two interdefinable implication connectives to the basic bilattice language. Section  \ref{sec:add} 
contains Arieli and Avron's original presentations, a semanical and a Hilbert-style one, of this logic, which we call $\lbs$.
In Section \ref{sec:lbs} 
we prove some properties of the Hilbert-style calculus of  Arieli and Avron that will be used  to show that the logic $\lbs$ is algebraizable. In the following section (\ref{sec:algstar}) we determine the equivalent algebraic semantics of $\lbs$.
We also show that this class of algebras, that we call ``implicative bilattices'', is a variety and provide an equational presentation for it.

Chapter \ref{ch:imp} 
is devoted to an algebraic study of the variety of implicative bilattices. In Section \ref{sec:repimp} 
we prove a representation theorem for implicative bilattices, analogous to the one proved in Chapter \ref{ch:int} 
for bilattices, stating that any implicative bilattice is isomoprhic to a certain product of two lattices satisfying some additional properties,
which we call \emph{classical implicative lattices}. Section \ref{sec:imp} 
 contains several results about the variety of implicative bilattices from the point of view of Universal Algebra. Section \ref{sec:dua} 
 is devoted to the study of the relationship between classical implicative lattices and another class of lattices that arose as (product bilattice) factors of the algebraic 
 models of $\lb$. The following two sections (\ref{sec:sub} and \ref{sec:other}) 
contain a description of some subreducts of implicative bilattices that seem to us to be particularly significant from a logical point of view. In particular, we introduce and characterize an interesting class of De Morgan lattices endowed with two additional operations forming a residuated pair. In the last section (\ref{sec:cat}) we consider most of the classes of bilattices studied in the literature from the point of view of category theory: in particular, we prove some equivalences between various categories of interlaced bilattices and the corresponding lattices arising from our representation theorems.

\section{Abstract Algebraic Logic} 
\label{sec:aal}

In this section we recall some definitions and results of Abstract Algebraic Logic that will be needed in order to understand our study of bilattices and bilattice-based logics. All the references and  proofs of the results can be found in  \cite{Cz01}  and \cite{FJa09}.

Let us start by giving the definion of what we mean by \emph{a logic} in the context of AAL.

A \emph{sentential logic} is a pair $\mathcal{L} = \la \Al[Fm], \mathbf{C}_{\mathcal{L}} \ra $ where $\Al[Fm]$ is the formula algebra of some similarity type and $ \mathbf{C}_{\mathcal{L}}$ is a structural (i.e.\ substitution-invariant) closure operator on  $\Al[Fm]$. In the present work, since we will not deal with first- or higher-order logic, normally we shall just say  \emph{a logic}, meaning a sentential logic. 

To any closure operator $ \mathbf{C}_{\mathcal{L}}$ of this kind we may associate a consequence relation, denoted by $ \vdash_{\mathcal{L}}$ or $ \vDash_{\mathcal{L}}$, defined as follows: for all $\Gamma \cup \{ \phi \} \subseteq Fm$, we set $\Gamma \vdash_{\mathcal{L}} \phi $ if and only if  $\phi \in  \mathbf{C}_{\mathcal{L}} (\Gamma)$. We will generally reserve the symbol $ \vdash_{\mathcal{L}}$ to consequence relations defined in a sintactical way, while $ \vDash_{\mathcal{L}}$ shall be used for semantically defined relations.

Recalling that instead of closure operators one can equivalently speak of closure systems, we note that another way to define a logic is as a pair $\la \Al[Fm], \mathcal{T}h\mathcal{L}\ra $,  where $\Al[Fm]$ is the formula algebra and $ \mathcal{T}h\mathcal{L} \subseteq P(Fm)$ is a family closed under inverse substitutions, i.e.\ such that for any endomorphism $\sigma : \Al[Fm] \rightarrow \Al[Fm]$ and for any $T \in \mathcal{T}h\mathcal{L}$, we have $\sigma^{-1}(T) \in  \mathcal{T}h\mathcal{L}$. As the notation suggests, $\mathcal{T}h\mathcal{L}$ is the closure system given by the family of all theories of the logic $\mathcal{L}$.

One of the main topics in Algebraic Logic  is the study of logical matrices, i.e.\, roughly speaking, algebraic models of sentential logics. Formally, a \emph{logical matrix}  is a pair $\la \Al, D \ra $ where $ \Al$ is an algebra and $D \subseteq A$ is a set of \emph{designated elements}. 

To each matrix $\la \Al, D \ra $  we can associate a set of congruences of $\Al$ which have a special logical interest, called \emph{matrix congruences}, and defined as follows: $\theta \in \textrm{Con}(\Al)$ is a matrix congruence of $\la \Al, D \ra $  when it is compatible with the set $D$, i.e.\ when, for all $a,b \in A $,  if $a \in D$ and $\la a, b \ra \in \theta$, then $b \in D$. 

It is known that, for any  $\la \Al, D \ra $,  the set of matrix congruences, ordered by inclusion, has always a maximum element: this is called the \emph{Leibniz congruence} of the matrix  $\la \Al, D \ra $, and is denoted by $\leib D$ or $\leib[] \la  \Al, D \ra$. We say that a matrix is \emph{reduced} when its Leibniz congruence is the identity.

In a matrix $\la \Al, D \ra $, the algebra with its operations can be thought of as a kind of generalized truth table, while the designated elements may be regarded as those values which are treated like \emph{true} in classical logic. We may then use any matrix as a truth table in order to define a logic, as follows. We define  $\Gamma \vDash_{\la \Al, D \ra} \phi $ if and only if, for any homomorphism $h: \Al[Fm] \rightarrow  \Al$,   $h[\gamma] \subseteq D$ implies $h(\phi) \in D$.

 A matrix $\la \Al, D \ra $  is said to be a \emph{model} of a logic  $ \mathcal{L}$ when $\Gamma \vdash_{\mathcal{L}} \phi $ implies $\Gamma \vDash_{\la \Al, D \ra} \phi $. In this case the set $D$ is called a\emph{ filter of the logic} $ \mathcal{L}$ or an $ \mathcal{L}$-\emph{filter} on $\Al$.
The set of all filters of a logic $ \mathcal{L}$ on a given algebra $\Al$ will be denoted by $\Fi_{\mathcal{L}} \Al$.

For any algebra $\Al$, the Leibniz congruence naturally determines a map, called the \emph{Leibniz operator}, from the power set of $A$ to the set of all congruences of $\Al$, for which we use the same symbol as for the Leibniz congruence: $\leib: P(A) \rightarrow \textrm{Con}(\Al)$. 

Recalling that the sets  $P(A)$ and $\textrm{Con}(\Al)$ are both lattices, one sees that it makes sense to consider properties of the Leibniz operator such as injectivity, surjectivity, but also monotonicity, etc. The study of these  properties is  very important in Abstract Algebraic Logic and it allowed to build a hierarchy of logics (called the \emph{Leibniz hierarchy}) which presents a classification of all logics (in the sense defined above) according to their algebraic behaviour. 

There are, for instance, logics that have a very close relationship with their associated classes of algebras, so that most or all of the interesting properties of the logic can be formulated and proved as properties of the associated class of algebras and viceversa. These logics, known as \emph{algebraizable logics}, appear at the top of the hierarchy: among them are classical logic, intuitionistic logic, many fuzzy logics, etc. The logic $\lbs$, that we will study in Chapter \ref{ch:add}, is also an example of algebraizable logic. 

At the other end of the Leibniz hierarchy is the class of  \emph{protoalgebraic logics}, which has a special interest for our work. It is the broader class that includes all logics that are, so to speak, reasonably ``well-behaved'' from an algebraic point of view. 
Both classes, that of algebraizable and of protoalgebraic logics, can be characterized in terms of the behaviour of the Leibniz operator: the protoalgebraic, for instance, are the logics for which the the Leibniz operator is monotone on the set of all filters of the logic.

 The general theory of Abstract Algebraic Logic provides a method to associate with any logic $ \mathcal{L}$ a canonical class of algebraic models, sometimes called the \emph{algebraic counterpart of} $ \mathcal{L}$, defined as the class of algebraic reducts of all reduced matrices of $ \mathcal{L}$, and denoted by $\algstar \mathcal{L}$. This method works very well for  protoalgebraic logics, but there are examples of non-protoalgebraic logics in which we do not get a satisfactory result, in the sense that the class of algebras we obtain does not coincide with the one that seems most natural for a given logic. 
 
One way of overcoming this difficulty is to work not with matrices but with generalized matrices. By \emph{generalized matrix} or \emph{g-matrix} we mean a pair $\la \Al, \mathcal{C} \ra $, where $\Al $ is an algebra and $ \mathcal{C}$ is a closure system on the set $A$. From this perspective, a logic $\mathcal{L}$ can be seen as a particular case of generalized matrix of the form $\la \Al[Fm], \mathcal{T}h\mathcal{L}\ra $. 

Instead of g-matrices, it is sometimes more convenient to work with the equivalent notion of \emph{abstract logic}, by this meaning a structure  $\la \Al, \mathbf{C} \ra $ where $\Al$ is an algebra and  $\mathbf{C}$ a closure operator on $A$.

A semantics of g-matrices may be developed as a natural generalization of the semantics of matrices sketched before. To a given g-matrix $\la \Al, \mathcal{C} \ra $ we may associate a logic by defining   $\Gamma \vDash_{\la \Al, \mathcal{C} \ra} \phi $ if and only if, for any homomorphism $h: \Al[Fm] \rightarrow  \Al$ we have 
$h(\phi) \subseteq \mathbf{C} (h[\Gamma)] )$, where $\mathbf{C}$ is the closure operator corresponding to $\mathcal{C}$. Similarly, we say that a g-matrix $\la \Al, \mathcal{C} \ra $ is a g-model of a  logic $\mathcal{L} $ when $ \mathcal{C} \subseteq  \Fi_{\mathcal{L}} \Al$.

The role of the Leibniz congruence is played in this context by the \emph{Tarski congruence} of a g-matrix $\la \Al, \mathcal{C} \ra $, usually denoted by $\tarski_{\Al} \mathcal{C}$, and defined as the greatest congruence compatible with all $F \in \mathcal{C}$. The Tarski congruence can be characterized in terms of the Leibniz congruence, as follows:
$$\tarski_{\Al} \mathcal{C} = \bigcap_{F \in \mathcal{C}} \leibniz_{\Al}F.$$


The Tarski congruence can be equivalently defined as the greatest congruence below the interderivability relation, which in AAL contexts is usually called the \emph{Frege relation}. For a given closure operator $\mathbf{C}$ on a set $A$,  the Frege relation $\mathbf{\Lambda C} $ is defined as follows: 
$$\mathbf{\Lambda C} = \{ \la a, b  \ra \in A \times A : \mathbf{C}(a) = \mathbf{C}(b)  \}. $$

It is obvious that, if $\mathbf{C}$ is the closure operator associated with some logical consequence relation $ \vdash_{\mathcal{L}}$, then the Frege relation corresponds to the interderivability relation, which we usually denote $ \dashv \vdash_{\mathcal{L}}$.

An alternative definition of the Tarski congruence of a g-matrix $\la \Al, \mathcal{C} \ra$ is thus the following:
$$\tarski_{\Al} \mathcal{C} =  \max \{ \theta \in \textrm{Con} \Al : \theta \subseteq \mathbf{\Lambda C}  \} .$$

We say that a g-matrix is \emph{reduced} when its Tarski congruence is the identity. We may then associate to a logic $\mathcal{L}$ another class of algebras, which we denote by $\alg \mathcal{L}$, defined as the class of algebraic reducts of all reduced g-matrices of  $ \mathcal{L}$.

A central notion is also that of \emph{bilogical morphism} between two g-matrices $\la \Al, \mathcal{C} \ra$ and $\la \Al', \mathcal{C}' \ra$: by this we mean an epimorphism $h: \Al \rightarrow \Al' $ such that $\mathcal{C} = \{ h^{-1}[T] : T \in \mathcal{C}' \}$. In terms of closure operators, the previous condition may be expressed as follows: $a \in \mathbf{C}(X)$ if and only if $h(a) \in \mathbf{C'}(h[X])$ for all $a \in A$ and all $X \subseteq A$.

Using the notion of bilogical morphism it is possible to isolate an interesting subclass of the g-models of a logic $\mathcal{L}$: the class of full models of $\mathcal{L}$. A g-matrix $\la \Al, \mathcal{C} \ra$  is a \emph{full model} of a logic $\mathcal{L}$ when there is a bilogical morphism between $\la \Al, \mathcal{C} \ra$ and a g-matrix of the form $\la \Al', \Fi_{\mathcal{L}} \Al' \ra$. These special models are particularly significant because they inherit some interesting metalogical properties from the corresponding logic, something which does not hold for all models (we shall see an example of this in Chapter 
\ref{ch:log}). It is also worth noting that $\alg \mathcal{L}$ can be alternatively defined as the class of algebraic reducts of reduced full models.

The theory of g-matrices allows to obtain results that can be legitimately considered generalizations of those relative to matrices. For our purposes, it is useful to recall that, for any logic $ \mathcal{L}$, we have $ \algstar \mathcal{L} \subseteq \alg \mathcal{L}$. More precisely, we have that $ \alg \mathcal{L} = \mathbb{P}_{SD}\algstar \mathcal{L}$, where $ \mathbb{P}_{SD}$ denotes the subdirect product operator.

For most logics the two classes are indeed identical: in particular, it is a well-known result that for protoalgebraic logics they must coincide. It is interesting to note that, in the known  cases where they do not coincide, it is the class  $ \alg \mathcal{L}$ that seems to be the more naturally associated with the logic $\mathcal{L}$: examples of this include the $\{ \land, \lor\}$-fragment of classical propositional logic, the Belnap-Dunn logic and, as we shall see in Chapter 
\ref{ch:log}, also the logic $\lb$.

It is interesting to observe that in many cases, including those we have just mentioned, the class of algebras naturally associated with a logical system can be obtained also through another process of algebraization, which can be seen as a generalization of the one introduced by Blok and Pigozzi. This is achieved by shifting our attention from logics conceived as deductive systems (semantically defined, or through Hilbert-style calculi) to logics conceived as Gentzen systems. This study, developed in \cite{ReV93b} and \cite{ReV95-p}, led to the definition of a notion of algebraizability for Gentzen systems parallel to the standard one for sentential logics. It turns out that some logical systems, especially logics without implication, although not algebraizable (or not even protoalgebraic), have an associated Gentzen system that is algebraizable. This is true, as we shall see, also of the logic $\lb$.


\section{Pre-bilattices and bilattices}
\label{sec:bil}

In this section we collect the basic definitions and some known results on bilattices that will be used thoughout our work. First of all, let us note that the terminology concerning bilattices is not
uniform\footnote{
This was already pointed out in \cite[p.~111]{MoPiSlVo00}.
}, not even as far as the basic definitions are concerned. In this work we shall  reserve the name
``bilattice'' to the algebraic structures that sometimes are called ``bilattices with negation'': this terminology seems to us to be the most perspicuous, 
and is becoming more or less standard in recent
papers about bilattices.

\begin{definition} \label{def:pbl}

A \emph{pre-bilattice} is an algebra $\Al[B] = \left\langle B, \land,
\lor, \otimes, \oplus \right\rangle$  such that $\langle B, \land, \lor
\rangle$ and $\langle B, \otimes, \oplus \rangle$ are both lattices. 
\end{definition} 

The
order associated with the lattice $\langle B, \land, \lor \rangle$, which we shall sometimes call the \emph{truth lattice} or \emph{t-lattice}, is
denoted by $\leq_t$ and is called the \emph{truth order}, while the order
$\leq_k$ associated with $\langle B, \otimes, \oplus \rangle$, sometimes called the \emph{knowledge lattice} or \emph{k-lattice},  is the
\emph{knowledge order}. 

As it happens with lattices, a pre-bilattice can be also viewed as a (doubly) partially ordered set. When focusing our attention on this aspect, we will denote a pre-bilattice by $ \left\langle B, \leq_{t}, \leq_{k} \right\rangle$ instead of $ \left\langle B, \land,
\lor, \otimes, \oplus \right\rangle$.

Usually in the literature it is required that
the lattices be complete or at least bounded, but here none of these
assumptions is made. The minimum and maximum of the truth lattice, in
case they exist, will be denoted by $\false$ and $\true$; similarly,
$\bot$ and $\top$ will refer to the minimum and maximum of the knowledge
lattice. 

Of course the interest on pre-bilattices increases when there is some
connection between the two orders. At least two ways of establishing such a connection have been investigated in the literature. The first one is to impose certain monotonicity 
properties to the connectives of the two orders, as in the following definition, due to Fitting \cite{Fi90}. 

\begin{definition} \label{def:intpbl}
A pre-bilattice  $\Al[B] = \left\langle B, \land,
\lor, \otimes, \oplus \right\rangle$  is
\emph{interlaced} whenever each one of the four lattice operations $\land,
\lor, \otimes$ and $\oplus$ is monotonic with respect to both partial
orders $\leq_t$ and $\leq_k$. That is, when the following quasi-equations hold:
\begin{equation*}
  \label{eq:interlaced}
\begin{split}
 x \leq_t y \: & \Rightarrow \: x \otimes z \leq_t y \otimes z  \qquad
 \qquad x \leq_t y \: \Rightarrow \: x \oplus z \leq_t y \oplus z \\
  x \leq_k y \: & \Rightarrow \: x \land z \leq_k y \land z  
  \qquad
 \qquad x \leq_k y \: \Rightarrow \: x \lor z \leq_k y \lor z.  
\end{split}
\end{equation*}
(Here, of course, the inequality $x \leq_t y$ is an abbreviation for the identity $ x \land y \approx x$, and similarly $x \leq_k y$ stands for $ x \otimes y \approx x$.)
\end{definition} 

A weaker notion, called \emph{regularity}, has been considered by Pynko \cite{Py00}: a pre-bilattice is \emph{regular} if it satisfies the last two quasi-equations of Definition \ref{def:intpbl}, i.e.\ if the truth lattice operations are monotonic w.r.t. the knowledge order. In the present work we shall not deal with this weaker notion, but it may be worth noting that from Pynko's results it follows that, for bounded pre-bilattices, being regular is equivalent to being interlaced.

On the other hand, the interlacing conditions may be strengthened through the following definition due to Ginsberg \cite{Gi88}:

\begin{definition} \label{def:distr}
A pre-bilattice is \emph{distributive} when
all twelve distributive laws concerning the four lattice operations, i.e.\ any identity of the following form, hold:
\begin{equation*}
  \label{eq:dist}
\begin{split}
 x \circ ( y \bullet z )  \approx (x \circ y) \bullet (x
 \circ z) \mbox{\qquad for every } \circ, \bullet \in \{ \land, \lor,
 \otimes, \oplus \} \mbox{ with } \circ \neq \bullet.
\end{split}
\end{equation*}
\end{definition}

We will denote, respectively, the classes of
pre-bilattices, of interlaced pre-bilattices and of distributive
pre-bilattices by $\PBL$, $\IPBL$ and $\DPBL$. 

Obviously $\PBL$ is an equational class, axiomatized by the lattice identities for the two lattices, and so is  $\DPBL$, which can be axiomatized by adding the twelve distributive laws to the lattice identities (this axiomatization is of course not minimal, since not all distributive laws are independent from each other). It is known that $\IPBL$ is also a variety\footnote{
A proof of this fact can be found in \cite{Av96a}: even if Avron assumes that pre-bilattices are always bounded in both orders, it is easy to check that his proofs do not use such an  assumption.
}, axiomatized by the identities for pre-bilattices, plus the following ones: 
\begin{equation*}
  \label{eq:inteq}
\begin{split}
(x \land y) \otimes z & \leq_t y \otimes z \qquad
 \qquad (x \land y) \oplus z \leq_t y \oplus z \\
(x \otimes y) \land z & \leq_k y \land z \qquad
 \qquad (x \otimes y) \lor z \leq_k y  \lor z.
\end{split}
\end{equation*}

It is also known, and easily checked, that being distributive implies being interlaced: hence we have that $\DPBL \subseteq \IPBL \subseteq \PBL$, and all of these inclusions are strict, as we shall see later examining some examples of bilattices.

From an algebraic point of view, $\IPBL $ is perhaps the most interesting subclass of pre-bilattices: its interest lies mainly in the fact that any interlaced pre-bilattice can be represented as a special kind of product of two lattices. This result is well known for bounded pre-bilattices, but in the present work we will generalize it to the unbounded case. 

Focusing on the bounded case, we may list some basic properties of interlaced pre-bilattices (all proofs can be found in \cite{Av96a}). 

\begin{proposition} \label{prop:intpbl}

Let  $\Al[B] =\left\langle B, \land, \lor,
\otimes, \oplus, \false, \true, \bot, \top \right\rangle$ be a bounded interlaced pre-bilattice. Then the following equations are satisfied: 

\begin{equation}
  \label{eq:intbo1}
\begin{split}
\false \otimes  \true & \approx \bot
 \qquad \qquad 
\false \oplus \true \approx \top  
 \\  
 \bot \land \top & \approx \false
 \qquad \qquad 
\bot \lor \top \approx \true
 \end{split}
\end{equation} 
\begin{equation}
  \label{eq:intbo2}
\begin{split}
x \land \bot \approx x \otimes \false \qquad \qquad
x \land \top \approx x \oplus \false \\
x \lor \bot \approx x \otimes \true \qquad \qquad
x \lor \top \approx x \oplus \true
 \end{split}
\end{equation} 
\begin{equation}
 \label{eq:intbo3}
\begin{split}
x & \approx (x \land \bot ) \oplus  (x \lor \bot )   \approx (x \otimes \false ) \oplus  (x \otimes \true ) \\ 
x & \approx (x \land \top ) \otimes  (x \lor \top )   \approx (x \oplus \false ) \oplus  (x \oplus \true )  \\ 
 x & \approx (x \otimes \false ) \lor  (x \oplus \false )  \approx (x  \land \bot ) \lor  (x \land \top )  \\ 
x & \approx (x \otimes \true ) \land  (x \oplus \true )   \approx (x  \lor \bot ) \land  (x \lor \top ).
 \end{split}
\end{equation} 

\begin{equation}
 \label{eq:intbo4}
\begin{split}
x \wedge y & \approx (x \otimes \mathsf{f}) \oplus (y \otimes \mathsf{f}) \oplus (x \otimes y \otimes \mathsf{t}) \\ 
x \vee y & \approx (x \otimes \mathsf{t}) \oplus (y \otimes \mathsf{t}) \oplus (x \otimes y \otimes \mathsf{f}) \\
x \otimes y & \approx  (x \wedge \bot) \vee (y \wedge \bot) \vee (x \wedge y \wedge \top)  \\
x \oplus y & \approx   (x \wedge \top) \vee (y \wedge \top) \vee (x \wedge y \wedge \bot).
 \end{split}
\end{equation} 

\end{proposition} 

The last four equations (\ref{eq:intbo4}) show that in the bounded case we can explicitely define the lattice operations of one of the lattice orders using the operations of the other order. Indeed, a stronger and interesting result, due to Avron \cite{Av96a}, can be stated.

 Given a lattice $\Al[L] =\left\langle L, \otimes, \oplus \right\rangle$, we say that an element $a \in L $ is \emph{distributive} when each equation of the form $x \circ ( y \bullet z )  \approx (x \circ y) \bullet (x  \circ z)$, where  $\circ, \bullet \in \{  \otimes, \oplus \} $, holds in case $a = x$ or $a = y $ or $a = z$. Now we have the following:
 
 \begin{proposition} \label{prop:poleq}
 Let $\Al[B] =\left\langle B, \otimes, \oplus,  \bot, \top \right\rangle$ be a bounded lattice, with minimum $ \bot$ and maximum $\top$, such that there are distributive elements $\false, \true \in B$ which are complements of each other, i.e.\ satisfying   that $\false \otimes  \true = \bot
$ and $\false \oplus \true = \top  $. Then the structure $\Al[B] =\left\langle B, \land, \lor,
\otimes, \oplus, \false, \true, \bot, \top \right\rangle$, where the operations $\land$ and $\lor$ are defined as in Proposition \ref{prop:intpbl} (\ref{eq:intbo4}), is a bounded interlaced pre-bilattice. 
  \end{proposition} 
It is clear, by duality, that a similar result can be proved starting from the bounded lattice $\Al[B] =\left\langle B, \land, \lor,  \false, \true \right\rangle$.

Notice that none of the conditions we have considered so far precludes the possibility that a pre-bilattice be degenerated, in the sense that the two orders may coincide, or that one may be the dual of the other (we will come back to this observation when we deal with product pre-bilattices). These  somehow less interesting cases are ruled out when we come to the second way of connecting the two lattice orders, which consists in expanding the algebraic language with a unary operator. This is the method Ginsberg  originally
used to introduce bilattices.  

\begin{definition} \label{def:bil}
A \emph{bilattice} is an algebra $\Al[B] =\left\langle B, \land, \lor,
\otimes, \oplus, \neg \right\rangle$ such that the reduct
$\left\langle B, \land, \lor,
\otimes, \oplus \right\rangle$ is a pre-bilattice and the
\emph{negation} $\neg$ is a unary operation satisfying that for every $a,b \in B$,
\begin{enumerate}[]
 \item \quad \textbf{(neg1)} \qquad if $a \leq_{t} b$,  then  $\neg b \leq_{t} \neg a$
 \item \quad \textbf{(neg2)} \qquad  if $a \leq_{k} b$,  then  $\neg a \leq_{k} \neg b$ 
 \item \quad \textbf{(neg3)} \qquad  $ a = \neg \neg a  $.
\end{enumerate}
\end{definition}

The interlacing and distributivity properties extend to bilattices in the obvious way: we say that a bilattice is interlaced (distributive) when its pre-bilattice reduct is interlaced (distributive). 

Figure  \ref{fig:hasse1} shows the double Hasse diagram of some of the most important pre-bilattices. The diagrams should be read as follows: $a \leq_t b$ if there is a path from $a$ to $b$ which goes uniformly from left to right, while  $a \leq_k b$ if  there is a path from $a$ to $b$ which goes uniformly from the bottom to the top\footnote{ It is worth pointing out that, unlike lattices, not all finite bilattices can be represented in this way: for more on this, see the notions introduced by Avron \cite{Av95} of ``graphically
  representable'' and ``precisely representable'' pre-bilattice.}. The four lattice operations are thus uniquely determined by the diagram, while  negation, if there is one, corresponds to reflection along the vertical axis connecting $\bot$ and $\top$. 
  
  It is then clear that all the pre-bilattices shown in Figure  \ref{fig:hasse1} can be endowed with a negation in a unique way, and so turned into bilattices. When no confusion is likely to arise, we shall use the same name to denote a particular pre-bilattice and its associated bilattice: the names used in the diagrams are by now more or less standard in the literature ($\7 $ is sometimes called $\default$, which is the name originally used by Ginsberg \cite{Gi88}, since this bilattice was introduced with applications to default logic in mind).

\begin{center}
\begin{figure}[t]
\vspace{15pt}
\begin{center}
\begin{tabular}{cccc}
\vspace{5pt}

\begin{minipage}{2cm}
\setlength{\unitlength}{1.2cm}
\begin{center}
\begin{picture}(2,2)(0.15,0)
\put(1,0){\makebox(0,0)[l]{ $\bot$}} 
\put(0,1){\makebox(0,0)[r]{$\false$ }} 
\put(2,1){\makebox(0,0)[l]{ $\true$}} 
\put(1,2){\makebox(0,0)[r]{$\top$ }} 

\put(1,0){\circle*{0.2}} 
\put(0,1){\circle*{0.2}} 
\put(2,1){\circle*{0.2}}
\put(1,2){\circle*{0.2}}

\put(1,0){\line(1,1){1}} 
\put(1,0){\line(-1,1){1}} 
\put(0,1){\line(1,1){1}} 
\put(2,1){\line(-1,1){1}}

\end{picture}
\end{center}
\end{minipage}

&

\begin{minipage}{3cm}
\setlength{\unitlength}{0.8cm}
\begin{center}
\begin{picture}(2,2)(0,-0.25)
\put(1,-1){\makebox(0,0)[l]{ $\bot$}} 
\put(1,0){\makebox(0,0)[l]{ $a$}} 
\put(0,1){\makebox(0,0)[r]{$\false$ }} 
\put(2,1){\makebox(0,0)[l]{ $\true$}} 
\put(1,2){\makebox(0,0)[r]{$\top$ }} 

\put(1,-1){\circle*{0.2}} 
\put(1,0){\circle*{0.2}} 
\put(0,1){\circle*{0.2}} 
\put(2,1){\circle*{0.2}}
\put(1,2){\circle*{0.2}}

\put(1,0){\line(1,1){1}} 
\put(1,0){\line(-1,1){1}} 
\put(0,1){\line(1,1){1}} 
\put(2,1){\line(-1,1){1}} 

\put(1,-1){\line(0,1){1}} 
\put(1,-1){\line(-1,2){1}} 
\put(1,-1){\line(1,2){1}} 

\end{picture}
\end{center}
\end{minipage}

&

\begin{minipage}{3cm}
\setlength{\unitlength}{0.7cm}
\begin{center}
\begin{picture}(3,4)(-0.25,0)
\multiput(1,0)(1,1){3}{\circle*{0.3}}
\multiput(0,1)(1,1){3}{\circle*{0.3}}
\multiput(-1,2)(1,1){3}{\circle*{0.3}}

\put(1,0){\line(1,1){2}} 
\put(0,1){\line(1,1){2}} 
\put(-1,2){\line(1,1){2}} 
\put(1,0){\line(-1,1){2}} 
\put(2,1){\line(-1,1){2}} 
\put(3,2){\line(-1,1){2}} 

\put(1,0){\makebox(0,0)[l]{ $\bot$}} 
\put(-1,2){\makebox(0,0)[r]{$\false$ }} 
\put(3,2){\makebox(0,0)[l]{ $\true$}} 
\put(1,4){\makebox(0,0)[r]{$\top$ }} 

\end{picture}
\end{center}
\end{minipage}

&

\begin{minipage}{3cm}
\setlength{\unitlength}{0.9cm}
\begin{center}
\begin{picture}(2,2)(0,-0.25)
\put(1,-1){\makebox(0,0)[l]{ $\bot$}} 
\put(0.5,-0.5){\makebox(0,0)[r]{$a$ }} 
\put(1.5,-0.5){\makebox(0,0)[l]{ $b$}} 
\put(1,0){\makebox(0,0)[l]{ $c$}} 
\put(0,1){\makebox(0,0)[r]{$\false$ }} 
\put(2,1){\makebox(0,0)[l]{ $\true$}} 
\put(1,2){\makebox(0,0)[r]{$\top$ }} 

\put(1,-1){\circle*{0.2}} 
\put(0.5,-0.5){\circle*{0.2}}
\put(1.5,-0.5){\circle*{0.2}}
\put(1,0){\circle*{0.2}} 
\put(0,1){\circle*{0.2}} 
\put(2,1){\circle*{0.2}}
\put(1,2){\circle*{0.2}}

\put(1,0){\line(1,1){1}} 
\put(1,0){\line(-1,1){1}} 
\put(0,1){\line(1,1){1}} 
\put(2,1){\line(-1,1){1}} 

\put(1,-1){\line(-1,1){0.5}} 
\put(1,-1){\line(1,1){0.5}} 

\put(1,0){\line(1,-1){0.5}} 
\put(1,0){\line(-1,-1){0.5}} 

\put(0,1){\line(1,-3){0.5}} 
\put(2,1){\line(-1,-3){0.5}} 

\end{picture}
\end{center}
\end{minipage}
\\ \\

$\four$ & $\five$ & $\nine$ & $\7$
\end{tabular}

\caption{Some examples of (pre-)bilattices} \label{fig:hasse1}
\end{center}
\end{figure}
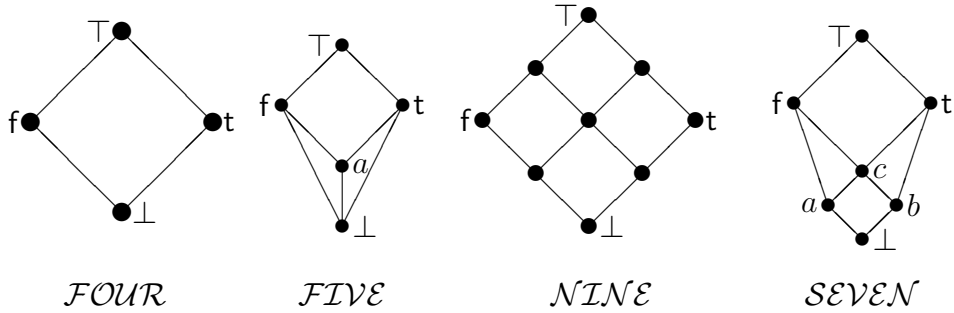
\end{center}

The smallest non-trivial bilattice is $\four$. This algebra has a key role among bilattices, both from an algebraic and from a logical point of view, as we shall see. 

$\four$ is distributive and, as a bilattice, it is a simple algebra. In fact it is, up to isomorphism, the only subdirectly irreducible bounded distributive bilattice (this is proved,  for instance, in \cite{MoPiSlVo00}). 

Let us also note that the $\{ \wedge, \vee, \neg \}$-reduct of $\four$ coincides with the four-element De Morgan algebra that was used by Belnap \cite{Be76} to define the Belnap-Dunn four-valued logic. In fact, we shall see that the logic of distributive bilattices (both with and without implication) turns out to be a conservative expansion  of the Belnap-Dunn logic.

\begin{proposition}[De Morgan laws]
\label{prop:dm}

The following equations  hold in any bilattice: 
\begin{align*}
 \label{eq:dm}
\neg (x \land y) & \approx \neg x \lor \neg y     &&& \neg (x \lor y)  & \approx \neg x \land \neg y \\
\neg (x \otimes y) & \approx \neg x \otimes \neg y &&&   \neg (x \oplus y) & \approx \neg x \oplus \neg y.
\end{align*} 
Moreover, if the bilattice is bounded, then $\neg \top = \top$, $\neg \bot = \bot$, $\neg \true = \false$ and
$\neg \false = \true$.
\end{proposition}

So, if a bilattice $\Al[B] =\left\langle B, \land, \lor, \otimes, \oplus, \neg \right\rangle$ is distributive, or at least the truth lattice of $\Al[B]$ is distributive, then the reduct $\la B, \wedge, \vee, \neg\ra$ is a De Morgan lattice. It is also easy to check that the four De Morgan laws imply that the negation operator satisfies \textbf{(neg1)} and \textbf{(neg2)}.  Then, it is obvious that the class of
bilattices, denoted by $\BL$, is equationally axiomatizable. Analogously
to what we did in the case of pre-bilattices, we will
denote by $\IBL$ and $\DBL$ the classes of \emph{interlaced bilattices} and
\emph{distributive bilattices}, which are also equationally
axiomatizable. It is obvious that $\DBL \subseteq \IBL \subseteq \BL$, and these inclusions are all strict, as we shall see presently.

Further expansions of the similarity type $\{ \land, \lor,
\otimes, \oplus, \neg \}$, which may be considered the standard bilattice language, have also been considered in the literature. Fitting  \cite{Fi94}, for instance, introduced a kind of dual negation operator, which he called \emph{conflation}, and an implication-like connective called \emph{guard}, while Arieli and Avron \cite{ArAv96} investigated different choices for a bilattice implication. However, throughout this work we will always deal only with the basic language $\{ \land, \lor, \otimes, \oplus, \neg \}$, except for the last two chapters, where we will consider the expansion obtained by adding one of Arieli and Avron's implication connectives. 

An interesting class of (pre-)bilattices can be constructed as a kind of product of two lattices. We shall see that this construction, due to Fitting\footnote{The essential of the definition are already in \cite{Gi88}, but Ginsberg considered only a special case of the construction, what he called ``world-based bilattices''. } \cite{Fi90} has a natural intuitive interpretation, and gives rise to a class of structures that enjoys nice algebraic properties.

\begin{definition} \label{defn:prodpbl}
Let $\Al[L_1] = \langle L_1,  \sqcap_1, \sqcup_1 \rangle$ and
$\Al[L_2] = \langle L_2,  \sqcap_2, \sqcup_2 \rangle$ be two
lattices with associated orders $\leq_1$ and $\leq_2$. Then the
\emph{product pre-bilattice}  $\Al[L_1] \odot \Al[L_2] = \langle L_{1}
\times L_{2}, \land, \lor, \otimes, \oplus \rangle$ is defined as
follows. For all $\left\langle a_{1}, a_{2} \right\rangle, \left\langle
b_{1}, b_{2} \right\rangle \in L_{1} \times L_{2}$, 
\begin{align*}
\left\langle a_{1}, a_{2} \right\rangle \wedge \left\langle b_{1}, b_{2} \right\rangle &= \left\langle a_{1} \sqcap_{1} b_{1} ,\: a_{2} \sqcup_{2} b_{2} \right\rangle  \\ 
\left\langle a_{1}, a_{2} \right\rangle \vee \left\langle b_{1}, b_{2} \right\rangle &= \left\langle a_{1} \sqcup_{1} b_{1} ,\: a_{2} \sqcap_{2} b_{2} \right\rangle \\ 
\left\langle a_{1}, a_{2} \right\rangle \otimes \left\langle b_{1}, b_{2} \right\rangle &= \left\langle a_{1} \sqcap_{1} b_{1} ,\: a_{2} \sqcap_{2} b_{2} \right\rangle \\ 
\left\langle a_{1}, a_{2} \right\rangle \oplus \left\langle b_{1}, b_{2}
\right\rangle &= \left\langle a_{1} \sqcup_{1} b_{1} ,\: a_{2}
\sqcup_{2} b_{2} \right\rangle.
\end{align*} 
\end{definition}

It easy to check that the structure $\Al[L_1] \odot \Al[L_2]$ is always
an interlaced pre-bilattice, and it is distributive if and only if both  $\Al[L_1]$ and $\Al[L_2]$ are distributive. From the definition it is also obvious that 
\begin{center}
  $\la a_1,a_2 \ra \leq_k \la b_1, b_1 \ra $ \qquad iff \qquad $a_1 \leq_1 b_1$ and
  $a_2 \leq_2 b_2$
\end{center}
and
\begin{center}
  $ \la a_1,a_2 \ra \leq_t  \la b_1, b_1 \ra$ \qquad iff \qquad $a_1 \leq_1 b_1$ and
  $a_2 \geq_2 b_2$.
\end{center}

The construction, as we have said, has a natural interpretation: we can think of the first component of each element of the form $\la a_1,a_2 \ra$ as representing evidence \emph{for} the truth of some sentence, while the second component can be thought of as representing the evidence \emph{against}  the truth (or for the falsity) of that sentence. 

It is not difficult to convince oneself that the truth lattice operations $\land$ and $ \lor$  act on each component according to our intuitions, as generalizations of classical conjunction and disjunction: for instance $\land$ takes the infimum of the ``truth component'' and the supremum of the ``falsity component''. More unusual, perhaps, are the two knowledge lattice connectives. As Fitting \cite{Fi91b} puts it: 

\begin{quotation}
If we think of $\leq_k$ as being an ordering by knowledge, then $\otimes$ is a consensus operator: $p \otimes q$ is the most that $p$ and $q$ can agree on. Likewise $\oplus$ is a `gullability' operator: $p \oplus q$  accepts and combines the knowledge of $p$ with that of $q$, whether or not there is a conflict. Loosely, it believes whatever it is told.
\end{quotation}

If the two lattices $\Al[L_1]$ and $\Al[L_2]$ are isomorphic (so we may assume that they coincide, and denote both lattices just by $\Al[L]$), then it is possible to define a negation in $\Al[L] \odot \Al[L]$, so we speak of \emph{product bilattice} instead of  product pre-bilattice. Negation is defined as 
$$\neg \la a_1,a_2 \ra =  \la a_2, a_1 \ra.$$

 Once again, it is easy to see that the behaviour of this operation is consistent with the intuitive interpretation we have proposed.
 
Using the construction we have described, we are now able to settle the question of whether the inclusions between the subvarieties of (pre-)bilattices mentioned above are strict. It is easy to see that $\5$ and $\7$ are not interlaced, hence we have $\IPBL \varsubsetneq \PBL$. To see that $\DPBL \varsubsetneq \IPBL $  it is enough to consider a product pre-bilattice $\Al[L] \odot \Al[L]$ where $\Al[L] $ is a non-distributive lattice. Since all the examples of pre-bilattices considered can be turned into bilattices, as an immediate consequence we also have  $\DBL \varsubsetneq \IBL \varsubsetneq \BL$. 
 
Before proceeding, let us note that there is an important difference between the two variants of the construction described; this fact, although easily seen, has not  received much attention  in the literature on bilattices so far. The difference is that the product pre-bilattice construction can be regarded as a particular case of a direct product, while this is not the case for the product bilattice.
 
As anticipated above, all
lattices $\Al[L] = \langle L, \sqcap, \sqcup \rangle$ can be seen as
degenerated pre-bilattices in at least four different ways. We can
consider the following four algebras:
\begin{align*}
\Al[L^{++}] &  = \langle L, \sqcap, \sqcup, \sqcap, \sqcup \rangle \\
\Al[L^{+-}]  &  = \langle L, \sqcap, \sqcup, \sqcup, \sqcap \rangle \\
\Al[L^{-+}] & = \langle L, \sqcup, \sqcap, \sqcap, \sqcup \rangle  \\
\Al[L^{--}] &  = \langle L, \sqcup, \sqcap, \sqcup, \sqcap \rangle.
\end{align*}

The first superscript, $+$ or $-$, says whether we are taking as
truth order the same order than in the original lattice or the dual one;
and the second superscript refers to the same for the knowledge order.
Using this notation, it is easy to see that the product pre-bilattice
$\Al[L_1] \odot \Al[L_2]$ coincides with the direct product
$\Al[L_1^{++}] \times \Al[L_2^{-+}]$. Notice also  that 
$\Al[L_1^{++}] = \langle L_1, \sqcap_1, \sqcup_1, \sqcap_1, \sqcup_1 \rangle$ and
$\Al[L_2^{-+}] = \langle L_2, \sqcup_2, \sqcap_2, \sqcap_2,  \sqcup_2 \rangle$.

In the next chapter we will come back to this construction, relating it to the representation theorem for unbounded pre-bilattices; for now it suffices to note that, of course, the product bilattice is not a direct product, because in general the factor lattice need not have a negation.

We close this section stating the known representation theorem in its two versions: for bounded interlaced pre-bilattices and for bounded interlaced bilattices. This theorem has been stated and proved in several works, several versions, and different degrees of generality\footnote{
For a brief review of these versions, see \cite{MoPiSlVo00}.}. The last and perhaps deeper work on it, and in general on interlaced bounded (pre-)bilattices, is Avron's \cite{Av96a}.

\begin{thm}[Representation, 1]
  \label{thm:reprpre}
  Let $\Al[B]$ be a bounded pre-bilattice. The following statements are
  equivalent.
  \begin{enumerate}[(i)]
    \item $\Al[B]$ is an interlaced pre-bilattice.
    \item There are two bounded lattices $\Al[L_1]$ and $\Al[L_2]$ such that
      $\Al[B]$ is isomorphic to $\Al[L_1] \odot \Al[L_2]$.
  \end{enumerate}
\end{thm}

Although, as we have pointed out, many versions of the theorem are to be found in the literature, all of them use essentially the same proof strategy, of which we present here a sketch in order to help understand why this kind of proof does not work in the unbounded case. 

Of course, that (ii) implies (i) is immediate. To prove the other implication we need to construct $\Al[L_1]$ and $\Al[L_2]$. This can be done by considering principal upsets and/or downsets  of some of the bounds, 
together with the lattice operations inherited from the pre-bilattice. For this, having just one of the bounds is sufficient; of course, if we use $\bot$ or $\top$, then we have to consider upsets and downsets relative to the truth order, and similarly with $\true$ or $\false$ we need to use the knowledge order. 

Let us take, for instance, $\bot$ and the order $\leq_t$. Then we have 
\begin{align*}
\Al[L_1] & =  \la \{  a \in B : a \geq_t \bot \}, \otimes, \oplus, \bot, \true  \ra = \la \{  a \in B : a \geq_t \bot \}, \land, \lor, \bot, \true  \ra \\
\Al[L_2] & =  \la \{  a \in B : a \leq_t \bot \}, \otimes, \oplus, \bot, \false  \ra = \la \{  a \in B : a \leq_t \bot \}, \lor, \land, \bot, \false  \ra.
\end{align*}

Taking a look at the Hasse diagrams in Figure \ref{fig:hasse1}, one may observe that, from a geometrical point of view, we are making a kind of projection of each point of the pre-bilattice on the two axes connecting $\bot$ to $\true$ and $\bot$ to $\false$, fixing  $\bot$ as the origin. 

The isomorphism $h: B  \rightarrow  L_1 \times L_2 $ is in this case defined as, for all $a \in B$,  $$h(a) = \la a \lor \bot, a \land \bot \ra.$$ Its inverse $h^{-1}:  L_1 \times L_2  \rightarrow B  $ is defined as $$h^{-1}(\la a_1, a_2 \ra) = a_1 \oplus a_2. $$ Injectivity of these maps is easily proved using Proposition \ref{prop:intpbl} (\ref{eq:intbo2}) and (\ref{eq:intbo3}), which can be also used to give altenative decompositions, using the other bounds of the pre-bilattice. We stress that the key point here is that there is at least one bound (geometrically, a point which can be taken to be the origin of the axes on which we are making the projections).

The representation theorem for bilattices is just a special case of the former:

\begin{thm}[Representation, 2]
  \label{thm:reprbil}
  Let $\Al[B]$ be a bounded bilattice. The following statements are
  equivalent.
  \begin{enumerate}[(i)]
    \item $\Al[B]$ is an interlaced bilattice.
    \item There is a bounded lattice $\Al[L]$  such that
      $\Al[B]$ is isomorphic to $\Al[L] \odot \Al[L]$.
  \end{enumerate}
\end{thm}

Everything works as in the case of pre-bilattices, but now we have that $\Al[L_1]$ and $\Al[L_2]$ are isomorphic via the map given by the negation operator.  


As a corollary of the representation theorem, we get a characterization of subdirectly irreducible bounded interlaced (pre-)bilattices (see for instance \cite{MoPiSlVo00}). We have that a bounded pre-bilattice $\Al[L_1] \odot \Al[L_2]$ is subdirectly irreducible if and only if  $\Al[L_1] $ is a  subdirectly irreducible lattice and $\Al[L_2] $ is trivial or viceversa,  $\Al[L_2] $ is a  subdirectly irreducible lattice and $\Al[L_1] $ is trivial. For bilattices, we have that $\Al[L] \odot \Al[L]$ is subdirectly irreducible if and only if $\Al[L]$ is a subdirectly irreducible lattice.

\include{Ch_2_Int_(Pre)bilattices}





\chapter{Logical bilattices:  the logic $\lb$} 
\label{ch:log}

\section{Semantical and Gentzen-style presentations}
\label{sec:sem}

In this chapter we will study the logic $\lb$, introduced by Arieli and Avron \cite{ArAv96}, from the standpoint of Abstract Algebraic Logic. We start by giving a semantical presentation of $\lb$, and then consider a sequent calculus that is complete with respect to this semantics. 

Our semantical presentation of $\lb$ differs from Arieli and Avron's original one in that they use a whole class of matrices (called ``logical bilattices'') to define their logic, while we will consider only $\four$. However, as we shall see, the two definitions have been proved to be equivalent. 

Recall that $\four$ is the smallest non-trivial bilattice and its  $\{ \land, \lor, \neg \}$-reduct is  a four-element De Morgan algebra which is known to generate the variety of De Morgan lattices. Indeed, as we have anticipated, the Belnap-Dunn four-valued
logic is the logic defined by the logical matrix $\langle \Al[M_{4}], \Tr
\rangle$ where $\Al[M_{4}]$ is this four-element De Morgan algebra and
$\Tr$ is the set $\{ \top, \true \}$ (see~\cite[Proposition~2.3]{F97}).

According to the interpretation proposed by Belnap and Dunn,
the elements of $\four$ may be thought of as: only \emph{true} $(\true)$,
only \emph{false} $(\false)$, both \emph{true} and \emph{false}
$(\top)$, and neither \emph{true} nor \emph{false} $(\bot)$. Thus, taking  $\Tr= \{ \top,
\true \}$ as the set of designated elements corresponds to the intuitive idea of preferring those values which are at
least true (but possibly also false).  Arieli and Avron followed the same intuition when they introduced the
logic $\lb$. Let us give the formal definition:

\begin{definition}
   Let $\LB = \langle \Al[Fm], \Log \rangle$ be the logic defined by the  matrix
  $\langle \four, \Tr \rangle$.
\end{definition}

\begin{table}[ht!]
\fbox{\parbox{\textwidth}{
\vspace{0,15cm}
\begin{enumerate}[ ]
  \item \quad \quad \quad \textbf{Axiom:} \qquad $(Ax) \quad \Gamma, \varphi \rhd \varphi,
    \Delta$.
    \item
  \item \quad \quad \quad \textbf{Rules:} \qquad \;  Cut Rule plus the following logical rules:
\end{enumerate}
\begin{center} {
  \begin{align*}
&(\land \rhd) \quad \frac{\Gamma, \varphi, \psi \rhd \Delta}{\Gamma, \varphi \land \psi \rhd \Delta} 
& & (\rhd \land) \quad \frac{\Gamma \rhd \Delta, \varphi \quad \Gamma \rhd \Delta, \psi }{\Gamma \rhd \Delta, \varphi \land \psi} \\ \\
& (\neg \land \rhd) \quad \frac{\Gamma, \neg \varphi \rhd \Delta \quad \Gamma, \neg \psi \rhd \Delta }{\Gamma, \neg (\varphi \land \psi) \rhd \Delta}  
& &  (\rhd \neg \land ) \quad \frac{\Gamma \rhd \Delta, \neg \varphi,\neg \psi}{\Gamma \rhd \Delta, \neg (\varphi \land \psi)} \\ \\
& (\lor \rhd) \quad \frac{\Gamma, \varphi \rhd \Delta \quad \Gamma, \psi \rhd \Delta }{\Gamma, \varphi \lor \psi \rhd \Delta}  & &  (\rhd \lor ) \quad \frac{\Gamma \rhd \Delta, \varphi, \psi}{\Gamma \rhd \Delta, \varphi \lor \psi} \\ \\
&(\neg \lor \rhd) \quad \frac{\Gamma, \neg \varphi, \neg \psi \rhd \Delta}{\Gamma, \neg (\varphi \lor \psi) \rhd \Delta} & & (\rhd \neg \lor) \quad \frac{\Gamma \rhd \Delta, \neg \varphi \quad \Gamma \rhd \Delta, \neg \psi }{\Gamma \rhd \Delta, \neg (\varphi \lor \psi)} \\ \\
&(\otimes \rhd) \quad \frac{\Gamma, \varphi, \psi \rhd \Delta}{\Gamma, \varphi \otimes \psi \rhd \Delta} 
& & (\rhd \otimes) \quad \frac{\Gamma \rhd \Delta, \varphi \quad \Gamma \rhd \Delta, \psi }{\Gamma \rhd \Delta, \varphi \otimes \psi} \\ \\
& (\neg \otimes \rhd) \quad \frac{\Gamma, \neg \varphi, \neg \psi \rhd \Delta}{\Gamma, \neg (\varphi \otimes \psi) \rhd \Delta}  
& & (\rhd \neg \otimes) \quad \frac{\Gamma \rhd \Delta, \neg \varphi \quad \Gamma \rhd \Delta, \neg \psi }{\Gamma \rhd \Delta, \neg (\varphi \otimes \psi)}  \\ \\
& (\oplus \rhd) \quad \frac{\Gamma, \varphi \rhd \Delta \quad \Gamma, \psi \rhd \Delta }{\Gamma, \varphi \oplus \psi \rhd \Delta}  & &  (\rhd \oplus ) \quad \frac{\Gamma \rhd \Delta, \varphi, \psi}{\Gamma \rhd \Delta, \varphi \oplus \psi} \\ \\
&(\neg \oplus \rhd) \quad \frac{\Gamma, \neg \varphi \rhd \Delta \quad
\Gamma, \neg \psi \rhd \Delta }{\Gamma, \neg (\varphi \oplus \psi) \rhd
\Delta} & & (\rhd \neg \oplus ) \quad \frac{\Gamma \rhd \Delta, \neg
\varphi,\neg \psi}{\Gamma \rhd \Delta, \neg (\varphi \oplus \psi)} \\ \\
& (\neg \neg \rhd) \quad \frac{\Gamma, \varphi \rhd \Delta }{\Gamma, \neg \neg \varphi \rhd \Delta}  & &  (\rhd \neg \neg) \quad \frac{\Gamma \rhd \Delta, \varphi}{\Gamma \rhd \Delta, \neg \neg \varphi}
\end{align*}
}
\end{center}
}}

\caption{A complete sequent calculus for the logic
$\mathcal{LB}$} \label{tab:GentRules_DBL}
\end{table}


As usual, the algebra $\Al[Fm]$ of formulas is the free algebra generated by a countable set $\var$ of variables
using the algebraic language $\{ \land, \lor, \otimes, \oplus,
\neg \}$. 
Note that we do not include constants in the language.  
By definition, for every set $\Gamma \cup \{ \varphi
\}$ of formulas it holds that $\: \Gamma \Log \varphi \:$ if and only if, for every valuation $h \in \textrm{Hom}(\Al[Fm], \four)$, if
    $h[\Gamma] \subseteq \Tr$ then $h(\varphi) \in \Tr$.

We will now remind two important results obtained in
\cite{ArAv96}. The first is
the introduction of a 
complete axiomatization of $\lb$ by means of a sequent calculus\footnote{An alternative sequent calculus, also complete w.r.t.\ the semantics of $\lb$, was  introduced in \cite{Ko01b}.}.
Here by  \emph{sequent} we mean a pair $\langle\Gamma,\Delta\rangle$ where
$\Gamma$ and $\Delta$ are both finite non-empty \emph{sets} of formulas;
to denote the sequent $\langle\Gamma,\Delta\rangle$ we will usually write $\Gamma\seq \Delta$ in order to avoid
misunderstandings with other symbols that are sometimes used as
sequent separator, such as $\vdash\,,\to$ or $\Rightarrow$. The Gentzen
system defined by the axiom and rules given in
Table~\ref{tab:GentRules_DBL}, that we call $\mathcal{G_{LB}}$, is the one introduced  in \cite{ArAv96} by Arieli and Avron\footnote{
Note that, unlike Arieli and Avron's, our presentation requires that both sides of sequents be
non-empty.
However, it is straightforward to see that the two presentations generate
essentially the same consequence relation.}.
We will denote by $\Gensin$ the consequence relation determined on the set of sequents by this calculus, so  $$\{ \Gamma_{i} \seq \Delta_{i} : i \in I \} \Gensin \Gamma \seq \Delta$$ means that the sequent $\Gamma \seq \Delta$ is derivable from the sequents $\{ \Gamma_{i} \seq \Delta_{i} : i \in I \}$. By this we mean that there is a finite sequence $\Sigma = S_{1}, \ldots S_{n}$ of sequents such that $S_{n} = \Gamma \seq \Delta$ and, for each $S_{m} \in \Sigma$, either $S_{m}$ is an instance of $(Ax)$ or $S_{m} \in \{ \Gamma_{i} \seq \Delta_{i} : i \in I \}$ or there are $S_{j}, S_{k} \in \Sigma$ such that $j, k < m$ and $S_{m}$ has been obtained from $S_{j}$ and $S_{k}$ by the application of a rule of $\mathcal{G_{LB}}$.

Since both the left- and right-hand side of our sequents
are (finite) sets of formulas, rather than multisets or sequences, it
is not necessary to include the structural rules of contraction and
exchange; they are, so to speak, built-in in the formalism. Note also that, using (Ax), Cut, $(\land \rhd)$ and $(\rhd \lor)$, it is easy to prove that the sequent $\Gamma \rhd \Delta$ is equivalent to $\bigwedge \Gamma \rhd \bigvee \Delta$. Taking this into account, we may obtain formal proofs of the rules of left weakening $(W \rhd)$ and
right weakening $(\rhd W)$, as follows:


\begin{center}

\begin{tabular}{c@{\qquad}c}
  {\small
  \AxiomC{$(Ax)$}
  \UnaryInfC{$\bigwedge \Gamma, \varphi \rhd \bigwedge \Gamma$}
  \AxiomC{$\Gamma \rhd \Delta$}
  \UnaryInfC{$\bigwedge \Gamma \rhd \bigvee \Delta$}
  \LeftLabel{$(Cut) \ $}
  \BinaryInfC{$\bigwedge \Gamma, \varphi \rhd \bigvee \Delta$}
  \UnaryInfC{$\Gamma, \varphi \rhd \Delta$}
  \DisplayProof}
 & 
  {\small
  \AxiomC{$\Gamma \rhd \Delta$}
  \UnaryInfC{$\bigwedge \Gamma \rhd \bigvee \Delta$}
  \AxiomC{$(Ax)$}
  \UnaryInfC{$\bigvee \Delta \rhd \bigvee \Delta, \varphi$}
  \RightLabel{$\ (Cut)$}
  \BinaryInfC{$\bigwedge \Gamma \rhd \bigvee \Delta, \varphi$}
  \UnaryInfC{$\Gamma \rhd \Delta, \varphi$}
  \DisplayProof}
\end{tabular}

\end{center}


Hence,  $\mathcal{G_{LB}}$ has all the structural rules. In
\cite{ArAv96} it is proved that this calculus admits Cut Elimination
(i.e., the Cut Rule is admissible) and is complete with respect to the semantics of $\lb$, in the following sense:

\begin{thm}
  \label{thm:GentzenComplet}
  The sequent calculus $\mathcal{G_{LB}}$ is 
  complete with respect to $\Log$. That is, for any $\Gamma \cup \{ \phi\} \subseteq Fm$, we have
$$\Gamma \Log \varphi \quad \textrm{iff} \quad \emptyset \Gensin \Gamma \rhd \varphi.$$
\end{thm}
The previous result can also be expressed saying that the Gentzen system  $\mathcal{G_{LB}}$ is \emph{adequate} for the logic $\lb$.   

The second important result we want to cite
from~\cite{ArAv96}, which justifies why $\LB$ is called
\emph{the logic of logical bilattices}, shows that the consequence relation $\Log$ may be defined using many other logical matrices instead of $\langle \four, \Tr \rangle$.  In order to state it, we need the following:

\begin{definition} \label{def:logbil}
A \emph{logical bilattice} is a pair $\langle \Al[B], F \rangle$ where $\Al[B]$ is a
bilattice and $F$ is a prime bifilter of $\Al[B]$.
\end{definition}

It is obvious that
logical bilattices are also matrices in the sense of AAL: so
each logical bilattice determines a logic. Note also that, since $\four$ has (only) one proper bifilter,
 $\langle \four, \Tr \rangle$ is a logical bilattice, namely the one we used to introduce $\lb$. A key result of ~\cite{ArAv96} is then that all logical bilattices define the same consequence relation (i.e.\ $\Log $):

\begin{thm}
  \label{thm:complet_logicalbilattice}
  If $\langle \Al[B], F \rangle$ is a logical bilattice then
  the logic determined by the matrix $\langle \Al[B], F \rangle$
  coincides with $\lb$. That
  is, for every set $\Gamma \cup \{ \varphi \}$ of formulas it holds
  that
$$
    \Gamma \Log \varphi \qquad \textrm{iff} \qquad \Gamma \models_{\langle
    \Al[B], F \rangle} \varphi.
  $$
\end{thm}

This last theorem is indeed a straightforward consequence of
the following lemma ({see~\cite[Theorem~2.17]{ArAv96}}).

\begin{lem}
  \label{l:four_logical_bilat}
  Let $\Al[B]$ be a bilattice and let $F \subsetneq B$. Then the following
  statements are equivalent:
  \begin{enumerate}[(i)]
    \item $F$ is a prime bifilter of $\Al[B],$
    \item there is a unique epimorphism $\pi_F: \Al[B] \longrightarrow \four$ such
      that $F= \pi_F^{-1} [\Tr],$
    \item there is an epimorphism $\pi_F: \Al[B] \longrightarrow \four$ such
      that $F= \pi_F^{-1} [\Tr]$.
  \end{enumerate}
\end{lem} 

We stress that the
epimorhism $\pi_F$ is the map defined, for all $a \in B$, by
\[
\pi_F (a) := \left\{ 
\begin{array}{ll}
  \top & \textrm{ if } a \in F \textrm{ and } \neg a \in F \\
  \true & \textrm{ if } a \in F \textrm{ and } \neg a \not \in F \\
  \false & \textrm{ if } a \not \in F \textrm{ and } \neg a \in F \\
  \bot & \textrm{ if } a \not \in F \textrm{ and } \neg a \not \in F \\
\end{array}
\right.
\]

Theorem \ref{thm:complet_logicalbilattice} justifies the claim that the logic of logical bilattices is indeed the logic of the matrix $\langle \four, \Tr \rangle$. In Section \ref{sec:lb} we will see that, from an algebraic point of view, the logic $\lb$ may be also considered in some sense as \emph{the logic of distributive bilattices}.

\section{Hilbert-style presentation}
\label{sec:hil}

In the literature a Hilbert-style presentation for
the logic $\LB$ has not yet been given.  The aim of this section is to fill
this gap, introducing a strongly complete Hilbert-style  calculus for this logic. 

It is well known that, from a proof theoretic point of view, sequent
calculi (especially those enjoying cut elimination and the
subformula property) are better suited for searching  proofs than
Hilbert-style calculi. However, from the point of view of AAL, having a Hilbert-style presentation provides a lot of
benefits, since it allows to characterize on any algebra the filters of the logic (i.e.\ those sets of elements of the algebra that are closed under the rules of the logic). This kind of considerations, besides its intrinsic interest, motivated the introduction of our calculus.

From the semantical definition of $\LB$, is it obvious that this logic is a
conservative expansion of the Belnap-Dunn four-valued logic. This observation suggests that, in order
to find a Hilbert-style presentation for $\LB$, we can just expand any axiomatization of the Belnap-Dunn logic. We shall  consider the one given by Font in \cite{F97}, which consists of the first fifteen rules of Table~\ref{tab:Rules_DBL}.

\begin{table}[ht!]
\fbox{\parbox{\textwidth}{
\begin{center} {
  \hspace{-0.7cm}
  \begin{tabular}{ccc}
    \AxiomC{$p \land q$}
    \LeftLabel{\,(R1)}
    \UnaryInfC{$p$}
    \DisplayProof &
    \AxiomC{$p \land q$}
    \LeftLabel{\,(R2)}
    \UnaryInfC{$q$}
    \DisplayProof &
    \AxiomC{$p \!$}
    \AxiomC{$\! q$}
    \LeftLabel{\,(R3)}
    \BinaryInfC{$p \land q$}
    \DisplayProof \\ \\ 

    \AxiomC{$p$}
    \LeftLabel{\,(R4)}
    \UnaryInfC{$p \lor q$}
    \DisplayProof &
    \AxiomC{$p \lor q$}
    \LeftLabel{\,(R5)}
    \UnaryInfC{$q \lor p$}
    \DisplayProof &
    \AxiomC{$p \lor p$}
    \LeftLabel{\,(R6)}
    \UnaryInfC{$p$}
    \DisplayProof \\ \\ 
 
    \AxiomC{$p \lor (q \lor r)$}
    \LeftLabel{\,(R7)}
    \UnaryInfC{$(p \lor q) \lor r$}
    \DisplayProof &
    \AxiomC{$p \lor (q \land r)$}
    \LeftLabel{\,(R8)}
    \UnaryInfC{$(p \lor q) \land (p \lor r)$}
    \DisplayProof &
    \AxiomC{$(p \lor q) \land (p \lor r)$}
    \LeftLabel{\,(R9)}
    \UnaryInfC{$p \lor (q \land r)$}
    \DisplayProof \\ \\ 
 
    \AxiomC{$p \lor r$}
    \LeftLabel{\,(R10)}
    \UnaryInfC{$\neg \neg p \lor r$}
    \DisplayProof &
    \AxiomC{$\neg \neg p \lor r$}
    \LeftLabel{\,(R11)}
    \UnaryInfC{$p \lor r$}
    \DisplayProof &
    \AxiomC{$\neg (p \lor q) \lor r$}
    \LeftLabel{\,(R12)}
    \UnaryInfC{$ (\neg p \land \neg q) \lor r$}
    \DisplayProof \\ \\ 
 
    \AxiomC{$(\neg p \land \neg q) \lor r$}
    \LeftLabel{\,(R13)}
    \UnaryInfC{$\neg (p \lor q) \lor r$}
    \DisplayProof &
    \AxiomC{$\neg (p \land q) \lor r$}
    \LeftLabel{\,(R14)}
    \UnaryInfC{$(\neg p \lor \neg q) \lor r$}
    \DisplayProof &
    \AxiomC{$(\neg p \lor \neg q) \lor r$}
    \LeftLabel{\,(R15)}
    \UnaryInfC{$\neg (p \land q) \lor r$}
    \DisplayProof \\ \\ \\ 
 
    \AxiomC{$( p \otimes q) \lor r$}
    \LeftLabel{\,(R16)}
    \UnaryInfC{$(p \land q) \lor r$}
    \DisplayProof &
    \AxiomC{$(p \land q) \lor r$}
    \LeftLabel{\,(R17)}
    \UnaryInfC{$(p \otimes q) \lor r$}
    \DisplayProof &
    \AxiomC{$(p \oplus q) \lor r$}
    \LeftLabel{\,(R18)}
    \UnaryInfC{$(p \lor q) \lor r$}
    \DisplayProof \\ \\ 
 
    \AxiomC{$( p \lor q) \lor r$}
    \LeftLabel{\,(R19)}
    \UnaryInfC{$(p \oplus q) \lor r$}
    \DisplayProof &
    \AxiomC{$(\neg p \otimes \neg q) \lor r$}
    \LeftLabel{\,(R20)}
    \UnaryInfC{$\neg (p \otimes q) \lor r$}
    \DisplayProof &
    \AxiomC{$\neg (p \otimes q) \lor r$}
    \LeftLabel{\,(R21)}
    \UnaryInfC{$(\neg p \otimes \neg q) \lor r$}
    \DisplayProof \\ \\ 
 
    \AxiomC{$(\neg p \oplus \neg q) \lor r$}
    \LeftLabel{\,(R22)}
    \UnaryInfC{$\neg (p \oplus q) \lor r$}
    \DisplayProof &
    \AxiomC{$\neg (p \oplus q) \lor r$}
    \LeftLabel{\,(R23)}
    \UnaryInfC{$(\neg p \oplus \neg q) \lor r$}
    \DisplayProof & \\
  \end{tabular} } 
\end{center} 
}} \vspace{0,4cm}

\caption{A complete Hilbert-style calculus for the logic $\LB$}
\label{tab:Rules_DBL}
\end{table}

Note that, like Font's, our calculus has no axioms: this is due to the fact that $\lb$ has no theorems, just like the Belnap-Dunn logic. To see this, it is sufficient to observe that $\{ \bot \}$ is a subalgebra of $\four$ and $\bot$ is not a designated element in the matrix $\langle \four, \Tr \rangle$. Let us stress that here it is crucial that we do not have any of the constants $\{ \top,\true,\false \}$ in the language. 

Hence, all Hilbert-style
presentations for $\LB$ must be free of axioms and consist only
of (proper) rules. Of course, as noted by Font \cite{F97}, and contrary to what is claimed in
\cite[p.~37]{ArAv96}, this absence of theorems does not mean  that there may not be Hilbert-style
presentations for $\LB$.

Let us introduce formally the consequence relation determined by our rules:

\begin{definition} \label{definition:rules}
The logic $\vdash_{\scriptscriptstyle H}$ is the consequence relation
defined through the rules of
Table~\ref{tab:Rules_DBL}. The closure operator associated with
$\vdash_{\scriptscriptstyle H}$ will be denoted $\mathbf{C}_H$.
\end{definition}

We shall devote the rest of the section to prove that this calculus is strongly
complete with respect to the semantics of $\LB$. The strategy of our proof is very similar to the
one used in \cite{F97} for the Belnap-Dunn logic, and is based on a normal
form representation of formulas.

First of all, let us verify that $\vdash_{\scriptscriptstyle H}$ is sound:

\begin{proposition}[Soundness]
\label{proposition:soundness}
Given a set of formulas $\Gamma \subseteq Fm$ and a formula $\phi \in
Fm,$ if $\Gamma \vdash_{\scriptscriptstyle H} \phi $, then $\Gamma
\Log \phi$.
\end{proposition}
\begin{proof}
It is sufficient to check that in $\four$ the set $\Tr$  is closed w.r.t.\ all
rules given in Table~\ref{tab:Rules_DBL}.
\end{proof}

In the following propositions (from \ref{proposition:ab} to \ref{proposition:lemma3}) we state some lemmas that will be needed to prove our normal form theorem (Theorem \ref{thm:nf}).

\begin{proposition} \label{proposition:ab}
The following rules follow from  \textrm{(R1)} to \textrm{(R23)}:
\begin{enumerate}[(i)]
  \item  The rule
    \AxiomC{$\varphi$}
    \LeftLabel{\,(Ri$^+$)}
    \UnaryInfC{$\psi$}
    \DisplayProof for each one of the rules 
    \AxiomC{$\varphi \lor r$}
    \LeftLabel{\,(Ri)}
    \UnaryInfC{$\psi \lor r$}
    \DisplayProof, \\ where $i \in \{10, \ldots, 23 \}$.

  \item The rule
    \AxiomC{$\varphi \land r$}
    \UnaryInfC{$\psi \land r$}
    \DisplayProof in the same cases.
\end{enumerate}
\end{proposition}

\begin{proof}
(i) From $\phi$ by (R4) we obtain $\phi \vee \psi$. Then we apply (R\emph{i}) to obtain $\psi \vee \psi$ and by (R6) we have $\psi$.

(ii) From $\phi \wedge r$ by (R1) we obtain $\phi$. Now using (i) we obtain $\psi$. Also from $\phi \wedge r$, by (R2), follows $r$. Thus applying (R3) we obtain $\psi \wedge r$.
\end{proof}

The following properties are also easily proved (we omit the proof):

\begin{proposition}
From \textrm{(R1), \ldots, (R9)} and \textrm{(R}$16^{+}), \ldots,$ \textrm{(R}$19^{+})$ we can derive the following rules:
\begin{align*}
&\textrm{\textrm{(R1')}} \quad \frac{p \otimes q }{p} & &\textrm{\textrm{(R2')}} \quad \frac{p \otimes q }{q} & &\textrm{\textrm{(R3')}} \quad \frac{p \quad q }{p \otimes q}\\ \\
&\textrm{\textrm{(R4')}} \quad \frac{p}{p \oplus q} & &\textrm{\textrm{(R5')}} \quad \frac{p \oplus q }{q \oplus p} & &\textrm{\textrm{(R6')}} \quad \frac{p \oplus p }{p} \\ \\
&\textrm{\textrm{(R7')}} \quad \frac{p \oplus \left(q \oplus r\right)}{\left(p \oplus q\right) \oplus r} & & \textrm{\textrm{(R8')}} \quad \frac{p \oplus \left(q \otimes r \right) } {\left(p \oplus q \right) \otimes \left( p \oplus r \right)} & & \textrm{\textrm{(R9')}} \quad \frac{\left(p \oplus q \right) \otimes \left( p \oplus r \right) }{p \oplus \left(q \otimes r \right)}
\end{align*}
\end{proposition}


\vspace{6pt}

\begin{proposition} \label{proposition:congr}
The interderivability relation $\dashv \vdash_{\scriptscriptstyle H} $ is a congruence w.r.t.\ the operations $\wedge$ and $\vee$.
\end{proposition}
\begin{proof}
It is sufficient to show that  the following two rules
\begin{center}
  \begin{tabular}{c@{\qquad}c}
    \AxiomC{$p \land r$}
    \AxiomC{$q \land r$}
    \BinaryInfC{$(p \land q) \land r$}
    \DisplayProof & 
    \AxiomC{$p \lor r$}
    \AxiomC{$q \lor r$}
    \BinaryInfC{$(p \land q) \lor r$}
    \DisplayProof 
  \end{tabular}
\end{center}
together with the rules
\begin{center}
    \AxiomC{$\varphi \lor r$}
    \UnaryInfC{$\psi \lor r$}
    \DisplayProof 
\quad and \quad
    \AxiomC{$\varphi \land r$}
    \UnaryInfC{$\psi \land r$}
    \DisplayProof 
\end{center}
(for each rule 
    \AxiomC{$\varphi$}
    \UnaryInfC{$\psi$}
    \DisplayProof 
in Table~\ref{tab:Rules_DBL})
are all derivable in $\vdash_{\scriptscriptstyle H}$.
For the rules in Table~\ref{tab:Rules_DBL} that belong to the $\{
\wedge, \vee\}$-fragment, it is known that they 
follow just from rules (R1) to (R9). And for (R10) to (R23) the conjunction
case is shown by Proposition~\ref{proposition:ab} (ii), while the
disjunction case can be easily shown by using the associativity of
$\vee$.

Then we know that $\phi \hi \psi $ implies  $\phi \land \gamma \hi \psi \land \gamma$ and $\phi \lor \gamma \hi \psi \lor \gamma$ for any $\gamma \in Fm$. So, assuming $\phi_{1} \hi \psi_{1} $ and $\phi_{2} \hi \psi_{2} $, from the former we obtain $\phi_{1} \land \phi_{2}  \hi \psi_{1} \land \phi_{2} $ and from the latter $\psi_{1} \land \phi_{2} \hi \psi_{1} \land \psi_{2}$. Hence  $ \phi_{1} \land \phi_{2} \hi \psi_{1} \land \psi_{2}$. By symmetry, we may conclude that  $\phi_{1} \dashv \hi \psi_{1} $ and $\phi_{2} \dashv \hi \psi_{2} $ imply $\phi_{1} \land \phi_{2} \dashv \hi \psi_{1} \land \psi_{2}$. A similar reasoning shows that $\dashv \vdash_{\scriptscriptstyle H} $ is also a congruence w.r.t.\ $ \vee$.
\end{proof}
  

\begin{definition}
$\mathcal{L}it = \var \cup \left\{\neg p : p \in \var \right\}$ is the set of \emph{literals}. $\mathcal{C}l$, the set of \emph{clauses}, is the least set containing $\mathcal{L}it$ and closed under $\vee$. For any $\phi \in Fm$, the set $var\left(\phi\right)$ of \emph{variables} of $\phi$ is defined in the usual way. For $\Gamma \subseteq Fm$, we set
$$var\left(\Gamma\right)=\bigcup_{\phi \in \Gamma} var\left(\phi\right). $$ For any $\phi \in \mathcal{C}l$, the set $lit\left(\phi\right)$ of \emph{literals} of $\phi$ is defined inductively by $lit\left(\phi\right)= \left\{\phi\right\}$ if $\phi \in \mathcal{L}it$ and $lit\left(\phi \vee \psi \right)= lit\left(\phi \right) \cup lit\left( \psi \right)$. For $\Gamma \subseteq \mathcal{C}l$, we set $$lit\left(\Gamma\right)=\bigcup_{\phi \in \Gamma} lit\left(\phi\right). $$
\end{definition}

\begin{proposition} \label{proposition:lemma}
For all $\phi \in Fm$, there is a finite $\Gamma_{\phi} \subseteq \mathcal{C}l$ such that $var\left(\phi\right)=var\left(\Gamma\right)$ and for every $\psi \in Fm$, $$\mathbf{C}_{H} \left(\phi \vee \psi\right)= \mathbf{C}_{H} \left(\left\{\gamma \vee \psi: \gamma \in \Gamma\right\} \right).$$
\end{proposition} 

\begin{proof}
By induction on the length of $\phi$.

\begin{enumerate}[]
  \item
If $\phi=p \in \var $, then $\Gamma_{\phi} =\left\{p\right\}$.

  \item
If $\phi=\phi_{1} \wedge \phi_{2}$ and by inductive hypothesis $\Gamma_{\phi_1}, \Gamma_{\phi_2}$ correspond respectively to $\phi_{1}$ and $\phi_{2}$, then we may take $\Gamma_{\phi} =\Gamma_{\phi_1} \cup \Gamma_{\phi_2}$ and we have $var\left(\phi\right)=var\left(\Gamma_{\phi} \right)$. We also have
\begin{eqnarray*}
\mathbf{C}_{H} \left(\phi \vee \psi\right) & = & \mathbf{C}_{H} \left(\left(\phi_{1} \wedge \phi_{2}\right) \vee \psi\right)  \\
& = & \mathbf{C}_{H}\left(\left(\phi_{1} \vee \psi \right) \wedge \left(\phi_{2} \vee \psi\right) \right)  \\
& = & \mathbf{C}_{H} \left(\phi_{1} \vee \psi , \phi_{2} \vee \psi\right)  \\
& = & \textrm{by (R1), (R2), (R3)} \\
& = & \mathbf{C}_{H} \left( \mathbf{C}_{H}\left( \phi_{1} \vee \psi\right) \cup \mathbf{C}_{H} \left(\phi_{2} \vee \psi\right)\right)  \\
& = & \mathbf{C}_{H} \left( \mathbf{C}_{H} \left(\left\{\gamma_{1} \vee \psi: \gamma_{1} \in \Gamma_{\phi_1}\right\} \right) \cup \mathbf{C}_{H}\left(\left\{\gamma_{2} \vee \psi: \gamma_{2} \in \Gamma_{\phi_2}\right\} \right)\right) \\
& = & \mathbf{C}_{H} \left(\left\{\gamma \vee \psi: \gamma \in \Gamma_{\phi} \right\} \right).
\end{eqnarray*}

  \item
If $\phi=\phi_{1} \vee \phi_{2}$ and $\Gamma_{\phi_1}, \Gamma_{\phi_2}$ correspond respectively to $\phi_{1}$ and $\phi_{2}$, then we take 
$$\Gamma_{\phi} = \left\{\gamma_{1} \vee \gamma_{2} : \gamma_{1} \in \Gamma_{\phi_1}, \gamma_{2} \in \Gamma_{\phi_2}\right\}$$ 
and we have $var\left(\phi\right)=var\left(\Gamma_{\phi} \right)$. We also have:
\begin{eqnarray*}
\mathbf{C}_{H} \left(\phi \vee \psi\right) & = & \mathbf{C}_{H} \left(\left(\phi_{1} \vee \phi_{2}\right) \vee \psi\right)  \\
& = & \mathbf{C}_{H} \left(\phi_{1} \vee \left(\phi_{2} \vee \psi\right) \right)  \\
& = &   \textrm{(by inductive hypothesis)} \\
& = & \mathbf{C}_{H} \left(\left\{\gamma_{1} \vee \left(\phi_{2} \vee \psi\right) : \gamma_{1} \in \Gamma_{\phi_1}\right\} \right)  \\
& = & \mathbf{C}_{H} \left(\left\{\phi_{2} \vee \left(\gamma_{1} \vee \psi\right) : \gamma_{1} \in \Gamma_{\phi_1}\right\} \right) \\
& = & \mathbf{C}_{H} \left(\left\{\gamma_{2} \vee \left(\gamma_{1} \vee \psi\right) : \gamma_{1} \in \Gamma_{\phi_1}, \gamma_{2} \in \Gamma_{\phi_2} \right\} \right) \\
& = &\mathbf{C}_{H} \left(\left\{\left(\gamma_{1} \vee \gamma_{2}\right) \vee \psi : \gamma_{1} \in \Gamma_{\phi_1}, \gamma_{2} \in \Gamma_{\phi_2} \right\} \right).
\end{eqnarray*}

  \item
If $\phi=\phi_{1} \otimes \phi_{2}$, then $\mathbf{C}_{H} \left(\phi \vee \psi\right)=\mathbf{C}_{H} \left( \left(\phi_{1} \otimes \phi_{2}\right) \vee \psi\right)$. By (R16) and (R17) we have 
$$\mathbf{C}_{H} \left( \left(\phi_{1} \otimes \phi_{2}\right) \vee \psi\right)=\mathbf{C}_{H} \left( \left(\phi_{1} \wedge \phi_{2}\right) \vee \psi\right).$$
So we may apply the procedure for $\phi=\phi_{1} \wedge \phi_{2}$.

  \item
If $\phi=\phi_{1} \oplus \phi_{2}$, then 
$$\mathbf{C}_{H} \left(\phi \vee \psi\right)=\mathbf{C}_{H} \left( \left(\phi_{1} \oplus \phi_{2}\right) \vee \psi\right).$$
By (R18) and (R19) we have 
$$\mathbf{C}_{H} \left( \left(\phi_{1} \oplus \phi_{2}\right) \vee \psi\right)=\mathbf{C}_{H} \left( \left(\phi_{1} \vee \phi_{2}\right) \vee \psi\right).$$
So we may apply the procedure for $\phi=\phi_{1} \vee \phi_{2}$.

  \item
If $\phi=\neg \phi^{\prime}$, then we have to distinguish several cases on
$\varphi'$.
\begin{enumerate}[]
  \item
If $\phi^{\prime}=p \in \var $, then $ \phi \in \mathcal{L}it \subseteq \mathcal{C}l$, so we may take $\Gamma_{\phi} =\left\{\phi\right\}$.

  \item 
If $\phi^{\prime}=\neg \phi^{\prime\prime}$, then $\phi=\neg \neg \phi^{\prime\prime}$ and by (R10) and (R11) we have 
$$\mathbf{C}_{H} \left(\phi \vee \psi\right)= \mathbf{C}_{H} \left(\phi^{\prime\prime} \vee \psi\right).$$ 
Now just note that $\phi^{\prime\prime}$ is shorter that $\phi$ and its corresponding set $\Gamma_{\phi}$ also works for $\phi$.

  \item
If $\phi^{\prime}=\phi_{1} \wedge \phi_{2}$, then $\phi= \neg \left(\phi_{1} \wedge \phi_{2}\right)$ and by (R14) and (R15) we have 
$$\mathbf{C}_{H} \left(\phi \vee \psi \right)=\mathbf{C}_{H} \left(\left(\neg \phi_{1} \vee \neg \phi_{2}\right) \vee \psi \right).$$ 
Both $\neg \phi_{1}$ and $\neg \phi_{2}$ are shorter than $\neg \left(\phi_{1} \wedge \phi_{2}\right)$, so the same procedure for the case of $\phi=\phi_{1} \vee \phi_{2}$ works.

  \item
If $\phi^{\prime}=\phi_{1} \vee \phi_{2}$, then $\phi= \neg \left(\phi_{1} \vee \phi_{2}\right)$ and by (R12) and (R13) we have 
$$\mathbf{C}_{H} \left(\phi \vee \psi \right)=\mathbf{C}_{H} \left(\left(\neg \phi_{1} \wedge \neg \phi_{2}\right) \vee \psi \right).$$ Both $\neg \phi_{1}$ and $\neg \phi_{2}$ are shorter than $\neg \left(\phi_{1} \vee \phi_{2}\right)$, so the same procedure for the case of $\phi=\phi_{1} \wedge \phi_{2}$ works.

  \item
If $\phi^{\prime}=\phi_{1} \otimes \phi_{2}$, then $\phi= \neg \left(\phi_{1} \otimes \phi_{2}\right)$ and by (R20) and (R21) we have  
$$\mathbf{C}_{H} \left(\phi \vee \psi\right)=\mathbf{C}_{H} \left( \left( \neg \phi_{1} \otimes \neg \phi_{2}\right) \vee \psi\right).$$
Both $\neg \phi_{1}$ and $\neg \phi_{2}$ are shorter than $\neg \left(\phi_{1} \otimes \phi_{2}\right)$, hence the procedure applied for $\phi=\phi_{1} \otimes \phi_{2}$ works.

  \item
If $\phi^{\prime}=\phi_{1} \oplus \phi_{2}$, then $\phi= \neg \left(\phi_{1} \oplus \phi_{2}\right)$ and by (R22) and (R23) we have  
$$\mathbf{C}_{H} \left(\phi \vee \psi\right)=\mathbf{C}_{H} \left( \left( \neg \phi_{1} \oplus \neg \phi_{2}\right) \vee \psi\right).$$ 
Both $\neg \phi_{1}$ and $\neg \phi_{2}$ are shorter than $\neg \left(\phi_{1} \oplus \phi_{2}\right)$. Once again, the procedure applied for $\phi=\phi_{1} \oplus \phi_{2}$ works.
\end{enumerate}
\end{enumerate}
\end{proof}

\begin{proposition} \label{proposition:lemma3}
For all $\phi \in Fm$ there is a finite $\Gamma_{\phi} \subseteq \mathcal{C}l$ such that $var\left(\phi\right)=var\left(\Gamma_{\phi}\right)$ and $\mathbf{C}_{H} \left(\phi\right)= \mathbf{C}_{H} \left(\Gamma_{\phi}\right).$
\end{proposition}

\begin{proof}
By induction on the length of $\phi$.

\begin{enumerate}[]
  \item 
If $\phi=p \in \var $, then take $\Gamma_{\phi}=\left\{\phi\right\}$.

  \item 
If $\phi=\phi_{1} \wedge \phi_{2}$ by (R1), (R2) and (R3) we have $\mathbf{C}_{H} \left(\phi\right)=\mathbf{C}_{H} \left(\phi_{1}, \phi_{2}\right)$. So we may take $\Gamma_{\phi}= \Gamma_{\phi_{1}} \cup \Gamma_{\phi_{2}}$ and we are done.

  \item 
If $\phi=\phi_{1} \vee \phi_{2}$ then by Proposition \ref{proposition:lemma} and (R5) we have: 
\begin{eqnarray*}
\mathbf{C}_{H} \left(\phi\right) & = & \mathbf{C}_{H}  \left(\left\{\gamma_{1} \vee \phi_{2}: \gamma_{1} \in \Gamma_{\phi_1}\right\}\right)  \\
& = & \mathbf{C}_{H}  \left(\left\{ \phi_{2} \vee \gamma_{1}  : \gamma_{1} \in \Gamma_{\phi_1}\right\}\right)  \\
& = & \mathbf{C}_{H}  \left(\left\{ \gamma_{2} \vee \gamma_{1}  : \gamma_{1} \in \Gamma_{\phi_1}, \gamma_{2} \in \Gamma_{\phi_2}\right\}\right).
\end{eqnarray*}
 Since $\Gamma_{\phi_1}, \Gamma_{\phi_2} \subseteq \mathcal{C}l $ are finite, $\Gamma_{\phi}=  \left\{ \gamma_{1} \vee \gamma_{2}  : \gamma_{1} \in \Gamma_{\phi_1}, \gamma_{2} \in \Gamma_{\phi_2}\right\} \subseteq \mathcal{C}l $ is also finite and we are done.

  \item 
If $\phi=\phi_{1} \otimes \phi_{2}$, by (R$16^{+}$) and (R$17^{+}$) we have $\mathbf{C}_{H} \left(\phi\right)=\mathbf{C}_{H} \left(\phi_{1}, \phi_{2}\right)$, so we may take $\Gamma_{\phi}= \Gamma_{\phi_{1}} \cup \Gamma_{\phi_{2}}$ and we are done.

  \item 
If $\phi=\phi_{1} \oplus \phi_{2}$, since by (R$18^{+}$) and (R$19^{+}$) we have $\mathbf{C}_{H} \left(\phi_{1} \oplus \phi_{2}\right)=\mathbf{C}_{H} \left(\phi_{1} \vee \phi_{2}\right)$, we may apply the procedure for $\phi=\phi_{1} \vee \phi_{2}$.

  \item 
If $\phi=\neg \phi^{\prime}$ we have to distinguish several cases.
\begin{enumerate}[]
  \item 
If $\phi^{\prime}=p \in \var $, then $ \phi \in \mathcal{C}l$, so we may take $\Gamma_{\phi}=\left\{\phi\right\}$.

  \item 
If $\phi^{\prime}=\neg \phi^{\prime\prime}$, then by (R$10^{+}$) and (R$11^{+}$) we have $\mathbf{C}_{H} \left(\phi\right)= \mathbf{C}_{H} \left(\phi^{\prime\prime}\right)$ and since $\phi^{\prime\prime}$ is shorter that $\phi$ we are done.

  \item 
If $\phi^{\prime}=\phi_{1} \wedge \phi_{2}$ then by (R$14^{+}$) and (R$15^{+}$) we have $\mathbf{C}_{H} \left(\phi\right)= \mathbf{C}_{H} \left( \neg \phi_{1} \vee \neg \phi_{2} \right)$, so we may apply the procedure for the $\lor$-disjunction case.

  \item 
If $\phi^{\prime}=\phi_{1} \vee \phi_{2}$ then by (R$12^{+}$) and (R$13^{+}$) we have $\mathbf{C}_{H} \left(\phi\right)= \mathbf{C}_{H} \left( \neg \phi_{1}, \neg \phi_{2} \right)$, so applying the inductive hypothesis we are done.

  \item 
If $\phi^{\prime}=\phi_{1} \otimes \phi_{2}$, then $\phi= \neg \left(\phi_{1} \otimes \phi_{2}\right)$ and by (R$20^{+}$) and (R$21^{+}$) we have  $\mathbf{C}_{H} \left(\phi \right)=\mathbf{C}_{H}  \left( \neg \phi_{1} \otimes \neg \phi_{2}\right)$, so the procedure applied for the $\otimes$-conjunction works.

  \item 
If $\phi^{\prime}=\phi_{1} \oplus \phi_{2}$, then $\phi= \neg \left(\phi_{1} \oplus \phi_{2}\right)$ and by (R$22^{+}$) and (R$23^{+}$) we have  $\mathbf{C}_{H} \left(\phi \right)=\mathbf{C}_{H} \left( \neg \phi_{1} \oplus \neg \phi_{2}\right) $, so the procedure applied for the $\oplus$-disjunction works.
\end{enumerate} 
\end{enumerate} 
\end{proof}

\begin{thm}[Normal Form] \label{thm:nf}
Every formula is equivalent, both through $\dashv
\vdash_{\scriptscriptstyle H}$ and $\eqLog$, to a $\wedge$-conjunction of clauses with the same variables.
\end{thm}
\begin{proof}
By Proposition \ref{proposition:lemma3} we have that $\phi \dashv
\vdash_{\scriptscriptstyle H} \bigwedge \Gamma_{\phi}$, where $\bigwedge
\Gamma_{\phi}$ is any conjunction of all the clauses in $\Gamma_{\phi}$.
Now, by Proposition \ref{proposition:soundness}, this implies also that $
\bigwedge \Gamma_{\phi} \eqLog \phi $.
\end{proof}

The following lemma will allow us to prove the completeness of our Hilbert calculus.

\begin{lem} \label{lem:compl}
For all $\Gamma \subseteq \mathcal{C}l$ and $\phi \in \mathcal{C}l$, the following are equivalent:

\begin{enumerate}[(i)]
	\item  $\Gamma \vdash_{\scriptscriptstyle H} \phi$,
 \item $\Gamma \Log \phi$,
 \item $\exists \gamma \in \Gamma$ such that $lit\left( \gamma \right) \subseteq lit\left( \phi \right)$,
 \item $\exists \gamma \in \Gamma$ such that $\gamma \vdash_{\scriptscriptstyle H} \phi$.
\end{enumerate}
\end{lem}

\begin{proof}
(i) $\Rightarrow$ (ii) follows from Proposition \ref{proposition:soundness}.

(ii) $\Rightarrow$ (iii). For a fixed $\phi \in \mathcal{C}l$, define a homomorphism $h: \Al[Fm] \rightarrow \four$ as follows. For every $p \in \var$:

\[
 	h\left(p\right) = \left\{ \begin{array}{cl}
	\true & \textrm{if $p \notin lit\left(\phi\right)$ and $\neg p \in lit\left(\phi\right)$}\\
	\top & \textrm{if $p, \neg p \notin lit\left(\phi\right)$} \\
  \mathbf{\bot} & \textrm{if $p, \neg p \in lit\left(\phi\right)$} \\
	\false & \textrm{if $p \in lit\left(\phi\right)$ and $\neg p \notin lit\left(\phi\right)$}
	\end{array} \right.
\]

\vspace{7pt}
If $p \in lit\left( \phi \right)$,  then $h\left(p\right) \in
\left\{\false,
\bot \right\}$ and also  $h\left( \neg p\right) \in \left\{\false, \bot
\right\}$ 	when $\neg p \in lit\left( \phi \right)$. Since $\false
\leq_{t} \bot$, we have $h\left(\phi\right) \in \left\{\false, \bot
\right\}$. Suppose (iii) fails: then for any $\gamma \in \Gamma$ there
would be $\psi_{\gamma} \in lit\left( \gamma \right)$ such that
$\psi_{\gamma} \notin lit\left( \phi \right)$. Then we would have
$h\left(\psi_{\gamma}\right) \in \left\{\true, \top \right\}$ and as a
consequence $h\left(\gamma\right) \in \left\{\true, \top \right\}$. Thus
we would have, against (ii), $h\left[\Gamma\right] \subseteq
\left\{\true, \top \right\}$ while $h\left(\phi\right) \notin
\left\{\true, \top \right\}$.

(iii) $\Rightarrow$ (iv).	If $lit\left( \gamma \right) \subseteq lit\left( \phi \right)$ and $\gamma, \phi \in \mathcal{C}l$, then $\phi$ is a disjunction of the same literals appearing in $\gamma$ plus other ones, modulo some associations, permutations etc. Therefore, applying rules (R4) to (R7) and repeatedly using Proposition \ref{proposition:congr}, we obtain $\gamma \vdash_{\scriptscriptstyle H} \phi$.

(iv) $\Rightarrow$ (i). Immediate.
\end{proof}

\begin{thm}[Completeness]
For all $\Gamma \subseteq Fm$ and $\phi \in Fm,$ it holds that $\Gamma
\Log \phi$ iff $\Gamma  \vdash_{\scriptscriptstyle H} \phi$.
\end{thm}
\begin{proof}
By Lemma \ref{lem:compl} and Theorem \ref{thm:nf}.
\end{proof}

\section{Tarski-style characterizations}
\label{sec:tar}

With the help of the Hilbert calculus introduced in the previous section, we will now investigate our logic from the point of view of Abstract Algebraic Logic. In particular, we study the algebraic models and g-models of $\lb$, characterize the classes $\mathbf{Alg}\mathcal{LB}$ and $\mathbf{Alg^*}\mathcal{LB}$ and compare them with the class of algebraic reducts of logical bilattices, which we will denote by $\mathsf{LoBiLat}$. We will also prove that the Gentzen calculus introduced in Section \ref{sec:sem} %
is algebraizable and individuate its equivalent algebraic semantics.



Let us start by checking that $\LB$ has no consistent extensions. We shall need the following:

\begin{lem} \label{l:submat}
Let $\la \Al[B], F \ra$ be a matrix such that $\Al[B]$ is a distributive bilattice and $F$ is a proper and non-empty bifilter of $\Al[B]$, i.e.\ $\emptyset \neq F \varsubsetneq B$. Then the logic defined by $\la \Al[B], F \ra$ is weaker than $\LB$.
\end{lem}

\begin{proof}
Reasoning by contraposition, we will prove that $\Gamma \nvDash_{\LB} \phi$ implies $\Gamma \nvDash_{\la \Al[B], F \ra} \phi$ for all $\Gamma \cup \{ \phi \} \subseteq Fm$. In order to do this, it will be enough to show that $\la \four, \Tr \ra$ is a submatrix of any matrix of the form $\la \Al[B], F \ra$. By assumption  $F$ is proper and non-empty, so there are $a, b \in B$ such that $a \notin F$ and $b \in F$. Let us denote by $\bot(a,b)$ the element $a \otimes b \otimes \neg a \otimes \neg b$. Similarly, let $\top(a,b) = a \oplus b \oplus \neg a \oplus \neg b $, $\true(a,b) = \bot(a,b) \lor \top(a,b)$  and $\false(a,b) = \bot(a,b) \land \top(a,b)$. 
Since $F$ is a bifilter, from the assumptions it follows that  $\top(a,b), \true(a,b) \in F$ and  $\bot(a,b), \false(a,b) \notin F$. 
It is easy to check that $\four$ is embeddable into $\Al[B]$ via the map $f$ defined as $f(x) = x(a,b)$ for all $x \in \{\bot, \top, \true, \false \}$. Moreover, $\Tr = f^{-1}[F]$. So if $h: \Al[Fm] \longrightarrow \four$ is a homomorphism such that $h[\Gamma] \subseteq \Tr $ but $h (\phi) \notin \Tr$, then also $f [h[\Gamma]] \subseteq F $ but $f (h(\phi)) \notin F $. Recalling that $\LB$ is the logic defined by the matrix $\la \four, \Tr \ra$, we may then conclude that  $\Gamma \nvDash_{\LB} \phi$ implies $\Gamma \nvDash_{\la \Al[B], F \ra} \phi$.
\end{proof}

Let us say that a logic $\mathcal{L} = \la \Al[Fm], \vdash_{\mathcal{L}} \ra$ is \emph{consistent} if there exist $ \phi, \psi \in Fm$ such that $\phi \nvdash_{\mathcal{L}} \psi$. Then the previous lemma allows to obtain the following: 

\begin{proposition} \label{p:extens}
If a logic $\mathcal{L} = \la \Al[Fm], \vdash_{\mathcal{L}} \ra$ is a consistent extension of $\LB$, then ${\vdash_{\mathcal{L}}} = {\vDash_{\mathcal{LB}}} $. 
\end{proposition}

\begin{proof}
By \cite[Proposition 2.27]{FJa09}, we know that any reduced matrix for $\mathcal{L}$ is of the form $\la \Al[B], F \ra$, where $\Al[B]$ is a distributive bilattice and $F$ is a bifilter. By the assumption of consistency, we may assume that there is at least one reduced matrix for $\mathcal{L}$ such that  $F$ is proper and non-empty. By Lemma \ref{l:submat}, we know that the logic defined by such a matrix is weaker than $\LB$; this implies that the class of all reduced matrices for  $\mathcal{L}$ defines a weaker logic than $\LB$. Since  any logic is complete with respect to the class of its reduced matrices (see \cite{W88}), we may conclude that  $\mathcal{L}$ itself is weaker than $\LB$, so they must be equal.
\end{proof}

The two completeness results stated in the previous section allow us to give a characterization of $\lb$ in terms of some metalogical properties which are sometimes called \emph{Tarski-style conditions}. In particular, we shall consider the following: the \emph{Property of Conjunction} (PC) w.r.t.\ both conjunctions $\land$ and $\otimes$, the \emph{Property of Disjunction} (PD) w.r.t.\ both disjunctions $\lor$ and $\oplus$, the \emph{Property of Double Negation} (PDN) and the \emph{Properties of De Morgan} (PDM).

Let us denote the closure operator associated with our logic by $\clb$. Then we may state the following: 

\begin{proposition} \label{p:metalog}
 The logic $\lb = \la \Al[Fm], \clb \ra$ satisfies the following properties: for all $\Gamma \cup \{ \phi, \psi \} \subseteq Fm$,
 
\begin{tabular}[l]{clllll}
\\
\emph{(PC)}  &  $ \clb(\phi \land \psi ) =  \clb(\phi \otimes \psi )  =  \clb(\phi, \psi ) $ \\  \\
\emph{(PDI)}   &  $ \clb(\Gamma, \phi \lor \psi ) =  \clb(\Gamma, \phi \oplus \psi )  =  \clb(\Gamma, \phi)  \cap \clb(\Gamma, \psi) $ \\ \\
\emph{(PDN)}  &  $ \clb(\phi ) =  \clb(\neg \neg \phi )$ \\ \\
\emph{(PDM)}  &  $ \clb(\neg (\phi \land \psi)) =  \clb(\neg \phi \lor \neg \psi)) $ \\
  &  $ \clb(\neg (\phi \lor \psi)) =  \clb(\neg \phi \land \neg \psi)) $ \\
     &  $ \clb(\neg (\phi \otimes \psi)) =  \clb(\neg \phi \otimes \neg \psi)) $ \\
        &  $ \clb(\neg (\phi \oplus \psi)) =  \clb(\neg \phi \oplus \neg \psi)) $. 
\\ \\
\end{tabular}

Moreover, $\lb$ is the only consistent logic satisfying them.
\end{proposition}
\begin{proof}
In \cite[Theorem~4.1]{Py95a} it is proved that the Belnap-Dunn logic is the least logic satisfying all the above properties except those involving $\otimes$ and $\oplus$. Since our logic is a conservative expansion of the Belnap-Dunn, we need only to check that $\lb$ satisfies the conditions where $\otimes$ or $\oplus$ appears. (PC) is easily proved using the derivable rules (R$16^+$) and (R$17^+$) of our Hilbert calculus (see the first item of Proposition \ref{proposition:ab}). Recalling that $\lb$ is finitary, to prove (PDI) we may use (PC), (R$18^+$) and (R$19^+$). Finally, the last two equalities of (PDM) are easily proved using rules from (R$20^+$) to (R$23^+$).

Hence $\lb$ satisfies all the above properties. Moreover, it is the weakest one that satisfies them. In fact, any logic $\mathcal{L} = \la \Al[Fm], \vdash_{\mathcal{L}} \ra$ satisfying the same properties will be closed under the rules of the Gentzen calculus $\mathcal{G_{LB}}$, which is complete w.r.t.\ the semantics of $\lb$. So any derivation in $\mathcal{G_{LB}}$ will produce only sequents which are derivable in   $\mathcal{L}$. Hence, by completeness,
if $\Gamma \Log \phi$, then $\Gamma \vdash_{\mathcal{L}} \phi.$ 
Now, applying Lemma \ref{l:submat}, we may conclude that $\vdash_{\mathcal{L}} \ = \  \Log $.
\end{proof}
%



Another interesting feature of $\lb$ is the \emph{variable sharing property} (VSP), that can be formulated as follows: if $\phi \Log \psi$, then  $var(\phi) \cap var(\psi) \neq \emptyset$. Note that any logic $\mathcal{L} = \la \Al[Fm], \vdash_{\mathcal{L}} \ra$ satisfying the (VSP) will be consistent, for it will hold that $p \nvdash_{\mathcal{L}} q$ for any two distinct propositional variables $p$ and $q$. So from the previous result it also follows that $\lb $ is the only logic satisfying (PC), (PDI), (PDN), (PDM) and (VSP).

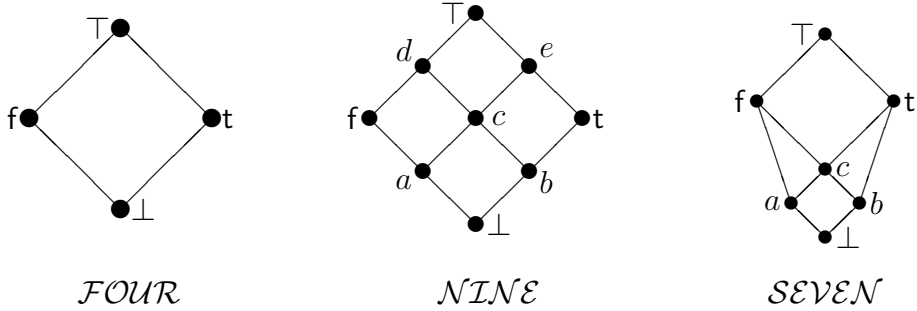
\begin{figure}[t]
\vspace{15pt}
\begin{center}
\begin{tabular}{cccc}
\vspace{5pt}

\begin{minipage}{2,5cm}
\setlength{\unitlength}{1.2cm}
\begin{center}
\begin{picture}(2,2)(0.15,0)
\put(1,0){\makebox(0,0)[l]{ $\bot$}} 
\put(0,1){\makebox(0,0)[r]{$\false$ }} 
\put(2,1){\makebox(0,0)[l]{ $\true$}} 
\put(1,2){\makebox(0,0)[r]{$\top$ }} 

\put(1,0){\circle*{0.2}} 
\put(0,1){\circle*{0.2}} 
\put(2,1){\circle*{0.2}}
\put(1,2){\circle*{0.2}}

\put(1,0){\line(1,1){1}} 
\put(1,0){\line(-1,1){1}} 
\put(0,1){\line(1,1){1}} 
\put(2,1){\line(-1,1){1}}

\end{picture}
\end{center}
\end{minipage}

&



%




\begin{minipage}{6cm}
\setlength{\unitlength}{0.7cm}
\begin{center}
\begin{picture}(3,4)(-0.25,0)
\multiput(1,0)(1,1){3}{\circle*{0.3}}
\multiput(0,1)(1,1){3}{\circle*{0.3}}
\multiput(-1,2)(1,1){3}{\circle*{0.3}}

\put(1,0){\line(1,1){2}} 
\put(0,1){\line(1,1){2}} 
\put(-1,2){\line(1,1){2}} 
\put(1,0){\line(-1,1){2}} 
\put(2,1){\line(-1,1){2}} 
\put(3,2){\line(-1,1){2}} 

\put(1,0){\makebox(0,0)[l]{ $\bot$}} 
\put(-1,2){\makebox(0,0)[r]{$\false$ }} 
\put(3,2){\makebox(0,0)[l]{ $\true$}} 
\put(1,4){\makebox(0,0)[r]{$\top$ }}

\put(-0.2,0.8){\makebox(0,0)[r]{ $a$}} 
\put(2,0.8){\makebox(0,0)[l]{ $b$}} 
\put(1.1,2){\makebox(0,0)[l]{ $c$}} 
\put(-0.2,3.3){\makebox(0,0)[r]{ $d$}} 
\put(2,3.3){\makebox(0,0)[l]{ $e$}} 

\end{picture}
\end{center}
\end{minipage}

&

\begin{minipage}{2cm}
\setlength{\unitlength}{0.9cm}
\begin{center}
\begin{picture}(2,2)(0,-0.25)
\put(1,-1){\makebox(0,0)[l]{ $\bot$}} 
\put(0.5,-0.5){\makebox(0,0)[r]{$a$ }} 
\put(1.5,-0.5){\makebox(0,0)[l]{ $b$}} 
\put(1,0){\makebox(0,0)[l]{ $c$}} 
\put(0,1){\makebox(0,0)[r]{$\false$ }} 
\put(2,1){\makebox(0,0)[l]{ $\true$}} 
\put(1,2){\makebox(0,0)[r]{$\top$ }} 

\put(1,-1){\circle*{0.2}} 
\put(0.5,-0.5){\circle*{0.2}}
\put(1.5,-0.5){\circle*{0.2}}
\put(1,0){\circle*{0.2}} 
\put(0,1){\circle*{0.2}} 
\put(2,1){\circle*{0.2}}
\put(1,2){\circle*{0.2}}

\put(1,0){\line(1,1){1}} 
\put(1,0){\line(-1,1){1}} 
\put(0,1){\line(1,1){1}} 
\put(2,1){\line(-1,1){1}} 

\put(1,-1){\line(-1,1){0.5}} 
\put(1,-1){\line(1,1){0.5}} 

\put(1,0){\line(1,-1){0.5}} 
\put(1,0){\line(-1,-1){0.5}} 

\put(0,1){\line(1,-3){0.5}} 
\put(2,1){\line(-1,-3){0.5}} 

\end{picture}
\end{center}
\end{minipage}

\\ \\

$\four$ 
& $\nine$ & $\7$
\end{tabular}

\caption{Some bilattices} \label{fig:hasse_aal}
\end{center}
\end{figure}

\section{AAL study of $\lb$}
\label{sec:lb}


Let us 
now classify our logic according to some of the criteria of Abstract Algebraic Logic. Recall that, in the context of AAL, a logic $\mathcal{L}$ is said to be \emph{protoalgebraic} if and only if, on any algebra, the Leibniz operator is monotone on the $\mathcal{L}$-filters (this is not the original definition, but a characterization that has by now become standard; see, for instance, \cite{BP86}).
A logic is said to be \emph{selfextensional} when the interderivability relation is a congruence of the formula algebra. The following proposition shows that our logic falls outside of both these categories:

\begin{proposition} \label{p:proto}
The logic $\vDash_{\LB}$ is non--protoalgebraic and non--selfextensional.
\end{proposition}
\begin{proof}

Consider the bilattice $\mathcal{NINE}$, repeated in Figure \ref{fig:hasse_aal}. The only proper and non--empty $\mathcal{LB}$--filters on $\mathcal{NINE}$ are $ F_{1}=\left\{e, \top,  \true \right\} \subseteq \left\{ b, c, d, e, \top, \true \right\}=F_{2}$. It is easy to check that $ \left\langle \true, e \right\rangle \in \leib[] \left\langle \mathcal{NINE}, F_{1}\right\rangle$ but, because of negation, we have $ \left\langle \true, e \right\rangle \notin \leib[]\left\langle  \mathcal{NINE}, F_{2}\right\rangle$. Hence, the Leibniz operator is not monotone on $\mathcal{LB}$-filters. 


As to the second claim, note that for any $p, q \in Fm$ we have $p \oplus q \eqLog p \vee q$, but we can easily check that we do not have $\neg \left(p \oplus q\right) \eqLog \neg \left(p \vee q\right)$. For instance in $\mathcal{FOUR}$ we have $\neg \left( t  \oplus \top \right)= \top \in \left\{t, \top\right\}$ but $\neg \left( t  \vee \top \right)= f \notin \left\{t, \top\right\}$.
\end{proof}

The fact that $\mathcal{LB}$ is not selfextensional constitutes one of the main difficulties of the AAL approach to it. As we have seen, this is due to the behaviour of the negation operator, and it is possible to see that this exception to selfextensionality is essentially the only one. We need the following lemmas.

\begin{lem} \label{l:lem}
Let $\varphi, \psi \in Fm$ be two formulas. The following statements are equivalent:
\begin{enumerate}[(i)]
\item $\mathcal{FOUR} \vDash \varphi \land (\varphi \otimes \psi) \approx \varphi$

\item $\varphi \vdash_{\scriptscriptstyle H} \psi$.
\end{enumerate}
\end{lem} 

\begin{proof}
(i) $\Rightarrow$ (ii). Let $h: \mathbf{Fm} \rightarrow \mathcal{FOUR}$ be a homomorphism. If $h(\varphi) = \mathsf{t}$, then $ \mathsf{t} \otimes h(\psi) = \mathsf{t}$, i.e. \ $h(\psi) \geq_k \mathsf{t}$, therefore $h(\psi) \in \{ \top, \mathsf{t} \}$. If $h(\varphi) = \top$, then  $\top \land h(\psi) = \top$, i.e. \ $h(\psi) \geq_t \top$, hence $h(\psi) \in \{ \top, \mathsf{t} \}$.

(ii) $\Rightarrow$ (i). Let $h: \mathbf{Fm} \rightarrow \mathcal{FOUR}$ be a homomorphism and assume that $\varphi \vdash_{\scriptscriptstyle H} \psi$. 

If $h(\varphi) = \mathsf{t}$, then $h(\psi) \in \{ \top, \mathsf{t} \}$, so $h(\psi) \geq_k h(\varphi)$. Hence $h(\varphi) \land (h(\varphi) \otimes h(\psi)) = h(\varphi) \land h(\varphi) = h(\varphi)$.


If $h(\varphi) = \top$, then $h(\psi) \in \{ \top, \mathsf{t} \}$, so $h(\psi) \geq_{t}  h(\varphi) $ and obviously $h(\varphi) \geq_{k} h(\psi)$. Hence we have $h(\phi) \land   (h(\phi) \otimes h(\psi)) = h(\phi) \land  h(\psi) =  h(\phi) $.


If $h(\varphi) = \bot$, then $h(\varphi) \land (h(\varphi) \otimes h(\psi)) = \bot \land \bot = \bot = h(\varphi)$. Finally, the case where $h(\varphi) = \mathsf{f}$ is immediate.
\end{proof}

As an immediate consequence of the preceding result, we have the following:

\begin{lem} \label{l:tfsae}
Let $\varphi, \psi \in Fm$ be two formulas. The following statements are equivalent:
\begin{enumerate}[(i)]
\item $\mathcal{FOUR} \vDash \varphi \approx \psi$,

\item $\varphi \dashv \vdash_{\scriptscriptstyle H} \psi$ and $\neg \varphi \dashv \vdash_{\scriptscriptstyle H} \neg \psi$.
\end{enumerate}
\end{lem} 

\begin{proof}
The only non-trivial implication is (ii)$\Rightarrow$(i). By Lemma \ref{l:lem}, (ii) implies that in $\mathcal{FOUR}$ the following equations hold:
\begin{align}
 \varphi    & \approx \varphi \land (\varphi \otimes \psi) \label{e:eq1} \\ 
\psi & \approx \psi \land (\varphi \otimes \psi) \label{e:eq2}   \\
 \neg \psi  & \approx \neg \psi \land (\neg \varphi \otimes \neg \psi)  \label{e:eq3} \\
 \neg \varphi & \approx  \neg \varphi \land (\neg \varphi \otimes \neg \psi). \label{e:eq4} 
\end{align}


Negating both sides of \ref{e:eq3} and using De Morgan's laws, we obtain 
\begin{align*}
\neg \neg \psi \approx \psi & \approx  \neg (\neg \psi \land (\neg \varphi \otimes \neg \psi)) \\
& \approx  \neg \neg \psi \lor \neg (\neg \varphi \otimes \neg \psi) \\
& \approx  \psi \lor (\neg \neg \varphi \otimes \neg \neg \psi) \\
& \approx  \psi \lor (\varphi \otimes \psi).
\end{align*}
From this and 
\ref{e:eq2}  it follows  that $\psi \approx \varphi \otimes \psi  $. A similar reasoning shows that 
\ref{e:eq1} and \ref{e:eq4}  
imply $\varphi \approx \varphi \otimes \psi $. Hence $\varphi \approx \psi$. 
\end{proof}

The preceding result enables us to characterize the Tarski congruence associated with $\mathcal{LB}$ as the relation defined by the  equations valid in $\mathcal{FOUR}$:

\begin{thm} \label{t:tarski}
The Tarski congruence associated with $\lb = \la \Al[Fm], \Log \ra$ is 
$$\widetilde{\leib[]}(\mathcal{LB})= \{ \langle \varphi , \psi \rangle: \mathcal{FOUR} \vDash \phi \approx \psi \}.$$
\end{thm}

\begin{proof}
Obviously, the relation $\{ \langle \varphi , \psi \rangle: \mathcal{FOUR} \vDash \phi \approx \psi \}$ is a congruence and, by Lemma \ref{l:tfsae} (ii), it is also clear that it is the maximal congruence below the Frege relation.
\end{proof}

Recalling \cite[Propositions 1.23 and 2.26]{FJa09}, we can conclude that  both $\mathbf{Alg^{*}}\mathcal{LB}$ and $\mathbf{Alg}\mathcal{LB}$ are classes of algebras generating the same variety as $\mathcal{FOUR}$ (which is, as we have seen in Chapter 
\ref{ch:int}, the variety $\DBL$ of distributive bilattices). In fact, we have the following:

\begin{thm} \label{t:alg}
The class $\mathbf{Alg}\mathcal{LB}$ is the variety generated by $\mathcal{FOUR}$, i.e.\ the variety of distributive bilattices.
\end{thm}

\begin{proof}
It is clear that $\mathcal{FOUR} \in \mathbf{Alg^{*}}\mathcal{LB} \subseteq \mathbf{Alg}\mathcal{LB}$. By  \cite[Theorem 2.23]{FJa09} we also have 
$$\mathcal{FOUR} \in \mathbf{Alg}\mathcal{LB} = P_{s}(\mathbf{Alg^{*}}\mathcal{LB}) \subseteq V(\mathcal{FOUR}).$$ 
Recall that $V(\four) = \DBL $ is congruence-distributive. Hence we may apply J\'{o}nsson's Lemma  \cite[Corollary IV.6.10]{BuSa00} to conclude that the subdirectly irreducible members of $V(\four)$ belong to $HS(\four)$, and clearly the only algebras in  $HS(\four)$ are the trivial one and $\four$ itself. Then we may conclude that
$V(\mathcal{FOUR})=P_{s}(\mathcal{FOUR}) \subseteq P_{s}(\mathbf{Alg^{*}}\mathcal{LB}).$ Hence we obtain $P_{s}(\mathbf{Alg^{*}}\mathcal{LB})=\mathbf{Alg}\mathcal{LB}=  V(\mathcal{FOUR}).$
\end{proof}

An immediate corollary of the previous result concerning the algebraic reducts of logical bilattices is that $\mathsf{LoBiLat} \nsubseteq  \mathbf{Alg}\mathcal{LB}$. This is so because $\la \7, \{\top, \true \}\ra$ is a logical bilattice, but $\7 \notin \mathbf{Alg}\mathcal{LB}$, for this bilattice is not distributive (not even interlaced, as one can easily see by cardinality condiderations). We can also verify that $\mathcal{NINE} \in \mathbf{Alg}\mathcal{LB}$, since this bilattice is distributive. Taking into account the results of the previous chapter, this last claim follows from the fact that $\mathcal{NINE} \cong \Al[3] \odot \Al[3]$, where $\Al[3]$ denotes the three-element lattice, which is of course distributive.

Having individuated a class which, according to the general theory of \cite{FJa09}, may be regarded as the algebraic counterpart of the logic $\LB$, we may wonder if this class could also be the algebraic counterpart of some other logic. Thanks to the general results of \cite{BP89}, in some cases one may be able to prove that a certain class of algebras cannot be  the equivalent algebraic semantics of any algebraizable logic (such a result has been obtained, for instance, for the varieties of distributive lattices and of De Morgan lattices: see \cite{FV91} and \cite{F97}). This, however, is not the case with distributive bilattices, for it is possible to define a logic which is algebraizable w.r.t.\  the class $\DBL$. Consider the following:

\begin{ex} \label{ex:reglog}
Let $\RB = \la \Al[Fm], \vdash_{\RB} \ra $ be the logic defined, for all $\Gamma \cup \{ \phi \} \subseteq Fm$, as follows: $\Gamma  \vdash_{\RB} \phi $ iff $ \tau (\Gamma) \vDash_{\DBL} \tau (\phi)$, where $\tau$ is a translation from formulas into equations defined as $\tau(\phi) = \{ \phi \approx \neg \phi \}$ for all $\phi \in Fm$. Note that, by definition, the least $\RB$-filter on any distributive bilattice coincides with the set of regular elements defined in the previous chapter (whence the name we have chosen for the logic). 
It also follows from the definition that $\RB$ satisfies one of the two conditions for being algebraizable w.r.t.\ the variety $\DBL$, hence it will be sufficient to show that it satisfies the other one as well, namely the existence of a translation $\rho$ from equations into formulas s.t.\ $\phi \approx \psi =\joinrel\mid \models_{\DBL}  \tau (\rho (\phi \approx \psi) )$. It is not difficult to check that, defining $$ \rho (\phi \approx \psi) = \{ \neg \phi \otimes \psi, (\phi \oplus \neg \phi) \land ( \psi \oplus \neg \psi) \},$$ the condition is satisfied.  We have to prove that $\phi \approx \psi =\joinrel\mid \models_{\DBL} \{ \neg \phi \otimes \psi \approx \neg (\neg \phi \otimes \psi), (\phi \oplus \neg \phi) \land ( \psi \oplus \neg \psi) \approx \neg ((\phi \oplus \neg \phi) \land ( \psi \oplus \neg \psi)) \} .$ The rightwards direction is immediate; for the other one, note that $\neg \phi \otimes \psi \approx \neg (\neg \phi \otimes \psi)$ is equivalent to $\neg \phi \otimes \psi \approx \phi \otimes \neg \psi$ and $(\phi \oplus \neg \phi) \land ( \psi \oplus \neg \psi) \approx \neg ((\phi \oplus \neg \phi) \land ( \psi \oplus \neg \psi))$ is equivalent to $\phi \oplus \neg \phi \approx \psi \oplus \neg \psi$. Now let  $\Al[B] \in \DBL$ and $a, b \in \Al[B]$ such that $\neg a \otimes b = a \otimes \neg b$ and  $a \oplus \neg a \approx b \oplus \neg b$. Using the absorption and the distributive laws, we obtain 
\begin{align*}
    a & = a \otimes (a \oplus \neg a)  \\
    & =  a \otimes (b \oplus \neg b) \\ 
    & = (a \otimes b) \oplus (a \otimes \neg b) \\ 
    & = (a \otimes b) \oplus (\neg a \otimes b) \\ 
    & = b \otimes (a \oplus \neg a) \\ 
    & = b \otimes (b \oplus \neg b) \\ 
    & = b.
\end{align*}
%
%
\end{ex}  

In order to describe the class of g-models of $\lb$, we shall use the following  characterization of $\lb$-filters:

\begin{proposition} \label{prop:lb_filters} 
Let $\Al[B]$ be a distributive bilattice and $F \subseteq B$. Then $F$ is an $\lb$-filter if and only if $F$ is a bifilter of $\Al[B]$ or $F = \emptyset$.
\end{proposition}

\begin{proof} 
For $F$ empty the proof is trivial, so assume it is not. By rules (R3), (R4), (R3') and (R4') of our Hilbert calculus $\vdash_{\scriptscriptstyle H}$, it is obvious that any $\lb$-filter on $\Al[B]$ is a bifilter. It is also easy to see that, in a distributive bilattice, any bifilter is closed w.r.t.\  all rules of our Hilbert calculus. To see that it  is closed under (R18) and (R19), recall that any interlaced (hence, any distributive) bilattice satisfies that $a \lor b \leq_{k} a \oplus b$ and $a \oplus b \leq_{t} a \lor b$ for all $a,b \in B$. Therefore, since any bifilter $F$ is upward closed w.r.t.\ both lattice orders, we have that $a \lor b \in F$ iff $a \oplus b \in F$. 
\end{proof}

Combining the result of the previous proposition with that of Theorem \ref{t:alg}, we immediately obtain the following:


\begin{proposition} \label{prop:g_models}
If a g-matrix $\la \Al, \mathcal{C} \ra$ is a reduced g-model of $\lb$, then $\Al$ is a distributive bilattice and any non-empty $F \in \mathcal{C} $ is a bifilter. 

\end{proposition}


One may wonder if the result of Proposition \ref{prop:g_models} could be strengthened, proving that if a g-matrix $\la \Al, \mathcal{C} \ra$ is a reduced g-model of $\lb$, then $\Al$ is a distributive bilattice and $\mathcal{C} $ is the family of all bifilters of  $\Al$ (possibly plus the empty set). This is not the case, as we shall see later (Example \ref{e:g_models}). 

On the other hand, in Theorem \ref{t:fullmod} we shall see that the  g-models of $\lb$ that satisfy this requirement (i.e.\ the full models of $\lb$) are exactly the g-models without theorems that inherit the metalogical properties stated in Proposition \ref{p:metalog}.

We will use the following results:

\begin{lem} \label{l:fullmod}
A g-matrix $\la \Al, \mathcal{C} \ra$ is a full model of $\lb$ if and only if there is a bilogical morphism between $\la \Al, \mathcal{C} \ra$ and a g-matrix $\la \Al', \mathcal{C'} \ra$, where  $\Al'$ is a distributive bilattice and  $ \mathcal{C'} = \FF(\Al) \cup \{ \emptyset \} $.
\end{lem}

\begin{proof}
Follows from the general result of \cite[Proposition 2.21]{FJa09} together with our Theorem \ref{t:alg} and Proposition \ref{prop:lb_filters}. 
\end{proof}

\begin{lem} \label{l:cong}
Let $\la \Al, \Al[C] \ra$ be an abstract logic satisfying properties (PC), (PDI), (PDN) and (PDM). Then the relation 
$$ \tarski(\Al[C]) = \{ \la a, b \ra \in A \times A: \: \Al[C](a) = \Al[C](b) \textrm{ and } \Al[C](\neg a) = \Al[C](\neg b) \} $$
is a congruence of $\Al$ and the quotient  algebra $\Al / \tarski(\Al[C])$  is a distributive bilattice.
\end{lem}

\begin{proof}
Clearly $\tarski(\Al[C])$ is an equivalence relation and, using properties (PC) to (PDM), it is not difficult to prove that it is also a congruence of $\Al$. For instance, to see that $\la a_{1}, b_{1} \ra, \la a_{2}, b_{2} \ra \in \tarski(\Al[C])$ implies $\la a_{1} \land a_{2}, b_{1} \land b_{2} \ra \in \tarski(\Al[C])$, note that we have
\begin{align*} 
\Al[C](a_{1} \land a_{2}) & = \Al[C](a_{1}, a_{2}) & \text{ by (PC)} \\ 
& = \Al[C](\Al[C](a_{1}), \Al[C](a_{2}))  \\ 
& = \Al[C](\Al[C](b_{1}), \Al[C](b_{2})) & \text{ by hypothesis}  \\ 
& = \Al[C](b_{1}, b_{2})  \\ 
& = \Al[C](b_{1} \land b_{2}) & \text{ by (PC)} 
\end{align*}
and 
\begin{align*} 
\Al[C](\neg (a_{1} \land a_{2})) & = \Al[C](\neg a_{1} \lor \neg a_{2}) & \text{ by (PDM)} \\
& = \Al[C](\neg a_{1}) \cap \Al[C](\neg a_{2}) & \text{ by (PDI)} \\
& = \Al[C](\neg b_{1}) \cap \Al[C](\neg b_{2}) & \text{ by hypothesis}  \\
& = \Al[C](\neg b_{1} \lor \neg b_{2}) & \text{ by (PDI)} \\
& = \Al[C](\neg (b_{1} \land b_{2})) & \text{ by (PDM).} 
\end{align*}
A similar reasoning shows that $\tarski(\Al[C])$ is compatible with the other bilattice connectives. To show that the quotient $\Al / \tarski(\Al[C])$ is a distributive bilattice, we need to check that, for any equation $\phi \approx \psi$ axiomatizing the variety $\DBL$, we have $\la \phi, \psi \ra \in \tarski(\Al[C])$. This is not difficult, but quite long. Let us check, for instance, just one of the distributive identities. We have
\begin{align*} 
\Al[C](a \land (b \lor c)) & = \Al[C](a, b \lor c) & \text{ by (PC)} \\
& = \Al[C](a, b) \cap \Al[C](a, c) & \text{ by (PDI)} \\
& = \Al[C](a \land b) \cap \Al[C](a \land c) & \text{ by (PC)} \\
& = \Al[C]((a \land b) \lor (a \land c)) & \text{ by (PDI)}
\end{align*}
and
\begin{align*} 
\Al[C](\neg (a \land (b \lor c))) & = \Al[C](\neg a \lor \neg  (b \lor c)) & \text{ by (PDM)} \\
& = \Al[C](\neg a) \cap \Al[C](\neg  (b \lor c)) & \text{ by (PDI)} \\
& = \Al[C](\neg a) \cap \Al[C](\neg  b \land \neg c)) & \text{ by (PDM)} \\
& =  \Al[C](\neg a) \cap \Al[C](\neg  b, \neg c) & \text{ by (PC)} \\
& =  \Al[C](\neg a, \neg a) \cap \Al[C](\neg  a, \neg b)  \cap \Al[C](\neg  a, \neg c)  \cap \Al[C](\neg  b, \neg c) \\
& =  \Al[C](\neg a, \neg a \lor \neg b ) \cap \Al[C](\neg c, \neg  a \lor \neg b )  & \text{ by (PDI)} \\
& =  \Al[C](\neg a \lor \neg c, \neg a \lor \neg b )  & \text{ by (PDI)} \\
& =  \Al[C](\Al[C] (\neg a \lor \neg c), \Al[C](\neg a \lor \neg b) ) \\
& =  \Al[C](\Al[C] (\neg ( a \land c)), \Al[C](\neg (a \land b) ))  & \text{ by (PDM)} \\
& =  \Al[C](\neg ( a \land b), \neg (a \land c) ) \\
& =  \Al[C](\neg ( a \land b) \land \neg (a \land c) ) & \text{ by (PC)} \\
& =  \Al[C](\neg (( a \land b) \lor (a \land c)) ) & \text{ by (PDM)}.
\end{align*}
\end{proof}

\begin{thm} \label{t:fullmod}
An abstract logic $\la \Al, \Al[C] \ra$ is a full model of $\lb$ if and only if it is finitary and satisfies, for all $a, b \in A$ and all $X \subseteq A$, the following properties: 

\begin{tabular}[l]{clllll}
\\
\emph{(E)}  &  $ \Al[C](\emptyset ) =  \emptyset $ \\  \\
\emph{(PC)}  &  $ \Al[C](a \land b ) =  \Al[C](a  \otimes b )  =  \Al[C](a , b ) $ \\  \\
\emph{(PDI)}   &  $ \Al[C](X, a \lor b ) =  \Al[C](X, a \oplus b )  =  \Al[C](X, a)  \cap \Al[C](X, b) $ \\ \\
\emph{(PDN)}  &  $ \Al[C](a ) =  \Al[C](\neg \neg a )$ \\ \\
\emph{(PDM)}  &  $ \Al[C](\neg (a \land b)) =  \Al[C](\neg a \lor \neg b)) $ \\
  &  $ \Al[C](\neg (a \lor b)) =  \Al[C](\neg a \land \neg b)) $ \\
     &  $ \Al[C](\neg (a \otimes b)) =  \Al[C](\neg a \otimes \neg b)) $ \\
        &  $ \Al[C](\neg (a \oplus b)) =  \Al[C](\neg a \oplus \neg b)) $. 
\\
\end{tabular}
\end{thm}

\begin{proof}
$(\Rightarrow)$ Let $\la \Al, \Al[C] \ra$ be a full model of $\lb$ and let $\mathcal{C}$ be the closure system associated with $\Al[C]$. As any full model, $\la \Al, \Al[C] \ra$ will be finitary. By Lemma \ref{l:fullmod}, we know that there is a bilogical morphism $h: A \rightarrow B$ onto an abstract logic of the form $\la \Al[B], \FF \ra$, where $\Al[B] \in \DBL$ and $\FF$ is the operator of bifilter generation (if we set $\FF(\emptyset) = \emptyset $). This last condition implies that $\la \Al, \Al[C] \ra$ satisfies condition $(E)$, for the least closed set $T \in \mathcal{C}$ will be the empty set. To prove (PC), note that, for any $a, b, c \in A$ we have
\begin{align*}
   c \in  \Al[C] (a \land b)  & \:  \textrm{ iff } \: h(c) \in \FF (h (a \land b) )  \\
    &  \: \textrm{ iff } \: h(c) \in \FF (h (a) \land h (b) ) \\
    & \:  \textrm{ iff } \: h(c) \in \FF (h (a), h (b) ) \\
     & \:  \textrm{ iff } \:  c \in  \Al[C] (a, b).
\end{align*}
%
The first two equivalences and the last one hold because $h$ is a bilogical morphism, while the third follows from the definition of bifilter. The same reasoning shows that  $\Al[C](a  \otimes b )  =  \Al[C](a , b ) $. As to (PDI), recall that, by Proposition \ref{p:metalog}, $\lb$ satisfies it: then we may apply \cite[Theorem 2.52]{FJa09} to conclude that any full model of $\lb$ will have the (PDI) as well. Finally, (PDN) and (PDM) are easily proved using the double negation and De Morgan's laws for bilattices together with the fact that $h$ is a bilogical morphism.

$(\Leftarrow)$  Let $\la \Al, \Al[C] \ra$ be a finitary logic that satisfies properties (E) to (PDM). We will prove that there is a bilogical morphism between $\la \Al, \Al[C] \ra$ and  an abstract logic of the form $\la \Al[B], \FF \ra$, where $\Al[B] \in \DBL$ and $\FF$ is the operator of bifilter generation. The morphism is given by the canonical projection associated with the following congruence:
$$ \tarski(\Al[C]) = \{ \la a, b \ra \in A \times A: \: \Al[C](a) = \Al[C](b) \textrm{ and } \Al[C](\neg a) = \Al[C](b) \}. $$
By Lemma  \ref{l:cong} we have that $\tarski(\Al[C])$ is a congruence and that the quotient algebra $\Al / \tarski(\Al[C])$  is a distributive bilattice. By definition, the canonical projection $ \pi : \Al \rightarrow \Al / \tarski(\Al[C])$ is an epimorphism, so we only need to prove that $a \in \Al[C](X) $ iff $\pi(a) \in \FF (\pi[X]) $ for all $X \cup \{ a \} \subseteq A$.

 If $X$ is empty, then it is immediate. Assume then $X \neq \emptyset$ and $a \in \Al[C](X)$. By finitarity, there is a finite set $\{a_{1}, \ldots, a_{n} \} \subseteq X$ such that $a \in \Al[C](\{a_{1}, \ldots, a_{n} \})$. In order to simplify the notation, note that by (PC) and the properties of closure operators we have $\Al[C](\{a_{1}, \ldots, a_{n} \}) =\Al[C]( \{a_{1} \land \ldots \land a_{n} \} ) $, so let $b = a_{1} \land \ldots \land a_{n}$. Then we have to prove that $\pi(a) \in \FF (\pi( b )) $. By Corollary \ref{cor:bifilters}, 
 this happens when $(\pi(a) \otimes \pi(b)) \land \pi(b) = \pi(b)$, i.e. when $\pi ( (a \otimes b) \land b ) = \pi (b ) $. So we have to prove that $\Al[C] ( (a \otimes b) \land b ) = \Al[C] (b ) $  and $\Al[C] ( \neg ((a \otimes b) \land b) ) = \Al[C] ( \neg b ) $. Applying (PC), the first equality becomes $\Al[C] ( a, b ) = \Al[C] (b ) $, which is true since by hypothesis $a \in \Al[C] (b )$. As to the second equality, applying (PC), (PDM) and (PDI) we have
\begin{align*}
\Al[C] ( \neg b )   &  = \Al[C] ( \neg a, \neg b) \cap \Al[C] ( \neg  b)  \\
&  = \Al[C] ( \neg a \otimes \neg b) \cap \Al[C] ( \neg  b)  \\
&  = \Al[C] ( \neg (a \otimes b)) \cap \Al[C] ( \neg  b)  \\
& = \Al[C] ( \neg (a \otimes b) \lor \neg b) ) \\
& = \Al[C] ( \neg ((a \otimes b) \land b) ).
\end{align*}

Conversely, assume  $\pi(a) \in \FF (\pi[X]) $. Again by Corollary \ref{cor:bifilters}, 
we know that this last condition is equivalent to the existence of $a_{1}, \ldots, a_{n} \in X$ such that 
$$(\pi(a) \otimes (\pi(a_1) \land \ldots \land \pi(a_n))) \land   \pi(a_1) \land \ldots \land \pi(a_n) = \pi(a_1) \land \ldots \land \pi(a_n).$$
Letting  $b = a_{1} \land \ldots \land a_{n}$ and using the fact that $\pi$ is a homomorphism, we obtain
$$\pi( (a \otimes b) \land  b) = \pi(b).$$
This implies $ \Al[C] ( (a \otimes b) \land  b) =  \Al[C] (b)$, i.e.\ $ \Al[C] ( a, b) =  \Al[C] (b)$, i.e.\ $a \in \Al[C] (b)$. Hence $a \in \Al[C](X)$.
\end{proof}

Abstract logics can be studied as models of Gentzen systems (see \cite{FJa09}). In this context, we say that an abstract logic $\la \Al, \Al[C] \ra$ is a \emph{model} of a Gentzen system $\g$ when for any family of sequents  $\{ \Gamma_{i} \seq \phi_{i} : i \in I \}$ and for any sequent $\Gamma \seq \phi$ such that $\{ \Gamma_{i} \seq \phi_{i} : i \in I \} \Gens \Gamma \seq \phi$ it holds that for any homomorphism $h:  \Al[Fm] \rightarrow  \Al$ such that $h(\phi_{i}) \in   \Al[C] (h[\Gamma_{i}])$ for all $i \in I$, also $h(\phi) \in   \Al[C] (h[\Gamma])$. Recall also that a Gentzen $\g$ system is said to be \emph{adequate} for a logic $\mathcal{L} = \la \Al[Fm], \vdash_{\mathcal{L}} \ra$ when $\Gamma \vdash_{\mathcal{L} } \varphi$ if and only if $ \emptyset \Gensin \Gamma \rhd \varphi$ for any $\Gamma \cup \{ \phi \} \subseteq Fm$. We say that a Gentzen system $\g$ is \emph{fully adequate} for a logic $\mathcal{L} = \la \Al[Fm], \vdash_{\mathcal{L}} \ra$ when any abstract logic $\la \Al, \Al[C] \ra$ is a full model of $\mathcal{L}$ if and only if it is a finitary model of $\g$ (with theorems if $\mathcal{L}$ has, otherwise  without theorems). 


We have seen in Theorem \ref{thm:GentzenComplet} that $\mathcal{G_{LB}}$ is adequate for the logic $\lb$. Now, using Theorem \ref{t:fullmod}, we immediately obtain the following as a corollary:

 
\begin{thm} \label{t:fullyad}
The Gentzen system $\mathcal{G_{LB}}$ is fully adequate for the logic $\lb$.
\end{thm}



In order to characterize the class of matrix models of $\mathcal{LB}$, we will now turn to the study of the Leibniz congruence of $\mathcal{LB}$.

\begin{proposition} \label{p:leibcong}
Let $\langle \Al, F \rangle$ be a model of the logic $\lb$. Then, for all $a, b \in A$, the
  following are equivalent:
  \begin{enumerate}[(i)]
    \item  $\langle a, b \rangle \in \leibniz_{\Al} (F)$,
    \item $\{ c \in A: a \lor c \in F \} = \{ c
  \in A: b \lor c \in F \}$ and \\ $\{ c \in A:  \neg a  \lor c \in F \} = \{ c
  \in A:  \neg b  \lor c \in F \}$.
  \item $\{ c \in A: a \oplus c \in F \} = \{ c
  \in A: b \oplus c \in F \}$  and \\ $\{ c \in A:  \neg a  \oplus c \in F \} = \{ c
  \in A:  \neg b  \oplus c \in F \}$.
 \end{enumerate}
\end{proposition}

\begin{proof} 
(i) $ \Rightarrow$ (ii). It is easy to see that any congruence $\theta$ compatible with $F$ must satisfy (ii). For instance, if $\langle a, b \rangle \in \theta$, then, for any $c \in A$, we have $\langle a \lor c, b \lor c \rangle \in \theta$ as well. Hence we have that $a \lor c \in F $ if and only if $b \lor c \in F$. A similar argument shows also that (i) implies (iii).

(ii) $ \Rightarrow$ (i). Let $\theta$ be the relation defined by the conditions of (ii), that is, for all $a, b \in A$, we set $\la a, b \ra \in \theta $ if and only if  
$\{ c \in A: a \lor c \in F \} = \{ c
  \in A: b \lor c \in F \}$ and $\{ c \in A:  \neg a  \lor c \in F \} = \{ c
  \in A:  \neg b  \lor c \in F \}.$
%
Clearly, to prove that $\theta \subseteq \leibniz_{\Al} (F)$, it is sufficient to check that $\theta$ is a congruence compatible with $F$. Taking into account the fact that $F $ is an $\lb$-filter, it is not difficult to see that $\theta$ is a congruence. We need to prove, for instance, that $\la a_{1}, b_{1} \ra, \la a_{2}, b_{2} \ra  \in \theta $ implies $\la a_{1} \land a_{2},  b_{1} \land b_{2} \ra \in \theta$. For this, assume $(a_{1} \land a_{2}) \lor c \in F$ for some $c \in A$. This implies
\begin{align*}
c \lor  (a_{1} \land a_{2})             \: & \in F & \text{ by (R5)} \\
(c \lor  a_{1}) \land (c \lor a_{2})  \: & \in F & \text{ by (R8)} \\
(c \lor  a_{1}), (c \lor a_{2})          \: & \in F & \text{ by (R1) and (R2)} \\
(a_{1} \lor c), (a_{2} \lor c)           \: & \in F & \text{ by (R5)} \\
(b_{1} \lor c), (b_{2} \lor c)          \: & \in F & \text{ by definition of } \theta  \\
(c \lor b_{1}), (c \lor b_{2})          \: & \in F & \text{ by (R5)}  \\
(c \lor b_{1})  \land (c \lor b_{2})          \: & \in F & \text{ by (R3)}  \\
c \lor (b_{1} \land b_{2})                \: & \in F & \text{ by (R9)}  \\
(b_{1} \land b_{2})  \lor c              \: & \in F & \text{ by (R5)}.
\end{align*}
Hence the first condition of (ii) is satisfied. A similar argument allows to prove the second one as well, so that we may conclude that $\la a_{1} \land a_{2},  b_{1} \land b_{2} \ra \in \theta$.

To see  that $\theta$ is compatible with $F$, assume $\la a, b \ra \in \theta $ and $a \in F$. We have:
\begin{align*}
a \lor b      	       				\: & \in F & \text{ by (R4)} \\
b \lor b						  \: & \in F & \text{ by definition of } \theta  \\
b 					     \: & \in F & \text{ by (R6)} .
\end{align*}
(ii) $ \Leftrightarrow$ (iii). This is almost immediate, since by (R18) and (R19) we have that $a \lor b \in F$ iff $a \oplus b \in F$ for any $a, b \in A$ and any $\lb$-filter $F$.
\end{proof}

As a consequence of Proposition \ref{p:leibcong}, we obtain the following characterization of the reduced matrix models of $\mathcal{LB}$:

\begin{thm} \label{thm:red_mod}
Let $\Al[A]$  be a non-trivial algebra. Then the following conditions are equivalent:
 \begin{enumerate}[(i)]
 \item  $\la \Al, F \ra$ is a reduced matrix for $\lb$,
 \item  $\Al \in \DBL$ and $F$ is a bifilter s.t., for all $a, b \in A$, if $a <_{t} b$, then there is $c \in A$ s.t. either $a \lor c \notin F$ and $b \lor c \in F$ or $\neg a \lor c \in F$ and $\neg b \lor c \notin F$,
\item $\Al \in \DBL$ and $F$ is a bifilter s.t., for all $a, b \in A$, if $a <_{k} b$, then there is $c \in A$ s.t. either $a \lor c \notin F$ and $b \lor c \in F$ or $\neg a \lor c \notin F$ and $\neg b \lor c \in F$.
\end{enumerate}
\end{thm}

\begin{proof} 
(i) $\Rightarrow$ (ii). Assume $\la \Al, F \ra$ is a reduced matrix for $\lb$. That $\Al \in \DBL$ follows from Theorem \ref{t:alg}, while Proposition \ref{prop:lb_filters} implies that $F$ is a bifilter (the assumption that $\Al$ is not trivial guarantees that $F \neq \emptyset$). Notice that  $a <_{t} b$ implies that $b \in \FF(a)$ and $\neg a \in \FF(\neg b)$; obviously it also implies that $\langle a, b \rangle \notin \leibniz_{\Al} (F)$. By Proposition \ref{p:leibcong}, this means that either $\{ c \in A: a \lor c \in F \} \neq \{ c   \in A: b \lor c \in F \}$ or $\{ c \in A:  \neg a  \lor c \in F \} \neq \{ c \in A:  \neg b  \lor c \in F \}$. If the first is the case, then, for some $c \in A$, either $a \lor c \notin F$ and $b \lor c \in F$ or $a \lor c \in F$ and $b \lor c \notin F$. The latter hypothesis is impossible, for $b \in \FF(a)$ implies $ \FF(b \lor c) = \FF(b) \cap \FF(c) \subseteq \FF(a) \cap \FF(c) = \FF(a \lor c) $. So if $a \lor c \in F$, then $b \lor c \in F$ for any bifilter $F$. Hence the former hypothesis must be true. A similar argument can be applied to the case of $\{ c \in A:  \neg a  \lor c \in F \} \neq \{ c \in A:  \neg b  \lor c \in F \}$. Recalling that  $a <_{k} b$ implies  $b \in \FF(a)$ and $\neg b \in \FF(\neg a)$, it is easy to apply the same reasoning in order to show also  that (i) $\Rightarrow$ (iii).

(ii) $\Rightarrow$ (i). Assume that $\Al \in \DBL$ and $F$ is a bifilter satisfying (ii). Assume also $a \neq b$. Then $a \land b <_{t} a \lor b$, hence we may apply the assumption and Proposition \ref{p:leibcong} to conclude that $\langle a \land b, a \lor b \rangle \notin \leibniz_{\Al} (F)$. Since we are in a lattice, this implies $\langle a, b \rangle \notin \leibniz_{\Al} (F)$. Hence $ \leibniz_{\Al} (F) = Id_{\Al}$. A similar reasoning shows that (iii) $\Rightarrow$ (i).
\end{proof} 

Notice that, using the characterization given by item (iii) instead of (ii) of Proposition \ref{p:leibcong}, 
we could equivalently formulate conditions (ii) and (iii) of Theorem \ref{thm:red_mod} using $\oplus$ instead of $\lor$, thus obtaining the following:
%


\begin{cor} \label{cor:red_mod}
Let $\Al[A]$  be a non-trivial algebra. The following conditions are equivalent:
 \begin{enumerate}[(i)]
 \item  $\la \Al, F \ra$ is a reduced matrix for $\lb$,
 \item  $\Al \in \DBL$ and $F$ is a bifilter s.t., for all $a, b \in A$, if $a <_{t} b$, then there is $c \in A$ s.t. either $a \oplus c \notin F$ and $b \oplus c \in F$ or $\neg a \oplus c \in F$ and $\neg b \oplus c \notin F$,
\item $\Al \in \DBL$ and $F$ is a bifilter s.t., for all $a, b \in A$, if $a <_{k} b$, then there is $c \in A$ s.t. either $a \oplus c \notin F$ and $b \oplus c \in F$ or $\neg a \oplus c \notin F$ and $\neg b \oplus c \in F$.
\end{enumerate}
\end{cor}

We know that all algebras in $ \mathbf{Alg^*}\LB$ are distributive bilattices, hence, by our Representation Theorem \ref{thm:interlaced_bl}, 
isomorphic to a product bilattice. The following lemma enables us to determine which requirements a lattice $\Al[L]$ must satisfy in order to have that $ \Al[L] \odot \Al[L] \in \mathbf{Alg^*}\LB$.


\begin{lem} \label{l:leib}
Let $\Al[B] = \la B,  \land, \lor, \otimes, \oplus, \neg \ra$ be an interlaced bilattice, let $D \subseteq \Reg(\Al[B])$ be a lattice filter of the lattice $ \mathbf{Reg}(\Al[B]) = \la \Reg(\Al[B]), \otimes, \oplus \ra $ and let $\theta \in \Con (\Al[B])$. Then: 
%
\begin{enumerate}[(i)]
  \item $D = \FF(D) \cap \Reg(\Al[B])$,
  \item $\theta$ is compatible with $\FF (D)$ if and only if $\ \theta \:  \cap {\Reg(\Al[B]) \times \Reg(\Al[B])}$ is compatible with $D$,
  \item $\leib[B] (\FF (D)) = Id_{B}$ if and only if  $\leib[]_{\mathbf{Reg}(\Al[B])} (D) = Id_{\Reg(\Al[B])}$.
\end{enumerate}
\end{lem}

\begin{proof} 
(i). Obviously $D \subseteq \FF(D) \cap \Reg(\Al[B])$. To prove the other inclusion, assume $a \in \FF(D) \cap \Reg(\Al[B])$. By Lemma \ref{lem:bifilter}, the assumption implies that there are  $a_{1}, \ldots, a_{n} \in D$ such that $a_{1} \otimes \ldots \otimes a_{n} \leq_{t} a_{1} \otimes \ldots \otimes a_{n} \otimes a$. Let $ b = a_{1} \otimes \ldots \otimes a_{n}$. Since $D$ is a filter,  $b \in D $, so we have that $b \leq_{t} b \otimes a$. Since $b \otimes a \in \Reg(\Al[B])$, this implies that $b = b \otimes a$. Hence $b \leq_{k} a$ and, using again the fact that $D$ is a filter, we conclude that $a \in D$.

(ii). Assume $\theta \in \Con (\Al[B])$ is compatible with $\FF (D)$, $a \in D$ and $\la a, b \ra \in  \theta \:  \cap {\Reg(\Al[B]) \times \Reg(\Al[B])}$. The assumptions imply $a \in \FF (D)$ and $\la a, b \ra \in \theta$, hence $b \in \FF (D)$. Now, using (i), we may conclude that $b \in D$.

Conversely, assume $\theta \:  \cap {\Reg(\Al[B]) \times \Reg(\Al[B])}$ is compatible with $D$,  $\la a, b \ra \in \theta$ and $a \in \FF(D)$. By Proposition \ref{prop:regular} (ii), this last assumption implies $\reg(a) \in \FF(D)$. Hence, using again (i), we have $\reg(a) \in D$. By Proposition \ref{p:blconf_cong} (i), $\la a, b \ra \in \theta$ implies  $\la \reg(a), \reg(b) \ra \in \theta$. Then, applying compatibility, we have $\reg(b) \in D$. Hence $\reg(b) \in \FF(D)$ and, applying again Proposition \ref{prop:regular} (ii), we conclude that $b \in \FF(D)$.

(iii). Recall that $\Al[B] \cong \mathbf{Reg}(\Al[B]) \odot \mathbf{Reg}(\Al[B])$. Hence, by Proposition  \ref{p:interbl_cong}, we have that $\Con (\Al[B] ) \cong \Con (\mathbf{Reg}(\Al[B]) )$. It is easy to see that the isomorphism is given by the map
$h: \Con (\Al[B] ) \ta\Con (\mathbf{Reg}(\Al[B]) )$
defined, for all $\theta \in  \Con (\Al[B])$, as $h(\theta) = \theta \cap \Reg(\Al[B]) \times \Reg(\Al[B])$ (see Proposition \ref{p:blconf_cong}). Then, applying (ii), the result easily follows.
\end{proof}

Now we can easily obtain the following characterization:


\begin{thm} \label{t:alg*}
Let $\Al[A]$ be a non-trivial algebra. Then a matrix $\langle \Al[A], F  \rangle$ is a reduced model of $\mathcal{LB}$ if and only if
$\Al[A]$ is a distributive bilattice isomorphic to $\mathbf{Reg}(\Al[A]) \odot \mathbf{Reg}(\Al[A])$ such that 
the following conditions are satisfied:
\begin{enumerate}[(i)]
	\item $\mathbf{Reg}(\Al[A])  = \langle \Reg(\Al), \otimes, \oplus \rangle$ is a distributive lattice with top element $\top$ satisfying the property that, for all $a, b \in  \Reg(\Al)$ such that $a <_{k} b$, there is $c \in  \Reg(\Al) $ such that $a \oplus c \neq \top$ and $b \oplus c = \top$,
	\item $F =  \FF(\top)$.
\end{enumerate}
\end{thm}

\begin{proof}
We know, by Theorem \ref{thm:red_mod}, that $\Al$ is a distributive bilattice and by assumption   $\Al $ is non-trivial, hence $F \neq \emptyset$.   By Proposition \ref{prop:lb_filters}, $F \subseteq B$ is an $\mathcal{LB}$--filter iff $F$ is a bifilter of $\Al[A]$. Moreover, note  that $F = \FF (F \cap \Reg(\Al))$. It is obvious that $ \FF (F \cap \Reg(\Al)) \subseteq \FF(F) = F $. As to the other inclusion, assume $a \in F$. By Proposition \ref{prop:regular} (ii), this implies $\reg(a) \in F$. Hence we have $$a \in \FF(a) = \FF (\reg(a)) \subseteq  \FF (F \cap \Reg(\Al)). $$
Then we may apply Lemma  \ref{l:leib} to conclude that the matrix $\langle \Al, F \rangle$ is reduced if and only if the matrix $\langle \mathbf{Reg}(\Al[A]), F \cap \Reg(\Al) \rangle$ is reduced. As shown in \cite{FGV91}, this last condition is equivalent to (i)
plus  $F \cap \Reg(\Al) = \{ \top \}$. 
\end{proof}



Theorem \ref{t:alg*}  tells us    that any $\Al[B] \in \mathbf{Alg^*}\mathcal{LB}$ must have a top element w.r.t.\
 the knowledge ordering, i.e.\ $\top$.
 This also implies that $\Al[B]$ has a minimal nonempty bifilter, namely $\FF(\top)  = \{ a \in B : a \geq_t \top  \}$.
 Another interesting consequence of the theorem is that the result of Proposition \ref{prop:g_models} concerning the g-models of $\lb$ cannot be strengthened. That is, it is not true that if a g-matrix $\la \Al, \mathcal{C} \ra$ is a reduced g-model of $\lb$, then $\Al$ is a distributive bilattice and $\mathcal{C} $ is the family of all bifilters of  $\Al$. Consider  the following:

\begin{ex} \label{e:g_models}
Let  $\Al[L]$ be any lattice  that satisfies property (i) of Theorem \ref{t:alg*} (for instance the four-element non-linear distributive lattice), and let us denote its top element by 1. Then we know that the matrix $\la \Al[L], \{ 1 \} \ra$ is reduced. It is easy to see that  the matrix  $\la \Al[L], \{ 1 \} \ra$ is isomorphic to $\la \mathbf{Reg}(\Al[L] \odot \Al[L]), \la 1, 1 \ra \ra$. Note also that $\FF(\la 1, 1 \ra) = \{ 1 \} \times L$. Then, by Lemma \ref{l:leib}, we have that the matrix $\la \Al[L] \odot \Al[L], \{ 1 \} \times L \ra$ is a reduced model of $\lb$. Hence, any g-matrix $\la \Al[L] \odot \Al[L], \mathcal{C} \ra$ such that $\{ 1 \} \times L \in \mathcal{C}$ will be reduced as well. So, if we take for example  $  \mathcal{C} = \{ \{ 1 \} \times L, L \times L  \}$, then $\la \Al[L] \odot \Al[L], \mathcal{C} \ra$ is a reduced g-model of $\lb$, and clearly there may be bifilters of $\Al[L] \odot \Al[L]$ that are not in $ \mathcal{C}$.
\end{ex}


The class of lattices satisfying property (i) of Theorem \ref{t:alg*} seems to have some interest in itself and to deserve further study. Indeed,  algebras satisfying a property in some sense dual to our (i)  have already been considered in the literature, i.e.\ lattices having a minimum element 0 and satisfying that, for all $a, b $ such that $a > b$, there is $c$ such that $a \sqcap c \neq 0$ and $b \sqcap c = 0$. This property has been called \emph{disjunction property}, and the corresponding lattices \emph{disjunctive lattices} (see for instance \cite{Wa38} and \cite{Ci91}). In the same spirit,  we will here adopt the name \emph{dual disjunctive} for those lattices satisfying property (i) of Theorem \ref{t:alg*}.  

As noted in \cite{FGV91}, all Boolean lattices are dual disjunctive lattices in our sense. In fact, this result can be sharpened:

\begin{proposition} \label{p:relgood}
Let $\Al[L] = \langle L, \sqcap, \sqcup \rangle $ be a Boolean lattice whose minimum and maximum element are 0 and 1, and let $F \subseteq L$ be a filter of $\Al[L]$.  Then the sublattice of $\: \Al[L]$ with universe $F$ is a dual disjunctive lattice.
\end{proposition} 

\begin{proof}
Let $a, b \in F$ be such that $a > b$ and let $a'$ be the complement of $a$. Clearly $a' \sqcup b \in F $, and note that $a' \sqcup b < 1 $, because otherwise we would have $$ a \sqcap (a' \sqcup b) = a > b = a \sqcap b = (a \sqcap a') \sqcup (a \sqcap b) = a \sqcap (a' \sqcup b).$$ Moreover, $a \sqcup a' \sqcup b = 1$, but $b \sqcup a' \sqcup b = a' \sqcup b < 1$, and this completes the proof.
\end{proof}

One may wonder whether the converse of Proposition \ref{p:relgood} is also true, i.e.\ if any dual disjunctive lattice can be proved to be isomorphic to a filter of some Boolean lattice. This is not the case, a counterexample being the following: 

\begin{ex} \label{ex:dual}
Let $F$ be a non-principal filter (so, without bottom element) of a Boolean lattice $\Al[L] = \langle L, \sqcap, \sqcup \rangle $ whose maximum element is 1. Define the structure $$\Al[F^*] = \langle F \cup \{ 0 \}, \sqcap, \sqcup, 1 \rangle $$ with universe $F$ augmented with a new element $0 \notin L$, and whose lattice order is the one inherited from $\Al[L]$, except that we have $0 < a$ for all $a \in F$. Clearly $\Al[F^*]$ is a bounded distributive lattice, so if it were the filter of some Boolean lattice, it would itself be a Boolean lattice. But it is not, since for all $a, b \in F$ we have $a \sqcap b \in F$, i.e. \ $a \sqcap b > 0$. Therefore, no element in $F$ has a complement. 

On the other hand, it is easy to see that $\Al[F^*]$ is dual disjunctive. Clearly if $0 < a < b$ the condition is satisfied because $a, b \in F$. If $a = 0$, then let $c \in F$ such that $0 = a < c < b$ (such an element must exist, because $F$ had no bottom element). If we denote by $b'$ the complement of $b$ in $\Al[L]$, then we have  $b' \sqcup c \in F$ and $b \sqcup b' \sqcup c = 1$, but $ 0 \sqcup b' \sqcup c = b' \sqcup c < 1 $. So $\Al[F^*]$ is a dual disjunctive lattice. 
\end{ex}

In Chapter \ref{ch:imp}, 
in connection with the study of the algebraic models of an expansion of $\lb$, we will investigate a bit further the class of dual disjunctive lattices, in particular characterizing those that are indeed isomorphic to filters of Boolean lattices. For now, let us observe that the results just stated allow us to gain some additional information on the class $\mathbf{Alg^*}\mathcal{LB}$.

First of all, we may check that $\mathbf{Alg^*}\mathcal{LB}$ is closed under direct products but not under subalgebras (so it is not a quasivariety). 

The first claims follows from the fact that  $\mathbf{Alg^*}\mathcal{LB}$ is definable by a first-order universal formula. So $P(\mathcal{FOUR}) \subseteq \mathbf{Alg^*}\mathcal{LB}$, and by cardinality reasons we may see that this inclusion is strict, because there are countable algebras in $\mathbf{Alg^*}\mathcal{LB}$: one just needs to consider any bilattice $\Al[B] \cong \Al[L] \odot \Al[L]$ where $\Al[L] $ is a countable Boolean lattice. 

The second claim can be proved by considering the nine--element distributive bilattice $\mathcal{NINE}$. It is easy to see that $\mathcal{NINE}$ is isomomorphic to a subalgebra of $\mathcal{FOUR} \times \mathcal{FOUR}$, but on the other hand, as we have observed, $\mathcal{NINE} \cong  \Al[3] \odot \Al[3]$. Since the three--element lattice $\Al[3]$ is not a dual disjunctive lattice,  we may conclude that $\mathcal{NINE} \notin  \mathbf{Alg^*}\mathcal{LB}$. This in turn implies that $\mathbf{Alg^{*}}\mathcal{LB} \varsubsetneq \mathbf{Alg}\mathcal{LB}$. 

As noted in the first chapter, it is significant that in the case of $\lb$ these two classes do not coincide, as well as the fact that it is $\mathbf{Alg}\mathcal{LB}$, the class of distributive bilattices, the one that seems to be more naturally associated with this logic.

\section{Algebraizability of the Gentzen calculus $\mathcal{G_{LB}}$}
\label{sec:gentz}

As we anticipated, since our logic is not protoalgebraic, hence not algebraizable, there is a particular interest in studying the algebraic properties of sequent calculi associated with $\lb$. We end the section on this issue, stating the algebraizability of the Gentzen calculus $\mathcal{G_{LB}}$ introduced in Section \ref{sec:sem}.

\begin{thm} \label{thm:gentzalg}
The Gentzen calculus  $\mathcal{G_{LB}}$ is algebraizable w.r.t.\ the variety $\DBL$ of distributive bilattices, with the following translations: 
\begin{align*}
\tau (\Gamma \rhd \Delta) & = \{ \bigwedge \Gamma \land (\bigwedge \Gamma \otimes \bigvee \Delta) \approx  \bigwedge \Gamma \}, \\ \\
\rho (\varphi \approx \psi) & = \{ \varphi \rhd \psi, \: \neg \varphi \rhd \neg \psi,  \: \psi \rhd \varphi, \:  \neg \psi \rhd \neg \varphi \}.
\end{align*}

\begin{proof}
We will use the characterization of \cite[Lemma 2.5]{ReV95-p}.

(i). We have to check that $\Gamma \rhd \Delta \Gensinl \Gensin  \rho \tau (\Gamma \rhd \Delta)$, i.e.\ that
\begin{align*}
\Gamma \rhd \Delta \Gensinl \Gensin  \{ & \bigwedge \Gamma \land (\bigwedge \Gamma \otimes \bigvee \Delta) \rhd \bigwedge \Gamma, \;  \neg (\bigwedge \Gamma \land (\bigwedge \Gamma \otimes \bigvee \Delta)) \rhd \neg \bigwedge \Gamma,  \\ 
& \bigwedge \Gamma \rhd \bigwedge \Gamma \land (\bigwedge \Gamma \otimes \bigvee \Delta), \; 
\neg \bigwedge \Gamma \rhd \neg (\bigwedge \Gamma \land (\bigwedge \Gamma \otimes \bigvee \Delta)) \}.
\end{align*}

Let us prove the rightward direction. By (Ax) we have $\bigwedge \Gamma, \bigwedge \Gamma \otimes \bigvee \Delta \rhd \bigwedge \Gamma$. Now by ($\land \rhd$) we have $\bigwedge \Gamma \land (\bigwedge \Gamma \otimes \bigvee \Delta) \rhd \bigwedge \Gamma$.

By (Ax) we have $\neg \bigwedge \Gamma, \neg \bigvee \Delta \rhd \neg \bigwedge \Gamma$, so by $(\neg \otimes \rhd)$ we obtain  $\neg (\bigwedge \Gamma \otimes \bigvee \Delta) \rhd \neg \bigwedge \Gamma$. By (Ax) we have $\neg \bigwedge \Gamma \rhd \neg \bigwedge \Gamma$, so by $(\neg \land \rhd)$ we obtain $\neg (\bigwedge \Gamma \land (\bigwedge \Gamma \otimes \bigvee \Delta)) \rhd \neg \bigwedge \Gamma$.

As we have noted, $\Gamma \rhd \Delta$ is equivalent to $\bigwedge \Gamma \rhd \bigvee \Delta$. So we may assume $\bigwedge \Gamma \rhd \bigvee \Delta$, and by (Ax) we have also $\bigwedge \Gamma \rhd \bigwedge \Gamma$. Now, applying $(\rhd \otimes)$ and $(\rhd \land)$, we obtain $\bigwedge \Gamma \rhd \bigwedge \Gamma \land (\bigwedge \Gamma \otimes \bigvee \Delta)$.

By (Ax) we have $\neg \bigwedge \Gamma \rhd \neg \bigwedge \Gamma, \neg (\bigwedge \Gamma \otimes \bigvee \Delta)$ and by $(\rhd \neg \land)$ we have $\neg \bigwedge \Gamma \rhd \neg (\bigwedge \Gamma \land (\bigwedge \Gamma \otimes \bigvee \Delta))$. 

To prove the leftward direction, note that by (Ax) we have $\bigwedge \Gamma, \bigwedge \Gamma, \bigvee \Delta \rhd  \bigvee \Delta$, so by $(\otimes \rhd)$ we obtain $\bigwedge \Gamma, \bigwedge \Gamma \otimes \bigvee \Delta \rhd  \bigvee \Delta$. Now by $(\land \rhd)$ we have $\bigwedge \Gamma \land (\bigwedge \Gamma \otimes \bigvee \Delta) \rhd  \bigvee \Delta$ and by assumption $\bigwedge \Gamma \rhd \bigwedge \Gamma \land (\bigwedge \Gamma \otimes \bigvee \Delta)$, so using Cut we obtain $\bigwedge \Gamma \rhd \bigvee \Delta$.

(ii). We have to check that $\varphi \approx \psi =\joinrel\mid \models_{\DBL} \tau \rho (\varphi \approx \psi)$, i.e.\ that
\begin{align*}
\varphi \approx \psi =\joinrel\mid \models_{\DBL} \{ & \varphi \land (\varphi \otimes \psi) \approx  \varphi, \: \neg \varphi \land (\neg \varphi \otimes \neg \psi) \approx  \neg  \varphi,  \\ 
& \: \psi \land (\varphi \otimes \psi) \approx  \psi, \neg \psi \land (\neg \varphi \otimes \neg \psi) \approx  \neg \psi \}. 
\end{align*}
The rightward direction is clear. As to the other, note that from $\neg \varphi \land (\neg \varphi \otimes \neg \psi) \approx  \neg  \varphi$ we have $\neg (\neg \varphi \land (\neg \varphi \otimes \neg \psi)) \approx \varphi \lor (\varphi \otimes \psi)) \approx \varphi \approx \neg \neg  \varphi$, so $\varphi \approx \varphi \otimes \psi $. Similarly we obtain $\psi \approx \varphi \otimes \psi $. Hence $\varphi  \approx \psi $. 

(iii). We have to check that, for any distributive bilattice $\Al[B] \in \DBL$, the set $R = \{ \langle X, Y \rangle : \bigwedge X \leq_t \bigwedge X \otimes \bigvee Y \} $ is closed under the rules of our Gentzen calculus, where $X, Y \subseteq B $ are finite and non-empty.

(Ax). Clearly $\langle \bigwedge \Gamma \land \varphi, \varphi \lor \bigvee \Delta \rangle \in R$, since $ \bigwedge \Gamma \land \varphi \leq_t \varphi \lor \bigvee \Delta $, so by the interlacing conditions  $ \bigwedge \Gamma \land \varphi \leq_t (\bigwedge \Gamma \land \varphi) \otimes (\varphi \lor \bigvee \Delta) $.

The proof for rules $(\land \rhd)$, $(\lor \rhd)$, $(\neg \neg \rhd)$ and $(\rhd \neg \neg )$ is immediate.

$(\rhd \land)$. Assume $ \langle \bigwedge \Gamma, \bigvee \Delta \lor \varphi \rangle \in R$ and $ \langle \bigwedge \Gamma, \bigvee \Delta \lor \psi\rangle \in R$, i.e. \ $ \bigwedge \Gamma \leq_t  \bigwedge \Gamma \otimes (\bigvee \Delta \lor \varphi)$ and $ \bigwedge \Gamma \leq_t  \bigwedge \Gamma \otimes (\bigvee \Delta \lor \psi)$. Using the interlacing conditions and distributivity we have $ \bigwedge \Gamma \leq_t  (\bigwedge \Gamma \otimes (\bigvee \Delta \lor \varphi)) \land (\bigwedge \Gamma \otimes (\bigvee \Delta \lor \psi)) = \bigwedge \Gamma \otimes ((\bigvee \Delta \lor \varphi) \land (\bigvee \Delta \lor \psi)) =  \bigwedge \Gamma \otimes (\bigvee \Delta \lor (\varphi \land \psi)) $.

$(\neg \land \rhd)$. Assume $ \bigwedge \Gamma \land \neg \varphi \leq_t (\bigwedge \Gamma \land \neg \varphi) \otimes \bigvee \Delta$ and $ \bigwedge \Gamma \land \neg \psi \leq_t  (\bigwedge \Gamma \land \neg \psi) \otimes \bigvee \Delta$. Then, using distributivity and De Morgan's laws, we have 
\begin{align*}
(\bigwedge \Gamma \land \neg \varphi) \lor (\bigwedge \Gamma \land \neg \psi) & = \bigwedge \Gamma \land  (\neg \varphi \lor  \neg \psi) \\ & =  \bigwedge \Gamma \land \neg (\varphi \land  \psi) \\ & \leq_t (\bigwedge \Gamma \land \neg (\varphi \land  \psi) ) \otimes \bigvee \Delta \\ & = (\bigwedge \Gamma \land (\neg \varphi \lor  \neg \psi) ) \otimes \bigvee \Delta \\ & = ((\bigwedge \Gamma \land \neg \psi) \lor (\bigwedge \Gamma \land \neg \varphi)) \otimes \bigvee \Delta \\ & = ((\bigwedge \Gamma \land \neg \psi) \otimes \bigvee \Delta) \lor ((\bigwedge \Gamma \land \neg \varphi) \otimes \bigvee \Delta).
\end{align*}
$(\rhd \neg \land )$. Assume $\bigwedge \Gamma \leq_t \bigwedge \Gamma \otimes (\bigvee \Delta \lor \neg \varphi \lor \neg \psi)$. Then, applying De Morgan's laws, we immediately obtain $\bigwedge \Gamma \leq_t \bigwedge \Gamma \otimes (\bigvee \Delta \lor \neg (\varphi \land \psi) $.

$(\lor \rhd)$. Assume $ \bigwedge \Gamma \land \varphi \leq_t (\bigwedge \Gamma \land \varphi) \otimes \bigvee \Delta$ and $ \bigwedge \Gamma \land \psi \leq_t (\bigwedge \Gamma \land \psi)  \otimes \bigvee \Delta$. Then, by distributivity 
\begin{align*}
(\bigwedge \Gamma \land \varphi) \lor (\bigwedge \Gamma \land \psi) & = \bigwedge \Gamma \land (\varphi \lor \psi) 
\\ & \leq_t ((\bigwedge \Gamma \land \varphi) \otimes \bigvee \Delta) \lor ((\bigwedge \Gamma \land \psi)  \otimes \bigvee \Delta) 
\\ & = ((\bigwedge \Gamma \land \varphi) \lor (\bigwedge \Gamma \land \psi)) \otimes \bigvee \Delta
\\ & =  (\bigwedge \Gamma \land (\varphi \lor \psi )) \otimes \bigvee \Delta.
\end{align*}

$(\neg \lor \rhd)$. Assume $\bigwedge \Gamma \land \neg \varphi \land \neg \psi \leq_t (\bigwedge \Gamma \land \neg \varphi \land \neg \psi) \otimes \bigvee \Delta$. Then by De Morgan's laws we immediately obtain  $\bigwedge \Gamma \land \neg (\varphi \lor \psi) \leq_t (\bigwedge \Gamma \land \neg (\varphi \lor \psi)) \otimes \bigvee \Delta$.

$(\rhd \neg \lor )$. Assume $\bigwedge \Gamma \leq_t \bigwedge \Gamma \otimes (\bigvee \Delta \lor \neg \varphi)$ and $\bigwedge \Gamma \leq_t \bigwedge \Gamma \otimes (\bigvee \Delta \lor \neg \psi)$. Then, using distributivity and De Morgan's laws, we obtain
\begin{align*}
\bigwedge \Gamma & \leq_t (\bigwedge \Gamma \otimes (\bigvee \Delta \lor \neg \varphi)) \land (\bigwedge \Gamma \otimes (\bigvee \Delta \lor \neg \psi))
\\ & = \bigwedge \Gamma \otimes ((\bigvee \Delta \lor \neg \varphi) \land (\bigvee \Delta \lor \neg \psi))
\\ & = \bigwedge \Gamma \otimes ( \bigvee \Delta \lor ( \neg \varphi \land \neg \psi))
\\ & =  \bigwedge \Gamma \otimes ( \bigvee \Delta \lor \neg (\varphi  \lor \psi) ).
\end{align*}

$(\otimes \rhd)$. Assume $\bigwedge \Gamma \land \varphi \land \psi  \leq_t (\bigwedge \Gamma \land \varphi \land \psi) \otimes \bigvee \Delta$. By the interlacing conditions we have 
$$
\bigwedge \Gamma \land \varphi \land \psi \leq_t (\bigwedge \Gamma \land \varphi \land \psi) \otimes \bigvee \Delta   \leq_t (\bigwedge \Gamma \land (\varphi \otimes \psi) ) \otimes \bigvee \Delta. $$

Hence we also have
\begin{align*}
(\bigwedge \Gamma \land \varphi \land \psi) \otimes (\bigwedge \Gamma \land (\varphi \otimes \psi)) & = \bigwedge \Gamma \land (\varphi \otimes \psi) \\ & \leq_t (\bigwedge \Gamma \land (\varphi \otimes \psi) ) \otimes \bigvee \Delta \otimes (\bigwedge \Gamma \land (\varphi \otimes \psi)) \\ & = (\bigwedge \Gamma \land (\varphi \otimes \psi) ) \otimes \bigvee \Delta.
\end{align*}

$(\rhd \otimes)$. Assume $\bigwedge \Gamma \leq_t \bigwedge \Gamma \otimes (\bigvee \Delta \lor \varphi )$ and $\bigwedge \Gamma \leq_t \bigwedge \Gamma \otimes (\bigvee \Delta \lor \psi )$. By distributivity we have 
\begin{align*}
\bigwedge \Gamma & \leq_t \bigwedge \Gamma \otimes (\bigvee \Delta \lor \varphi )  \otimes (\bigvee \Delta \lor \psi ) \\ & = \bigwedge \Gamma \otimes (\bigvee \Delta \lor (\varphi \otimes \psi)  ).
\end{align*}

$(\neg \otimes \rhd)$. Assume $ \bigwedge \Gamma \land \neg \varphi \land \neg \psi  \leq_t (\bigwedge \Gamma \land \neg \varphi \land \neg \psi) \otimes \bigvee \Delta $. Using $( \otimes \rhd)$ we have $ \bigwedge \Gamma \land (\neg \varphi \otimes \neg \psi)  \leq_t (\bigwedge \Gamma \land (\neg \varphi \otimes \neg \psi)) \otimes \bigvee \Delta $, and now by De Morgan's laws we obtain $ \bigwedge \Gamma \land \neg ( \varphi \otimes \psi)  \leq_t (\bigwedge \Gamma \land \neg ( \varphi \otimes \psi)) \otimes \bigvee \Delta $.    

$(\rhd \neg \otimes )$. Assume $\bigwedge \Gamma \leq_t \bigwedge \Gamma \otimes (\bigvee \Delta \lor \neg \varphi )$ and $\bigwedge \Gamma \leq_t \bigwedge \Gamma \otimes (\bigvee \Delta \lor \neg \psi )$. Using $( \rhd \otimes )$ we obtain $\bigwedge \Gamma \leq_t \bigwedge \Gamma \otimes (\bigvee \Delta \lor (\neg \varphi \otimes \neg \psi)  )$, and by De Morgan's laws $\bigwedge \Gamma \leq_t \bigwedge \Gamma \otimes (\bigvee \Delta \lor \neg  (\varphi \otimes \psi) )$.

$(\oplus \rhd)$. Assume $ \bigwedge \Gamma  \land \varphi \leq_t (\bigwedge \Gamma  \land \varphi) \otimes \bigvee \Delta$ and $ \bigwedge \Gamma  \land \varphi \leq_t (\bigwedge \Gamma  \land \psi) \otimes \bigvee \Delta$. Using distributivity we have
\begin{align*}
(\bigwedge \Gamma  \land \varphi) \oplus (\bigwedge \Gamma  \land \psi) & =  \bigwedge \Gamma  \land (\varphi \oplus \psi) \\ & \leq_t ((\bigwedge \Gamma  \land \varphi) \otimes \bigvee \Delta) \oplus((\bigwedge \Gamma  \land \psi) \otimes \bigvee \Delta) \\ & =   ((\bigwedge \Gamma  \land \varphi) \oplus (\bigwedge \Gamma  \land \psi)) \otimes \bigvee \Delta \\ & = (\bigwedge \Gamma  \land  (\varphi \oplus \psi) ) \otimes \bigvee \Delta.
\end{align*}

$(\rhd \oplus )$. Assume $\bigwedge \Gamma \leq_t \bigwedge \Gamma \otimes (\bigvee \Delta \lor \varphi \lor \psi)$. By the interlacing conditions we have 
$$\bigwedge \Gamma \leq_t \bigwedge \Gamma \otimes (\bigvee \Delta \lor \varphi \lor \psi) \leq_k \bigwedge \Gamma \otimes (\bigvee \Delta \lor (\varphi \oplus \psi))
$$ 
Therefore 
$$ 
\bigwedge \Gamma \land (\bigwedge \Gamma \otimes (\bigvee \Delta \lor \varphi \lor \psi)) = \bigwedge \Gamma \leq_k \bigwedge \Gamma \land (\bigwedge \Gamma \otimes (\bigvee \Delta \lor (\varphi \oplus \psi))).
$$ 
As we have seen, in interlaced bilattices the previous condition is equivalent to 
$$\bigwedge \Gamma \leq_t \bigwedge \Gamma \otimes (\bigwedge \Gamma \otimes (\bigvee \Delta \lor (\varphi \oplus \psi))) = \bigwedge \Gamma \otimes (\bigvee \Delta \lor (\varphi \oplus \psi)).
$$ 

$(\neg \oplus \rhd)$. Assume $\bigwedge \Gamma \land \neg \varphi \leq_t (\bigwedge \Gamma \land \neg \varphi) \otimes \bigvee \Delta $ and $\bigwedge \Gamma \land \neg \varphi \leq_t (\bigwedge \Gamma \land \neg \psi) \otimes \bigvee \Delta $. Then, as shown in the proof of $(\oplus \rhd)$, we have $\bigwedge \Gamma \land (\neg \varphi \oplus \neg \psi) \leq_t (\bigwedge \Gamma \land (\neg \varphi \oplus \neg \psi)) \otimes \bigvee \Delta $. Now using De Morgan's laws we immediately obtain the result. 

$(\rhd \neg \oplus )$. As shown in the proof of $(\rhd \oplus )$, we have that $\bigwedge \Gamma \leq_t \bigwedge \Gamma \otimes (\bigvee \Delta \lor \neg \varphi \lor \neg \psi)$ implies $\bigwedge \Gamma \leq_t \bigwedge \Gamma \otimes (\bigvee \Delta \lor (\neg \varphi \oplus \neg \psi))$. Now again, using De Morgan's laws, we immediately obtain the result.

(iv). We have to show that $\theta_{T} \in \Con_{\DBL}(\Al[Fm])$ for all $T \in Th \mathcal{G_{LB}}$, where 
$\theta_{T} = \{ \la \phi, \psi \ra  \in Fm \times Fm : \rho (\la \phi, \psi \ra) \subseteq T \}.$ 

To prove that $\theta_{T}$ is a congruence, it is sufficient to prove that, if $$\{ \varphi \rhd \psi, \: \neg \varphi \rhd \neg \psi,  \: \psi \rhd \varphi, \:  \neg \psi \rhd \neg \varphi \} \subseteq T, $$ then for all $\vartheta \in Fm$ we have:

(a). $\{ \varphi \land \vartheta \rhd \psi \land \vartheta, \: \neg (\varphi \land \vartheta) \rhd \neg (\psi \land \vartheta),  \: \psi \land \vartheta \rhd \varphi \land \vartheta, \:  \neg (\psi \land \vartheta) \rhd \neg (\varphi \land \vartheta) \} \subseteq T $,

(b). $\{ \varphi \lor \vartheta \rhd \psi \lor \vartheta, \: \neg (\varphi \lor \vartheta) \rhd \neg (\psi \lor \vartheta),  \: \psi \lor \vartheta \rhd \varphi \lor \vartheta, \:  \neg (\psi \lor \vartheta) \rhd \neg (\varphi \lor \vartheta) \} \subseteq T $,

(c). $\{ \varphi \otimes \vartheta \rhd \psi \otimes \vartheta, \: \neg (\varphi \otimes \vartheta) \rhd \neg (\psi \otimes \vartheta),  \: \psi \otimes \vartheta \rhd \varphi \otimes \vartheta, \:  \neg (\psi \otimes \vartheta) \rhd \neg (\varphi \otimes \vartheta) \} \subseteq T $,

(d). $\{ \varphi \oplus \vartheta \rhd \psi \oplus \vartheta, \: \neg (\varphi \oplus \vartheta) \rhd \neg (\psi \oplus \vartheta),  \: \psi \oplus \vartheta \rhd \varphi \oplus \vartheta, \:  \neg (\psi \oplus \vartheta) \rhd \neg (\varphi \oplus \vartheta) \} \subseteq T $,

(e). $\{ \neg \varphi \rhd \neg  \psi, \: \neg \neg \varphi \rhd \neg \neg  \psi,  \: \neg  \psi \rhd \neg \varphi, \:  \neg  \neg \psi \rhd \neg \neg \varphi \} \subseteq T $. 

We will prove just the first two cases of each item, for the others are symmetric.

(a).
\begin{center}
               \AxiomC{$\varphi \rhd \psi$}
               \LeftLabel{$(W \rhd)$}
               \UnaryInfC{$\varphi, \vartheta \rhd \psi$}
               \AxiomC{$(Ax)$}
               \UnaryInfC{$\varphi, \vartheta \rhd \vartheta$}
               \LeftLabel{$(\rhd \land)$}
               \BinaryInfC{$\varphi, \vartheta \rhd \psi \land \vartheta$}
               \LeftLabel{$(\rhd \land)$}
               \UnaryInfC{$\varphi \land \vartheta \rhd \psi \land \vartheta$}
               \DisplayProof
\end{center}

\begin{center}
               \AxiomC{$\neg \varphi \rhd \neg \psi$}
               \LeftLabel{$( \rhd W)$}
               \UnaryInfC{$\neg \varphi \rhd \neg \psi, \neg \vartheta$}
               \LeftLabel{$(\rhd \neg \land  )$}
               \UnaryInfC{$\neg \varphi \rhd \neg ( \psi \land  \vartheta)$}
               \AxiomC{$(Ax)$}
               \UnaryInfC{$\neg \vartheta \rhd \neg \vartheta$}
               \RightLabel{$( \rhd W)$}
               \UnaryInfC{$\neg \vartheta \rhd \neg \psi, \neg \vartheta$}
               \RightLabel{$(\rhd \neg \land  )$}
               \UnaryInfC{$\neg \varphi \rhd \neg ( \psi \land  \vartheta)$}
               \LeftLabel{$(\neg \land \rhd )$}
               \BinaryInfC{$\neg (\varphi \land \vartheta) \rhd \neg (\psi \land \vartheta)$}
               \DisplayProof
\end{center}

(b). 
\begin{center}
               
  \AxiomC{$\varphi \rhd \psi$}
               \LeftLabel{$( \rhd W)$}
               \UnaryInfC{$\varphi \rhd \psi, \vartheta$}
               \LeftLabel{$(\rhd \lor )$}
               \UnaryInfC{$ \varphi \rhd \psi \lor  \vartheta$}
               \AxiomC{$(Ax)$}
               \UnaryInfC{$\vartheta \rhd  \vartheta$}
               \RightLabel{$( \rhd W)$}
               \UnaryInfC{$ \vartheta \rhd  \psi, \vartheta$}
               \RightLabel{$(\rhd \lor  )$}
               \UnaryInfC{$\vartheta \rhd \psi \lor \vartheta$}
               \LeftLabel{$(\lor \rhd )$}
               \BinaryInfC{$\varphi \lor \vartheta \rhd \psi \lor \vartheta$}
               \DisplayProof             
\end{center}

\begin{center}
               
               \AxiomC{$\neg \varphi \rhd \neg\psi$}
               \LeftLabel{$(W \rhd)$}
               \UnaryInfC{$\neg \varphi, \neg \vartheta \rhd \neg \psi$}
               \LeftLabel{$(\neg \lor \rhd )$}
               \UnaryInfC{$ \neg (\varphi \lor \vartheta) \rhd \neg \psi$}
               \AxiomC{$(Ax)$}
               \UnaryInfC{$\neg \varphi, \neg \vartheta \rhd \neg \vartheta$}
               \RightLabel{$(\neg \lor \rhd)$} 
               \UnaryInfC{$ \neg (\varphi \lor \vartheta) \rhd  \neg \vartheta$}
               \RightLabel{$(\rhd \neg  \lor  )$}
               \BinaryInfC{$ \neg (\varphi \lor \vartheta) \rhd  \neg ( \psi \lor \vartheta)$}
               \DisplayProof             
\end{center}

(c).
\begin{center}
               \AxiomC{$\varphi \rhd \psi$}
               \LeftLabel{$(W \rhd)$}
                \UnaryInfC{$\varphi, \vartheta \rhd \psi$}
                \LeftLabel{$(\otimes \rhd)$}
                \UnaryInfC{$\varphi \otimes \vartheta \rhd \psi$}
                \AxiomC{$(Ax)$}
               \UnaryInfC{$\varphi, \vartheta \rhd \vartheta$}
                \RightLabel{$(\otimes \rhd)$}
                \UnaryInfC{$\varphi \otimes \vartheta \rhd \vartheta$}
                \LeftLabel{$(\rhd \otimes )$}
               \BinaryInfC{$\varphi \otimes \vartheta \rhd \psi \otimes \vartheta$}
               \DisplayProof
\end{center}

\begin{center}
               \AxiomC{$\neg \varphi \rhd \neg \psi$}
               \LeftLabel{$(W \rhd)$}
                \UnaryInfC{$\neg \varphi, \neg \vartheta \rhd \neg \psi$}
                \LeftLabel{$(\neg \otimes \rhd)$}
                \UnaryInfC{$\neg (\varphi \otimes \vartheta) \rhd \neg \psi$}
                \AxiomC{$(Ax)$}
               \UnaryInfC{$\neg \varphi, \neg \vartheta \rhd \neg \vartheta$}
                \RightLabel{$(\neg \otimes \rhd)$}
                \UnaryInfC{$\neg (\varphi \otimes \vartheta) \rhd \neg \vartheta$}
                \LeftLabel{$(\rhd \neg  \otimes )$}
               \BinaryInfC{$\neg  (\varphi \otimes \vartheta) \rhd \neg (\psi \otimes \vartheta)$}
               \DisplayProof
\end{center}

(d). 
\begin{center}
               \AxiomC{$ \varphi \rhd  \psi$}
               \LeftLabel{$(\rhd W)$}
                \UnaryInfC{$\varphi \rhd \psi, \vartheta$}
                \LeftLabel{$(\rhd \oplus)$}
                \UnaryInfC{$\varphi \rhd \psi \oplus \vartheta$}
                \AxiomC{$(Ax)$}
               \UnaryInfC{$\vartheta \rhd \psi, \vartheta$}
                \RightLabel{$(\rhd \oplus)$}
                \UnaryInfC{$\vartheta \rhd \psi \oplus \vartheta$}
                \LeftLabel{$(\oplus \rhd)$}
               \BinaryInfC{$\varphi \oplus \vartheta \rhd \psi  \oplus \vartheta$}
               \DisplayProof
\end{center}

\begin{center}
               \AxiomC{$ \neg \varphi \rhd \neg  \psi$}
               \LeftLabel{$(\rhd W)$}
                \UnaryInfC{$\neg \varphi \rhd \neg  \psi, \neg \vartheta$}
                \LeftLabel{$(\rhd \neg \oplus)$}
                \UnaryInfC{$\neg \varphi \rhd \neg (\psi \oplus \vartheta)$}
                \AxiomC{$(Ax)$}
               \UnaryInfC{$\neg \vartheta \rhd \neg \psi, \neg \vartheta$}
                \RightLabel{$(\rhd \neg \oplus)$}
                \UnaryInfC{$\neg \vartheta \rhd \neg (\psi \oplus \vartheta)$}
                \LeftLabel{$(\neg \oplus \rhd)$}
               \BinaryInfC{$\neg (\varphi \oplus \vartheta) \rhd \neg (\psi  \oplus \vartheta)$}
               \DisplayProof
\end{center}

(e). This last case is trivial.

It remains only to prove that $ \Al[Fm] / \theta_{T} \in \DBL$, i.e.\ that, for any equation $\phi \approx \psi $ valid in the variety $ \DBL$, we have  $\{ \varphi \rhd \psi, \: \neg \varphi \rhd \neg \psi,  \: \psi \rhd \varphi, \:  \neg \psi \rhd \neg \varphi \} \subseteq T. $ This is not difficult, altough quite long. Let us see, as an example, just one of the four cases of the distributivity law: $\phi \land (\psi \lor \tet) \approx (\phi \land \psi) \lor (\phi \land \tet)$. 

The proof is the following: 

{\footnotesize 

\begin{center}
  \AxiomC{$(Ax)$}          
  \UnaryInfC{$\varphi, \phi \lor \tet \rhd \phi, \phi$}
              \UnaryInfC{$\varphi \land (\phi \lor \tet) \rhd \phi, \phi$}
                          \AxiomC{$(Ax)$}    
              \UnaryInfC{$\varphi, \psi \lor \tet \rhd \phi, \psi$}             
              \UnaryInfC{$\varphi \land (\psi \lor \tet) \rhd \phi, \psi$}

                        \BinaryInfC{$ \phi \land (\psi \lor \tet) \rhd (\phi \land \psi), \phi $}
                   
                    \AxiomC{$(Ax)$}    
                                 \UnaryInfC{$\varphi, \psi \lor \tet \rhd \tet, \phi$}
              \UnaryInfC{$\varphi \land (\psi \lor \tet) \rhd \tet, \phi$}
               
               \AxiomC{$(Ax)$}   
              \UnaryInfC{$\varphi, \psi \lor \tet \rhd \tet \lor \psi$}
               \RightLabel{$(\land \rhd  )$}
              \UnaryInfC{$\varphi \land (\psi \lor \tet) \rhd \tet, \psi$}
              \RightLabel{$( \rhd \land )$}
              \BinaryInfC{$ \phi \land (\psi \lor \tet) \rhd (\phi \land \psi), \tet $}
              \RightLabel{$(\rhd \land )$}
              \BinaryInfC{$ \phi \land (\psi \lor \tet) \rhd (\phi \land \psi), (\phi \land  \tet) $}
              \RightLabel{$(\rhd \lor )$}
              \UnaryInfC{$ \phi \land (\psi \lor \tet) \rhd (\phi \land \psi) \lor (\phi \land  \tet) $}

             \DisplayProof             
\end{center}
}

\end{proof}
\end{thm}

%










%



\chapter{Adding implications: the logic $\lbs$}
\label{ch:add}

\section{Semantical and Hilbert-style Presentations} 
\label{sec:add}

As we have seen, the logic $\lb$ lacks an implication connective. This fact may be seen as a deficiency for a logical system; in order to overcome it, Arieli and Avron \cite{ArAv96} introduced an expansion of $\lb$ obtained by adding to it two interdefinable implication connectives that they called \emph{weak} and \emph{strong implication}. In this chapter we study this logic, which we will call $\lbs$. Our main goals will be to prove that $\lbs$ is algebraizable, that its equivalent algebraic semantics is a variety, and to give a presentation of this class of algebras.

In this section we recall some definitions and results concerning $\lbs$ which are due to Arieli and Avron \cite{ArAv96}.

\begin{definition} 

 \label{impl}
{\rm Let $\left\langle \mathfrak{B}, F \right\rangle$ be a logical bilattice,  and let $\mathsf{t}$ denote the maximum element of $B$ w.r.t.\  the truth ordering. Define the operation $\supset$ as follows: for any $a, b \in B$,
\[
 	a \supset b = \left\{ \begin{array}{cl}
	\mathsf{t} & \textrm{if $a \notin F$}\\
	b & \textrm{if $a \in F.$}
	\end{array} \right.
\]
}
\end{definition}

Note that the previous definition requires the existence of the maximum w.r.t.\ the truth ordering. Note also that, in general, the behaviour of the operation $\supset$ in the algebra  $\left\langle B, \land, \lor, \otimes, \oplus, \supset, \neg \right\rangle$ depends on the bifilter that we consider.  However, since $\mathcal{FOUR}$ has only one proper bifiliter, i.e.\ $\Tr = \left\{\mathsf{t}, \top\right\}$, we can unequivocally denote  by $\mathcal{FOUR}_\supset$ the algebra obtained by adding the operation $\supset$ of the logical bilattice $\langle \mathcal{FOUR}, \Tr \rangle$. The behaviour of this new operation is described by the following table:

\begin{center}
\begin{tabular}[c]{c|c|c|c|c}
 $\supset$ & $\mathsf{f}$ & $\bot$ & $\top$ & $\mathsf{t}$ \\
\cline{1-5} 
$\mathsf{f}$             & $\mathsf{t}$ & $\mathsf{t}$  & $\mathsf{t}$  & $\mathsf{t}$  \\
\cline{1-5} 
$\bot$                   & $\mathsf{t}$ & $\mathsf{t}$  & $\mathsf{t}$  & $\mathsf{t}$  \\
\cline{1-5} 
$\top$                   & $\mathsf{f}$ & $\bot$  & $\top$   & $\mathsf{t}$  \\
\cline{1-5} 
$\mathsf{t}$             & $\mathsf{f}$ & $\bot$  & $\top$   & $\mathsf{t}$  \\
 \end{tabular}
\end{center}

The previous definition allows to prove an analogue of the fundamental Lemma \ref{l:four_logical_bilat}: 

\begin{lem} 
\label{l:epimp}
Let $\left\langle  \mathfrak{B}_\supset, F \right\rangle$ be the logical bilattice $\left\langle  \mathfrak{B}, F \right\rangle$ enriched with the operation $\supset$ defined as in Definition \ref{impl}.
Then there is a unique epimorphism $h: \mathfrak{B}_\supset \rightarrow \mathcal{FOUR}_\supset$ such that for all $x \in B$, $h\left(x\right) \in \Tr$ iff $x \in F$.
\end{lem}

From now on, we shall denote by $Fm$ the set of formulas in the language $\left\{ \land, \lor, \otimes, \oplus, \supset, \neg \right\}$ and by $\mathbf{Fm}$ the corresponding algebra of formulas.
  
\begin{definition} 
{\rm The consequence relation $\vDash_{\scriptstyle \mathcal{LB_{\supset}}}$ on $Fm$ is defined as follows. For any $\Gamma \cup \left\{\varphi\right\} \subseteq Fm$,
$$\text{$\Gamma  \vDash_{\scriptstyle \mathcal{LB_{\supset}}} \phi$ iff for every $\left\langle \mathfrak{B}_\supset, F \right\rangle$ and  every  $v \in Hom \left(\mathfrak{Fm}, \mathfrak{B}\right)$}$$ $$\text{ if $v\left(\psi\right) \in F$ for all $\psi \in \Gamma$, then $v\left(\varphi\right) \in F$.}$$ 
We denote the logic $\langle \mathbf{Fm}, \vDash_{\scriptstyle \mathcal{LB_{\supset}}}\rangle$ by  $\mathcal{LB}_{\supset}$.
}
\end{definition}

As a corollary of Lemma \ref{l:epimp}, we obtain an analogue of   Theorem \ref{thm:complet_logicalbilattice}: 

\begin{thm} 
 \label{cons}
For every  $\Gamma \cup \left\{\varphi\right\} \subseteq Fm$, $$\text{$\Gamma  \vDash_{\scriptstyle \mathcal{LB_{\supset}}} \varphi$ iff  $\ \ \Gamma  \vDash_{\scriptstyle \left\langle \mathcal{FOUR_{\supset}}, \Tr  \right\rangle} \varphi$.}$$
\end{thm}

Arieli and Avron \cite{ArAv96} provided a Hilbert-style axiomatization for $\lbs$, which we repeat here:

\begin{definition} 
\label{d:hilbert}
{\rm Let  $H_{\supset}  = \langle \mathfrak{Fm}, \vdash_{\scriptstyle H_{\supset}}\rangle$  be the sentential logic defined through the  Hilbert style calculus with axioms, 
%

\begin{align*}
&(\supset 1)  & &p \supset (q \supset p) \\
&(\supset 2)  & &  (p \supset (q \supset r)) \supset ((p\supset q) \supset (p\supset r)) \\
&(\supset 3)  & &  ((p \supset q) \supset p) \supset p \\
&(\land \supset )  & &  (p \land q) \supset p \qquad \qquad (p \land q) \supset q \\
&(\supset \land )  & &  p \supset (q \supset (p \land q)) \\
&(\otimes \supset )  & &  (p \otimes q) \supset p \qquad \qquad (p \otimes q) \supset q \\
&(\supset \otimes )  & &  p \supset (q \supset (p \otimes q)) \\
&(\supset \lor)  & &  p  \supset (p \lor q) \qquad \qquad q \supset (p \lor q) \\
&(\lor \supset )  & &  (p \supset r) \supset ((q \supset r) \supset ((p \lor q) \supset r)) \\
&(\supset \oplus)  & &  p  \supset (p \oplus q) \qquad  \qquad q \supset (p \oplus q) \\
&(\oplus \supset )  & &  (p \supset r) \supset ((q \supset r) \supset ((p \oplus q) \supset r)) \\
\end{align*}
\begin{align*}
&(\neg \land )  & &   \neg (p \land q ) \equiv (\neg p \vee \neg q)  \\
&(\neg \lor )  & &  \neg (p \vee q) \equiv (\neg p \land \neg q ) \\
&(\neg \otimes )  & &   \neg (p \otimes q ) \equiv (\neg p \otimes \neg q)  \\
&(\neg \oplus )  & &   \neg (p \oplus q ) \equiv (\neg p \oplus \neg q)  \\
&(\neg \supset )  & &   \neg (p \supset q ) \equiv (p \land \neg q)  \\
&(\neg \neg) & & p \equiv \neg \neg p \\
\end{align*}
where  $\phi \equiv \psi$ abbreviates   $(\phi \supset \psi) \land (\psi \supset \phi)$, and with  modus ponens (MP) as the only inference rule:
$$
\frac{p \quad p \supset q }{q}
$$

The consequence operator associated with $\vdash_{\scriptstyle H_{\supset}}$ will be denoted by $\mathrm{C}_{H_{\supset}}$. }
\end{definition}

A remarkable feature of $H_{\supset}$ is that it enjoys the classical Deduction\--Detachment Theorem; this is proved in \cite{ArAv96} using the completeness theorem, but in general it is known to hold for any calculus that has axioms $(\supset 1)$ and $(\supset 2)$ and MP as the only rule.

\begin{thm}[DDT
]
 \label{t:ddt}
Let $\Gamma \cup \left\{\varphi, \psi \right\} \subseteq Fm$. Then $$\text{$\Gamma, \varphi  \vdash_{\scriptstyle H_{\supset}} \psi$ \ \ iff $\ \ \Gamma \vdash_{\scriptstyle H_{\supset}} \varphi \supset \psi$.}$$
\end{thm}

The following result shows that the calculus introduced in  Definition \ref{d:hilbert} is complete w.r.t.\ the semantics of $\lbs$:

\begin{thm}
 \label{t:compl}
Let $\Gamma \cup \left\{\varphi\right\} \subseteq Fm$. The following are equivalent:

\begin{enumerate}[(i)]
\item $\Gamma  \vdash_{\scriptstyle H_{\supset}} \varphi$.

\item$\Gamma  \vDash_{\scriptstyle \mathcal{LB_{\supset}}} \varphi$.

\item $\Gamma  \vDash_{\scriptstyle \left\langle \mathcal{FOUR_{\supset}}, \Tr  \right\rangle} \varphi$. 
\end{enumerate}
\end{thm}

Adopting the notation of \cite{ArAv96}, we will use the following abbreviations:
\begin{align*}
p \rightarrow q & \: : = \:  (p \supset q) \land (\neg q \supset \neg p) \\
p \leftrightarrow q & \:  : = \:  (p \rightarrow q) \land (q \rightarrow p).
\end{align*}

To finish the section, let us cite a useful result that follows immediately from Theorem \ref{t:compl} and \cite[Proposition 3.27]{ArAv96}:

\begin{proposition} \label{congr}
For any $\left\{ \varphi_i : i \in I \right\}  \cup \left\{ \psi_i : i \in I \right\} \cup\left\{\varphi, \psi\right\} \subseteq Fm$, the following conditions are equivalent:
\begin{enumerate}[(i)]
 \item  $ \left\{ \varphi_i \approx \psi_i : i \in I \right\} \vDash_{\mathcal{FOUR}_{\supset}} \varphi \approx \psi$.
 \item $ \left\{ \varphi_i \leftrightarrow \psi_i : i \in I \right\} \vdash_{\scriptstyle H_{\supset}} \varphi \leftrightarrow \psi$.
\end{enumerate}
\end{proposition}

\section{Some properties of the calculus $H_{\s}$}
\label{sec:lbs}

Our next aim is to prove that the logic $\lbs$ is algebraizable, and that its equivalent algebraic semantics is a variety of algebras that we will call \emph{implicative bilattices}. In the next chapter we will study this variety; in particular, we will show that it is generated by $\fours$. In this section we begin by stating some properties of the consequence relation $\h$ that will be needed in the proof of algebraizability.

\begin{rem} \label{r:1}
{\rm Let $\Gamma \: \cup \: \{ \phi \} \subseteq Fm$ be formulas in the language $\{ \land, \lor, \s \}$. Denote by $\vdash_{CPC}$
the derivability relation of the corresponding fragment of classical propositional logic, where the connectives $\{ \land, \lor, \s \}$ are interpreted respectively as classical conjunction, disjunction and implication. Then $\Gamma \vdash_{CPC} \phi$ implies $\Gamma \h \phi$. This follows from the fact that the axioms and rules of $H_{\supset}$ involving only  $\{ \land, \lor, \s \}$ constitute an axiomatization (see, for instance, the one given in \cite{Cu77}) of the   $\{ \land, \lor, \s \}$-fragment of classical logic. The same reasoning shows that the same holds for formulas in the language $\{ \otimes, \oplus, \s \}$ when we interpret these connectives as respectively classical conjunction, disjunction and implication. 
It is also possible to prove the converse implication, i.e.\ that, under the same assumptions,  $\Gamma \h \phi$ implies $\Gamma \vdash_{CPC} \phi$: this follows from the fact that the $\{ \land, \lor, \s \}$-fragment of classical logic is maximally consistent, that is, it has no axiomatic extensions (however, we shall not need this result here).}
\end{rem}

The preceding remark will be used in the following proofs; also, we will often make use of the DDT without notice. Moreover, recall that, by structurality of the derivability relation $\h$, the proof of a derivation implies also that of all its substitution instances (possibly containing connectives that did not appear in the original formulas). 


\begin{proposition} \label{lemmalg}
For all formulas $\varphi, \psi, \vartheta, \varphi_1, \varphi_2, \psi_1, \psi_2 \in Fm$,
\begin{enumerate}[(i)]
\item $ \varphi \vdash_{\scriptstyle H_{\supset}}  \psi \land \vartheta $ if and only if $ \varphi \vdash_{\scriptstyle H_{\supset}}  \psi$ and $ \varphi \vdash_{\scriptstyle H_{\supset}}  \vartheta$ if and only if $ \varphi \vdash_{\scriptstyle H_{\supset}}  \psi \otimes \vartheta $ 

\item $ \varphi \supset \psi, \psi \supset \vartheta \vdash_{\scriptstyle H_{\supset}} \varphi \supset  \vartheta$
\item $ \vdash_{\scriptstyle H_{\supset}} \varphi \supset \varphi$
\item $ \vdash_{\scriptstyle H_{\supset}} \neg (\varphi \supset \varphi) \supset \neg \varphi$
\item $\varphi \vdash_{\scriptstyle H_{\supset}} \varphi \leftrightarrow (\varphi \supset \varphi)$
\item $\varphi \leftrightarrow (\varphi \supset \varphi) \vdash_{\scriptstyle H_{\supset}} \varphi $
\item  $\vdash_{\scriptstyle H_{\supset}} \varphi \leftrightarrow  \varphi$
\item $\varphi \leftrightarrow \psi \vdash_{\scriptstyle H_{\supset}} \psi \leftrightarrow \varphi$
\item $\varphi \leftrightarrow \psi, \psi \leftrightarrow \vartheta \vdash_{\scriptstyle H_{\supset}} \varphi \leftrightarrow \vartheta$
\item $\varphi \leftrightarrow \psi \vdash_{\scriptstyle H_{\supset}} \neg \varphi \leftrightarrow  \neg \psi$

\item $\varphi_1 \rightarrow \psi_1, \varphi_2 \rightarrow \psi_2 \vdash_{\scriptstyle H_{\supset}} (\varphi_1 \land \varphi_2) \rightarrow (\psi_1 \land \psi_2)$

\item $\varphi_1 \rightarrow \psi_1, \varphi_2 \rightarrow \psi_2 \vdash_{\scriptstyle H_{\supset}} (\varphi_1 \lor \varphi_2) \rightarrow (\psi_1 \lor \psi_2)$

\item $\varphi_1 \rightarrow \psi_1, \varphi_2 \rightarrow \psi_2 \vdash_{\scriptstyle H_{\supset}} (\varphi_1 \otimes \varphi_2) \rightarrow (\psi_1 \otimes \psi_2)$

\item $\varphi_1 \rightarrow \psi_1, \varphi_2 \rightarrow \psi_2 \vdash_{\scriptstyle H_{\supset}} (\varphi_1 \oplus \varphi_2) \rightarrow (\psi_1 \oplus \psi_2)$

\item $\psi_1  \rightarrow \varphi_1,  \varphi_2  \rightarrow \psi_2  \vdash_{\scriptstyle H_{\supset}} (\varphi_1 \supset \varphi_2) \rightarrow (\psi_1 \supset \psi_2)$

\item $\phi, \psi \dashv \h  \phi \land \psi $ and $ \phi \land \psi \dashv \h \phi \otimes \psi$

\item if $\phi \dashv \h \psi$, then $\h (\phi \s \chi) \da (\psi \s \chi)$ for all $\chi \in Fm$ 

\item if $\phi_1 \h \psi_1 $ and $\phi_2  \h \psi_2 $, then $\phi_1 \lor \phi_2 \h \psi_1 \lor \psi_2$

\item if $\phi_1 \h \psi_1 $ and $\phi_2  \h \psi_2 $, then $\phi_1 \oplus \phi_2 \h \psi_1 \oplus \psi_2$

\item $\h \phi \ta \psi $ if and only if $\h (\phi \land \psi) \da \psi $

\end{enumerate}
\end{proposition}

\begin{proof}
(i). The rightward implication is easily proved using $(\land \supset)$. As to the leftward one, note that by $(\supset \land )$ we have $ \varphi \vdash_{\scriptstyle H_{\supset}} \psi \supset (\vartheta \supset (\psi \land \vartheta)) $, so applying MP twice we obtain $ \varphi \vdash_{\scriptstyle H_{\supset}}  \psi \land \vartheta $. The proof for the case of $\otimes$ is similar, we just need to use $(\otimes \supset)$ and $(\supset \otimes )$ instead of $(\land \supset)$ and $(\supset \land )$.

(ii). By $(\supset 1)$ and MP we have $\psi \supset \vartheta \vdash_{\scriptstyle H_{\supset}} \varphi \supset (\psi \supset \vartheta)$ and by $(\supset 2)$ we have  $\vdash_{\scriptstyle H_{\supset}} (\varphi \supset (\psi \supset \vartheta)) \supset (( \varphi \supset \psi) \supset ( \varphi \supset \vartheta ))$. So, applying MP, we have $\psi \supset \vartheta \vdash_{\scriptstyle H_{\supset}} ( \varphi \supset \psi) \supset ( \varphi \supset \vartheta ) $. Hence, by MP, $\psi \supset \vartheta,  \varphi \supset \psi \vdash_{\scriptstyle H_{\supset}} ( \varphi \supset \vartheta ) $. 

(iii). $(\varphi \supset ((\psi \supset \varphi) \supset \varphi)) \supset ((\varphi \supset ( \psi \supset \varphi )) \supset (\varphi \supset \varphi ))$ is an instance of $(\supset 2)$ and $( \varphi \supset (( \psi \supset \varphi ) \supset \varphi ))$ and $(\varphi \supset ( \psi \supset \varphi ))$ are instances of $(\supset 1)$. So applying MP twice we obtain $ \vdash_{\scriptstyle H_{\supset}} \varphi  \supset \varphi $.

(iv). $ \neg (\varphi \supset \varphi) \supset (\varphi \land \neg \varphi)$ is an instance of $(\neg \supset)$ and $(\varphi \land \neg \varphi)  \supset \neg \varphi$ is an instance of $(\land \supset)$. So, by (ii), we obtain $ \vdash_{\scriptstyle H_{\supset}} \neg (\varphi \supset \varphi) \supset \neg \varphi$. 

(v). Taking into account (i), it is sufficient to prove the following: $\varphi \vdash_{\scriptstyle H_{\supset}} \varphi \supset (\varphi \supset \varphi)$, $\varphi \vdash_{\scriptstyle H_{\supset}} \neg (\varphi \supset \varphi) \supset \neg  \varphi $, $\varphi \vdash_{\scriptstyle H_{\supset}}  (\varphi \supset \varphi) \supset \varphi$ and $\varphi \vdash_{\scriptstyle H_{\supset}} \neg \varphi \supset \neg  (\varphi \supset \varphi) $. The first follows immediately from $(\s 1)$, while the second follows from (iv). The third amounts to  $\varphi, \varphi \supset \varphi \vdash_{\scriptstyle H_{\supset}} \varphi$, which is obvious, and the fourth to $\varphi, \neg \varphi \vdash_{\scriptstyle H_{\supset}} \neg  (\varphi \supset \varphi) $, which easily follows from $(\neg \supset)$.

(vi). It is sufficient to prove that $(\varphi \supset \varphi) \supset \varphi\vdash_{\scriptstyle H_{\supset}} \varphi $, and this follows from (iii) by MP.

(vii). Follows immediately from (iii).

(viii). Immediate. 

(ix). Follows easily, using 
(i) and (ii).

(x). It is sufficient to prove that $\varphi \supset \psi \vdash_{\scriptstyle H_{\supset}} \neg \neg  \varphi \supset  \neg \neg \psi$ and $  \psi \supset \varphi  \vdash_{\scriptstyle H_{\supset}} \neg \neg \psi  \supset  \neg \neg \varphi $, and this follows easily using $(\neg \neg )$ and the transitivity of $\s$.

(xi). We will prove that $\varphi_1 \supset \psi_1, \varphi_2 \supset \psi_2 \vdash_{\scriptstyle H_{\supset}} (\varphi_1 \land \varphi_2) \supset (\psi_1 \land \psi_2)$ and  $\neg \psi_1 \supset \neg \varphi_1 , \neg \psi_2 \supset  \neg \varphi_2 \vdash_{\scriptstyle H_{\supset}} \neg (\psi_1 \land \psi_2) \supset \neg ( \varphi_1 \land \varphi_2) $.
The former is equivalent to $\varphi_1 \supset \psi_1, \varphi_2 \supset \psi_2, \varphi_1 \land \varphi_2 \vdash_{\scriptstyle H_{\supset}} (\psi_1 \land \psi_2)$, which is easily proved. As to the latter, by $(\supset \lor)$ and the transitivity of $\h,$ we have $\neg \psi_1 \supset \neg \varphi_1 , \neg \psi_2 \supset  \neg \varphi_2 \vdash_{\scriptstyle H_{\supset}} \neg \psi_1 \supset (\neg \varphi_1 \lor \neg \varphi_2)$ and $\neg \psi_1 \supset \neg \varphi_1 , \neg \psi_2 \supset  \neg \varphi_2 \vdash_{\scriptstyle H_{\supset}} \neg \psi_2 \supset (\neg \varphi_1 \lor \neg \varphi_2)$. Then, using $(\lor \supset )$, we obtain $\neg \psi_1 \supset \neg \varphi_1 , \neg \psi_2 \supset  \neg \varphi_2 \vdash_{\scriptstyle H_{\supset}} ( \neg \psi_1 \lor \neg \psi_2) \supset (\neg \varphi_1 \lor \neg \varphi_2)$. By $(\neg \land)$ we have $  \vdash_{\scriptstyle H_{\supset}} \neg ( \psi_1 \land \psi_2) \supset ( \neg \psi_1 \lor \neg \psi_2)$ and $\vdash_{\scriptstyle H_{\supset}}   (\neg \varphi_1 \lor \neg \varphi_2) \supset \neg ( \varphi_1 \land \varphi_2)$. So, applying (ii), we obtain the result. 

(xii). We will prove that $\varphi_1 \supset \psi_1, \varphi_2 \supset \psi_2 \vdash_{\scriptstyle H_{\supset}} (\varphi_1 \lor \varphi_2) \supset (\psi_1 \lor \psi_2)$ and  $\neg \psi_1 \supset \neg \varphi_1 , \neg \psi_2 \supset  \neg \varphi_2 \vdash_{\scriptstyle H_{\supset}} \neg (\psi_1 \lor \psi_2) \supset \neg ( \varphi_1 \lor \varphi_2) $. As to the first, we have that $\varphi_1 \supset \psi_1 \vdash_{\scriptstyle H_{\supset}} \varphi_1 \supset (\psi_1 \lor \psi_2)$ and  $\varphi_2 \supset \psi_2 \vdash_{\scriptstyle H_{\supset}} \varphi_2 \supset (\psi_1 \lor \psi_2)$. Now, using $(\lor \supset)$ we obtain $\varphi_1 \supset \psi_1, \varphi_2 \supset \psi_2 \vdash_{\scriptstyle H_{\supset}} (\varphi_1 \lor \varphi_2) \supset (\psi_1 \lor \psi_2)$. As to the second, using (xi) we have that $\neg \psi_1 \supset \neg \varphi_1 , \neg \psi_2 \supset  \neg \varphi_2 \vdash_{\scriptstyle H_{\supset}} (\neg \psi_1 \land \neg \psi_2) \supset (\neg  \varphi_1 \land \neg \varphi_2) $. Now using $(\neg \lor)$ and transitivity we obtain the result. 

(xiii).  We will prove that $\varphi_1 \supset \psi_1, \varphi_2 \supset \psi_2 \vdash_{\scriptstyle H_{\supset}} (\varphi_1 \otimes \varphi_2) \supset (\psi_1 \otimes \psi_2)$ and  $\neg \psi_1 \supset \neg \varphi_1 , \neg \psi_2 \supset  \neg \varphi_2 \vdash_{\scriptstyle H_{\supset}} \neg (\psi_1 \otimes \psi_2) \supset \neg ( \varphi_1 \otimes \varphi_2) $.
A proof of the first one can be obtained from that of (xi), just replacing any occurence of $\land$ with $\otimes$ and using the corresponding axioms for $\otimes$. As to the second, it is easy to prove  $\neg \psi_1 \supset \neg \varphi_1 , \neg \psi_2 \supset  \neg \varphi_2,  \neg \psi_1 \otimes \neg \psi_2 \vdash_{\scriptstyle H_{\supset}} \neg \varphi_1 \otimes \neg \varphi_2 $ and from this, using $(\neg \otimes)$, we obtain the result.

(xiv). We will prove that $\varphi_1 \supset \psi_1, \varphi_2 \supset \psi_2 \vdash_{\scriptstyle H_{\supset}} (\varphi_1 \oplus \varphi_2) \supset (\psi_1 \oplus \psi_2)$ and  $\neg \psi_1 \supset \neg \varphi_1 , \neg \psi_2 \supset  \neg \varphi_2 \vdash_{\scriptstyle H_{\supset}} \neg (\psi_1 \oplus \psi_2) \supset \neg ( \varphi_1 \oplus \varphi_2) $. A proof of the first one can be obtained from that of (xii), just replacing any occurence of $\lor$ with $\oplus$ and using the corresponding axioms for $\oplus$. As to the second, it is easy to prove that $\neg \psi_1 \supset \neg \varphi_1 , \neg \psi_2 \supset  \neg \varphi_2 \vdash_{\scriptstyle H_{\supset}} (\neg \psi_1 \oplus \neg \psi_2) \supset (\neg \varphi_1 \oplus \neg \varphi_2) $ and from this, using $(\neg \oplus)$ and transitivity, we obtain the result.

(xv). We will prove that $ \psi_1 \supset \varphi_1, \varphi_2 \supset \psi_2 \vdash_{\scriptstyle H_{\supset}} (\varphi_1 \supset \varphi_2) \supset (\psi_1 \supset \psi_2)$ and $ \psi_1 \supset \varphi_1, \neg \psi_2 \supset  \neg \varphi_2 \vdash_{\scriptstyle H_{\supset}} \neg (\psi_1 \supset \psi_2) \supset \neg ( \varphi_1 \supset \varphi_2)$. The former is equivalent to $ \psi_1 \supset \varphi_1, \varphi_2 \supset \psi_2, \varphi_1 \supset \varphi_2, \psi_1 \vdash_{\scriptstyle H_{\supset}} \psi_2$, which is easily proved by transitivity. In order to prove the latter, note that $ \psi_1 \supset \varphi_1, \neg \psi_2 \supset  \neg \varphi_2, \psi_1, \neg \psi_2  \vdash_{\scriptstyle H_{\supset}} \varphi_1 \land \neg \varphi_2$, so $ \psi_1 \supset \varphi_1, \neg \psi_2 \supset  \neg \varphi_2 \vdash_{\scriptstyle H_{\supset}} (\psi_1 \land \neg \psi_2)  \supset (\varphi_1 \land \neg \varphi_2)$. Now, by $(\neg \supset)$ we have $\vdash_{\scriptstyle H_{\supset}} \neg (\psi_1 \supset \psi_2) \supset (\psi_1 \land \neg \psi_2)$ and $\vdash_{\scriptstyle H_{\supset}} (\varphi_1 \land \neg \varphi_2) \supset \neg (\varphi_1 \supset  \varphi_2)$. So by transitivity we obtain the result.

(xvi). Easy, using using $(\land \s )$, $(\s \land)$, $(\otimes \s )$ and $(\s \otimes)$. In the following proofs we will sometimes make use of this property without notice.

(xvii). Assume $\phi \dashv \h \psi$. Using (i), we will prove that for all $\chi \in Fm$:
\begin{align*}
& \h (\phi \s \chi) \s (\psi \s \chi) \\
& \h  (\psi \s \chi) \s (\phi \s \chi) \\
& \h \neg (\phi \s \chi) \s \neg (\psi \s \chi) \\
& \h \neg (\psi \s \chi) \s \neg  (\phi \s \chi). 
\end{align*}
The first two are equivalent to $\phi \s \chi, \psi \h \chi $ and  $\psi \s \chi, \phi \h  \chi$, so they are easily proved. As to the second two, they amount to showing that $\neg (\phi \s \chi) \dashv \h \neg (\psi \s \chi)$. Now note that by $(\neg \s)$ we have $\neg (\phi \s \chi) \dashv \h  \phi  \land \neg \chi$ and $\neg (\psi \s \chi) \dashv \h  \psi  \land \neg \chi$. So, it is sufficient to prove that $\phi  \land \neg \chi \dashv \h  \psi  \land \neg \chi$ and, using (i) and (xvi), this is easy.

(xviii). Assume $\phi_1 \h \psi_1 $ and $\phi_2  \h \psi_2 $. Note that 
$$ (\phi_1 \s (\psi_1 \lor \psi_2)) \s ( (\phi_2 \s (\psi_1 \lor \psi_2)) \s ((\phi_1 \lor \phi_2) \s (\psi_1 \lor \psi_2))  ) $$
is an instance of $(\lor \s)$, and that using $( \s \lor)$ one can easily derive $\h \phi_1 \s (\psi_1 \lor \psi_2)$ from the first assumption and $\h \phi_2 \s (\psi_1 \lor \psi_2)$ from the second. Hence, by MP, we have $\h (\phi_1 \lor \phi_2) \s (\psi_1 \lor \psi_2)$, which implies $\phi_1 \lor \phi_2 \h \psi_1 \lor \psi_2$.

(xix). The proof can be easily obtained from that of (xviii), just using the corresponding axioms for $\oplus$ instead of those for $\lor$. 

(xx). Note that $\h (\phi \land \psi) \ta \phi$, because $\h (\phi \land \psi) \s \phi$ is an instance of $(\land \s )$ and $\h \neg \phi \s \neg (\phi \land \psi)$ follows from the fact that, by $(\neg \land)$, we have $\neg \phi \lor \neg \psi \h \neg (\phi \land \psi)$ and, by $(\s \lor)$, we have $\neg \phi \h \neg \phi \lor \neg \psi$. 

So, assuming $\h \phi \ta \psi$, we only need to prove that $\h \phi  \ta (\phi \land \psi)$. It is not difficult to prove that $\h \phi \da (\phi \land \phi)$, and by (vi) we also have $\h \phi \ta \phi$. Now, using (xi) and the assumption, we obtain $\h (\phi \land \phi)  \ta (\phi \land \psi)$. Finally, using (ix), we obtain $\h \phi  \ta (\phi \land \psi)$.

Conversely, assume $\h \phi  \ta (\phi \land \psi)$. Clearly, the same proof of $(\phi \land \psi) \ta \phi$ shows that  $\h (\phi \land \psi) \ta \psi$, so we can apply (i) and the transitivity of $\s$ and the result follows easily. 
\end{proof}

The following proposition will be needed in order to characterize the class $\mathbf{Alg^{*}}\lbs$. Note that, in order to prove  results of the form $\h \phi \da \psi$, we need to show that 
$$\text{ (a) $\h \phi \equiv \psi$ \hspace{2cm} (b) $\h \neg \phi \equiv \neg \psi$.}$$
In the next proposition, part (a) follows from  Remark \ref{r:1}, so we will prove only  part (b).

\begin{proposition} \label{lemlat}

For all formulas $\varphi, \psi, \vartheta \in Fm$,
\begin{enumerate}[(i)]

\item $\h \phi \da (\phi \land \phi)$
\item $\h (\phi \land \psi) \da (\psi \land \phi)$
\item $\h (\phi \land (\psi \land \tet)) \da ((\phi \land \psi) \land \tet)$

\item $\h \phi \da (\phi \lor \phi)$
\item $\h (\phi \lor \psi) \da (\psi \lor \phi)$
\item $\h (\phi \lor (\psi \lor \tet)) \da ((\phi \lor \psi) \lor \tet)$

\item $\h (\phi \land (\phi \lor \psi)) \da \phi $
\item $\h (\phi \lor (\phi \land \psi)) \da \phi $

\item $\h \phi \da (\phi \otimes \phi)$
\item $\h (\phi \otimes \psi) \da (\psi \otimes \phi)$
\item $\h (\phi \otimes (\psi \otimes \tet)) \da ((\phi \otimes \psi) \otimes \tet)$

\item $\h \phi \da (\phi \oplus \phi)$
\item $\h (\phi \oplus \psi) \da (\psi \oplus \phi)$
\item $\h (\phi \oplus (\psi \oplus \tet)) \da ((\phi \oplus \psi) \oplus \tet)$

\item $\h (\phi \otimes (\phi \oplus \psi)) \da \phi $
\item $\h (\phi \oplus (\phi \otimes \psi)) \da \phi $

\item $\h \neg (\phi \land \psi) \da (\neg \phi \lor \neg \psi) $
\item $\h \neg (\phi \lor \psi) \da (\neg \phi \land \neg \psi) $

\item $\h \neg (\phi \otimes \psi) \da (\neg \phi \otimes \neg \psi) $
\item $\h \neg (\phi \oplus \psi) \da (\neg \phi \oplus \neg \psi) $

\item $\h  \phi \da  \neg \neg \phi$

\item $\h (((\phi \s \phi) \oplus \neg (\phi \s \phi)) \s \psi) \da \psi $

\item $\h ((\phi \land \psi) \s \tet) \da (\phi \s (\psi \s \tet))$ 

\item $\h ((\phi \land \psi) \s \tet) \da ((\phi \otimes \psi) \s \tet) $

\item $\h  ((\phi \s \phi) \oplus \neg (\phi \s \phi)) \ta ( ((\psi \s \tet) \s \psi) \s \psi)$

\item $\h ((\phi \lor \psi) \s \tet) \da ((\phi \oplus \psi) \s \tet) $
\item $\h \phi \ta ((\phi \s \psi) \s (\phi \otimes \psi))$
\item $\h (\neg (\phi \s \psi ) \s \tet) \da ((\phi \land \neg \psi ) \s \tet)$

\item $\h ((\phi \lor \psi) \s \tet) \da ((\phi \s \tet) \land (\psi \s \tet))$.

\end{enumerate}
\end{proposition}

\begin{proof}
%

(i). We have to prove that $\h  \neg (\phi \land \phi) \s \neg \phi $ and $\h \neg \phi \s \neg (\phi \land \phi)$. As to the former, note that by $(\neg \land)$ we have $\neg (\phi \land \phi) \h  \neg \phi \lor \neg \phi$, so it will be enough to prove $\neg \phi \lor \neg \phi \h  \neg \phi $, and this is easily done using  $(\lor \s)$. As to the latter, by  $( \s \lor)$ we have $\h \neg \phi \s \neg \phi \lor \neg \phi$ and from this, using $(\neg \land)$, we easily obtain the result.

(ii). To prove that $\h \neg (\phi \land \psi) \equiv \neg (\psi \land \phi)$, using $(\neg \land)$, it suffices to show that $ \h (\phi \lor \psi) \equiv (\psi \lor \phi)$, and this follows  by  Remark \ref{r:1}.

(iii). To prove that $\h \neg (\phi \land (\psi \land \tet)) \equiv \neg ((\phi \land \psi) \land \tet)$, note that by $(\neg \land)$ we have, on the one hand, $\h \neg (\phi \land (\psi \land \tet)) \equiv \neg \phi \lor  \neg (\psi \land \tet)$ and, using also Proposition \ref{lemmalg} (xviii),  $\h  \neg \phi \lor  \neg (\psi \land \tet) \equiv  \neg \phi \lor  (\neg \psi \lor \neg \tet)$. On the other hand we have $\h  \neg ((\phi \land \psi) \land \tet) \equiv \neg (\phi \land \psi) \lor \neg \tet$ and $\h  \neg (\phi \land \psi) \lor \neg \tet \equiv (\neg \phi \lor \neg \psi) \lor \neg \tet$. Hence, it suffices to prove that  $ \neg \phi \lor  (\neg \psi \lor \neg \tet) \equiv (\neg \phi \lor \neg \psi) \lor \neg \tet$, and this follows from Remark \ref{r:1}.

(iv). To prove that $\h \neg \phi \equiv \neg (\phi \lor \phi)$, we only need to use $(\neg \lor)$ and (i).

(v). To prove that $\h \neg (\phi \lor \psi) \equiv \neg (\psi \lor \phi)$, we only need to use $(\neg \lor)$ and (ii).

(vi). To prove that $\h \neg (\phi \lor (\psi \lor \tet)) \equiv \neg ((\phi \lor \psi) \lor \tet)$ we can use $(\neg \lor)$ and Proposition \ref{lemmalg} (xviii).

(vii). To prove that $\h \neg (\phi \land (\phi \lor \psi)) \equiv \neg \phi $, note that by $(\neg \land)$ we have $\h \neg (\phi \land (\phi \lor \psi)) \equiv \neg  \phi \lor \neg (\phi \lor \psi))$ and, using $(\neg \lor )$ and Proposition \ref{lemmalg} (xviii), we have $\h \neg  \phi \lor \neg (\phi \lor \psi)) \equiv  \neg  \phi \lor ( \neg \phi \land \neg \psi)$. Now we may use Remark \ref{r:1} to obtain the result. 

(viii). To prove that $\h \neg (\phi \lor (\phi \land \psi)) \equiv \neg \phi $ we may reason as in (vii), just using $(\neg \land )$ instead of $(\neg \lor )$.

(ix). To prove that $\h \neg \phi \equiv \neg (\phi \otimes \phi)$, it is sufficient to observe that by $(\neg \otimes)$ we have $\h  \neg (\phi \otimes \phi) \equiv \neg \phi \otimes \neg \phi$, so we may use again Remark \ref{r:1}. 

(x). To prove that $\h \neg (\phi \otimes \psi) \equiv \neg (\psi \otimes \phi)$, it is sufficient to observe that by $(\neg \otimes)$ we have $\h  \neg (\phi \otimes \psi) \equiv \neg \phi \otimes \neg \psi$, so we may use again Remark \ref{r:1}.

(xi). To see that $\h \neg (\phi \otimes (\psi \otimes \tet)) \equiv \neg ((\phi \otimes \psi) \otimes \tet)$, as in the previous cases, we may use $(\neg \otimes)$ and Remark \ref{r:1}.

(xii). We may proceed as in (ix), just using $(\neg \oplus)$ instead of $(\neg \otimes)$.

(xiii). We may proceed as in (x), just using $(\neg \oplus)$ instead of $(\neg \otimes)$.

(xiv). We may proceed as in (xi), using $(\neg \oplus)$ instead of $(\neg \otimes)$, together with Proposition \ref{lemmalg} (xix).

(xv). To prove $\h \neg (\phi \otimes (\phi \oplus \psi)) \equiv \neg \phi $, we use $(\neg \otimes)$ to obtain $\h \neg (\phi \otimes (\phi \oplus \psi)) \equiv \neg \phi \otimes \neg (\phi \oplus \psi)$. Using $(\neg \oplus)$, it is easy to obtain $\h  \neg \phi \otimes \neg (\phi \oplus \psi) \equiv \neg \phi \otimes (\neg \phi \oplus \neg \psi)$. Now we may apply Remark \ref{r:1} again to obtain the result.

(xvi). We may proceed as in (xv), using the property stated in Proposition \ref{lemmalg} (xix).

(xvii). By $(\neg \land)$ we have $\h \neg (\phi \land \psi) \equiv (\neg \phi \lor \neg \psi )$. To prove $\h \neg \neg (\phi \land \psi) \equiv \neg (\neg \phi \lor \neg \psi) $, observe that, by  $(\neg \neg)$, we have $\h \neg \neg (\phi \land \psi) \equiv (\phi \land \psi) $ and, moreover, $\phi, \psi \dashv \h \neg \neg \phi, \neg \neg \psi$. By Proposition \ref{lemmalg} (xvi), this means that $\h (\phi \land \psi) \equiv (\neg \neg \phi \land \neg \neg \psi) $. By $(\neg \lor)$, we have $\h  (\neg \neg \phi \land \neg \neg \psi) \equiv \neg (\neg \phi \lor \neg \psi)$. Now, using Proposition \ref{lemmalg} (ii), we obtain the result.

(xviii). We may procced as in (xvii), just using $(\neg \lor) $ instead of  $(\neg \land)$ and viceversa.

(xix). By $(\neg \otimes)$ we have $\h \neg (\phi \otimes \psi) \equiv (\neg \phi \otimes \neg \psi) $. To prove that  $\h \neg \neg (\phi \otimes \psi) \equiv \neg (\neg \phi \otimes \neg \psi) $, note that by $(\neg \neg)$ we have $\h \neg \neg (\phi \otimes \psi) \equiv (\phi \otimes \psi) $ and by $(\neg \otimes)$ we have $\h \neg (\neg \phi \otimes \neg \psi) \equiv (\neg \neg \phi \otimes \neg \neg \psi)$. Now we apply $(\neg \neg)$ to obtain the result. 

(xx). By $(\neg \oplus)$ we have $\h \neg (\phi \oplus \psi) \equiv (\neg \phi \oplus \neg \psi) $. To prove $\h \neg \neg (\phi \oplus \psi) \equiv \neg (\neg \phi \oplus \neg \psi) $ we may proceed as in (xix), using Proposition \ref{lemmalg} (xix).

(xxi). By $(\neg \neg)$ we have $\h  \phi \equiv  \neg \neg \phi$, and $\h  \neg \phi \equiv  \neg \neg \neg \phi$ is also an instance of $(\neg \neg)$.

(xxii). In Proposition \ref{lemmalg} (iii) we proved that $\h \phi \s \phi$. Hence, by $(\s \oplus )$ and MP, we have that  $\h (\phi \s \phi) \oplus \neg (\phi \s \phi) $. From this it follows immediately that $$ ((\phi \s \phi) \oplus \neg (\phi \s \phi)) \s \psi \h  \psi$$
therefore
$$\h (((\phi \s \phi) \oplus \neg (\phi \s \phi)) \s \psi) \s \psi.$$ 

That $\h \psi \s (((\phi \s \phi) \oplus \neg (\phi \s \phi)) \s \psi)$ holds is also clear, since it is an instance of $(\s 1)$. To prove that  $\h \neg (((\phi \s \phi) \oplus \neg (\phi \s \phi)) \s \psi) \s \neg \psi, $ it is enough to note that, by $(\neg \s)$, we have $$ \neg (((\phi \s \phi) \oplus \neg (\phi \s \phi)) \s \psi) \h  ((\phi \s \phi) \oplus \neg (\phi \s \phi)) \land \neg \psi.$$ 
So, by the transitivity of ${\h}$, the result follows. Finally, to prove that $\h \neg \psi \s \neg (((\phi \s \phi) \oplus \neg (\phi \s \phi)) \s \psi)$, note that by $(\neg \s)$ we have
$$ ((\phi \s \phi) \oplus \neg (\phi \s \phi)) \land \neg \psi \h \neg (((\phi \s \phi) \oplus \neg (\phi \s \phi)) \s \psi),$$ so  the result again follows easily.

(xxiii). Using Proposition \ref{lemmalg} (i), we will prove that
\begin{align*}
& \h ((\phi \land \psi) \s \tet) \s (\phi \s (\psi \s \tet)) \\
& \h (\phi \s (\psi \s \tet)) \s ((\phi \land \psi) \s \tet) \\
& \h \neg ((\phi \land \psi) \s \tet) \s \neg (\phi \s (\psi \s \tet)) \\
& \h \neg (\phi \s (\psi \s \tet)) \s \neg ((\phi \land \psi) \s \tet). 
\end{align*}
The first two are easily proved, for they amount to $(\phi \land \psi) \s \tet, \phi, \psi \h \tet$ and $\phi \s (\psi \s \tet), \phi \land \psi \h \tet$. As to the second two, using $(\neg \s)$, we will prove that $   (\phi \land \psi) \land \neg \tet \dashv \h  \phi \land \neg (\psi \s \tet)$. Applying $(\neg \s)$ again, it is easy to see that this follows from the fact that $(\phi \land \psi) \land \neg \tet \dashv \h \phi \land (\psi \land \neg \tet)$. 

(xxiv). Follows immediately from Proposition \ref{lemmalg} (xvi) and (xvii).

(xxv). Clearly $\h  ((\psi \s \tet) \s \psi) \s \psi$, since it is an instance of $(\s 3)$. From this we easily obtain 
$$\h  ((\phi \s \phi) \oplus \neg (\phi \s \phi)) \s ( ((\psi \s \tet) \s \psi) \s \psi).$$
Similarly, to prove that 
$$\h  \neg ( ((\psi \s \tet) \s \psi) \s \psi) \s \neg ((\phi \s \phi) \oplus \neg (\phi \s \phi)),$$ 
we will show that $\h \neg  ((\phi \s \phi) \oplus \neg (\phi \s \phi))$. To see this, note that by (iii) and $(\neg \neg)$ we have $\h \neg \neg (\phi \s \phi)$. By $(\s \oplus)$ and MP we obtain $\h \neg (\phi \s \phi)  \oplus  \neg \neg (\phi \s \phi) $ and by $(\neg \oplus)$ we have $$ \neg (\phi \s \phi)  \oplus  \neg \neg (\phi \s \phi) \h \neg ((\phi \s \phi) \oplus  \neg (\phi \s \phi)). $$ Now, using the transitivity of $\h$, we obtain the desired result.

(xxvi). Using Proposition \ref{lemmalg} (xvii), it will be enough to prove that $ \phi \lor \psi \dashv \h \phi \oplus \psi $ for all $\phi, \psi \in Fm$.  We have that 
$$  (\phi \supset (\phi \oplus \psi)) \supset ((\psi \supset (\phi \oplus \psi)) \supset ((\phi \lor \psi) \supset (\phi \oplus \psi)))$$
is an instance of $(\lor \supset )$. Now, since $\phi \supset (\phi \oplus \psi)$ and $\psi \supset (\phi \oplus \psi)$ are instances of $( \supset \oplus)$, we may apply MP two times to obtain $\h (\phi \lor \psi) \supset (\phi \oplus \psi)$, hence $\phi \lor \psi \h  \phi \oplus \psi$. The same reasoning, using $(\oplus \supset )$ and $( \supset \lor)$ instead of $(\lor \supset )$ and $( \supset \oplus)$, allows us to conclude that $  \phi \oplus  \psi \h \phi \lor \psi $.

(xxvii). Clearly  $\phi, \phi \s \psi \h  \phi$ and by MP we have $\phi, \phi \s \psi \h  \psi$. Now, using (i), we obtain $\phi, \phi \s \psi \h  \phi \otimes \psi$, i.e. $\h \phi \s ((\phi \s \psi) \s (\phi \otimes \psi))$. To prove that $\h \neg ((\phi \s \psi) \s (\phi \otimes \psi)) \s \neg \phi$, note that by $(\neg \s)$ we have $\neg ((\phi \s \psi) \s (\phi \otimes \psi)) \h  (\phi \s \psi) \land \neg (\phi \otimes \psi)$. By $(\land \s )$ we have $ (\phi \s \psi) \land \neg (\phi \otimes \psi) \h \neg (\phi \otimes \psi) $ and by $(\neg \otimes )$ we obtain  $ \neg (\phi \otimes \psi) \h \neg \phi \otimes \neg \psi$. Now, since $ \neg \phi \otimes \neg \psi \h \neg \phi$ by $(\otimes \s )$, using the transitivity of $\h$ we obtain $ \neg ((\phi \s \psi) \s (\phi \otimes \psi))  \h \neg \phi$.

(xxviii). Follows immediately from $(\neg \s)$ and Proposition \ref{lemmalg} (xvii).

(xxix).  To see that $\h ((\phi \lor \psi) \s \tet) \s ((\phi \s \tet) \land (\psi \s \tet))$, just note that $\phi \s (\phi \lor \psi)$ is an instance of $(\s \lor)$, so by the transitivity of $\s$ we have $ (\phi \lor \psi) \s \tet, \phi \h \tet$ and similarly  $ (\phi \lor \psi) \s \tet, \psi \h \tet$. Hence, using Proposition \ref{lemmalg} (i), we obtain the result. 

To prove that $\h \neg ((\phi \s \tet) \land (\psi \s \tet)) \s \neg ((\phi \lor \psi) \s \tet)$, note that $ ((\phi \land \neg \tet) \s ((\phi \lor \psi) \land \neg \tet)) \s ( ((\psi \land \neg \tet) \s ((\phi \lor \psi)) \land \neg \tet) \s ( (\phi \land \neg \tet) \lor (\psi \land \neg \tet)) \s ((\phi \lor \psi) \land \neg \tet) ) $
is an instance of $(\lor \s)$. It is not difficult to prove that $\h (\phi \land \neg \tet) \s ((\phi \lor \psi) \land \neg \tet)$ and $\h (\psi \land \neg \tet) \s ((\phi \lor \psi) \land \neg \tet)$, so applying MP we obtain
$$(\phi \land \neg \tet) \lor (\psi \land \neg \tet)) \h (\phi \lor \psi) \land \neg \tet.$$ 
Now observe that by $(\neg \land)$ we have $$\neg ((\phi \s \tet) \land (\psi \s \tet)) \h \neg (\phi \s \tet) \lor \neg (\psi \s \tet)$$ and, applying $(\neg \s)$ and (xviii), we obtain $$\neg (\phi \s \tet) \lor \neg (\psi \s \tet)) \h (\phi \land \neg \tet) \lor (\psi \land \neg \tet).$$ Hence, by the transitivity of $\h$, we have $$\neg ((\phi \s \tet) \land (\psi \s \tet)) \h (\phi \lor \psi) \land \neg \tet.$$ Now, by $(\neg \s)$ we have $ (\phi \lor \psi) \land \neg \tet \h \neg ((\phi \lor \psi) \s \tet),$ so the result immediately follows.

To see that $\h ((\phi \s \tet) \land (\psi \s \tet)) \s ((\phi \lor \psi) \s \tet)$, note that using $(\lor \s)$ we obtain $\phi \s \tet, \psi \s \tet \h (\phi \lor \psi) \s \tet$, Hence, by Proposition \ref{lemmalg} (xvi), the result easily follows.

It remains to prove that $$\h \neg ((\phi \lor \psi) \s \tet) \s \neg ((\phi \s \tet) \land (\psi \s \tet)).$$ By $(\neg \s)$ we have $$ \neg ((\phi \lor \psi) \s \tet) \h  (\phi \lor \psi) \land \neg \tet.$$ Using again $(\neg \s)$ and Proposition \ref{lemmalg} (xviii), we have  $$ (\phi \land \neg \tet) \lor (\psi \land \neg \tet) \h \neg (\phi \s \tet) \lor \neg (\psi \s \tet). $$ By $(\neg \land)$ we have $$ \neg (\phi \s \tet) \lor \neg (\psi \s \tet) \h \neg ((\phi \s \tet) \land (\psi \s \tet)). $$ Hence, by the transitivity of $\h$, it will be enough to prove that $$ (\phi \lor \psi) \land \neg  \tet \h (\phi \land \neg \tet) \lor (\psi \land \neg \tet). $$ To see this, note that by $(\s \lor)$ it is easy to show that $ \phi, \neg \tet \h (\phi \land \neg \tet) \lor (\psi \land \neg \tet)$ and $ \psi, \neg \tet \h (\phi \land \neg \tet) \lor (\psi \land \neg \tet).$ Hence we have $$ \phi \h \neg \tet \s ((\phi \land \neg \tet) \lor (\psi \land \neg \tet))$$ and $$ \psi \h \neg \tet \s ((\phi \land \neg \tet) \lor (\psi \land \neg \tet)).$$ Now, using $(\lor \s)$, we obtain $$ \phi \lor \psi \h \neg \tet \s ((\phi \land \neg \tet) \lor (\psi \land \neg \tet)), $$ hence  $$ \phi \lor \psi, \neg \tet \h  (\phi \land \neg \tet) \lor (\psi \land \neg \tet).$$ Now from this the result easily follows.
\end{proof}
 
The previous properties enable us to obtain the following result:

\begin{thm} \label{t:sintalg}
The logic $H_{\s}$ is algebraizable with equivalence formula $\varphi \leftrightarrow \psi$ and defining equation $ \varphi \approx \varphi  \supset \varphi $.
\end{thm}

\begin{proof}
Using the intrinsic characterization given by Blok and Pigozzi (\cite{BP89}, Theorem 4.7), it is sufficient to check that the following conditions hold: for all formulas $\varphi, \psi, \vartheta \in Fm$,

\begin{enumerate}[(i)]

\item $\varphi \dashv \vdash_{\scriptstyle H_{\supset}} \varphi \leftrightarrow (\varphi \supset \varphi)$

\item  $\vdash_{\scriptstyle H_{\supset}} \varphi \leftrightarrow  \varphi$

\item $\varphi \leftrightarrow \psi \vdash_{\scriptstyle H_{\supset}} \psi \leftrightarrow \varphi$

\item $\varphi \leftrightarrow \psi, \psi \leftrightarrow \vartheta \vdash_{\scriptstyle H_{\supset}} \varphi \leftrightarrow \vartheta$

\item $\varphi \leftrightarrow \psi \vdash_{\scriptstyle H_{\supset}} \neg \varphi \leftrightarrow  \neg \psi$

\item $\varphi_1 \leftrightarrow \psi_1, \varphi_2 \leftrightarrow \psi_2 \vdash_{\scriptstyle H_{\supset}} (\varphi_1 \bullet \varphi_2) \leftrightarrow (\psi_1 \bullet \psi_2)$, for all formulas $\varphi_1, \varphi_2, \psi_1, \psi_2 \in Fm$ and for any connective $\bullet \in \left\{\land, \lor, \otimes, \oplus, \supset \right\}.$
\end{enumerate}
And this follows directly from Proposition \ref{lemmalg} (v) to (xv).  
\end{proof}
 
Taking into account Arieli and Avron's completeness result (Theorem \ref{t:compl}), we may conclude that the logic $\lbs$ is algebraizable. In the following section we will determine its associated class of algebras $\mathbf{Alg^{*}}\lbs$.
 
\section{The equivalent algebraic semantics of $\lbs$} 
\label{sec:algstar}

We will now introduce a class of algebras that will later be proved to be the equivalent algebraic semantics of the logic $\lbs$. 

\begin{definition} \label{bilimpl}
{\rm An \emph{implicative bilattice} is an algebra $\Al[B] = \left\langle B, \land, \lor, \otimes, \oplus, \supset, \neg \right\rangle$ such that $\left\langle B, \land, \lor, \otimes, \oplus, \neg \right\rangle$ is a bilattice and the following equations are satisfied:
\begin{enumerate}[ ]
\item (IB1) \hspace{2mm} $(x \supset x) \supset y \approx y$ 
\item (IB2)  \hspace{2mm}  $x \supset (y \supset z) \approx (x \land y) \supset z \approx (x \otimes y) \supset z$
\item(IB3)   \hspace{2mm}  $((x \supset y) \supset x) \supset x  \approx x \s x $ 
\item(IB4)  \hspace{2mm}  $(x \lor y) \supset z \approx (x \supset z) \land  (y \supset z) \approx (x \oplus y) \supset z  $ 
\item (IB5)  \hspace{2mm}  $x \land ((x \supset y) \supset (x \otimes y)) \approx x $
\item (IB6)  \hspace{2mm}  $\neg (x \supset y ) \supset z  \approx (x \land \neg y)  \supset z.$
\end{enumerate}
We denote by $\ib$ the variety of implicative bilattices.}
\end{definition}

In the following propositions we shall prove some facts about implicative bilattices that will be needed to study the relationship between this class of algebras and our logic. In order to simplify the notation, we will abbreviate  the term $(x \s x) \oplus \neg (x \s x)$ as $\top(x)$ and, for any element $a$ of an implicative bilattice, we will write $E(a)$ as a shorthand for  $ a = a \s a$.

\begin{proposition} \label{p:biglem}
Let $\Al[B]$ be an implicative bilattice. Then, for all $a, b, c \in B$: 
\begin{enumerate}[(i)]
\item $a = b \s b$ implies $a \s c = c$ and $E(a)$.  
\item $a \leq_t b \supset a $.
\item $\top(a) = \neg \top(a)$.
\item $\top(a) \s b = b$.
\item $\top(a) \leq_t b$ implies $E(b)$.
\item $E(a \s b)$ implies $a \leq_t a \otimes b$.
\item $E(a \s b)$ and $E(\neg b \s \neg a)$ imply $a \leq_t b$.
\item $E(a \s b)$ and $E(\neg a \s \neg b)$ imply $a \leq_k b$.
\item $\top(a) \leq_t b \s b$.
\item $\top(a) \leq_t b$ if and only if $E(b)$.
\item $\top(a) = \top(b)$.
\item $a \leq_k \top(a)$.
\item If $ a \land \top(a) = b \land \top(a)$, then $a \supset d = b \supset d $ for all $d \in B.$
\item $a \land (a \s b) \land \top(a) \leq b$.
\item $a \supset (b \supset c) = (a \supset b) \supset (a \supset c)$.
\item $a \s a = \neg a \s \neg a $.
\end{enumerate}
\end{proposition}

\begin{proof}

(i). By  (IB1) we have $(b \supset b) \supset c = c$ and $E(b \supset b)$, so the result immediately follows.

(ii). Using (IB1) and (IB4) we have 
\begin{align*} 
a \land (b \supset a) 
& =  ((c \supset c) \supset a) \land (b \supset a) \\
& =  ((c \supset c) \lor b) \supset a \\
& =  (c \supset c) \s (((c \supset c) \lor b) \supset a) \\
& =  ((c \supset c) \land ((c \supset c) \lor b)) \supset a \\
& =  (c \supset c) \supset a \\
& =  a.
\end{align*}

(iii). Immediate, using the properties of the bilattice negation.

(iv). By (IB4) we have $\top(a) \s b = ((a \s a ) \s b) \land (\neg (a \s a ) \s b)$. By (IB1) and (ii) we have $$((a \s a ) \s b) \land (\neg (a \s a ) \s b) = b \land (\neg (a \s a ) \s b)  = b.$$

(v). Assume $\top(a) \leq_t b$.  Then we have
\begin{align*} 
b \s b & = (b \lor \top(a)) \s b  & \text{ from the assumption} \\ 
& = (b \s b) \land (\top(a) \s b)  & \text{ by (IB4)} \\  
& =  (b \s b) \land b & \text{ by (iv)} \\
& = b & \text{ by (ii)}.
\end{align*}

(vi). Assume $E(a \s b)$. Then, by (IB5) and 
(i), we have $$ a  \leq_t (a \s b ) \s (a \otimes b) = a \otimes b. $$

(vii). Assume $E(a \s b)$ and $E(\neg b \s \neg a)$. Then, using (vi), we obtain $ a\leq_t a \otimes b$ and $\neg b \leq_t \neg b \otimes \neg a.$ 
By the properties of the bilattice negation, this implies  $ b = \neg \neg b \geq_t  \neg (\neg b \otimes \neg a) = a \otimes b$. Hence $a \leq_t a \otimes b \leq_t b$, so the result immediately follows.

(viii). Assume $E(a \s b)$ and $E(\neg a \s \neg b)$. Reasoning as in (vii), we obtain  $a \leq_t a \otimes b$ and $a \geq_t a \otimes b$. Hence $a = a \otimes b$, i.e. $a \leq_k b$.

(ix). We shall prove that $E(\top(a) \s (b \s b))$ and $E( \neg (b \s b) \s \neg \top(a)  )$, so that, by (vii), the result will follow. As to the first, by (iv) we have $\top(a) \s (b \s b) = b \s b$. Now, applying 
(i), the result immediately follows. As to the second, note that, by (viii), we have $\neg (b \s b) \s \neg \top(a) = \neg (b \s b) \s \top(a)  $. By (ii), we have $\top(a) \leq_t \neg (b \s b) \s \top(a)$, so applying (v) we obtain the result. 

(x). The rightwards implication has been proved in (v), so we only need to prove that $E(b)$ implies $\top(a) \leq_t b$,
and this  follows immediately from (viiii).

(xi). By symmetry, it is sufficient to show that $\top(a) \leq_t \top(b)$, i.e., using (vii), that $E(\top(a) \s \top(b))$ and $E (\neg \top(b) \s \neg \top(a)) $. By (iii) we have $\neg \top(b) \s \neg \top(a) = \top(b) \s \top(a)$, so, again by symmetry, it will be enough just to check that $E(\top(a) \s \top(b))$. By (i) we have $\top(b) \leq_t \top(a) \s \top(b) $, so, using (x), the result immediately follows.   

(xii). We shall prove that $E(a \s \top(a))$ and $E(\neg a \s \top(a))$, so that (iii) will imply the result. By (xi) we have $\neg a \s \top(a) = \neg a \s \top(\neg a)$, so it will be enough to prove that $E(a \s \top(a))$. By (ii) we have  $\top(a) \leq_t a \s \top(a)$, which, by (x), implies $E(a \s \top(a))$ .

(xiii) Assume $ a \land \top(a) = b \land \top(a)$. Note that by (IB2) and (iv), we have  $$(a \land \top(a)) \s d = \top(a) \s (a \s d ) = a \s d$$ for every $d \in B$, and similarly we have $(b \land \top(a)) \s d = b \s d$. From the assumption then it follows that 
$a \s d = b\s d$ for every $d \in B$. 

(xiv) We will prove that $E((a \land (a \s b) \land \top(a)) \s b)$ and $E(\neg b \s \neg (a \land (a \s b) \land \top(a)))$. The result will then follow from (vii). On the one hand, by (IB2) and (IB1), we have
\begin{align*}
(a \land (a \s b) \land \top(a)) \s b   & =  
\top(a) \s ((a \land (a \s b)) \s b)    \\ & = 
(a \land (a \s b)) \s b    \\ & = 
(a \s b) \s (a \s b).
\end{align*}
Hence, by (i), we obtain $E((a \land (a \s b) \land \top(a)) \s b)$.
On the other hand, using De Morgan's laws and (ii), we have
\begin{align*}
\neg b \s \neg (a \land (a \s b) \land \top(a))  & =  
\neg b \s  (\neg (a \land (a \s b)) \lor \neg \top(a))  \\ & \geq 
\neg (a \land (a \s b)) \lor \top(a) \\ & \geq \top(a).
\end{align*}
Now from (x) we obtain $E(\neg b \s \neg (a \land (a \s b) \land \top(a)))$. 

(xv).  By (ii) we have $b \leq a \s b$, so $ a \land b \land \top(a) \leq a \land (a \s b) \land \top(a)$.  By (xiv) we have  $a \land (a \s b) \land \top(a) \leq b$. Hence $a \land (a \s b) \land \top(a) = a \land b \land \top(a)$. By (xiii), this implies that $(a \land (a \s b)) \s c = (a \land b) \s c$ for every $c \in B$. Using (IB2),  we obtain $ (a \s b) \s (a \s c) = (a \land (a \s b)) \s c  = (a \land b) \s c = a \s (b \s c)$, so we are done.

(xvi). We shall prove that $E ((a \s a) \equiv (\neg a \s \neg a))$ and $E (\neg (a \s a) \equiv \neg (\neg a \s \neg a))$, so that the result will follow by (vii). The first one is obvious. As to the second, applying (IB6), we have
\begin{align*}
\neg (a \s a ) \s \neg (\neg a \s \neg a) 
& = (a \land \neg a) \s \neg (\neg a \s \neg a)  \\
& = (\neg a \land \neg \neg a) \s \neg (\neg a \s \neg a)  \\
& = \neg (\neg a \s \neg a) \s \neg (\neg a \s \neg a)
\end{align*}
and
\begin{align*}
 \neg (\neg a \s \neg a)  \s \neg (a \s a )
& = (\neg a \land \neg \neg a) \s  \neg (a \s a ) \\
& = (a \land \neg a)   \s  \neg (a \s a ) \\
& = \neg (a \s a )  \s  \neg (a \s a ).
\end{align*}
Hence, using (i), the result follows.

\end{proof}

From Proposition \ref{p:biglem} (xi) it follows that $\top(a) = (a \s a) \oplus \neg (a \s a)$ defines an algebraic constant in every  $\Al[B] \in \ib$. Moreover, by (xi), this constant is the maximum element w.r.t.\ the knowledge order. So we can denote it just by $\top$. Using this notation, let us state some more arithmetical properties of implicative bilattices.

\begin{proposition} \label{p:lem2}  
Let $\Al[B]$ be an implicative bilattice. Then, for all $a, b, c \in B$: 
\begin{enumerate}[(i)] 
 \item $\top \leq_t  a\supset (b \supset a). $
\item $a \leq_t b $ or $a \leq_k b $ implies $\top \leq_t a \supset b $. 
\item $ \top \leq_t a $ and $ \top \leq_t a \supset b  $ imply $ \top  \leq_t b$.
\item $\top \leq_t a \s b $ and $\top \leq_t b \s c $ imply $\top \leq_t a \s c $.
\item $a \leq_t a \otimes b$ if and only if  $  \top  \leq_t a\supset b$.
\item $a \leq_t b$ if and only if $ \top  \leq_t a \rightarrow b $.
\item $a \leq_k b$ if and only if $ \top  \leq_t a \supset b $ and $  \top \leq_t \neg a \supset \neg b$.
\item $a \leq_t (a \rightarrow b) \supset b $.
\item $\top \leq_t  (a \supset b) \lor a $.
\item $a \leq_t b$ implies $c \supset a \leq_t c \supset  b $.
	\item $ a \supset (b \land c)  =  (a \supset b) \land (a \supset c)$.
 \item $a \rightarrow (b \rightarrow c)  =   b \rightarrow (a \rightarrow c)$.
  \item $a \otimes (a \supset b) \leq_k b$.
  \item $a \leq_k b$ implies $c \s a \leq_k c \s b$.
  \item $a \s (b \otimes c) = (a \s b) \otimes (a \s c)$.
	\item $a \ta b = (a \s b) \otimes (\neg b \s \neg a)$.
	\item  $ a \land \top = b \land \top$  if and only if $a \supset d = b \supset d$ for every $d \in B$.
	\end{enumerate}
\end{proposition}

\begin{proof}
(i). By (IB2) and Proposition \ref{p:biglem} (ii) and (ix), we have 
$$a \supset (b \supset a) = b \supset (a \supset a) \geq_t a \supset a \geq_t \top.$$

(ii). If $a \leq_t b $ or $a \leq_k b $, then $a \land b = a$ or $a \otimes b = a$, so by (i) either $a \supset (b \supset a) = (a \land b) \supset b  = a  \supset b \geq_t \top$ or $ a \supset (b \supset a) = (a \otimes b) \supset b = a  \supset b \geq_t \top$.

(iii). Clearly it is sufficient to prove that $a \geq_t \top $ implies $  b = a \s b $, and this follows immediately from Proposition \ref{p:biglem} (i) and (x). 
%



(iv) Assume $\top \leq a \supset b$ and $\top \leq b \supset c $. Note that by (i) and (IB2)  we have $\top \leq (a \land (b \s c)) \supset (b \s c) = (b \s c) \supset (a \supset (b \s c))$. Now, using (iii) and the second assumption, we obtain $\top \leq a \supset (b \supset c)$. By (IB2) and (vi), we have $\top \leq (a \supset b) \supset (a \supset c)$. Using  the first assumption and again (iii), we obtain $\top \leq (a \supset c)$.

(v). The leftwards implication follows from Proposition \ref{p:biglem} (vi) and (x). Conversely, if $a \leq_t a \otimes b$, then by (ii) we have $a \supset (a \otimes b) \geq_t \top $. Now, since $(a \otimes b) \supset b \geq_t \top $, applying  (iv), we obtain $a \supset b \geq_t \top$.

(vi). The leftwards implication follows from Proposition \ref{p:biglem} (vii). Conversely, assume $a \leq_t b$, which implies $\neg b \leq_t \neg a$. Then, by (ii), we have  $a \supset b \geq_t \top$ and $\neg b \supset \neg a \geq_t \top$, hence $a \rightarrow b \geq_t \top$.

(vii). The leftwards implication follows from Proposition \ref{p:biglem} (viii).	Conversely, assume $a \leq_k b$, which by definition implies  $\neg a \leq_k \neg b$. Then, by (ii), we have $\top \leq_t a \supset b$ and $\top \leq_t \neg a \supset \neg b$. 

(viii).  We will prove that $\top \leq_t a \rightarrow ((a \rightarrow b) \supset b)$, which implies that $a \leq_t (a \rightarrow b) \supset b $. So we need to show that 
$$\top \leq_t (a \supset ((a \rightarrow b) \supset b)) \land (\neg ((a \rightarrow b) \supset b) \supset \neg a).$$ From (IB2) and the definition of $\ta$  it follows that  
\begin{align*}
a \supset ((a \rightarrow b) \supset b)  & =  (a \rightarrow b) \supset (a \supset b)  \\ & 
=  ((a \supset b) \land (\neg b \supset \neg a)) \supset (a \supset b)  \\ & 
=  (a \supset b) \supset ((\neg b \supset \neg a) \supset (a \supset b)).
\end{align*}
Now, by (i), we have  $\top \leq_t (a \supset b) \supset ((\neg b \supset \neg a) \supset (a \supset b))$. Therefore, $\top \leq_t (a \supset ((a \rightarrow b) \supset b))$. On the other hand, again by (IB2) and the definition of $\ta$,  we have 
\begin{align*}
((a \rightarrow b) \land \neg b) \supset \neg a  & =  (a \rightarrow b) \supset (\neg b\supset  \neg a)  \\ & =  ((a \supset b) \land (\neg b \supset  \neg a)) \supset (\neg b \supset  \neg a) \\ & =   (\neg b \supset  \neg a) \supset ((a \supset b)  \supset (\neg b \supset  \neg a)).
\end{align*}
Using again (i), we have $\top \leq_t  (\neg b \supset  \neg a) \supset ((a \supset b)  \supset (\neg b \supset  \neg a))$. So it follows that $\top \leq_t  ((a \rightarrow b) \land \neg b) \supset \neg a$. By (IB6) $((a \rightarrow b) \land \neg b) \supset \neg a  =  \neg ((a \rightarrow b) \supset b)  \supset   \neg a$. Therefore $\top \leq_t  \neg ((a \rightarrow b) \supset b)  \supset   \neg a$. Hence $$\top \leq_t  (a \supset ((a \rightarrow b) \supset b)) \land (\neg ((a \rightarrow b) \supset b) \supset \neg a).$$
 
(ix). Since $(a \supset b) \leq_t ( (a \supset b) \lor a)$, by (ii) we have $\top \leq_t  (a \supset b) \supset ( (a \supset b) \lor a)$.  By Proposition \ref{p:biglem} (ii), it follows that
 $$\top \leq_t ((a \supset b) \supset b) \supset  ((a \supset b) \supset ( (a \supset b) \lor a)).$$ 
By (IB2) and (IB4), we have 
$((a \supset b) \supset b) \supset ((a \supset b) \supset ( (a \supset b) \lor a))$   = $(((a \supset b) \supset b) \land (a \supset b)) \supset ( (a \supset b) \lor a)  =  (((a \supset b) \lor a) \supset b) \supset ((a \supset b) \lor a).$
Hence $\top \leq_t (((a \supset b) \lor a) \supset b) \supset ((a \supset b) \lor a).$
By  (IB3) we have $$(((a \supset b) \lor a) \supset b) \supset ((a \supset b) \lor a)) \supset ((a \supset b) \lor a))  =  ((a \supset b) \lor a) \supset ((a \supset b) \lor a) ).$$
So $\top \leq_t  (((a \supset b) \lor a) \supset b) \supset ((a \supset b) \lor a)) \supset ((a \supset b) \lor a)) $. Now, applying (iii), we obtain  $\top \leq_t (a \supset b) \lor a$.

(x).  Assume $a \leq_t b$. We will prove that $(c \supset a) \rightarrow (c \supset b) \geq_t \top$, i.e. that $(c \supset a) \supset (c \supset b) \geq_t \top$ and $\neg (c \supset b) \supset \neg (c \supset a) \geq_t \top$. 

As to the first, note that by (ii) $a \leq_t b$ implies $\top \leq_t a \supset b$. Moreover, by Proposition \ref{p:biglem} (ix) we have $\top \leq_t (c \supset a) \supset (c \supset a)$. So, applying (IB2), we obtain  $$ (c \supset a) \supset (c \supset a) = ((c \supset a) \land c) \supset a \geq_t \top. $$ By (vii), it follows that $\top \leq_t ((c \supset a) \land c) \supset b$. Therefore, applying again (IB2), we have $\top \leq_t (c \supset a) \supset ( c \supset b)$. 

As to the second, note that $a \leq_t b$ implies $\top \leq_t \neg b \supset \neg a$. Reasoning as before, we have
\begin{align*}
\top & \leq_t  \neg (c \supset a) \supset \neg (c \supset a) & \text{ by Proposition \ref{p:biglem} (ix)} \\
& =  (\neg a \land c) \supset \neg (c \supset a)  & \text{ by (IB6)} \\
& =  \neg a \supset ( c \supset \neg (c \supset a)). & \text{ by (IB2)}
\end{align*}
  Now, using (vii) again,  we obtain $\top \leq_t \neg b \supset ( c \supset \neg (c \supset a))$. Hence, using (IB2) and (IB6), we have $\top \leq_t \neg (c \supset b) \supset \neg (c \supset a)$.

(xi). From (x) it follows that  $ a \supset (b \land c) \leq_t (a \supset b)$ and $ a \supset (b \land c) \leq_t (a \supset c)$, so $ a \supset (b \land c) \leq_t (a \supset b) \land (a \supset c)$. 
In order to prove the other inequality, we will show that $\top \leq_t  ((a \supset b) \land (a \supset c)) \rightarrow (a \supset (b \land c))$, i.e. that $\top \leq_t ((a \supset b) \land (a \supset c)) \supset (a \supset (b \land c))$ and $ \top \leq_t \neg (a \supset (b \land c)) \supset \neg ((a \supset b) \land (a \supset c))$. 

As to the first,  note that by (IB2) we have
$$((a \supset b) \land (a \supset c)) \supset (a \supset (b \land c))  =  (a \supset b) \supset ((a \supset c) \supset (a \supset (b \land c)).$$
By Proposition \ref{p:biglem} (xv)  we have 
$(a \supset b) \supset ((a \supset c) \supset (a \supset (b \land c))  =$  $=((a \supset b) \supset (a \supset c)) \supset ((a \supset b) \supset (a \supset (b \land c))).$
Hence
$((a \supset b) \land (a \supset c)) \supset (a \supset (b \land c)) =$ $ ((a \supset b) \supset (a \supset c)) \supset ((a \supset b) \supset (a \supset (b \land c))).$
Using (IB2), we obtain
$$((a \supset b) \land (a \supset c)) \supset (a \supset (b \land c))  =  (((a \supset b) \land a) \supset c) \supset (((a \supset b) \land a)  \supset  (b \land c)).$$
Applying again Proposition \ref{p:biglem} (xv)   and (IB2), we have 
\begin{align*}
& ((a \supset b) \land (a \supset c)) \supset (a \supset (b \land c)) =  \\
& =  ((a \supset b) \land a) \supset (c \supset (b \land c)) = \\
&  =  (a \supset b) \supset ( a \supset (c \supset (b \land c)))  = \\
&  =  (a \supset b) \supset ( (a \land c) \supset (b \land c))  = \\
&  =  (a \land c) \supset ( (a \supset b) \supset (b \land c))  = \\
&  =  ((a \land c) \supset (a \supset b)) \supset ((a \land c) \supset (b \land c)) =  \\
&  =  (((a \land c) \land a) \supset  b) \supset (((a \land c) \land (a \land c)) \supset (b \land c))  = \\
&  =  ((a \land c) \supset b) \supset ((a \land c) \supset ( (a \land c) \supset (b \land c)))  = \\
&  =  (a \land c) \supset (b \supset ((a \land c) \supset (b \land c)))  =  \\
&  =  (a \land c) \supset ((b \land a \land c ) \supset (b \land c))  =  \\ 
&  =  (b \land a \land c ) \supset (b \land c).
\end{align*}
Now, since $(b \land a \land c ) \leq_t (b \land c)$, we have $\top \leq_t  (b \land a \land c ) \supset (b \land c)$. Therefore we obtain $\top \leq_t ((a \supset b) \land (a \supset c)) \supset (a \supset (b \land c))$.

As to the second, applying (IB2), (IB4), (IB6) and De Morgan's laws, we have 
\begin{align*}
& \neg (a \supset (b \land c)) \supset \neg ((a \supset b) \land (a \supset c))  =  \\
&  =   (a \land \neg (b \land c)) \supset \neg ((a \supset b) \land (a \supset c))  =  \\
&  =   (a \land (\neg b \lor \neg c)) \supset  (\neg (a \supset b) \lor \neg (a \supset c))  =  \\
&  =   (\neg b \lor \neg c) \s (a \supset  (\neg (a \supset b) \lor \neg (a \supset c)))  =  \\
&  =   (\neg b \s (a \supset  (\neg (a \supset b) \lor \neg (a \supset c)))) \land  (\neg c \s (a \supset (\neg (a \supset b) \lor \neg (a \supset c))))  =  \\
&  =   ((\neg b \land a) \supset  (\neg (a \supset b) \lor \neg (a \supset c))) \land  ((\neg c \land a) \supset  (\neg (a \supset b) \lor \neg (a \supset c)))  =  \\
&  =   (\neg ( a \s b) \supset  (\neg (a \supset b) \lor \neg (a \supset c))) \land  (\neg (a \s c) \supset  (\neg (a \supset b) \lor \neg (a \supset c)))  =  \\
&  =   (\neg ( a \s b) \lor \neg (a \s c)) \s (\neg (a \supset b) \lor \neg (a \supset c)).
\end{align*}
Since $\top \leq_t (\neg ( a \s b) \lor \neg (a \s c)) \s (\neg (a \supset b) \lor \neg (a \supset c))$, it follows that $$\top \leq_t \neg (a \supset (b \land c)) \supset \neg ((a \supset b) \land (a \supset c)).$$  

(xii).  Using the definition of $\ta$, (xi), De Morgan's laws, (IB2), (IB4)  and (IB6), we  have
\begin{align*} 
& a \rightarrow (b \rightarrow c)   =  \\
&  =  (a \supset ( (b \supset c ) \land (\neg c \supset \neg b ) )) \land (\neg ( (b \supset c ) \land (\neg c \supset \neg b ) ) \supset \neg a)  =  \\
&  =  (a \supset (b \supset c ) ) \land (a \supset (\neg c \supset \neg b )) \land  ( (\neg (b \supset c ) \lor \neg (\neg c \supset \neg b ))  \supset \neg a)  =  \\ 
&  =   (b \supset (a \supset c ) ) \land ((a \land \neg c) \supset \neg b ) \land  (\neg (b \supset c ) \supset \neg a) \land  (\neg (\neg c \supset \neg b )  \supset \neg a)   =  \\
&  =  (b \supset (a \supset c ) ) \land (\neg (a \supset c) \supset \neg b ) \land  ((b \land \neg c ) \supset \neg a) \land  ((\neg c \land b )  \supset \neg a)  =  \\
&  =  (b \supset (a \supset c ) ) \land (\neg (a \supset c) \supset \neg b ) \land  ((b \land \neg c ) \supset \neg a)   =  \\
&  =  (b \rightarrow(a \supset c ) ) \land  (b \supset (\neg c \supset \neg a)) \land (\neg (\neg c \supset  \neg a ) \supset \neg b)  =  \\ 
&  =   (b \rightarrow(a \supset c ) ) \land (b \rightarrow (\neg c \supset \neg a ) )  =  \\
&  =  b \rightarrow (a \rightarrow c).
\end{align*}

(xiii). Using (vii), we will show that $\top \leq_t (a \otimes (a \supset b)) \supset b $ and $\top \leq_t \neg (a \otimes (a \supset b)) \supset \neg b$.  The former is clear, since by (IB2) we have 
$$\top \leq_t (a \supset b)  \supset  (a \supset b)  =  ((a \supset b) \otimes a) \supset b  =  (a \otimes (a \supset b)) \supset b.$$ As to the latter, applying De Morgan's laws, (IB2) and (IB6), we have 
\begin{align*}
\neg (a \otimes (a \supset b)) \supset \neg b  & =   (\neg a \otimes \neg (a \supset b)) \supset \neg b  
\\ & =    \neg a \supset (\neg (a \supset b) \supset \neg b)  \\ & =    \neg a \supset ((a \land \neg b) \supset \neg b).
\end{align*}
Since $(a \land \neg b) \leq_t \neg b$, we have $\top \leq_t (a \land \neg b) \supset \neg b$. So we may conclude that $\top \leq_t  \neg (a \otimes (a \supset b)) \supset \neg b$.

(xiv). Assume $a \leq_k b$. Using (vi),  we will prove that $(c \s a) \s (c \s b) \geq_t \top$ and $\neg (c \s a) \s \neg  (c \s b) \geq_t \top$. For the first, note that $a \leq_k b$ implies $a \supset b \geq_t \top$, and since $((c \supset a) \otimes c) \supset a \geq_t \top $, by transitivity we obtain $((c \supset a) \otimes c) \supset b = (c \supset a) \supset(c \supset b) \geq_t \top$. As to the second, by assumption we have $\neg a \leq_k \neg b$, which implies $\neg a \supset \neg b \geq_t \top$. By (IB6) we have $\neg b \supset ( c \supset \neg (c \supset b)) = (\neg b \land c) \supset \neg (c \supset b)) = \neg (c \supset b) \supset \neg (c \supset b)) \geq_t \top$. Now by (iv) we obtain $\neg a \supset ( c \supset \neg (c \supset b)) = (\neg a \land c) \supset \neg (c \supset b)) = \neg (c \supset a) \supset \neg (c \supset b)) \geq_t \top$. 

(xv). By (xiv) we have  $a \s (b \otimes c) \leq_k (a \s b) \otimes (a \s c)$, so it remains to prove that $a \s (b \otimes c) \geq_k (a \s b) \otimes (a \s c)$, i.e. that $((a \s b) \otimes (a \s c)) \s (a \s (b \otimes c)) \geq_t \top$ and $(\neg  (a \s b) \otimes \neg  (a \s c)) \s \neg  (a \s (b \otimes c)) \geq_t \top$.

As to the first, applying repeatedly (IB2) and Proposition \ref{p:biglem} (xv), we have
\begin{align*}
& ((a \supset b) \otimes (a \supset c)) \supset (a \supset (b \otimes c)) = \\
& = ((a \supset b) \supset (a \supset c)) \supset ((a \supset b) \supset (a \supset (b \otimes c))) =  \\
& =  (((a \supset b) \otimes a) \supset c) \supset (((a \supset b) \otimes a) \supset (b \otimes c))  = \\
& = ((a \supset b) \otimes a) \supset (c \supset (b \otimes c))  = \\
& = (a \supset b) \supset ( a \supset (c \supset (b \otimes c)))  = \\
& = (a \supset b) \supset ( (a \otimes c) \supset (b \otimes c))  = \\
& = (a \otimes c) \supset ( (a \supset b) \supset (b \otimes c))  = \\
& = ((a \otimes c) \supset (a \supset b)) \supset ((a \otimes c) \supset (b \otimes c))  = \\
& = ((a \otimes c) \supset b) \supset ((a \otimes c) \supset ( (a \otimes c) \supset (b \otimes c)))  = \\
& = (a \otimes c) \supset (b \supset ((a \otimes c) \supset (b \otimes c)))  = \\
& = (a \otimes c) \supset ((b \otimes a \otimes c ) \supset (b \otimes c)) \geq_t \top.
\end{align*}

As to the second, applying repeatedly (IB2) and (IB6), we have
\begin{align*}
& (\neg  (a \s b) \otimes \neg  (a \s c)) \s \neg  (a \s (b \otimes c)) = \\
 & =
\neg  (a \s b) \s (\neg  (a \s c) \s \neg  (a \s (b \otimes c)))  =  \\ & =
(a \land \neg b) \s ((a \land \neg c) \s \neg  (a \s (b \otimes c)))  =  \\ & =
(a \land \neg b \land \neg c) \s \neg (a \s (b \otimes c))  = \\ & =
a \s ((\neg b \land \neg c) \s \neg (a \s (b \otimes c)))  = \\ & =
a \s ((\neg b \otimes \neg c) \s \neg (a \s (b \otimes c))) =  \\ & =
a \s (\neg (b \otimes c) \s \neg (a \s (b \otimes c)))  = \\ & =
(a \land \neg (b \otimes c)) \s \neg (a \s (b \otimes c)))  = \\ & =
\neg (a \s (b \otimes c)) \s \neg (a \s (b \otimes c)))   \geq_t \top.
\end{align*}

(xvi). We will prove that $(a \ta b) \ta ((a \s b) \otimes (\neg b \s \neg a)) \geq_t \top$ and $ ((a \s b) \otimes (\neg b \s \neg a)) \ta (a \ta b)  \geq_t \top$.

By (IB2), it is obvious that the following two inequalities hold: 
\begin{align*}
((a \s b) \land (\neg b \s \neg a)) \s ((a \s b) \otimes (\neg b \s \neg a)) & \geq_t \top \\
((a \s b) \otimes (\neg b \s \neg a)) \s ((a \s b) \land (\neg b \s \neg a)) & \geq_t \top.
\end{align*}

Therefore, it remains only to prove that:
\begin{align*}
(\neg (a \s b) \lor \neg (\neg b \s \neg a)) \s (\neg (a \s b) \otimes \neg (\neg b \s \neg a)) & \geq_t \top \\
(\neg (a \s b) \otimes \neg (\neg b \s \neg a)) \s (\neg (a \s b) \lor \neg (\neg b \s \neg a)) & \geq_t \top.
\end{align*}
The second one is easy. As to the first, note that by (IB6) we have
\begin{align*}
\neg (a \s b) \s \neg (\neg b \s \neg a) & =
(a \land \neg b) \s \neg (\neg b \s \neg a) \\ & =
(\neg b \land \neg \neg a) \s \neg (\neg b \s \neg a) \\ & =
\neg (\neg b \s \neg a) \s \neg (\neg b \s \neg a) \\ & \geq_t \top.
\end{align*}
And similarly
\begin{align*}
\neg (\neg b \s \neg a) \s \neg (a \s b) & =
(\neg b \land a) \s \neg (a \s b) \\ & =
\neg (a \s b) \s \neg (a \s b) \\ & \geq_t \top.
\end{align*}
Now, using (ii), we have
\begin{align*}
& (\neg (a \s b) \s \neg (a \s b)) \otimes (\neg (a \s b) \s \neg (\neg b \s \neg a))  = \\
& \neg (a \s b) \s ( \neg (a \s b) \otimes \neg (\neg b \s \neg a)) \geq_t \top
\end{align*}
and
\begin{align*}
& (\neg (\neg b \s \neg a) \s \neg (a \s b)) \otimes (\neg (\neg b \s \neg a) \s \neg (\neg b \s \neg a))  = \\
& \neg (\neg b \s \neg a) \s ( \neg (a \s b) \otimes \neg (\neg b \s \neg a))  \geq_t \top.
\end{align*}
By (IB4), we have
$(\neg (a \s b) \s ( \neg (a \s b) \otimes \neg (\neg b \s \neg a)) )\land (\neg (\neg b \s \neg a) \s ( \neg (a \s b) \otimes \neg (\neg b \s \neg a))) = 
 (\neg (a \s b) \lor \neg (\neg b \s \neg a)) \s ( \neg (a \s b) \otimes \neg (\neg b \s \neg a)).$
Hence, applying De Morgan's laws and the interlacing conditions, we obtain
$$
\neg ((a \s b) \land (\neg b \s \neg a)) \s ( \neg (a \s b) \otimes \neg (\neg b \s \neg a)) \geq_t \top.
$$

(xvii). The rightwards implication has been proven in Proposition \ref{p:biglem} (xiii). As to the other one,
%
assume $a \supset c = b \supset c $ for all $c \in B$. Then, in particular, $a \supset b = b \supset b$ and $b \supset a = a \supset a$. We will show that $(a \land \top) \leftrightarrow (b \land \top) \geq_t \top$, so the result will follow from  (vi).  Notice that, since $\top \leq_t  a \supset a$ and $\top \leq_t b \supset b$, we obtain $\top \leq_t a \supset b$ and $\top \leq_t b \supset a$.   So we have
\begin{align*}
(a \land \top) \rightarrow (b \land \top) 
&  = ((a \land \top) \supset (b \land \top)) \land ( (\neg b \lor \top) \supset (\neg a \lor \top)) \\
&  = (\top \supset (a \supset (b \land \top))) \land (\neg a \lor \top) \\
&  = (a \supset (b \land \top)) \land (\neg a \lor \top) \\
&  = (a \supset b) \land (a \supset \top) \land (\neg a \lor \top)\\
 & \geq_t \top.
\end{align*}
Interchanging $b$ with $a$, the same proof shows that $(b \land \top) \rightarrow (a \land \top) \geq_t \top$.
%
\end{proof}

In the next proposition we state an important property of the bilattice reduct of any implicative bilattice.

\begin{proposition} \label{impinterla}
Let $\Al[B] = \left\langle B, \land, \lor, \otimes, \oplus, \supset, \neg \right\rangle$ be an implicative bilattice. Then the  reduct $\left\langle B, \land, \lor, \otimes, \oplus, \neg \right\rangle$ is a distributive bilattice.
\end{proposition}

\begin{proof}
We will prove first that the  reduct $\left\langle B, \land, \lor, \otimes, \oplus, \neg \right\rangle$ is an interlaced bilattice. Let $a, b \in B $ be such that $a \leq_t b$. To see that $a \otimes c \leq_t b \otimes c$ and $a \oplus c \leq_t b \oplus c$ for all $c \in B$, we  prove that $(a \otimes c) \ta (b \otimes c) \geq_t \top$ and $(a \oplus c) \ta (b \oplus c) \geq_t \top$. 

As to the first, using (IB2) and Proposition  \ref{p:lem2} (ii), we have
\begin{align*}
(a \otimes c) \s (b \otimes c)  & = 
((a \land b) \otimes c) \s (b \otimes c) \\ & = 
(a \otimes b \otimes c) \s (b \otimes c) \\ & \geq_t \top
\end{align*}
and, applying  De Morgan's laws and Proposition  \ref{p:biglem} (ii),
\begin{align*}
\neg (b \otimes c) \s \neg (a \otimes c)  & = 
(\neg b \otimes \neg c) \s  (\neg a \otimes \neg c)  \\ & = 
(\neg (a \lor b) \otimes \neg c) \s  (\neg a \otimes \neg c)  \\ & = 
((\neg a \land \neg b) \otimes \neg c) \s  (\neg a \otimes \neg c)  \\ & = 
(\neg a \otimes \neg b \otimes \neg c) \s  (\neg a \otimes \neg c)  \\ & \geq_t \top.
\end{align*}

As to the second, using (IB4), (IB2) and Proposition  \ref{p:biglem} (ii) and  Proposition  \ref{p:lem2} (ii), 
we have
\begin{align*}
(a \oplus c) \s (b \oplus c) & = 
(a \s (b \oplus c)) \land (c \s (b \oplus c)) \\ & = 
((a \land b) \s (b \oplus c)) \land (c \s (b \oplus c)) \\ & = 
(a \s (b \s (b \oplus c))) \land (c \s (b \oplus c)) \\ 
& = c \s (b \oplus c)\\ & \geq_t \top
\end{align*}
and, applying also De Morgan's laws,
\begin{align*}
\neg (b \oplus c) \s \neg (a \oplus c) & = 
(\neg b \oplus \neg c) \s (\neg a \oplus \neg c) \\ & = 
(\neg b \s (\neg a \oplus \neg c)) \land (\neg c \s (\neg a \oplus \neg c)) \\ & = 
(\neg (a \lor b) \s (\neg a \oplus \neg c)) \land (\neg c \s (\neg a \oplus \neg c)) \\ & = 
((\neg a \land \neg b) \s (\neg a \oplus \neg c)) \land (\neg c \s (\neg a \oplus \neg c)) \\ & = 
(\neg b \s (\neg a \s (\neg a \oplus \neg c))) \land (\neg c \s (\neg a \oplus \neg c)) \\ 
& = (\neg b \otimes \neg a) \s (\neg a \oplus \neg c)) \land (\neg c \s (\neg a \oplus \neg c)) \\
& \geq_t \top,
\end{align*}
because $\neg b \otimes \neg a \leq_k \neg a \oplus \neg c$ and $\neg c \leq_k \neg a \oplus \neg c$, which imply $\top \leq_t (\neg b \otimes \neg a)\s (\neg a \oplus \neg c)$ and $\top \leq_t \neg c \s (\neg a \oplus \neg c).$  

Now assume $a \leq_k b$. To see that $a \land c \leq_k b \land c$, we will prove that $(a \land c) \s (b \land c) \geq_t \top$ and $\neg (a \land c) \s \neg (b \land c) \geq_t \top$. Then  using Proposition \ref{p:lem2} (vii) we will obtain the desired conclusion.

As to the former, using (IB2) and Proposition  \ref{p:lem2} (ii), we have
\begin{align*}
(a \land c) \s (b \land c) & = 
((a \otimes b) \land c) \s (b \land c) \\ & = 
(a \land b \land c) \s (b \land c) \\ & \geq_t \top.
\end{align*}

As to the latter, using De Morgan's laws, (IB4), (IB2) and Proposition  \ref{p:lem2} (ii), we have
\begin{align*}
\neg (a \land c) \s \neg (b \land c) & = 
(\neg a \lor \neg c) \s (\neg b \lor \neg c) \\ & = 
(\neg a \s (\neg b \lor \neg c)) \land (\neg c \s (\neg b \lor \neg c)) \\ & = 
(\neg (a \otimes b) \s (\neg b \lor \neg c)) \land (\neg c \s (\neg b \lor \neg c)) \\ & = 
(( \neg a \otimes \neg b) \s (\neg b \lor \neg c)) \land (\neg c \s (\neg b \lor \neg c)) \\ & = 
(\neg a \s (\neg b \s (\neg b \lor \neg c))) \land (\neg c \s (\neg b \lor \neg c)) 
\\ 
& = (\neg a \land \neg b) \s (\neg b \lor \neg c))) \land (\neg c \s (\neg b \lor \neg c))\\& \geq_t \top.
\end{align*}
because $\neg c \leq_t (\neg b \lor \neg c)$ and $ (\neg a \land \neg b) \leq_t (\neg b \lor \neg c)$, and so $\top \leq_t (\neg a \s (\neg b \s (\neg b \lor \neg c))), (\neg c \s (\neg b \lor \neg c))$.

 To see that $a \lor c \leq_k b \lor c$, note that $a \leq_k b$ if and only if $\neg a \leq_k \neg b$. Applying what we have just proved, we have  $\neg a \land \neg c \leq_k \neg b \land \neg c$ and, therefore, $ \neg (\neg a \land \neg c) \leq_k \neg (\neg b \land \neg c)$. Now, using De Morgan's laws, we have $a \lor c = \neg (\neg a \land \neg c)  \leq_k   \neg (\neg b \land \neg c) = b \lor c$.
 
Therefore   $\left\langle B, \land, \lor, \otimes, \oplus, \neg \right\rangle$ is an interlaced bilattice. Hence, by Proposition \ref{prop:distributive_bl}, 
any of the twelve distributive laws implies the others. Let us check that $a \land (b \lor c) \leq_t (a \land b) \lor (a \land c)$ for all $a, b, c \in B$. As before, it is enough to  prove that  $(a \land (b \lor c)) \s ((a \land b) \lor (a \land c)) \geq_t \top$ and $\neg ((a \land b) \lor (a \land c))  \s \neg (a \land (b \lor c)) \geq_t \top$.

As to the former, using (IB2), (IB4) and Proposition \ref{p:lem2} (xi), we have
\begin{align*}
& (a \land (b \lor c)) \s ((a \land b) \lor (a \land c)) = \\ & = 
a \s ((b \lor c) \s ((a \land b) \lor (a \land c))) = \\ & = 
a \s ((b \s ((a \land b) \lor (a \land c))) \land (c \s ((a \land b) \lor (a \land c)))) = \\ & = 
(a \s (b \s ((a \land b) \lor (a \land c)))) \land (a \s (c \s ((a \land b) \lor (a \land c)))) = \\ & = 
((a \land b) \s ((a \land b) \lor (a \land c))) \land ((a \land c) \s ((a \land b) \lor (a \land c))) 
\geq_t \top.
\end{align*}

As to the latter, we will us use the following abbreviations:
\begin{align*}
d & = \neg a \s (\neg a \lor (\neg b \land \neg c)) \\
e & = \neg c \s (\neg a \lor (\neg b \land \neg c)).
\end{align*}
It is easy to see that $d \geq_t \top$ and $(\neg a \s  e) \geq_t \top $ and $(\neg b \s  e) \geq_t \top$. Now,  using De Morgan's laws, (IB2), (IB4) and Proposition  \ref{p:lem2} (xi), we have
\begin{align*}
& \neg ((a \land b) \lor (a \land c))  \s \neg (a \land (b \lor c)) = \\
& = (\neg (a \land b) \land \neg (a \land c))  \s  (\neg a \lor \neg (b \lor c)) = \\ 
& = ((\neg a \lor \neg b) \land  (\neg a \lor \neg c))  \s  (\neg a \lor (\neg b \land \neg c))  = \\ 
& = (\neg a \lor \neg b) \s  (((\neg a \lor \neg c))  \s  (\neg a \lor (\neg b \land \neg c))) = \\ 
& = (\neg a \lor \neg b) \s  (d \land e) = \\ 
& = ((\neg a \lor \neg b) \s  d) \land ((\neg a \lor \neg b) \s  e) = \\ 
& = (\neg a  \s  d) \land (\neg b  \s  d) \land (\neg a \s  e) \land (\neg b \s  e)  \geq_t \top. 
\end{align*}
\end{proof}

Proposition \ref{impinterla} allows us to establish some equivalences that give more insight into the structure of implicative bilattices. Recall that the relation $\sim_{1}$ is the one introduced in Definition~\ref{d:cong}, that $\FF$ denotes the operator of bifilter generation and that  $E(a)$ is an abbreviation for $a = a \s a$.

\begin{proposition} \label{p:some_equiv}
Let $\Al[B] = \left\langle B, \land, \lor, \otimes, \oplus, \supset, \neg \right\rangle$ be an implicative bilattice and $a, b \in B$. Then the following statements are equivalent:
\begin{enumerate}[(i)]
  \item \quad $a \sim_{1} b$
    \item  \quad $a \lor b = a \otimes b$
     \item  \quad $a \oplus b = a \land b$
     \item  \quad $\reg(a)  = \reg(b)$
     \item  \quad$\FF(a) = \FF(b)$
 \item  \quad $  \top  \leq_t a\supset b$ and $  \top  \leq_t b \supset a$
 \item  \quad $  E(a \s b)$ and $  E(b \s a)$
    \item \quad  $a \s c = b \s c$ for all $c \in B$. 
  \item  \quad $a \land \top  = b \land \top$
   \end{enumerate}
\end{proposition}

\begin{proof} 
The equivalence among (i), (ii) and (iii) has been proved in Proposition~\ref{prop:symetrization} (i). Moreover, (i) is also equivalent to (iv) by Proposition~\ref{prop:regular} (iv). Corollary~\ref{cor:bifilters} (ii) implies the equivalence of (i) and (v). Using Proposition~\ref{p:lem2} (v) and the interlacing conditions, it is obvious that (ii) and (vi) are equivalent; the equivalence between (vi) and (vii) follows from Proposition~\ref{p:biglem} (x). 

It is also easy to prove that (vii) and (viii) are equivalent. In fact, assuming (vii), we have, for all $c \in B$, 
\begin{align*}
a \s c 
& = \   (a \s b) \s (a \s c) & \textrm{by Proposition~\ref{p:biglem} (i)}  \\ 
& = \ a \s (b \s c) & \textrm{by Proposition~\ref{p:biglem} (xv)} \\ 
& = \ b \s (a \s c) & \textrm{by (IB2)}  \\ 
& = \ (b \s a) \s (b \s c) & \textrm{by Proposition~\ref{p:biglem} (xv)}  \\ 
& =  \ b \s c & \textrm{by Proposition~\ref{p:biglem} (i).}
\end{align*}
Conversely, assuming (viii) and using Proposition~\ref{p:biglem} (i) again, we have
$$
a \s b =  
b \s b = (a \s b) \s (a \s b).
$$
By symmetry, we also have $b \s a = (b \s a) \s (b \s a) $. Finally, the equivalence between (viii) and (ix) has been proved in Proposition~\ref{p:lem2} (xvii). 
\end{proof}
An interesting consequence of the previous proposition is the following:
\begin{cor} \label{c:cong1s}
In any implicative bilattice $\Al[B] = \left\langle B, \land, \lor, \otimes, \oplus, \supset, \neg \right\rangle$, the relation $\sim_{1}$ is a congruence of the reduct $ \left\langle B, \land, \lor, \otimes, \oplus, \supset \right\rangle$.
\end{cor}
\begin{proof}
We already know, by Proposition~\ref{prop:interlaced_pbl} (i), that $\sim_{1}$ is a congruence of $ \left\langle B, \land, \lor, \otimes, \oplus \right\rangle$. To prove that it is compatible with $\s$, assume $a_{1} \sim_{1} b_{1} $ and $a_{2} \sim_{1} b_{2}$ for some $a_{1}, a_{2}, b_{1}, b_{2} \in B$. This implies, by Proposition~\ref{p:some_equiv}, that $a_{1} \s  c= b_{1} \s c $  for all $c \in B$ and also  that $a_{2} \s b_{2} \geq_{t} \top $ and $b_{2} \s a_{2} \geq_{t} \top $. Then, using Proposition~\ref{p:biglem} (xv) and (ii), we have 
\begin{align*}
(a_{1} \s a_{2}) \s (b_{1} \s b_{2}) 
& = (a_{1} \s a_{2}) \s (a_{1} \s b_{2}) \\
& = a_{1} \s (a_{2} \s  b_{2})  \geq_{t} \top.
\end{align*}
By symmetry, we obtain $ (b_{1} \s b_{2}) \s (a_{1} \s a_{2})  \geq_{t} \top. $ Hence the desired results follows again by  by Proposition~\ref{p:some_equiv}.
\end{proof}

Let us also note that, if the bilattice reduct of $\Al[B]$ (which is interlaced) is a product bilattice $\Al[L] \odot \Al[L] $, then two elements $\la a_{1}, a_{2} \ra, \la b_{1}, b_{2} \ra \in L \times L$ satisfy any of the conditions of Proposition~\ref{p:some_equiv} if and only if $a_{1} = b_{1}$.

We will now turn to the study of the filters of the logic $\lbs$ on implicative bilattices. By definition, an  $\lbs$-filter on an implicative bilattice $\Al[B]$ is a set $F \subset B$ which contains the interpretation of all theorems of $\lbs$ for any homomorphism $h: \Al[Fm] \rightarrow \Al[B]$ and is closed under MP, i.e.\ such that $b \in F$ whenever $a, a \s b \in F$ for all $a, b \in B$. We shall see that, for the class of implicative bilattices, the $\lbs$-filters coincide with the deductive filters, which we define as follows:

\begin{definition} \label{filters}
{\rm Given an implicative bilattice  $\Al[B]$, a subset $F \subseteq B$ is a \emph{deductive filter} if and only if $\left\{a \in B : a \geq_t \top\right\} \subseteq F$ and, for all $a, b \in B$, if    $ a  \in F$ and $a \supset b \in F$, then $b \in F$. }
\end{definition}

To give a characterization of the $\lbs$-filters in purely algebraic terms, we shall need the following:  

\begin{lem} \label{l:ax}
For every axiom $\varphi$ of $H_{\s}$, the equation $\varphi \land \top \approx \top$ (sometimes abbreviated $\top \leq_t \phi$) is valid in the variety of implicative bilattices.
\end{lem}

\begin{proof}
For $(\s 1)$, this has been proved in Proposition \ref{p:lem2} (i). Also, Proposition \ref{p:biglem} (ix) and Proposition \ref{p:biglem}	(xv) prove the case of $(\s 2)$, while from (IB3) it is easily proved $(\s 3)$. 

To prove the case of $(\land \supset )$ and $(\otimes \supset )$, using  (IB2) and Proposition \ref{p:lem2} (i) we have
$$ (x \land y) \supset x \approx (x \otimes y) \supset x  \approx x \supset (y \supset x) \geq_t \top $$
and
$$ (y \land x) \supset y \approx (y \otimes x) \supset y  \approx y \supset (x \supset y) \geq_t \top. $$

$(\supset \land )$ and $(\supset \otimes )$: by (IB2)  and Proposition \ref{p:biglem} (ix) we have $x \supset (y \supset (x \land y)) \approx (x \land y) \supset (x \land y) \geq_t \top $ and $x \supset (y \supset (x \otimes y)) \approx (x \otimes y) \supset (x \otimes y) \geq_t \top $.

$(\supset \lor)$ and $(\supset \oplus)$: it is enough to note that, by Proposition \ref{p:lem2} (ii), if $x \leq_t y $ or $x \leq_k y $, then $x \supset y \geq_t \top$.

$(\lor \supset )$ and $(\oplus \supset )$: by (IB2), (IB4) and Proposition \ref{p:biglem} (ix) we have
\begin{align*}
(x \supset z) \supset ((y \supset z) \supset ((x \lor y) \supset z)) & \approx ((x \supset z) \land (y \supset z)) \supset ((x \lor y) \supset z)  \\
& \approx   ((x \lor y) \supset z) \supset ((x \lor y) \supset z) \\ & \geq_t \top
\end{align*}
and
\begin{align*}
(x \supset z) \supset ((y \supset z) \supset ((x \oplus y) \supset z)) & \approx ((x \supset z) \land (y \supset z)) \supset ((x \oplus y) \supset z) \\
& \approx   ((x \oplus y) \supset z) \supset ((x \oplus y) \supset z) \\ & \geq_t \top.
\end{align*}

$(\neg \land )$, $(\neg \lor )$, $(\neg \otimes )$, $(\neg \oplus )$ and $(\neg \neg)$ are easily proved using the identities that characterize negation within the variety of bilattices.

$(\neg \supset )$: by (IB6) and Proposition \ref{p:biglem} (ix) we have $\neg (x \supset y ) \supset (x \land \neg y) \approx (x \land \neg y) \supset (x \land \neg y) \geq_t \top$ and  $ (x \land \neg y) \supset \neg (x \supset y )  \approx  \neg (x \supset y )  \supset \neg (x \supset y )  \geq_t \top$.
\end{proof}

\begin{proposition} \label{filters2}
Given $\Al[B] \in \ib$ and $F \subseteq B$, the following conditions are equivalent:
\begin{enumerate}[(i)]
\item $F$ is a bifilter, i.e.\ $F$ is non--empty and the following condition holds: for all $a, b \in B$, $a \land b \in F$ iff $a \otimes b \in F$  iff $a \in F$ and $b \in F$.
\item $F$ is a  deductive filter.
\item $F$ is an $\lbs$-filter.
\item $F$ is a lattice filter of the truth ordering and $\top \in F$. 
\item $F$ is a lattice filter of the knowledge ordering and $\left\{a \in B : a \geq_t \top\right\} \subseteq F$.

\end{enumerate}
\end{proposition}

\begin{proof}
(i) $\Rightarrow$ (ii). Assume $F$ is a bifilter. Since $F$ is non--empty, $\top \in F$, which implies that $\left\{a \in B : a \geq_t \top\right\} \subseteq F$. To see that $F$ is closed under MP, assume $a \in F$ and $a \supset b \in F$, so that $(a \supset b) \land a \in F$. We have that $((a \supset b) \land a) \supset b \approx (a \supset b) \supset (a \supset b) \geq_t \top $. By Proposition \ref{p:lem2} (vi), it follows that  $(a \supset b) \land a \leq_t ((a \supset b) \land a) \otimes b$. So, $((a \supset b) \land a) \otimes b \in F$; hence $b \in F$. 

(ii) $\Rightarrow$ (iii). Assume $F$ is a deductive filter. Since $F$ is closed under MP by definition, we only have to check that, for every axiom $\varphi$ of $\h$ and every homomorphism $h: \Al[Fm] \rightarrow \Al[B]$, it holds that $h(\phi) \in F$. By Proposition \ref{l:ax}, for every axiom $\phi$, the equation $\varphi \land \top \approx \top$ is valid in $\Al[B]$, so
$h(\phi \wedge \top) = h(\top)$. Hence $\top = h(\top) \leq_t h(\phi)$ and, since   $\top \in F$, we conclude that $h(\phi) \in F$.

(iii) $\Rightarrow$ (i). Assume that $F$ is an $\lbs$-filter and $a, b \in F$. Since any interpretation of the axiom $(\supset \land )$ belongs to $F$, we have $a \supset (b \supset (a \land b)) \in F$, so by MP we obtain $a \land b \in F$. Similarly, using $(\supset \otimes )$, we obtain $a \otimes b \in F$. For the converse implication, assuming $a \land b \in F$ or $a \otimes b \in F$, we may use $(\land \supset  )$ and $(\otimes \supset  )$ and  MP to obtain the result. 

(i) $\Leftrightarrow$ (iv). It is clear that (i) implies (iv). To prove the converse, we have to show that if $F$ satisfies (iv), then $F$ is a lattice filter of the knowledge order. So let $a, b \in F$. By the interlacing conditions we have $a \land b \leq_t a \otimes b$, so $a \otimes b \in F$. Now let $c \in B$ such that $a \leq_k c $, so that $c = a \oplus c$. Since the bilattice reduct of $\Al[B]$ is distributive, we know (see for instance \cite{ArAv98}) that it satisfies the equation $x \oplus y \approx (x \land \top) \lor (y \land \top) \lor (x \land y)$. By hypothesis we have $ a \land \top \in F$, hence $ (a \land \top) \lor (c \land \top) \lor (a \land c) = a \oplus c \in F$.

(ii) $\Leftrightarrow$ (v). It is easy to show that (ii) $\Rightarrow$ (v), because (ii) implies (i). To prove the converse, assume that $F$ satisfies (v). We need only to check that $b \in F$ whenever $a, a \s b \in F$. Applying the hypothesis, we have $a \otimes (a \s b) \in F$, and now we may use Proposition \ref{p:lem2} (xiii) to obtain the result. 
\end{proof}

We are now able to determine the equivalent algebraic semantics of $\lbs$: 

\begin{thm} \label{t:alg3}
 $\lbs$ is algebraizable with respect to  the variety $\ib$ of implicative bilattices, with equivalence formula $\varphi \leftrightarrow \psi$ and defining equation $ \varphi \approx \varphi  \supset \varphi $.
\end{thm}

\begin{proof}
We will prove that $ \mathbf{Alg^{*}}\lbs = \ib $. By  \cite[Theorem 2.17]{BP89}, we know that the class $\mathbf{Alg^{*}}\lbs$ is axiomatized by the following equations and quasiequations:
\begin{enumerate}[(a)]
\item $\phi \approx \phi \s \phi $ for all axioms $\phi$ of $H_{\s}$
\item $x \approx x \s x \; \& \; x \s y \approx (x \s y) \s (x \s y) \; \Rightarrow \; y \approx y \s y$
\item $ x \da y \approx (x \da y) \s (x \da y) \; \Rightarrow \; x \approx y$.
\end{enumerate}

In order to prove that  $\ib \subseteq \mathbf{Alg^{*}}\lbs$, it is then sufficient to prove that any implicative bilattice satisfies (a) to (c). Note that by Proposition \ref{p:biglem} (x) we have that in any implicative bilattice $x \approx x \s x $ is equivalent to $\top \leq_t  x $. Hence we see that (a) has been proven in Lemma \ref{l:ax}. As to (b), it follows from Proposition \ref{p:lem2} (iii), while (c) follows from Proposition  \ref{p:lem2} (vi).

In order to prove that $ \mathbf{Alg^{*}}\lbs \subseteq  \ib$, we have to show that any $\Al[A] \in \mathbf{Alg^{*}}\lbs$ satisfies all equations defining the variety of implicative bilattices, i.e. all equations defining the variety of bilattices plus (IB1)-(IB6). To see this, using (a) and (c), it will be enough to prove that, for any equation $\phi \approx \psi$ axiomatizing the variety $\ib$, it holds that $\h \phi \da \psi$. And this has been shown in Proposition \ref{lemlat}. 
\end{proof}

By the previous theorem and Proposition~\ref{filters2}, we now have the following:

\begin{cor} \label{c:red_mod}
A matrix $\la \Al, F \ra$ is a reduced model of $\lbs$ if and only if $\Al \in \ib$ and $F = \{ a \in A: a = a \s a \} = \{ a \in A: \top \leq_{t} a  \} = \FF(\top)$.
\end{cor}

We end the chapter by proving that the logic $\lbs$, like its implicationless fragment $\lb$, has no consistent extensions. We need some preliminary results.


\begin{proposition} \label{p:foursub}
Let $\Al[B] \in \ib$ and let  $a \in B$ be such that $a >_{t} \top$. Then:
\begin{enumerate}[(i)]
\item $\neg a <_{t} a$
\item $ a = \neg a \s a $
\item $\neg a \s a =  \neg a \s \top  $
\item $ a = a \s a = \neg a \s \neg a $
\item $ a =  \neg a \s (a \otimes \neg a) $
\item $(a \otimes \neg a) \s b = \neg a \s b $ for all $b \in B$  
\item $(a \otimes \neg a ) \land \top = \neg a$  
\item $(a \otimes \neg a ) \lor \top = a$  
\item hence, the set $\{ a \otimes \neg a, \top, \neg a, a  \}$ is the universe of a subalgebra of $\Al[B]$ which is isomorphic to $\fours.$
\end{enumerate}
\end{proposition}

\begin{proof}
(i). Almost immediate, for the assumption implies $\neg a \leq_{t} \neg \top = \top <_{t} a$. 

(ii). Note that, by  Proposition~\ref{p:biglem} (ii), we have  $ a \leq_{t} \neg a \s a $. In order to prove the other inequality, we show that $E((\neg a \s a) \s a)$ and $E (\neg a \s \neg ( \neg a \s a ))$, so that the result will follow by Proposition~\ref{p:biglem} (vii). The first one is immediate; as to the second, using (IB6), we have
\begin{align*}
  \neg a \s \neg ( \neg a \s a )  & = (\neg a \land \neg a) \s \neg ( \neg a \s a )  \\
& = \neg (\neg a \s \neg \neg a) \s \neg ( \neg a \s a )  \\
& = \neg (\neg a \s  a) \s \neg ( \neg a \s a ).
\end{align*}
So the result easily follows.

(iii). From the assumptions and Proposition~\ref{p:lem2} (x), it follows that $\neg a \s a \geq_{t}  \neg a \s \top $, so we just need to prove the other inequality.  As in the proof of the previous item, we will show that $E((\neg a \s a) \s  (\neg a \s \top))  $ and $E(\neg (\neg a \s \top) \s  \neg (\neg a \s a))  $. The first one is almost immediate. The second, using (ii), is equivalent to $E(\neg (\neg a \s \top) \s  \neg a)  $. Then, using (IB6), we have 
\begin{align*}
  \neg (\neg a \s \top) \s  \neg a  & = (\neg a \land \neg \top) \s  \neg a  \\
& = (\neg a \land \top) \s  \neg a  \\
& = \top \s (\neg a \s  \neg a).
\end{align*}
So the result easily follows.

(iv). By Proposition~\ref{p:biglem} (xvi).

(v). Using Proposition~\ref{p:lem2} (xv) together with the previous items (ii) and (iv), we have 
$$
\neg a \s (a \otimes \neg a)  = (\neg a \s a) \otimes (\neg a \s \neg a)  = a \otimes a = a.
$$

(vi). The assumptions imply that, for all $b \in B$, 
$$(a \otimes \neg a ) \s b = a \s (\neg a \s b) = \neg a \s b.$$

(vii). Applying distributivity, we have 
$$(a \otimes \neg a ) \land \top = (a \land \top) \otimes (\neg a \land \top) = \top \otimes \neg a = \neg a.$$

(viii).  Applying distributivity, we have
$$(a \otimes \neg a ) \lor \top = (a \lor \top) \otimes (\neg a \lor \top) = a \otimes \top = a. $$  

(ix). Using the previous items, it is easy to check that the isomorphism is given by the map $h: \fours \ta \Al[B]$ defined as follows: $h(\bot) = a \otimes \neg a $, $h(\top) = \top$, $h(\false) = \neg a$ and $h(\true)= a $.
\end{proof}


It is now easy to prove the following:

\begin{lem} \label{l:submat_impl}
Let $\la \Al[B], F \ra$ be a reduced model of $\lbs$. Then the logic defined by $\la \Al[B], F \ra$ is weaker than $\lbs$.
\end{lem}

\begin{proof}
Reasoning by contraposition, we will prove that $\Gamma \nvDash_{\lbs} \phi$ implies $\Gamma \nvDash_{\la \Al[B], F \ra} \phi$ for all $\Gamma \cup \{ \phi \} \subseteq Fm$. In order to do this, it will be enough to show that $\la \fours, \Tr \ra$ is a submatrix of any matrix of the form $\la \Al[B], F \ra$. From Corollary~\ref{c:red_mod} it follows that $\Al[B]$ is an implicative bilattice and $F$ is the least bifilter of $\Al[B]$, i.e.\ the bifilter generated by $\top$. Given any element $a \in F$ such that $a \neq \top$, we have that the set $\{ a \otimes \neg a, \top, \neg a, a  \}$ is the universe of a subalgebra of $\Al[B]$ which is isomorphic to $\fours$ through the map $h$ defined as in Proposition~\ref{p:foursub} (ix). Note also that $\Tr = h^{-1}[F]$. So if $g: \Al[Fm] \ta \four$ is a homomorphism such that $g[\Gamma] \subseteq \Tr $ but $g (\phi) \notin \Tr$, then also $g [h[\Gamma]] \subseteq F $ but $g (h(\phi)) \notin F $. Recalling that $\lbs$ is the logic defined by the matrix $\la \four, \Tr \ra$ (Theorem~\ref{t:compl}), we may then conclude that  $\Gamma \nvDash_{\lbs} \phi$ implies $\Gamma \nvDash_{\la \Al[B], F \ra} \phi$.
%
%
\end{proof}

In Section~\ref{sec:tar} we defined a logic  $\mathcal{L} = \la \Al[Fm], \vdash_{\mathcal{L}} \ra$ to be consistent if there are $ \phi, \psi \in Fm$ such that $\phi \nvdash_{\mathcal{L}} \psi$. In this case, since any extension of $\lbs$ will have theorems, it would be sufficient to require a weaker condition, i.e.\ that there be $ \phi \in Fm$ such that $\nvdash_{\mathcal{L}} \phi$. By the previous lemma we may then obtain the following: 

\begin{proposition} \label{p:extens_impl}
If a logic $\mathcal{L} = \la \Al[Fm], \vdash_{\mathcal{L}} \ra$ is a consistent extension of $\lbs$, then ${\vdash_{\mathcal{L}}} = {\vDash_{\lbs}} $. 
\end{proposition}

\begin{proof}
By \cite[Proposition 2.27]{FJa09}, we know that any reduced matrix for $\mathcal{L}$ is of the form $\la \Al[B], F \ra$, where $\Al[B]$ is an implicative bilattice and $F$ is the least bifilter of \Al[B]. By the assumption of consistency, we may assume that there is at least one reduced matrix for $\mathcal{L}$ such that  $F$ is proper. By Lemma \ref{l:submat_impl}, we know that the logic defined by such a matrix is weaker than $\lbs$; this implies that the class of all reduced matrices for  $\mathcal{L}$ defines a weaker logic than $\lbs$. Since  any logic is complete with respect to the class of its reduced matrices (see \cite{W88}), we may conclude that  $\mathcal{L}$ itself is weaker than $\lbs$, so they must be equal.
\end{proof}

\chapter{Implicative bilattices}
\label{ch:imp}

\section{Representation Theorem and congruences} 
\label{sec:repimp}

In this chapter we will study the variety $\ib$  in more depth. We will obtain a representation theorem for implicative   bilattices analogous to the ones we have for interlaced pre-bilattices and   bilattices; we will turn to the study of the lattices that arise as factors from the decomposition given by our representation, comparing them to the class of lattices that arose from the study of the reduced models of the logic $\lb$. Using these results, we will prove that the variety  $\ib$ is generated by its four-element member.  Finally, we shall consider and characterize some subreducts of implicative   bilattices which seem to have a special logical significance.

We begin by showing that any implicative   bilattice is isomorphic to a special kind of   product  whose factors are upper-bounded relatively complemented distributive lattices. Let us recall that a lattice $\Al[L] = \la L, \sqcap, \sqcup \ra$ with maximum 1 is relatively complemented if any element has a complement in any interval in $L$ or, equivalently, if for any $a, b \in L$ such that $a \leq b$, there is $c \geq a$ such that $b \sqcap c = a $ and $b \sqcup c = 1$. In this case $c$ is said to be the relative complement of $b$ in the interval $[a, 1]$, and it is unique if the lattice is distributive.

The class of relatively complemented distributive lattices with maximum has already been considered in the literature as an algebraic counterpart of the $\{ \land, \lor, \rightarrow \}$-fragment of classical propositional logic. In \cite{Cu77} this fragment is called ``classical positive propositional algebra'', and the corresponding algebras  ``classical implicative lattices''. Here we will use the same terminology to denote this class of lattices. However, other names are available: in the context of universal algebra, relatively complemented distributive lattices with maximum are sometimes called ``generalized Boolean algebras'' (see for instance \cite{AbVaTo04}), while in other studies this name is used for relatively complemented distributive lattices having a minimum element. 

Our next aim is to verify that the class of classical implicative lattices can be axiomatized by means of equations only. We shall need some lemmas.

Given a  classical implicative lattice $\Al[L] = \langle L,  \sqcap, \sqcup, 1  \rangle$ and $a, b \in L$, we will denote by $a  \backslash  b $ the relative complement of $a$ in the interval $ \left[ a \sqcap b, 1 \right]  $, i.e.\ the unique element satisfying both $a \sqcap (a  \backslash  b ) = a \sqcap b $ and $a \sqcup (a  \backslash  b ) = 1 $. We will write $\Al[L] = \langle L,  \sqcap, \sqcup, \ba, 1  \rangle$ to emphasize the fact that we are considering these lattices as algebras in the extended similarity type.

\begin{proposition} \label{relcom} 
Let $\Al[L] = \langle L,  \sqcap, \sqcup, \ba, 1  \rangle$  be a classical implicative lattice. Then, for all $a, b, c \in L$:

\begin{enumerate}[(i)]
\item $a \sqcap b \leq c \;$ 	if and only if $\; a \leq b  \backslash  c $	
\item $1  \backslash  a = a $																					
\item $a  \backslash  (b \sqcap c) = (a  \backslash  b) \sqcap (a  \backslash  c)$						
\item $a \leq (b  \backslash  a)$																			
\item $(a \sqcup b)  \backslash  c = (a  \backslash  c) \sqcap (b  \backslash  c)$						
\item $(a  \backslash  b)  \backslash  a = a$																			
\item $a  \backslash  (b  \backslash  c) = (a  \backslash  b)  \backslash  (a  \backslash  c) = (a \sqcap b)  \backslash c $	
\item $a \sqcup (b \backslash c) = (a \sqcup b)  \backslash  (a \sqcup c)$  			
\item $(a  \backslash b ) \sqcup  (b  \backslash c ) = 1$													
\item $(a  \backslash  b)  \backslash  b = a \sqcup b$.														
\end{enumerate}
\end{proposition}

\begin{proof}

(i). Suppose $a \sqcap b \leq c $. Then $a \sqcap b \leq b \sqcap c $, so $b \sqcap (a \sqcup (b \backslash c)) = (b \sqcap a) \sqcup (b \sqcap (b \backslash c)) =  (b \sqcap a) \sqcup  (b \sqcap c) = b \sqcap c$ and $b \sqcup a \sqcup (b \backslash c) = a \sqcup 1 = 1$. Therefore $a \sqcup (b \backslash c) = (b \backslash c)$. Conversely, suppose $a \leq b  \backslash  c $. Then  $a \sqcap b \leq (b  \backslash  c) \sqcap b = b \sqcap c \leq c$.

From (i) it follows that the algebra $\langle L, \sqcap, \sqcup,  \backslash , 1  \rangle$ is a relatively pseudo-com\-plemented lattice, which implies (see \cite{Ra74}) that conditions (ii) to (vii) are satisfied.

(viii). We have $(a \sqcup b) \sqcap (a \sqcup (b \backslash c)) = a \sqcup (b \sqcap (b \backslash c)) = a \sqcup (b \sqcap c) = (a \sqcup b) \sqcap (a \sqcup c) $ and $a \sqcup b \sqcup a \sqcup (b \backslash c) = a \sqcup 1 = 1$. 

(ix). Using (viii), we have  $(a  \backslash b ) \sqcup  (b  \backslash c ) = (a \sqcup  (b  \backslash c ))  \backslash  (b \sqcup  (b  \backslash c )) = (a \sqcup  (b  \backslash c ))  \backslash  1 = 1$.	

(x). Using (iv), we have $(a  \backslash  b) \sqcap (a \sqcup b) = ((a  \backslash  b) \sqcap a) \sqcup ((a  \backslash  b) \sqcap b) = (a \sqcap b) \sqcup b = b$	and $ (a  \backslash  b) \sqcup a \sqcup b = 1 \sqcup b = 1$.												
\end{proof}


Recall that a lattice  $\Al[L] = \langle L, \sqcap, \sqcup, \backslash \rangle$ with a binary operation $\backslash$ is said to be \emph{relatively pseudo-complemented} (see \cite[p. 52]{Ra74}) when the following residuation condition is satisfied:
\begin{enumerate}[] 
\item (R) $\quad a \sqcap b \leq c$ if and only if $ b \leq a \backslash c$,  for all $a, b, c \in L $.
\end{enumerate} 

We may then characterize classical implicative lattices as follows:

\begin{proposition} \label{relpseu}
Let $\Al[L] = \langle L, \sqcap, \sqcup, \backslash \rangle$ be a relatively pseudo-complemented lattice satisfying the   equation $x \sqcup (x \backslash y) \approx x \backslash x$.
Then $\langle L, \sqcap, \sqcup \rangle$ is a classical implicative lattice. 
\end{proposition} 

\begin{proof}
It is known that condition (R)  implies that $\langle L, \sqcap, \sqcup \rangle$ is distributive, has a top element $1 = a \backslash a$ for all $a \in L$, and that $a \sqcap b = a \sqcap (a \backslash b)$ for all $a, b \in L$. Now, given an interval $[b,1]$, the satisfaction of  the equation $x \sqcup (x \backslash y) \approx x \backslash x$ guarantees that, for all $a \in [b,1]$, the element $a \backslash b$ is the relative complement of $a$ in $[b,1]$. This in turn implies that for an arbitrary interval $[b, c]$, any $a \in [b, c]$ has a complement in $[b, c]$, namely $(a \backslash b) \sqcap c $.
\end{proof}

As we have seen (Proposition \ref{relcom}), the converse implication is also true. That is, given a classical implicative lattice  $\langle L, \sqcap, \sqcup, 1 \rangle$, we can define an operation $\backslash : L^2 \longrightarrow L$ satisfying condition (R) above.

Since relatively pseudo-complemented lattices form a variety, it follows that the class of classical implicative lattices is also a variety, axiomatized by the identities for  relatively pseudo-complemented lattices plus  $x \sqcup (x \backslash y) \approx x \backslash x$.

It follows from the results of \cite[Chapter X, p. 236]{Ra74} that the variety of relatively pseudo-complemented lattices is the equivalent algebraic semantics of positive logic, the $\{ \land, \lor, \rightarrow \}$-fragment of intuitionistic logic. Therefore, the variety of classical implicative lattices is the equivalent algebraic semantics of the axiomatic extension of positive logic obtained by adding the axiom $ p \lor (p \rightarrow q)$ (the ``classical positive propositional algebra'' of \cite{Cu77}).


Let us now turn to the study of the relation between classical implicative lattices and implicative   bilattices. We start with the following result:

\begin{proposition} \label{reprimp} 
For any implicative   bilattice  $\Al[B] = \langle B, \land, \lor, \otimes, \oplus, \supset, \neg  \rangle $, the   bilattice reduct  $\langle B, \land, \lor, \otimes, \oplus, \neg \rangle$ is isomorphic to the     bilattice $\Al[L] \odot \Al[L]$ for some classical implicative lattice $\Al[L]$ (in particular, if $\Al[B]$ is bounded, then $\Al[L]$ is a Boolean lattice).
\end{proposition}

\begin{proof}

It is  known (see Theorem \ref{thm:reprbil}) 
that the bilattice reduct of $\Al[B]$ is isomorphic to the  product bilattice $\Al[B^-] \odot \Al[B^-]$, where 
$$\Al[B^-] = \la \{ a \in B : a \leq_t \top \}, \land, \lor  \ra = \la \{ a \in B : a \leq_t \top \}, \otimes, \oplus  \ra .$$ 
Moreover, we  know that $\Al[B^-]$ is distributive and has a maximum element $\top$.
Hence, to complete the proof it will be sufficient to show that $\Al[B^-]$ is relatively complemented.  Let then $a, b \in B$ be such that $ a \leq_t b \leq_t \top$. We will prove that the relative complement of $b$ in $[a, \top]$ is $ (b \supset a) \land \top$. This follows from the fact that, by Proposition \ref{p:lem2} (xxi), we have $b \land (b \supset a) \land \top = b \land a \land \top = a$, and that, using Proposition  \ref{p:lem2} (xii), we have $b \lor ((b \supset a) \land \top) = (b \lor (b \supset a)) \land (b \lor \top) = (b \lor (b \supset a)) \land \top = \top$.
\end{proof}

It is now possible to prove a result which may be regarded as a kind of converse of Proposition \ref{reprimp}.

Let $\Al[L] = \langle L, \sqcap, \sqcup,  1  \rangle$  a classical implicative lattice and let $\Al[L] \odot \Al[L] = \langle L \times L, \land, \lor, \otimes, \oplus, \neg  \rangle $ be the  product bilattice defined as usual. We define the operation $\supset: L \times L \longrightarrow L \times L$ as follows: for all $a_1, a_2, b_1, b_2 \in L, $
$$ \langle a_1, a_2 \rangle \supset  \langle b_1, b_2 \rangle  = \langle a_1  \backslash   b_1, a_1 \sqcap b_2 \rangle. $$ 
We have the following:

\begin{proposition} \label{defimp}
The structure $\langle \Al[L] \odot \Al[L], \supset \rangle $ is an implicative   bilattice.
\end{proposition}

\begin{proof}
Using the properties stated in Proposition \ref{relcom}, we will show that $\langle \Al[L] \odot \Al[L], \supset \rangle $ satisfies equations (IB1) to (IB6) of Definition \ref{bilimpl}. Let $a_{1}, a_{2}, b_{1}, b_{2}, c_{1}, c_{2} \in L$. Then:

(IB1). $\langle 1, 1 \rangle \supset \langle a_1, a_2 \rangle = \langle 1  \backslash  a_1, 1 \sqcap a_2 \rangle  = \langle a_1, a_2 \rangle $.

(IB2). We have 
\begin{align*}
\langle a_1, a_2 \rangle \supset ( \langle b_1, b_2 \rangle \supset \langle c_1, c_2 \rangle) & = \langle a_1  \backslash  (b_1  \backslash  c_1), a_1 \sqcap b_1 \sqcap c_2 \rangle \\ 
& = \langle (a_1 \sqcap b_1)  \backslash  c_1, a_1 \sqcap b_1 \sqcap c_2 \rangle \\ 
& = (\langle a_1, a_2 \rangle \land \langle b_1, b_2 \rangle) \supset \langle c_1, c_2 \rangle \\ 
& = (\langle a_1, a_2 \rangle \otimes \langle b_1, b_2 \rangle) \supset \langle c_1, c_2 \rangle.
\end{align*} 

(IB3). Recall that $(a  \backslash  b)  \backslash  a = a$  by Proposition \ref{relcom} (vi). Now we have
\begin{align*}
& ((\langle a_1, a_2 \rangle \s \langle b_1, b_2 \rangle) \s \langle a_1, a_2 \rangle) \s \langle a_1, a_2 \rangle = \\ & = 
((\langle a_1 \ba b_1, a_1 \sqcap b_2 \rangle) \s \langle a_1, a_2 \rangle) \s \langle a_1, a_2 \rangle = \\ 
& =
( \langle (a_1 \ba b_1) \ba a_1, (a_1 \ba b_1)  \sqcap a_2 \rangle) \s \langle a_1, a_2 \rangle = \\ 
& =
( \langle a_1, (a_1 \ba b_1)  \sqcap a_2 \rangle) \s \langle a_1, a_2 \rangle = \\ 
& =
 \langle a_1 \ba a_1, a_1 \sqcap a_2 \rangle =  \\ 
& =
\langle a_1 , a_2 \rangle \s  \langle a_1 , a_2 \rangle. 
\end{align*} 


(IB4). We have 
\begin{align*} 
(\langle a_1, a_2 \rangle \lor \langle b_1, b_2 \rangle) \supset \langle c_1, c_2 \rangle &  = \langle (a_1 \sqcup  b_1)  \backslash  c_1,  (a_1 \sqcup  b_1) \sqcap c_2  \rangle \\ 
& = (\langle a_1, a_2 \rangle \oplus \langle b_1, b_2 \rangle) \supset \langle c_1, c_2 \rangle \\ 
& = \langle (a_1  \backslash  c_1) \sqcap (b_1  \backslash  c_1), (a_1 \sqcup  b_1) \sqcap c_2 \rangle \\  
& = (\langle a_1, a_2 \rangle \supset \langle c_1, c_2 \rangle) \land (\langle b_1, b_2 \rangle \supset \langle c_1, c_2 \rangle).
\end{align*} 

(IB5). We have 
\begin{align*}
 & (\langle a_1, a_2 \rangle \supset \langle b_1, b_2 \rangle)  \supset (\langle a_1, a_2 \rangle \otimes \langle b_1, b_2 \rangle) = \\ & = \langle (a_1  \backslash  b_1)  \backslash  (a_1 \sqcap b_1), (a_1  \backslash  b_1) \sqcap a_2 \sqcap b_2 \rangle = \\ 
 & =    \langle ((a_1  \backslash  b_1)  \backslash  a_1) \sqcap ((a_1  \backslash  b_1)  \backslash  b_1), (a_1  \backslash  b_1) \sqcap a_2 \sqcap b_2 \rangle = \\
  & = \langle a_1 \sqcap (a_1 \sqcup b_1), (a_1  \backslash  b_1) \sqcap a_2 \sqcap b_2 \rangle = \\ 
  & = \langle a_1 , (a_1  \backslash  b_1) \sqcap a_2 \sqcap b_2 \rangle \geq_t \langle a_1,  a_2 \rangle.
 \end{align*} 

(IB6). We have that 
\begin{align*}
\neg (\langle a_1, a_2 \rangle \supset \langle b_1, b_2 \rangle) \supset \langle c_1, c_2 \rangle & = \langle a_1 \sqcap b_2 , a_1  \backslash   b_1 \rangle \supset \langle c_1, c_2 \rangle \\ 
& =  \langle (a_1 \sqcap b_2) \backslash c_1,  a_1 \sqcap b_2 \sqcap c_2   \rangle \\ 
& = (\langle a_1, a_2 \rangle \land \neg \langle b_1, b_2 \rangle) \supset \langle c_1, c_2 \rangle.
\end{align*} 
\end{proof}

Note that if $\Al[L] = \langle L, \sqcap, \sqcup  \rangle$ is a Boolean lattice, then the operation $ \backslash $ coincides with the Boolean implication, i.e.\ we have $a  \backslash  b = a' \sqcup b$ for all $a, b \in L$, where $a'$ denotes the complement of $a$.

Combining the results of Proposition  \ref{reprimp}  and Proposition \ref{defimp}, we obtain the following: 

\begin{thm}[Representation, 4] \label{t:reprimp2} 
An algebra $\Al[B] = \langle B, \land, \lor, \otimes, \oplus, \supset, \neg  \rangle$ is an implicative   bilattice if and only if  $\langle B, \land, \lor, \otimes, \oplus, \supset, \neg \rangle$ is isomorphic to the   product   bilattice $\Al[L] \odot \Al[L]$ for some classical implicative lattice $\Al[L] = \la L, \sqcap, \sqcup, \ba, 1 \ra $ endowed with the operation  $\s$  defined by 
$ \langle a_1, a_2 \rangle \supset  \langle b_1, b_2 \rangle  = \langle a_1  \backslash   b_1, a_1 \sqcap b_2 \rangle$  for all $a_1, a_2, b_1, b_2 \in L.$
\end{thm}

\begin{proof} The implication from right to left follows from Proposition \ref{defimp}. To prove the other implication, let $\Al[B] = \langle B, \land, \lor, \otimes, \oplus, \supset, \neg  \rangle$ be an implicative   bilattice. By the proof of Proposition \ref{reprimp}, the   bilattice reduct $\langle B, \land, \lor, \otimes, \oplus,  \neg \rangle$ is isomorphic to the 
  product   bilattice $\Al[L] \odot \Al[L]$ where $\Al[L]$ is the classical implicative lattice $\la B^{-}, \land, \lor, \top \ra$, with $$B^{-} = \{a\in B: a \leq_t \top\}. $$ The isomorphism is the map
$h: \la B, \land, \lor, \otimes, \oplus, \neg \ra \rightarrow \Al[L] \odot \Al[L]$ given by   
 $$h (a) = \la a \wedge \top , \neg a \wedge \top  \ra.$$
Next we show that $h$ is also a homomorphism w.r.t.\  the operation $\supset$ defined in  $\Al[L] \odot \Al[L]$  as in the statement of the proposition.  We have to check that, for all $a, b \in B$, $$h(a \supset b) = h(a) \supset h(b)$$
i.e.\ that
$$\la (a \s b) \wedge \top, \neg (a \s b) \wedge \top \ra = \la a \wedge \top, \neg a \wedge \top \ra  \supset \la b \wedge \top, \neg b \wedge \top\ra.$$
Since $\la a \wedge \top, \neg a \wedge \top \ra  \supset \la b \wedge \top, \neg b \wedge \top\ra = \la (a \wedge \top) \backslash (b \wedge \top), a \wedge \neg b \wedge \top \ra$, we need  to prove that 
 $ (a \s b) \wedge \top = (a \wedge \top) \backslash (b \wedge \top)$ and $ \neg (a \s b) \wedge \top = a \wedge \neg b \wedge \top$. Let us first show that the relative complement of 
 $a \land \top$ in the interval $[a \land b \land \top, \top]$ is $(a \s b) \wedge \top$. Indeed, by Proposition \ref{p:lem2} (xxi), we have  $(a \wedge \top) \land (a \supset b) \land \top = a \land (a \supset b) \land \top = a \land b \land \top$. And by Proposition \ref{p:lem2} (xii), $(a \land \top) \lor ((a \supset b) \land \top) = \top \land (a   \lor (a \supset b)) = \top$. Now, by (IB6) and \ref{p:lem2} (xx) we have $\neg (a \s b) \land \top  = 
a \land  \neg b \land \top  =
(a \land \top) \land (\neg b \land \top).$
\end{proof}

As in the case of bilattices, an alternative proof of the Representation Theorem can be obtained without using any constant by considering the regular elements (i.e.\ the fixed points of the negation operator) of the implicative bilattice. Let us see how.

Given an implicative bilattice $\Al[B] = \langle B, \land, \lor, \otimes, \oplus, \supset, \neg  \rangle$, we consider the algebra $ \langle \Reg(\Al[B]), \otimes, \oplus, \ba \rangle$, where $ \langle \Reg(\Al[B]), \otimes, \oplus \rangle$ is the sublattice of the k-lattice of $\Al[B]$ whose universe is the set of regular elements and the operation $\ba : \Reg(\Al[B]) \times \Reg(\Al[B])  \longrightarrow \Reg(\Al[B]) $ is defined, for all $a, b \in \Reg(\Al[B])$, as $a \ba b = \reg (a \s b)$. 

Note that, by Proposition~\ref{prop:regular} (ii), we have $a \sim_{1} \reg(a)$ for all $a \in B$ and, by Corollary~\ref{c:cong1s}, the relation $\sim_{1}$  is compatible with $\s$. By Proposition~\ref{p:some_equiv}, for all $a, b \in B$, we have that $a \sim_{1} b$ if and only if $\reg(a) = \reg(b)$. It is then clear that 
$$
a \ba b = \reg (a \s b) = \reg (\reg(a) \s \reg(b)) = \reg (\reg(a) \s b) = \reg (a\s \reg(b)).
$$
We shall sometimes use this fact without notice. We have then the following:
 
 \begin{thm} \label{t:reprimp_reg} 
Let $\Al[B] = \langle B, \land, \lor, \otimes, \oplus, \supset, \neg  \rangle$ be an implicative   bilattice. Then:
\begin{enumerate}[(i)]
  \item $ \langle \Reg(\Al[B]), \otimes, \oplus, \ba \rangle$ is a classical implicative lattice,
  \item $\Al[B]$ is isomorphic to the implicative bilattice $$ \langle \Reg(\Al[B]), \otimes, \oplus, \ba \rangle \odot \langle \Reg(\Al[B]), \otimes, \oplus, \ba \rangle.$$
\end{enumerate}
\end{thm}

\begin{proof}
(i). Since $\top$ is the maximum of the lattice $ \langle \Reg(\Al[B]), \otimes, \oplus \rangle$, we have to show that, for any $a, b \in \Reg(\Al[B])$, we have $a \otimes (a \ba b) = a \otimes b $ and $a \oplus (a \ba b) = \top$, i.e.\ that $a \otimes \reg (a \s b) = a \otimes b $ and $a \oplus \reg (a \s b) = \top$.

 As to the first, note that, by Proposition~\ref{prop:regular} (v), we have 
 $$a \otimes \reg (a \s b) = \reg(a) \otimes \reg (a \s b) =  \reg (a \otimes (a \s b)) =  \reg (a \land (a \s b))$$
and 
$$ a \otimes b = \reg(a) \otimes \reg(b) = \reg (a \otimes b) =  \reg (a \land b).$$
By Proposition~\ref{p:biglem} (xiv), we have $a \land (a \s b) \land \top = a \land b \land \top$. And this, by Proposition~\ref{p:some_equiv}, implies the desired result.

As to the second, reasoning as before, we have 
$$a \oplus \reg (a \s b) = \reg (a) \oplus \reg (a \s b) = \reg (a \oplus (a \s b))=  \reg (a \lor(a \s b)).$$
It will then be sufficient to check that $(a \lor(a \s b)) \land \top = \top $, and this has been proved in Proposition~\ref{p:lem2} (ix).  

(ii). Let us denote by $\s^*$ the implication defined in $ \langle \Reg(\Al[B]), \otimes, \oplus, \ba \rangle \odot \langle \Reg(\Al[B]), \otimes, \oplus, \ba \rangle$ as before, that is, for all  all $a_1, a_2, b_1, b_2 \in \Reg(\Al[B]), $
$$ \langle a_1, a_2 \rangle \supset^*  \langle b_1, b_2 \rangle  = \langle \reg  (a_1  \s  b_1), a_1 \otimes b_2 \rangle. $$ 
We shall prove that the isomorphism is given by the  same map we considered for bilattices, i.e.\ $f: B  \longrightarrow \Reg(\Al[B]) \times \Reg(\Al[B])$ defined, for all $a \in B$, as $$f(a) = \la \reg(a), \reg(\neg a) \ra.$$ 
We know that $f$ is a bijection and an isomorphism between the two bilattice reducts, so we just need to check that, for all $a, b \in B$,
\begin{align*}
f(a \s b)
& = \la \reg(a \s b), \reg(\neg (a \s b)) \ra   \\
&  = \la \reg(a), \reg(\neg a) \ra \s^* \la \reg(b), \reg(\neg b) \ra \\
&  = \la  \reg (\reg(a) \s \reg(b)), \reg(a)  \otimes \reg(\neg b) \ra \\
&  = f(a) \s^* f(b).
\end{align*}
This amounts to proving that $$\reg(a \s b) =  \reg (\reg(a) \s \reg(b))$$ 
and  
$$\reg(\neg (a \s b)) = \reg(a)  \otimes \reg (\neg b) = \reg (a \otimes \neg b) = \reg (a \land \neg b).$$ 
The first one is immediate.
As to the second, using Proposition~\ref{p:some_equiv}, we may prove that $\neg (a \s b) \s c = (a \land \neg b) \s c$ for all $c \in B$, which follows immediately from (IB6).

\end{proof}

The following result shows that, as in the case of interlaced   bilattices, there is a correspondence between the congruences of an implicative   bilattice and the congruences  of its associated lattice factor.

\begin{proposition} \label{p:impbl_cong} 
Let  $\Al[B] = \langle B, \land, \lor, \otimes, \oplus, \s, \neg \rangle$ be an implicative bilattice. Then:
\begin{enumerate}[(i)]
 \item for all $\theta \in \Con (\Al[B] )$ and for all $a, b \in B$, it holds that $\la a, b \ra \in \theta $ if and only if $\la \reg(a), \reg(b) \ra \in \theta $ and  $\la \reg(\neg a), \reg(\neg b) \ra \in \theta $, 
 \item   $\la \Con (\Al[B] ), \subseteq \ra \cong \la \Con (\la \Reg(\Al[B]), \otimes, \oplus, \ba \rangle), \subseteq \ra$.
\end{enumerate}
\end{proposition}

\begin{proof}
(i). By Proposition~\ref{p:blconf_cong} (i).

(ii). We shall follow the proof of Proposition~\ref{p:blconf_cong} (ii), showing that the isomorphism is given by the map
%
%
$$h: \Con (\Al[B] )  \longrightarrow \Con (\la \Reg(\Al[B]), \otimes, \oplus, \ba \rangle)$$
defined, for all $\theta \in  \Con (\Al[B])$, as $$h(\theta) = \theta \cap \Reg(\Al[B]) \times \Reg(\Al[B]).$$ 
From the proof of Proposition~\ref{p:blconf_cong} (ii) it follows that $h$ is well-defined and that it is an order embedding. Its inverse is $$h^{-1}: \Con (\la \Reg(\Al[B]), \otimes, \oplus, \ba \rangle)  \longrightarrow \Con (\Al[B] ) $$
defined, for all $\theta \in  \Con (\la \Reg(\Al[B]), \otimes, \oplus, \co \rangle)$, as follows: 
$$ \la a, b \ra \in h^{-1}(\theta) \textrm{ iff }  \la \reg(a), \reg(b) \ra \in \theta \textrm{ and  } \la \reg(\neg a), \reg(\neg b) \ra \in \theta.$$
We proved that $h^{-1}(\theta)$ is an equivalence relation compatible with all the lattice operations of both orders as well as with negation. As to implication, assume $\la a, b \ra, \la c, d \ra \in  h^{-1}(\theta)$, that is, $\la \reg(a), \reg(b) \ra,$ $\la \reg(\neg a), \reg(\neg b) \ra, \la \reg(c), \reg(d) \ra,$ $\la \reg(\neg c), \reg(\neg d) \ra \in \theta$. By the assumptions we have that 
$$\la \reg (\reg(a) \s \reg(c)), \reg (\reg(b) \s \reg(d)) \ra \in \theta. $$ 
From the proof of Theorem  \ref{t:reprimp_reg} (ii)  it follows that $\reg (a \s b) = \reg (\reg(a) \s \reg(b))$ and $\reg(\neg (a \s b))  = \reg (a \otimes \neg b)$ for all $a, b \in B$. From this we easily obtain  $\la \reg (a \s c), \reg(b \s d) \ra, \la \reg (\neg (a \s c)), \reg( \neg  (b \s d))   \ra  \in \theta$, and this completes the proof.
%
\end{proof}

From the previous proposition and Theorem \ref{t:reprimp2}, we immediately obtain the following:

\begin{cor} \label{c:congimp}
Let $ \Al[B] = \la B, \land, \lor, \otimes, \oplus, \supset, \neg, \top \ra $ be an implicative  bilattice. Then $\Con(\Al[B]) \cong \Con(\Al[B^{-}])$, where $\Al[B^{-}] = \la \{a\in B: a \leq_t \top\}, \land, \lor, \ba, \top \ra$  and the operation $\ba $ is defined as $a \ba b = (a \s (a \land b)) \land \top$. 
\end{cor}

The previous results suggest that the study of congruences of classical implicative lattices may give insight into  the congruences of implicative   bilattices. We now turn to this study, that  will  eventually  enable us to characterize the variety $\ib$ as generated by its four-element member.

The key result is the following:

\begin{proposition} \label{compat}
Let $\: \Al[L] = \langle L, \sqcap, \sqcup, \ba, 1 \rangle$ be a classical implicative lattice. Then $\Con(\Al[L]) = \Con(\langle L, \sqcap, \sqcup \rangle) $.
\end{proposition}

\begin{proof}
Obviously $\Con(\Al[L]) \subseteq \Con(\langle L, \sqcap, \sqcup \rangle) $. To prove the other inclusion, let $\theta \in \Con(\langle L, \sqcap, \sqcup \rangle)$ and let  $a_1, a_2, b_1, b_2 \in L$ be such that $ \la a_1, b_1 \ra, \la a_2, b_2 \ra \in \theta$. We have to prove that $\la a_1 \ba a_2, b_1 \ba b_2 \ra \in \theta$. By assumption we have $\la a_1 \sqcup (a_1 \ba a_2), b_1 \sqcup (a_1 \ba a_2)  \ra \in \theta$ and $\la b_1 \sqcup (b_1 \ba b_2), a_1 \sqcup (b_1 \ba b_2) \ra \in \theta$. 
Since $a_1 \sqcup (a_1 \ba a_2) = 1 = b_1 \sqcup (b_1 \ba b_2)$, we have also $\la b_1 \sqcup (a_1 \ba a_2), a_1 \sqcup (b_1 \ba b_2)  \ra \in \theta$. This implies that 
$$\la (a_1 \ba a_2) \sqcap (b_1 \sqcup (a_1 \ba a_2)),   (a_1 \ba a_2) \sqcap  (a_1 \sqcup (b_1 \ba b_2)) \ra \in \theta.$$
Using the absorption laws, we have 
$a_1 \ba a_2 = (a_1 \ba a_2) \sqcap (b_1 \sqcup (a_1 \ba a_2))$, hence
$$\la a_1 \ba a_2,   (a_1 \ba a_2) \sqcap  (a_1 \sqcup (b_1 \ba b_2)) \ra \in \theta.$$ 
Similarly we obtain
$\la b_1 \ba b_2,   (b_1 \ba b_2) \sqcap  (b_1 \sqcup (a_1 \ba a_2)) \ra \in \theta.$
Now, notice that from the assumption we have $\la a_1 \wedge a_2, b_1 \wedge b_2 \ra \in \theta$. Hence
$$\la (a_1 \sqcap a_2) \sqcup ((a_1 \ba a_2) \sqcap (b_1 \ba b_2)), (b_1 \sqcap b_2) \sqcup ((a_1 \ba a_2) \sqcap (b_1 \ba b_2)) \ra \in \theta.$$

Since $$(a_1 \ba a_2) \sqcap (a_1 \sqcup (b_1 \ba b_2)) = (a_1 \sqcap a_2) \sqcup ((a_1 \ba a_2) \sqcap (b_1 \ba b_2))$$
and $$(b_1 \ba b_2) \sqcap (b_1 \sqcup (a_1 \ba a_2)) = (b_1 \sqcap b_2) \sqcup ((a_1 \ba a_2) \sqcap (b_1 \ba b_2))$$ it follows that
$$\la (a_1 \ba a_2) \sqcap (a_1 \sqcup (b_1 \ba b_2)),  (b_1 \ba b_2) \sqcap (b_1 \sqcup (a_1 \ba a_2)) \ra \in \theta.$$
Now, using the transitivity of $\theta$, we obtain  $\la a_1 \ba a_2,  b_1 \ba b_2 \ra \in \theta.$ 
\end{proof}


An important consequence of the previous result is the following:

\begin{proposition}
The variety of classical implicative lattices in the similarity type $\la \sqcap, \sqcup, \ba \ra$ is generated by its two-element member. 
\end{proposition}

\begin{proof}
From Proposition  \ref{compat} it follows immediately that a classical implicative lattice $\Al[L] = \la \sqcap, \sqcup, \ba  \ra $  is subdirectly irreducible if and only if its $\{ \sqcap, \sqcup  \} $-reduct, which is a distributive lattice, is subdirectly irreducible. Hence, the only subdirectly irreducible algebra in this variety is the one whose $\{ \sqcap, \sqcup  \} $-reduct is isomorphic to the two-element Boolean lattice. Therefore this algebra generates the variety.
\end{proof}

Another interesting corollary of Proposition \ref{compat} is that an analogous property holds for implicative   bilattices:

\begin{proposition} \label{compat_reg}
Let $\Al[B] = \left\langle B, \land, \lor, \otimes, \oplus, \supset, \neg \right\rangle$ be an implicative   bilattice. 
Then $\Con(\Al[B]) = \Con(\left\langle B, \land, \lor, \otimes, \oplus,  \neg \right\rangle)$.
\end{proposition}
\begin{proof}
Obviously $\textrm{Con}(\Al[B]) \subseteq \Con(\left\langle B, \land, \lor, \otimes, \oplus,  \neg \right\rangle)$. To prove the other inclusion, assume $\theta \in \Con(\left\langle B, \land, \lor, \otimes, \oplus,  \neg \right\rangle)$ and $\la a, b \ra, \la c, d \ra \in \theta$. We will show that $\la a \s c, b \s d \ra \in h^{-1}(h(\theta))$, where the isomorphisms $h$ and $h^{-1}$ are defined as in the proof of Proposition \ref{p:impbl_cong} (ii). That is, we have to prove that $\la \reg (a \s c), \reg ( b \s d) \ra,$ $\la \reg (\neg (a \s c)), \reg (\neg ( b \s d) ) \ra \in h(\theta).$ 
The latter is easily shown. Using Proposition~\ref{p:blconf_cong} (i), from the assumptions we obtain $\la \reg(a), \reg(b) \ra,$ $\la \reg( \neg a), \reg( \neg b) \ra,$ $\la \reg( c), \reg (d) \ra,$
$\la \reg( \neg c), \reg (\neg d) \ra \in h(\theta)$. This implies
$$
\reg(a) \otimes \reg (\neg c) = \reg (a \otimes \neg c) \ h(\theta) \ \reg (b \otimes \neg d) = \reg(b) \otimes \reg (\neg d).
$$
As noted in the proof of Proposition~\ref{p:impbl_cong} (ii), we have that $\reg(\neg (a \s b))  = \reg (a \otimes \neg b)$ for all $a, b \in B$. Hence $\la \reg (\neg (a \s c)), \reg (\neg ( b \s d) ) \ra \in h(\theta).$ Taking into account Proposition~\ref{compat}, the assumptions also imply
$$
\reg( \reg(a) \s \reg (c)) = \reg (a \s c) \ h(\theta) \ \reg (b \s d) = \reg (\reg(b) \s \reg ( d)) 
$$
and this completes the proof.
\end{proof}

Taking into account the results of Chapter \ref{ch:int} 
(Proposition \ref{prop:interbl_cong}), we may state the following:

\begin{proposition} \label{congequiv2} 
Let $\Al[B] = \left\langle B, \land, \lor, \otimes, \oplus, \supset, \neg \right\rangle$ be an implicative   bilattice and let $\theta \subseteq B \times B$ be an equivalence relation. Then the following are equivalent:
\begin{enumerate}[(i)]
\item $\theta \in \Con(\Al[B])$.
\item $\theta$ is compatible with the operations $\{ \neg, \land \}$.
\item $\theta$ is compatible with $\{ \neg, \lor \}$.
\item $\theta$ is compatible with $\{ \neg, \otimes, \oplus \}$.
\item $\theta$ is compatible with $\{ \neg, \land, \lor, \otimes, \oplus \}$, i.e $\theta$ is a congruence of the   bilattice reduct $\left\langle B, \land, \lor, \otimes, \oplus, \neg \right\rangle$ of $\Al[B] $.
\end{enumerate}
\end{proposition}

\section{The variety of implicative bilattices}
\label{sec:imp}

We are now able to state the second main result of this chapter, i.e.\ that the variety $\ib$ of implicative   bilattices is generated by the algebra $\fours$. To see this, we will prove that $\fours$ is the only subdirectly irreducible algebra in this variety.

\begin{thm} \label{t:vargen}
The variety $\ib$ of implicative   bilattices is generated by the four-element implicative   bilattice $\fours$. As a consequence, we have that $$\mathbf{Alg^{*}}\lbs = V(\fours).$$
\end{thm}

\begin{proof}
By Proposition \ref{congequiv2}, an implicative   bilattice is subdirectly irreducible if and only if its   bilattice reduct is. By Proposition \ref{impinterla}, the   bilattice reduct of an implicative   bilattice is a distributive   bilattice, and we also know that $\four$ is the only subdirectly irreducible distributive   bilattice. Hence, the only subdirectly irreducible implicative   bilattice is the one whose   bilattice reduct is $\four$, i.e.\ $\fours$. Therefore this algebra generates the variety.
\end{proof}

It is not difficult to see that the previous result implies that $\ib$ has no proper sub-quasivarieties (this is also a consequence of Proposition \ref{p:foursub}). We may also note that Theorem \ref{t:vargen} provides also an alternative way to prove Arieli and Avron's completeness theorem for the Hilbert calculus $H_{\supset}$ (our Theorem \ref{t:compl}). In fact, we have that $\h \,  = \; \vDash_{\scriptscriptstyle \mathcal{LB_{\supset}}}$ and that the single algebra $\fours$ constitutes an equivalent algebraic semantics for $\lbs$.

The rest of this section is devoted to stating some purely algebraic results that give further insight into the structure of the variety $\ib$.

\begin{proposition} \label{discr} 
$\ib$ is a discriminator variety.
\end{proposition} 

\begin{proof}
We first prove that $\ib$ is arithmetical. Clearly it is congruence--distributive, since lattices are. To prove that it is congruence--permutable, by \cite[Theorem II.12.2]{BuSa00}, we just need to consider the following term: $$ p(x,y,z) = ((x \rightarrow y) \supset z) \land ((z \rightarrow y) \supset x) \land (x \lor z).$$ In fact, if  $\Al[B]$ is an implicative bilattice and $a, b \in \Al[B]$, then $ p(a,a,b) = b$ and  $ p(a,b,b) = a$. The first holds because $((a \rightarrow a) \supset b) \land ((b \rightarrow a) \supset a) \land (a \lor b) = b \land  (a \lor b) \land ((b \rightarrow a) \supset a) = b \land ((b \rightarrow a) \supset a)$ and, by Proposition   \ref{p:lem2} (xi), $b \land ((b \rightarrow a) \supset a) = b$. As to the second, we have  $((a \rightarrow b) \supset b) \land ((b \rightarrow b) \supset a) \land (a \lor b) = ((a \rightarrow b) \supset b) \land a \land (a \lor b) = ((a \rightarrow b) \supset b) \land a = a$.

To complete the proof, it is sufficient to show that the algebra $\mathcal{FOUR}_{\supset}$ is quasiprimal (see \cite[Definition IV.10.6]{BuSa00}; it would be possible, indeed, to prove a stronger result, i.e.\ that $\mathcal{FOUR}_{\supset}$ is semiprimal: see \cite[Exercise IV.10.6, p. 199]{BuSa00}). To see this, note that the only proper subalgebra of $\mathcal{FOUR}_{\supset}$ is the trivial one with universe $\{ \top \}$. So $\mathcal{FOUR}_{\supset}$ is hereditary simple \cite[Definition IV.10.5]{BuSa00}. Hence, applying \cite[Theorem IV.10.7]{BuSa00},  we conclude that $\mathcal{FOUR}_{\supset}$ is quasiprimal. So $V(\mathcal{FOUR}_{\supset}) = \ib$ is a discriminator variety.
\end{proof}

\begin{proposition} \label{finprod} 
Let $\Al[B] $ be an  implicative   bilattice. Then:
\begin{enumerate}[(i)]
	\item If $\left| B \right| > 4$, then $\Al[B] \cong \Al[B_1] \times \Al[B_2]$ for some nontrivial  $\Al[B_1], \Al[B_2] \in \ib$.
	\item If $\Al[B]$ is finite, then $\Al[B] \in P (\mathcal{FOUR}_{\supset})$.
\end{enumerate}
\end{proposition} 

\begin{proof}
(i) By \cite[Theorem IV.9.4]{BuSa00}, we know that the indecomposable algebras in $\ib$ are simple, and the only simple non-trivial algebra in this variety is $\mathcal{FOUR}_{\supset}$.

(ii) We know that $\Al[B] \in SP (\mathcal{FOUR}_{\supset})$. Therefore, by  \cite[Corollary IV.10.2]{BuSa00}, we have that $\Al[B] \cong \mathcal{FOUR}_{\supset}^{n}$ for some $n < \omega$.
\end{proof}

In the following propositions we will show that in an implicative   bilattice each of the two lattice orderings is definable using the lattice operations of the other order plus $\{ \neg, \s \}$. Definability of the knowledge order follows immediately from Proposition \ref{p:lem2} (recall that we abbreviate $a = a \s a$ as $E(a)$).

\begin{proposition} \label{p:defk}
Let $\Al[B]$ be an implicative   bilattice and $a, b \in B$. Then the following are equivalent: 
\begin{enumerate}[(i)]
  \item $a \leq_k b$.
  \item $a \s b \geq_t \top $ and  $\neg a \s \neg b \geq_t \top$.
	\item $(a \supset b) \land (\neg a \supset \neg b) \geq_t \top$.
	\item $E((a \supset b) \land (\neg a \supset \neg b))$.
\end{enumerate} 
\end{proposition}

\begin{proof}
(i)$\Leftrightarrow$(ii). This equivalence has been stated in Proposition \ref{p:lem2} (vii).

(ii)$\Leftrightarrow$(iii). One implication is obvious, while the other follows from the interlacing conditions.

(iii)$\Leftrightarrow$(iv). This equivalence has been stated in Proposition \ref{p:biglem} (x).
\end{proof}

A symmetric result holds for the truth order:

\begin{proposition} \label{p:deft}
Let $\Al[B]$ be an implicative   bilattice and $a, b \in B$. Then the following are equivalent: 
\begin{enumerate}[(i)]
  \item $a \leq_t b$.
  \item $a \ta b \geq_t \top$.
	\item $(a \supset b) \otimes (\neg b \supset \neg a) \geq_t \top$.
	\item $E((a \supset b) \otimes (\neg b \supset \neg a)) $.
\end{enumerate} 
\end{proposition}

\begin{proof}
(i)$\Leftrightarrow$(ii). This equivalence has been stated in Proposition \ref{p:lem2} (vi).

(ii)$\Leftrightarrow$(iii). This equivalence follows from Proposition \ref{p:lem2} (xvi).

(iii)$\Leftrightarrow$(vi). This equivalence has been stated in Proposition \ref{p:biglem} (x).
\end{proof}

\section{Classical implicative and dual disjunctive lattices}
\label{sec:dua}

In this section we will investigate the relationship between the class of classical implicative lattices and the class of dual disjunctive lattices, which arose from the study of the reduced models of $\lb$ (the implicationless fragment of $\lbs$). 

Recall that a lattice $\langle L, \sqcap, \sqcup, 1 \rangle $ is dual disjunctive if and only if it is distributive, has a top element 1 and satisfies the following property: for all $a, b \in L$, if $a > b$, then there is $c \in L$ such that $a \sqcup c = 1 > b \sqcup c$. 

It is not difficult to prove that the finite members of the two classes coincide (and coincide also with the finite Boolean lattices). Indeed, it is  proved in \cite{CoDa98} that any classical implicative lattice is isomorphic to an ultrafilter of a Boolean algebra. 

One inclusion between the two classes is easily shown:

\begin{proposition} \label{relgood}
Let $\Al[L] = \langle L, \sqcap, \sqcup, 1 \rangle $ be a classical implicative lattice. Then:
\begin{enumerate}[(i)]
\item $\Al[L]$ is dual disjunctive. 
\item Given $a, b \in L$, denote by $a \backslash b$ the relative complement of $a$ in $[a \sqcap b, 1]$. Then for $a \geq b$ we have $ a \backslash b = \min \{ c \in [b, 1] : a \sqcup c = 1 \} $. 
\end{enumerate}
\end{proposition} 

\begin{proof}
(i). Let $a, b \in L$ such that $a > b$ and let $a \backslash b$ be the relative complement of $a$ in $[a \sqcap b, 1] = [b, 1]$. Note that $a \backslash b < 1$, because otherwise we would have $ a \sqcap a \backslash b = a > b $, against the assumption. Moreover $a \sqcup a \backslash b = 1$, but since $a \backslash b \in [a,1]$, we have  $b \sqcup a \backslash b = a \backslash b < 1$. This proves that $\Al[L]$ is dual disjunctive.

(ii). Let $A = \{ c \in [b, 1] : a \sqcup c = 1 \}$. By the definition of relative complement, we have that $ a \backslash b \in  A.$ We will prove that if $ d \in  A$, then $d \geq a \ba b$. Note that if $ d \in  A$, then also $ d \sqcap (a \ba b) \in  A$, because clearly $ d \sqcap (a \ba b) \in  [b, 1]$ and by distributivity we have $$ (d \sqcap (a \ba b)) \sqcup a = (d \sqcup a) \sqcap ((a \ba b) \sqcup a) = 1 \sqcap 1 = 1.$$ But we also have $(d \sqcap (a \ba b)) \sqcap a = d \sqcap b = b$. Hence, by the uniqueness of the relative complement,  we conclude that $d \sqcap (a \ba b) = a \ba b$, i.e.\ $d \geq a \ba b$.
\end{proof}

It is easy to see that the other inclusion is not true, that is, not every dual disjunctive lattice is relatively complemented. Consider the following:

\begin{ex} \label{e:ddl} 
Let $\Al[L] = \langle L, \sqcap, \sqcup, 1 \rangle $ be a classical implicative lattice without bottom element. Define the structure $\Al[L'] = \langle L \cup \{ 0 \}, \sqcap, \sqcup, 1 \rangle $ whose universe is $L$ augmented with a new element $0 \notin L$ and whose order is the one inherited from $\Al[L]$ except that $0 < a$ for all $a \in L$. Clearly $\Al[L']$ is a bounded distributive lattice, so if it were 
relatively complemented it would be a Boolean lattice. But it is not, since for all $a, b \in L$ we have $a \sqcap b \in L$, i.e.\ $a \sqcap b > 0$, therefore no element in $L$ has a complement. On the other hand, it is easy to see that $\Al[L']$ is dual disjunctive. Clearly if $0 < a < b$ the condition is satisfied because $a, b \in L$. If $a = 0$, then let $c \in L$ such that $0 = a < c < b$ (such an element must exist, because by assumption $\Al[L]$ has no minimum). Denoting by $b \backslash c$ the relative complement of $b$ in $[c, 1]$, we have $b \sqcup b \backslash c = 1$ but $ 0 \sqcup b \backslash c = b \backslash c < 1 $. So $\Al[L']$ is a dual disjunctive lattice. 
\end{ex}

In order to characterize the dual disjunctive lattices that are also classical implicative lattices, we shall need the following:

\begin{lem} \label{l:goodlem}
Let $\Al[L] = \langle L, \sqcap, \sqcup, 1 \rangle $ be a dual disjunctive lattice. Then, for all $a, b \in L$:
\begin{enumerate}[(i)]
\item $a = b $ if and only if $ \{ x \in L : a \sqcup x = 1  \} = \{ x \in L : b \sqcup x = 1 \} $. 
\item The interval sublattice $\langle [a, 1], \sqcap, \sqcup, 1 \rangle $ is also a dual disjunctive lattice. 
\item If for all $c \in [a, 1]$ there exists $c^* = \min \{ x \in [a, 1] : c \sqcup x = 1 \}$, then any element of $[a, 1]$ has a relative complement in $[a, 1]$.
\end{enumerate}
\end{lem} 

\begin{proof}
(i). One direction is trivial. For the other, assume $ \{ x \in L : a \sqcup x = 1  \} = \{ x \in L : b \sqcup x = 1 \} $ and $a \neq b$. Since $\Al[L]$ is dual disjunctive, if $a < b$ or $b < a$, then  we are done. So suppose $a$ and $b$ are incomparable. Then $a < a \sqcup b$, so by hipothesis there is $c \in L$ s.t. $a \sqcup c < 1 = (a \sqcup b) \sqcup c $. But then $b \sqcup  (a \sqcup c) = 1$, while $a \sqcup  (a \sqcup c) = a \sqcup c < 1$. Therefore $a \sqcup c \in \{ x \in L : b \sqcup x = 1 \}$ and $a \sqcup c \notin \{ x \in L : a \sqcup x = 1 \}$, which contradicts our assumption.   

(ii). We have to show that, for all $a, b, c \in L$, if $a \leq b < c $, then there is $c' \in [a, 1]$ such that $b \sqcup c' < 1 = c \sqcup  c'$. Since $\Al[L]$ is dual disjunctive, we know that if $ b < c$ then there is $d \in L$ such that $b \sqcup d < 1 = c \sqcup  d$. Clearly $a \sqcup d \in [a, 1]$, so let $c' = a \sqcup d$. Then we have $b \sqcup c' = b \sqcup a \sqcup d = b \sqcup d < 1 = 1 \sqcup a = c \sqcup a \sqcup d = c \sqcup  c'$. Therefore we conclude that $\langle [a, 1], \sqcap, \sqcup, 1 \rangle $ is dual disjunctive.

(iii). Let $a \in L$ and $b \in [a, 1]$. We have to show that $b \sqcap b^* = a$. Note that $ \{ x \in [a, 1] : b \sqcup x = 1 \} = \{ x \in L : b^* \leq x \}$. By (ii) $\langle [a, 1], \sqcap, \sqcup, 1 \rangle $ is dual disjunctive, therefore this implies that for all $b, c \in [a, 1]$ we have  $b^* = c^* $ iff $b = c$. Hence the map $*: L \longrightarrow L$ is injective. Since $b \sqcup b^* = 1 $, we have that $ b \in \{ x \in [a, 1] : b^* \sqcup x = 1 \} $, so $b^{**} = \min \{ x \in [a, 1] : b^*  \sqcup x = 1 \} \leq b$. Moreover, if $b \leq c$, then $\{ x \in [a, 1] : b \sqcup x = 1 \} \subseteq \{ x \in [a, 1] : c \sqcup x = 1 \} $, so $c^* \leq b^* $. It follows that $b^{***} = b^*$ for all $b \in [a, 1]$, so by the injectivity of the map $*: L \longrightarrow L$ we conclude that $b^{**} = b$. Now, in order to prove the statement we only need to show that $(b \sqcap c)^* = b^* \sqcup c^*$, because then we would have $b \sqcap b^* = (b \sqcap b^*)^{**} = (b \sqcup b^*)^* = 1^* = a$. To see this, note that on the one hand $b \sqcap c \leq b$ and $b \sqcap c \leq c$ imply $ b^* \leq (b \sqcap c)^*$ and $ c^* \leq (b \sqcap c)^*$, so $ b^* \sqcup c^* \leq (b \sqcap c)^*$. On the other hand, note that $(b \sqcap c) \sqcup (b^* \sqcup c^*) = (b \sqcup b^* \sqcup c^*) \sqcap (c \sqcup b^* \sqcup c^*) = (1 \sqcup c^*) \sqcap (1\sqcup b^* ) = 1$ and this means that $b^* \sqcup c^* \in  \{ x \in [a, 1] : (b \sqcap c) \sqcup x = 1 \}$. Hence $ b^* \sqcup c^* \geq (b \sqcap c)^*$ and we are done.
\end{proof}

We immediately have the following:

\begin{cor} \label{c:goodrel}
A distributive lattice $\Al[L] = \langle L, \sqcap, \sqcup, 1 \rangle $ is a classical implicative lattice if and only if
it is a dual disjunctive lattice and for all $a, b \in L$ such that $b \in [a, 1]$ there exists $b^* = \min \{ x \in [a, 1] : b \sqcup x = 1 \}$.
\end{cor} 
\begin{proof}
The leftward implication has been proved in Proposition \ref{relgood}, while the rightward one follows immediately from Lemma \ref{l:goodlem}, (iii).
\end{proof}

Corollary  \ref{c:goodrel} implies that, as we have anticipated, any finite dual disjunctive lattice is a Boolean lattice. One may wonder if condition (iii) of Lemma \ref{l:goodlem} implies that the lattice is dual disjunctive. This is false, an easy counterexample being any chain with top element 1. In a chain we have that $\min \{ x \in [a, 1] : b \sqcup x = 1 \} = 1$ for all $a \leq b < 1$, so condition (iii) is always satisfied, but the only dual disjunctive lattice which is a chain is the two-element one.

\section{Residuated De Morgan lattices}
\label{sec:sub}

In this and the next section we will study some subreducts of implicative   bilattices that arise by considering fragments of the implicative   bilattice language $\{ \land, \lor, \otimes, \oplus, \s, \neg \}$ which seem to have  some logical significance. We will first consider the relation between implicative   bilattices and a certain class of De Morgan lattices having a residuated pair. We begin by showing that any implicative   bilattice has a reduct which is a residuated lattice.

\begin{proposition} \label{resid} 
Let $\Al[B] = \langle B, \land, \lor, \otimes, \oplus, \supset, \neg  \rangle  \in \ib$. We define the operation $* : B \times B \longrightarrow B$ as follows: $a * b = \neg (a \rightarrow \neg b)$ for all $a, b \in B$. Then:
\begin{enumerate}[(a)]
\item $ \langle B, *, \top  \rangle$ is a commutative monoid.
\item For every $a, b \in B $, $a \rightarrow b$ is the residuum of $a, b$ relative to $*$ and the lattice ordering $\leq_t$, i.e.\ $ a * b \leq_t c \;$ if and only if $\; a \leq_t b \rightarrow c$. 
\item For every $a, b \in B $, $a = \neg \neg a$ and $a \rightarrow \neg b = b \rightarrow \neg a$.
\end{enumerate}
\end{proposition}

\begin{proof}
(a). Clearly $*$ is commutative, since $\neg (a \rightarrow \neg b) = \neg (b \supset \neg a) \lor \neg (a \supset \neg b) $. To prove associativity, note first that $a \rightarrow b = \neg b \rightarrow \neg a$. Now, using Proposition \ref{p:lem2}  (xv), we have 
\begin{align*}
a * (b * c) & = \neg (a \rightarrow (b \rightarrow \neg c)) \\
& = \neg (a \rightarrow (c \rightarrow \neg b)) \\
& = \neg (c \rightarrow (a \rightarrow \neg b)) \\
& = c * (a * b) \\
& = (a * b) * c. 
\end{align*}
To prove that $\top $ is the identity, note first that $\neg a \leq_t a \supset \top $ because 
\begin{align*}
\neg a  \rightarrow (a \supset \top) & = (\neg a \supset (a \supset \top)) \land (\neg (a \supset \top) \supset \neg \neg a)  \\
& = ((\neg a \land a)  \supset \top) \land ((a \land \top) \supset a) \\ 
& = ((\neg a \land a)  \supset \top) \land (a \supset a) \\
& \geq_t \top.
\end{align*}
Now, using (IB1) and recalling that $\top = \neg \top$, we have 
\begin{align*}
\top * a & = a * \top \\
& = \neg ((a \supset \top) \land (\neg \top \supset \neg a))\\ 
& = \neg ((a \supset \top) \land \neg a) \\ 
& = \neg \neg a \\
& = a. 
\end{align*}
(b). Assume $ a * b \leq_t c$, i.e.\ $\neg (a \rightarrow \neg b) \leq_t c$. This means that $ \neg c \leq_t a \rightarrow \neg b $, so $ \neg c \rightarrow (a \rightarrow \neg b) \geq_t \top $. But  $ \neg c \rightarrow (a \rightarrow \neg b) = a \rightarrow (\neg c \rightarrow \neg b) = a \rightarrow (b \rightarrow c)$, therefore we have $a \rightarrow (b \rightarrow c) \geq_t \top$, i.e.\ $ a \leq_t b \rightarrow c$.

Conversely, if $ a \leq_t b \rightarrow c$, then $\top \leq_t a \rightarrow (b \rightarrow c) = \neg c \rightarrow (a \rightarrow \neg b)$, which implies $ a * b \leq_t c$.

(c). Follows immediately from what we have noted in (a) and from the definition of negation for   bilattices.
\end{proof}

The behaviour of the adjoint pair $\{ *, \rightarrow \}$ in $\mathcal{FOUR}_\supset$ is depicted in the table below:
\begin{center}
\begin{tabular}[c]{c|c|c|c|c c c c|c|c|c|c}
$*$ &                      $\mathsf{f}$ & $\bot$        & $\top$        & $\mathsf{t}$ & & & 
$\rightarrow$ & $\mathsf{f}$ & $\bot$       & $\top$       & $\mathsf{t}$  \\
\cline{1-5} \cline{8-12}
$\mathsf{f}$             & $\mathsf{f}$ & $\mathsf{f}$  & $\mathsf{f}$  & $\mathsf{f}$  & & & 
$\mathsf{f}$  & $\mathsf{t}$ & $\mathsf{t}$ & $\mathsf{t}$ & $\mathsf{t}$  \\
\cline{1-5} \cline{8-12}
$\bot$                   & $\mathsf{f}$ & $\mathsf{f}$  & $\bot$        & $\bot$        & & & 
$\bot$        & $\bot$       & $\mathsf{t}$ & $\bot$       & $\mathsf{t}$  \\
\cline{1-5} \cline{8-12}
$\top$                   & $\mathsf{f}$ & $\bot$        & $\top$        & $\mathsf{t}$  & & & 
$\top$        & $\mathsf{f}$ & $\bot$   & $\top$        & $\mathsf{t}$     \\
\cline{1-5} \cline{8-12}
$\mathsf{t}$             & $\mathsf{f}$ & $\bot$        & $\mathsf{t}$  & $\mathsf{t}$  & & & 
$\mathsf{t}$  & $\mathsf{f}$ & $\bot$ & $\mathsf{f}$    & $\mathsf{t}$     \\
\end{tabular}
\end{center}
\vspace{2.5mm}

Using the terminology of \cite{GaRa04}, we may conclude that the structure $ \langle B, *, \land, \lor,$ ${\rightarrow,} \neg, \top  \rangle $  is an involutive CDRL (commutative distributive residuated lattice). However, it satisfies also some additional properties, for instance it is not difficult to see that $a * a * a = a * a $ for all $a \in B$. So a question arises: which is the class of residuated lattices that correspond to the $\{  \land, \lor, \rightarrow, \neg, \top \}$-subreducts of implicative   bilattices? In order solve this problem, we introduce the following:

\begin{definition} \label{rdml}
{\rm A \emph{residuated De Morgan lattice} is an algebra $\Al = \langle A, \land, \lor, {\supset,}$ $\neg, \top  \rangle$ such that $\langle A, \land, \lor, \neg \rangle$ is a De Morgan lattice and the following equations are satisfied:
\begin{enumerate}[ ]
\item (RD0) $\quad \top  \approx \neg \top $ 
\item (RD1) $\quad \top \supset x \approx x$ 
\item (RD2) $\quad  x \supset (y \supset z) \approx (x \land y) \supset z$
\item (RD3) $\quad \top \land (((x \supset y) \supset x) \supset x) \approx \top  $ 
\item (RD4) $ \quad (x \lor y) \supset z \approx (x \supset z) \land  (y \supset z)$ 
\item (RD5) $\quad  x \land (((x \supset y) \land (\neg y \supset \neg x)) \supset y) \approx x$  
\item (RD6) $\quad \neg (x \supset y ) \supset z  \approx (x \land \neg y)  \supset z$.
\end{enumerate}
We will denote by $\rd$ the variety of residuated De Morgan lattices.}
\end{definition}

Adopting the notation of the previous sections, we will use the following abbreviations: 
\begin{align*}
a \rightarrow b &=_{def} (a \supset b) \land (\neg b \supset \neg a) \\
a \leftrightarrow b &=_{def} (a \rightarrow b) \land (b \rightarrow a) \\
a * b &=_{def} \neg (a \rightarrow \neg b)
\end{align*}

We will show that  (RD0) to (RD6) are necessary and sufficient properties for a De Morgan lattice to be a subreduct of an implicative   bilattice. Necessity follows from the fact that (RD0) to (RD6) hold in any implicative   bilattice; to prove sufficiency, we shall need the following lemma:

\begin{proposition} \label{lemmard}
Let $\Al = \langle A, \land, \lor, \supset, \neg, \top  \rangle \in \rd$. Then, for all $a, b, c \in A$:
\begin{enumerate}[(i)]
\item $\top \leq a$ implies $a \s d = d$ for every $d \in A$
\item $b \leq a \s b$
\item $\top \leq a \s a$
\item $\top \leq  (a \land b) \supset b $
\item $a \leq b$ iff $\top \leq  a \rightarrow b $
\item $\top \leq a$ implies $\top \leq b \s a$ for every $b \in A$
\item $\top \leq a$ and $  \top \leq a \supset b $ imply $\top \leq b$
\item $ a \land \top = b \land \top$ implies $a \supset d = b \supset d $ for all $d \in A$
\item $a \land (a \s b) \land \top \leq b$
\item $(a \land b) \supset c = (a \supset b) \supset (a \supset c)$
\item $\top \leq a \supset b  $ and $\top \leq  b \supset c $ imply $\top \leq a \supset c $
\item $a \leq b$ implies $c \supset a \leq c \supset  b $
\item $ a \supset (b \land c) = (a \supset b) \land (a \supset c)$
\item $ \neg (a \ta b) \s c = \neg (a \supset b) \supset c$
\item $a \rightarrow (b \rightarrow c) =  b \rightarrow (a \rightarrow c)$
\item $\top \leq (a \supset b) \lor a  $
\item if $a \supset d = b \supset d $ for all $d \in A$, then $a \land \top = b \land \top$.
\end{enumerate}
\end{proposition}

\begin{proof}
(i) Suppose that  $\top \leq  a$. Let $d \in A$.  Then using (RD1) and (RD2)  $d = \top \s d = (\top \land a ) \s d = \top \s (a \s d) = a \s d$.

(ii) Let $a, b \in A$. Since $\top \leq \top \vee a$, by (i)  $(\top \lor a ) \s b =
b$. Now using 
 (RD1) and (RD4) we have
$$
b \land (a \supset b) =  (\top \s b) \land (a \supset b) =  (\top \lor a ) \s b  =  b.
$$
%
Hence, $b \leq (a \s b)$.

(iii) By (ii) $\top \leq (a \s \top)$. Then by (i) $(a \s \top) \s a = a$. Now  by (RD3), $\top \leq ((a \s \top) \s a) \s a$. It follows that 
$\top \leq a \s a$. 

(iv) By (RD2) we have $(a \land b) \supset b = a  \s (b \s b)$ and by (iii) and (ii) we have 
$\top \leq b \s b \leq  a  \s (b \s b).$ So, $\top \leq (a \land b) \supset b$.

(v) Assume $a \leq b$. Then, using (iv), we have 
\begin{align*}
a \rightarrow b & =  (a \supset b) \land (\neg b \supset \neg a) \\
                & =  ((a \land b) \supset b) \land (( \neg b \land \neg a) \supset \neg a) 
								 \geq \top. \end{align*}
Conversely, assume $\top \leq a \rightarrow b$. Using  (i) we have $(a \rightarrow b) \supset b  = b$.
So, by (RD5), it follows that $a \leq b$.

(vi)  Assume that $\top \leq a$. Let $b \in A$. Then $\top \land  b \leq a$. So, from (v) follows that $\top \leq (\top \land b) \s a$. Hence, using (RD2) we obtain $\top \leq \top \s (b \s a)$. So by (RD1) we obtain that $\top \leq b \s a$. 

(vii) Assume $\top \leq a$ and $\top \leq a \supset b$. Then by (i) $a \s b = b$. So $ \top \leq  b $.

(viii) Assume $ a \land \top = b \land \top$. Note that by (RD1) and (RD2), we have  $(a \land \top) \s d = \top \s (a \s d ) = a \s d$ for every $d$, and similarly we have $(b \land \top) \s d = b \s d$. From the assumption then follows that 
$a \s d = b \s d$, for every $d$. 

(ix) We will prove that $\top \leq  (a \land (a \s b) \land \top) \ta b$. Then, by (v), we will obtain the desired conclusion. On the one hand, by (RD1), (RD2) and (i), we have
\begin{align*}
(a \land (a \s b) \land \top) \s b   & =  
\top \s ((a \land (a \s b)) \s b)    \\ & = 
(a \land (a \s b)) \s b    \\ & = 
(a \s b) \s (a \s b) \\ & \geq  \top.
\end{align*}
On the other hand, using De Morgan's laws, (RD0) and (ii), we have
\begin{align*}
\neg b \s \neg (a \land (a \s b) \land \top)  & =  
\neg b \s  (\neg (a \land (a \s b)) \lor \top)  \\ & \geq 
\neg (a \land (a \s b)) \lor \top \\ & \geq \top.
\end{align*}

(x) By (ii) we have $b \leq a \s b$, so $ a \land b \land \top \leq a \land (a \s b) \land \top$.  By (ix) we have  $a \land (a \s b) \land \top \leq b$. Hence, $a \land (a \s b) \land \top = a \land b \land \top$. By (viii), this implies that $(a \land (a \s b)) \s c = (a \land b) \s c$, for every $c$. By (RD2),  $(a \land (a \s b)) \s c = (a \s b) \s (a \s c)$, so we are done.

(xi) Assume $\top \leq a \supset b$ and $\top \leq b \supset c $. Note that by (iv) and (RD2)  we have $\top \leq (a \land (b \s c)) \supset (b \s c) = (b \s c) \supset (a \supset (b \s c))$. Now, using (vii) and the second assumption, we obtain $\top \leq a \supset (b \supset c)$. By (RD2) and (x), we have $\top \leq (a \supset b) \supset (a \supset c)$. Using  the first assumption and again (vii), we obtain $\top \leq (a \supset c)$.

(xii) Assume $a \leq b$.  We will prove that $\top \leq (c \supset a) \rightarrow (c \supset b)$, i.e.\ that $\top \leq (c \supset a) \supset (c \supset b)$ and $\top \leq \neg (c \supset b) \supset \neg (c \supset a)$. As to the first, note that $a \leq b$ implies $\top \leq a \supset b $. Moreover, by (iii) and (RD2), we have $\top \leq ((c \supset a) \land c) \supset a$.  Using (xi) and (RD2), it follows that   $\top \leq ((c \supset a) \land c) \supset b = (c \supset a) \supset(c \supset b)$. As to the second, note that from the assumption it follows that $\neg b \leq \neg a$, which  implies $\top \leq \neg b \supset \neg a$. On the other hand, by (iii) we have $\top \leq \neg (c \supset a) \supset \neg (c \supset a)$. Applying (RD6) and (RD2), we obtain $\top \leq (\neg a \land c) \supset \neg (c \supset a) = \neg a \supset ( c \supset \neg (c \supset a))$. Using (xi) as before, we have   $\top \leq \neg b \supset ( c \supset \neg (c \supset a))$. Now, applying (RD2) and (RD6), we obtain $\top \leq   (\neg b \land c) \supset \neg (c \supset a) = \neg (c \supset b) \supset \neg (c \supset a)$. 

(xiii) By (xii) we obtain that $ a \supset (b \land c) \leq  (a \supset b)$ and $ a \supset (b \land c) \leq  (a \supset c)$, so $ a \supset (b \land c) \leq  (a \supset b) \land (a \supset c)$. 
In order to prove the other inequality, we will show that $\top \leq ((a \supset b) \land (a \supset c)) \rightarrow (a \supset (b \land c))$, i.e.\ that $\top \leq ((a \supset b) \land (a \supset c)) \supset (a \supset (b \land c))$ and $\top \leq  \neg (a \supset (b \land c)) \supset \neg ((a \supset b) \land (a \supset c))$. 

For the first, applying repeatedly (RD2) and (x), we have
\begin{align*}
((a \supset b) \land (a \supset c)) \supset (a \supset (b \land c)) = &  \\ (a \supset b) \supset (a \supset c)) \supset ((a \supset b) \supset (a \supset (b \land c))) 
 = & \text{ by (x)}\\ ((a \supset b) \land a) \supset c) \supset (((a \supset b) \land a) \supset (b \land c)) = & \text{ by (RD2)}\\
 ((a \supset b) \land a) \supset (c \supset (b \land c)) = & \text{ by (RD2)}\\
 (a \supset b) \supset ( a \supset (c \supset (b \land c))) = & \text{ by (RD2)} \\
 (a \supset b) \supset ( (a \land c) \supset (b \land c)) = & \text{ by (RD2)} \\
 (a \land c) \supset ( (a \supset b) \supset (b \land c)) = & \text{ by (RD2)} \\
 ((a \land c) \supset (a \supset b)) \supset ((a \land c) \supset (b \land c)) = & \text{ by (RD2)} \\
 ((a \land c) \supset b) \supset ( (a \land c) \supset (b \land c)) =  & \text{ by (RD2)} \\
 (a \land c) \supset (b \supset ((a \land c) \supset (b \land c))) =  & \text{ by (RD2)}\\
 (a \land c) \supset ((b \land a \land c ) \supset (b \land c) =  & \text{ by (RD2)}\\ 
 (a \land b \land c)   \supset (b \land c). &
\end{align*} 
Since, using  (v), it follows that $\top \leq  (a \land b \land c) \supset   (b \land c)$, we obtain that $\top \leq ((a \supset b) \land (a \supset c)) \supset (a \supset (b \land c))$ as desired.

As to the second inequality, applying (RD4) and (RD6) we have 
\begin{align*}
& \neg (a \supset (b \land c)) \supset \neg ((a \supset b) \land (a \supset c)) = \\
& =  (a \land \neg (b \land c)) \supset \neg ((a \supset b) \land (a \supset c)) = \\
& = ((a \land \neg b) \lor (a \land \neg c)) \supset \neg ((a \supset b) \land (a \supset c)) = \\
& = ((a \land \neg b) \supset  (\neg (a \supset b) \lor \neg (a \supset c))) \land ((a \land \neg c) \supset (\neg (a \supset b) \lor \neg (a \supset c))) = \\
& = ((\neg a \supset  b) \supset  (\neg (a \supset b) \lor \neg (a \supset c))) \land ( \neg (a \supset c) \supset (\neg (a \supset b) \lor \neg (a \supset c))) \geq  \top.
\end{align*}

(xiv). We have
\begin{align*}
\neg (a \ta b) \s c & = \neg ((a \s b) \land (\neg b \s \neg a)) \s c    \\
& = (\neg (a \s b) \lor \neg (\neg b \s \neg a)) \s c     & \text{by De Morgan's law} \\
& = (\neg (a \s b) \s c) \land (\neg (\neg b \s \neg a) \s c)   & \text{by (RD4)}  \\
& = (\neg (a \s b) \s c) \land ((\neg b \land a) \s c)  & \text{by (RD6)}  \\
& = (\neg (a \s b) \s c) \land (\neg (a \s b) \s c)    & \text{by (RD6)}  \\
& = \neg (a \s b) \s c.
\end{align*}

(xv). 
We  have
\begin{align*} 
a \rightarrow (b \rightarrow c) 
&  =  (a \supset (b \ta c )) \land (\neg (b \ta c ) \s \neg a)   \\
&  =  (a \supset ( (b \supset c ) \land (\neg c \supset \neg b ) )) \land (\neg (b \s c ) \s \neg a)   & \text{by (xiv)} \\ 
&  =  (a \supset (b \supset c )) \land (a \s (\neg c \supset \neg b ) ) \land ((b \land \neg c ) \s \neg a)    & \text{by (xiii), (RD6)} \\ 
&  =  (b \supset (a \supset c )) \land ((a \land \neg c) \supset \neg b  ) \land (b \s (\neg c  \s \neg a) & \text{by (RD2), (RD6)}  \\ 
&  =  (b \supset ((a \supset c ) \land (\neg c  \s \neg a) ) \land (\neg (a \s c) \supset \neg b  )  & \text{by (xiii), (RD5)} \\
&  =  (b \supset (a \ta c )) \land (\neg (a \ta c) \supset \neg b  )  & \text{by (xiv)} \\
&  =  b \rightarrow (a \rightarrow c).
\end{align*}

(xvi) Since $(a \supset b) \leq ( (a \supset b) \lor a)$, it follows that $\top \leq  (a \supset b) \s ( (a \supset b) \lor a)$. Then using  (vi) we have 
$$\top \leq ((a \s b) \s b) \s ((a \supset b) \s ( (a \supset b) \lor a)).$$
Now by (RD2) and (RD4) we have 
\begin{align*}
& ((a \supset b) \supset b) \supset ((a \supset b) \supset ( (a \supset b) \lor a))  = \\ & =  (((a \supset b) \supset b) \land (a \supset b)) \supset ( (a \supset b) \lor a)  =
\\ & = (((a \supset b) \lor a) \supset b) \supset ((a \supset b) \lor a).
\end{align*}
Thus $\top \leq (((a \supset b) \lor a) \supset b) \supset ((a \supset b) \lor a)$. By (RD3 ) we have
$$\top \leq ((((a \supset b) \lor a) \supset b) \supset ((a \supset b) \lor a)) \s ((a \supset b) \lor a).$$
So, using (vii) it follows that 
 $\top \leq (a \supset b) \lor a$.

(xvii) Note that (ii) and (ix) imply $a \land (a \s b ) \land \top = a \land b \land \top$. By hypothesis we have $a \s b = b \s b$ and similarly $b \s a = a \s a$, so applying (iii) we obtain $a \land b \land \top = a \land (a \s b) \land \top = a \land (b \s b) \land \top = a \land \top$. Similarly we have  $a \land b \land \top = b \land (b \s a) \land \top = b \land \top$, so the result immediately follows.
\end{proof}

Before proceeding, let us check that any residuated De Morgan lattice is indeed an involutive CDRL (the proof is just an adaptation of that of Proposition \ref{resid}):

\begin{proposition} \label{residrd} 
Let $\Al = \langle A, \land, \lor, \supset, \neg, \top  \rangle \in \rd$. Then:
\begin{enumerate}[(i)]
\item $ \langle B, *, \top  \rangle$ is a commutative monoid.
\item For every $a, b \in B $, $a \rightarrow b$ is the residuum of $a, b$ relative to $*$ and the lattice ordering $\leq$, i.e.\ $ a * b \leq c \;$ if and only if $\; a \leq b \rightarrow c$. 
\item For every $a, b \in B $, $a = \neg \neg a$ and $a \rightarrow \neg b = b \rightarrow \neg a$.
\end{enumerate}
\end{proposition}

\begin{proof}
(i) Clearly $*$ is commutative, since $\neg (a \rightarrow \neg b) = \neg (b \supset \neg a) \lor \neg (a \supset \neg b) $. To prove associativity, note first that $a \rightarrow b = \neg b \rightarrow \neg a$. Now, using Proposition \ref{lemmard} (xv), we have 
\begin{align*}
a * (b * c) & = \neg (a \rightarrow (b \rightarrow \neg c)) \\
& = \neg (a \rightarrow (c \rightarrow \neg b)) \\
& = \neg (c \rightarrow (a \rightarrow \neg b)) \\
& = c * (a * b) \\
& = (a * b) * c. 
\end{align*}
To prove that $\top $ is the identity, note first that 
\begin{align*}
\neg a  \rightarrow (a \supset \top) & = (\neg a \supset (a \supset \top)) \land (\neg (a \supset \top) \supset \neg \neg a)  \\
& = ((\neg a \land a)  \supset \top) \land ((a \land \top) \supset a) \\ 
& = ((\neg a \land a)  \supset \top) \land (a \supset a). 
\end{align*}
So, since $\top \leq a \s a$ and $\top \leq (\neg a \land a)  \supset \top$, we obtain that 
$\top \leq \neg a  \rightarrow (a \supset \top)$. 
Therefore $\neg a \leq a \supset \top$.

Now we have 
\begin{align*}
\top * a & = a * \top \\
& = \neg ((a \supset \top) \land (\top \supset \neg a) \\ 
& = \neg ((a \supset \top) \land \neg a) \\ 
& = \neg \neg a \\
& = a. 
\end{align*}

(ii) Assume $ a * b \leq  c$, i.e.\  $\neg (a \rightarrow \neg b) \leq  c$. Therefore,  $\neg c \leq  a \rightarrow \neg b $. So $\top \leq  \neg c \rightarrow (a \rightarrow \neg b)$. But  $ \neg c \rightarrow (a \rightarrow \neg b) = a \rightarrow (\neg c \rightarrow \neg b) = a \rightarrow (b \rightarrow c)$. Therefore we have $\top \leq a \rightarrow (b \rightarrow c)$, and so $ a \leq  b \rightarrow c$.
Conversely, if $ a \leq  b \rightarrow c$, then $\top \leq  a \rightarrow (b \rightarrow c) = \neg c \rightarrow (a \rightarrow \neg b)$, i.e.\ $ a * b \leq  c$.

(iii) It follows immediately from what noted in (i) and from the definition of negation for De Morgan lattices.
\end{proof}

We are now able to prove what we claimed. The following result shows that any residuated De Morgan lattice is embeddable into an implicative   bilattice (by an embedding we mean here an injective map which is a homomorphism w.r.t.\ to the operations $\{\land, \lor, \s, \neg, \top \}$). Moreover, the defined embedding is in some sense a ``minimal'' one (see item (iv)).

\begin{thm} \label{t:embed} 
Let $\Al = \langle A, \land^A, \lor^A, \supset^A, \neg^A, \top^A  \rangle \in \rd$ and let $\Al[A^-] = \langle A^-, \land^A, \lor^A \rangle$ be the sublattice of $\Al$ with universe $A^- = \{ a \in A: a \leq \top^A \}$. Then: 
\begin{enumerate}[(i)]
\item $\Al[A^-]$ is dually isomorphic to $\Al[A^+] = \langle A^+, \land^A, \lor^A \rangle$, the sublattice of $\Al$ with universe $A^+ = \{ a \in A: a \geq \top^A \}$.
\item $\Al[A^-]$ and $\Al[A^+]$ are relatively complemented.
\item there is an embedding $h: A \ta A^- \times A^- $ of $\Al$ into the implicative   bilattice $\Al[B] = \langle \Al[A^-] \odot \Al[A^-], \supset^B \rangle$ (or into $ \langle \Al[A^+] \odot \Al[A^+], \supset \rangle$).
\item If $f: A \rightarrow B_{1}$ is a homomorphism from $\Al$ to an implicative   bilattice $\Al[B_{1}]$, then there is a unique map $f': A^- \times A^- \rightarrow B_{1}$ which is also a homomorphism of $\Al[B] = \langle \Al[A^-] \odot \Al[A^-], \supset^B \rangle$ into $\Al[B_{1}]$ such that $f' \cdot h = f$. Moreover, if $f$ is injective, so is $f'$.
\end{enumerate}
\end{thm}

\begin{proof}
(i). The isomorphism is given by the negation operation. It is easy to verify that it is a bijection. Moreover, we have $a \leq b$ if and only if $\neg b \leq \neg a$, so it reverses the order.

(ii). Consider $\Al[A^-]$ and Let $a, b \in A$  be such that $a \leq b \leq \top^A$. We have to show that there is $c \in A$ such that $a \leq c \leq \top^A$ and $b \land^A c = a$ and $b \lor^A c = \top^{A}$. Take $c = (b \supset^A a) \land^A \top^A$. On the one hand, by Proposition \ref{lemmard} (xii) we have $b \land^A (b \supset^A a) \land^A \top^A = b \land^A  a \land^A \top^A = a$. On the other hand, by Proposition \ref{lemmard} (xvi) we have $b \lor^A ((b \supset^A a) \land^A  \top^A) = (b \lor^A ((b \supset^A a)) \land^A (b \lor^A \top^A) = (b \lor^A ((b \supset^A a)) \land^A \top^A = \top^A$.

(iii). Let $\Al[A^-] = \langle A^-, \land^A, \lor^A, \supset^A, \top^{A} \rangle$. 
Since $\Al[A^-]$ is relatively complemented, by Proposition \ref{defimp} we know that the structure $\Al[B] = \langle \Al[A^-] \odot \Al[A^-], \supset^B \rangle$ is an implicative   bilattice and the implication is defined, for $a_1, a_2, b_1, b_2 \in A^-$, as follows:
$$
\langle a_1, a_2 \rangle \supset^B \langle b_1, b_2 \rangle = \langle (a_1 \supset^A (a_1 \land^A b_1)) \land \top^A,  a_1  \land^A b_2 \rangle.
$$
The embedding $h: A \longrightarrow A^- \times A^-$ is defined as follows: 
$$h(a) = \langle a \land^A \top^A, \neg^{A} a \land^A \top^A \rangle.$$ 
We have to check that $h$ is a homomorphism, i.e.\ that:
\begin{enumerate}[(a)]
	\item \quad $h(a \land^A b) = h(a) \land^B h(b)$.
	\item \quad $h(a \lor^A b) = h(a) \lor^B h(b)$.
	\item \quad $h(\neg^A a) = \neg^B(h(a))$. 
		\item \quad  $h(\top^A) = \top^B$.
	\item \quad $h(a \supset^A b) = h(a) \supset^B h(b)$.
\end{enumerate}

(a). We have 
\begin{align*}
h(a \land^A b) & =  \langle a \land^A b \land^A \top^A, \neg^A (a \land^A b) \land^A \top^A \rangle \\
							 & = \langle a \land^A b \land^A \top^A, (\neg^A a \lor^A \neg^A b) \land^A \top^A \rangle \\
							 & = \langle a \land^A b \land^A \top^A, (\neg^A a \land^A \top^A) \lor^A (\neg^A b \land^A \top^A) \rangle \\
						   & = \langle a \land^A \top^A, \neg^A a \land^A \top^A \rangle \land^B  \langle b \land^A \top^A, \neg^A b \land^A \top^A \rangle \\
& = h(a) \land^B h(b).
\end{align*}

Case (b) is similar to (a). Cases (c) and (d) are easy, so we omit them.

(e). Recall that by Proposition \ref{lemmard} we have $a \supset^A a \geq \top^A$, $\top^A \supset^A a \geq \top^A$  and $\neg^A (a \supset^A b) \land^A \top^A = a \land^A \neg^A b \land^A \top^A$. Now we have 
\begin{align*}
h(a \supset^A b) & = \langle (a \supset^A b) \land^A \top^A, \neg^A  (a \supset^A b) \land^A \top^A  \rangle \\
                 & = \langle (a  \supset^A a) \land^A (a \supset^A b) \land^A (a  \supset^A \top^A) \land^A \top^A, a \land^A \neg^A b \land^A \top^A \rangle \\
                 & = \langle a  \supset^A (a \land^A b \land^A \top^A), a \land^A \neg^A b \land^A \top^A \rangle \\
                 & = \langle (a \land^A \top^A) \supset^A (a \land^A b \land^A \top^A), a \land^A \top^A \land^A \neg^A b \land^A \top^A \rangle \\
                 & =  \langle a \land^A \top^A, \neg^A a \land^A \top^A \rangle \supset^B \langle b \land^A \top^A, \neg^A b \land^A \top^A \rangle \\
                 & =  h(a) \supset^B h(b).
\end{align*}

To prove injectivity, assume $h(a) = h(b)$, i.e.\ $\langle a \land^A \top^A, \neg^A a \land^A \top^A \rangle = \langle b \land^A \top^A, \neg^A b \land^A \top^A \rangle$, so $a \land^A \top^A = b \land^A \top^A$ and $\neg^A a \land^A \top^A  = \neg^A b \land^A \top^A$. By Proposition \ref{lemmard} (v), this implies $ (a \land^A \top^A) \leftrightarrow^A (b \land^A \top^A) \geq \top^A$ and $ (\neg^A a \land^A \top^A) \leftrightarrow^A (\neg^A  b \land^A \top^A) \geq \top^A$. From the first inequality we have
\begin{align*}
(a \land^A \top^A) \supset^A (b \land^A \top^A) & = a \supset^A (b \land^A \top^A) \geq \top^A \\
& = (a \supset^A b) \land^A (a \supset^A \top^A) \\
& \geq \top^A.
\end{align*}
Similarly we obtain $b \supset^A a \geq \top^A$. 

From the second inequality we have
\begin{align*}
(\neg^A a \land^A \top^A) \supset^A (\neg^A  b \land^A \top^A) & = \neg^A a \supset^A (\neg^A  b \land^A \top^A) \\
& = (\neg^A a \supset^A \neg^A  b) \land^A  ( \neg^A a \supset^A \top^A) \\
& \geq \top^A.  
\end{align*}
And similarly we obtain $\neg^A b \supset^A \neg^A  a \geq \top^A$. But this, again by Proposition \ref{lemmard} (v), implies $a=b$.

(iv). Assume $f: A \longrightarrow B_1$ is an embedding of $\Al$ into an implicative   bilattice $\Al[B_1]$. We know that the   bilattice reduct of $\Al[B_1]$ is isomorphic to the   product   bilattice $\Al[B_1^{-}] \odot \Al[B_1^{-}]$, where $\Al[B_1^{-}] = \langle \{a \in B_1: a \leq_t \top^{B_1} \}, \land^{B_1}, \lor^{B_1} \rangle$. 
We will prove that the desired embedding is given by the map $f': A^- \times A^- \longrightarrow B_{1}^{-} \times B_{1}^{-}$ defined as follows: $f'(\langle a , b \rangle) = \langle f(a) , f(b) \rangle$ for all $a,b \in A^-$.

From the definition it follows immediately that $f'$ is one-to-one. It remains to prove that it is indeed a homomorphism. Let us check the case of $\land$. We have 
\begin{align*}
f'( \langle a_1, a_2 \rangle \land^B \langle b_1, b_2 \rangle) 
& = f'(\langle a_1 \land^A b_1,  a_2 \lor^A b_2 \rangle) \\
& = \langle f(a_1 \land^A b_1), f(a_2 \lor^A b_2)  \rangle \\
& = \langle f(a_1) \land^{B_1} f(b_1), f(a_2) \lor^{B_1} f(b_2)  \rangle \\
& = \langle f(a_1), f(a_2)\rangle  \land^{B_1} \langle f(b_1), f(b_2)  \rangle \\
& = f'( \langle a_1, a_2 \rangle) \land^{B_1} f'( \langle b_1, b_2 \rangle).
\end{align*}
The proofs corresponding to the other  lattice connectives are similar. Let us check the case of implication:
\begin{align*}
& f'( \langle a_1, a_2 \rangle \supset^B \langle b_1, b_2 \rangle) = \\
& = f'( \langle (a_1 \supset^A (a_1 \land^A b_1)) \land \top^A,  a_1  \land^A b_2 \rangle) = \\
& = \langle f((a_1 \supset^A (a_1 \land^A b_1)) \land \top^A),  f(a_1  \land^A b_2) \rangle = \\
& = \langle (( f(a_1) \supset^{B_1} (f(a_1) \land^{B_1} f(b_1))) \land \top^{B_1}),  (f(a_1)  \land^{B_1} f(b_2)) \rangle =  \\
& = \langle f(a_1), f(a_2) \rangle \supset^{B_1} \langle f(b_1), f(b_2) \rangle) = \\
& = f'( \langle a_1, a_2 \rangle) \supset^{B_1} f'( \langle b_1, b_2 \rangle).
\end{align*}
This holds because, as we have seen in (iii), for any $a, b \in L $ we have that $(a \supset^{B_1} (a \land^{B_1} b)) \land \top^{B_1}$ is the relative complement of $a$ in the interval $[a, \top^{B_1}]$. Finally, it is easy to see that $f' \cdot h = f$, for we have, for all $a \in A$,
\begin{align*}
f' \cdot h (a) & = f' (\langle a \land^A \top^A, \neg^{A} a \land^A \top^A \rangle) \\
& = \la f(a \land^A \top^A), f(\neg^{A} a \land^A \top^A) \ra \\
& = \la f(a) \land^{B_1} \top^{B_1}, f(\neg^{A} a) \land^{B_1} \top^{B_1}) \ra \\
& = \la f(a) \land^{B_1} \top^{B_1}, \neg^{B_1} f( a) \land^{B_1} \top^{B_1}) \ra \\
& = f(a).
\end{align*}
From this it follows that the map $f'$ is unique, and it is also easy to see that, if $f$ is injective (an embedding), then $f'$ is also injective.
\end{proof}

Theorem \ref{t:embed} enables us to obtain some additional information about the variety $\rd$:

\begin{thm} \label{t:rdgen}
The variety of residuated De Morgan lattices is generated by the four-element algebra whose $\{ \land, \lor, \neg \}$-reduct is the four element De Morgan lattice.
\end{thm} 

\begin{proof}
We will prove that if an equation $\phi \approx \psi$ does not hold in the variety $\rd$, then it does not hold in the four-element residuated De Morgan lattice. By assumption we have that there is some residuated De Morgan lattice $\Al$ such that $\phi \approx \psi$ does not hold in $\Al$. By Theorem \ref{t:embed} (iii), we know that $\Al$ can be embedded into some implicative   bilattice $\Al[B]$. So $\Al[B]$ does not satisfy $\phi \approx \psi$. By Theorem \ref{t:vargen}, this implies that $\phi \approx \psi$ does not hold in $\fours$. Hence $\phi \approx \psi$ does not hold in the residuated De Morgan lattice reduct of $\fours$.
\end{proof}


\section{Other subreducts}
\label{sec:other}

In this section we will see that the construction described in Theorem \ref{t:embed} can be carried out even if we restrict our attention to a smaller fragment of the implicative  bilattice language.

From the point of view of AAL, the core of an algebraizable logic lies in the fragment of the language that is needed in order to define the interpretations between the logic and its associated class of algebras. In the case of the logic $\lbs$ the interpretations used  the connectives $\{ \s, \neg, \land \}$, but it is easy to see that $\land$ is not necessary, since the formula $p \da q$ can be replaced by the set $\{ p \s q, q \s p, \neg p \s \neg q, \neg q \s \neg p \}$. This fact seems to suggest that the  $\{ \s, \neg \}$-fragment of our language is a particularly interesting one. In order to justify this claim, let us introduce the following:

\begin{definition} \label{d:frag} 
{\rm An \emph{I-algebra} is an algebra $\Al = \left\langle A, \s, \neg \right\rangle$ satisfying the following equations:
\begin{enumerate}[ ]
\item (I1) $ \quad (x \supset x) \supset y \approx y$ 
\item (I2) $ \quad  x \supset (y \supset z) \approx (x \s y) \s (x \s z) \approx y \supset (x \supset z)$
\item (I3) $ \quad ((x \supset y) \supset x) \supset x  \approx x \s x $ 
\item (I4) $\quad  x \supset (\neg y \supset z) \approx \neg (x \s y) \s z$
\item (I5) $\quad \neg \neg x \approx x$
\item  (I6) $ \quad  p(x,y,x) \approx p(x,y,y)$
\end{enumerate}
where $p(x,y,z)$ is an abbreviation for $$(x \s y) \s ((y \s x) \s ((\neg x \s \neg y) \s ((\neg y \s \neg x) \s z))).$$}
We shall denote by $\ia$ the variety of $I$-algebras.
\end{definition} 

It follows from Proposition \ref{p:lem2} that the $\{ \s,  \neg \}$-reduct of any implicative   bilattice satisfies axioms (I1) to (I6), hence is an $I$-algebra. Let us now state some properties of these algebras that will be used in the rest of the section.  As we have done in the former chapter, we abbreviate  $ a = a \s a$ as $E(a)$.

\begin{proposition}
\label{p:ilem-0}
Let $\Al = \langle A, \s, \neg \rangle $ be an $I$-algebra. Then, for all $a, b, c, d, e \in A$: 
\begin{enumerate}[(i)]
\item  $E(a \s a)$,
\item  $a \s b = a \s (a \s b) = (a \s a) \s (a \s b)$,
\item If $a \s b = d \s d$ and $b \s c = e \s e$, then $a \s c = (a \s e) \s (a \s e)$,
\item $\neg(a \s \neg b) \s c = a \s (b \s c)$.
\item $a \s a = \neg a \s \neg a$.
\end{enumerate}
\end{proposition}
\begin{proof}
(i) Follows immediately from (I1).

(ii) Let $a, b \in A$. By (I2) and (I1), we have $a \s (a \s b) = (a \s a) \s (a \s b) = a \s b$.

(iii) Assume $a \s b = d \s d$ and $b \s c = e \s e$. By  (I2), we have $a \supset (b \supset c) =  (a \s b) \s (a \s c)$.  Since $a \s b = d \s d$, by (I1) we have $(a \s b) \s (a \s c) = a \s c$. So $a \s c = a \supset (b \supset c)$. Now, by (I2) again, we have $a \supset (e \supset e) = (a  \s e) \s (a  \s e)$. Since $b \s c = e \s e$, it follows that  $a \supset (b \supset c) = (a  \s e) \s (a  \s e)$. Hence, $a \s c =  (a  \s e) \s (a  \s e)$.

(iv) By (I4) $\neg(a \s \neg b) \s c = a \s (\neg\neg b \s c)$. Then by (I5) we obtain $\neg(a \s \neg b) \s c = a \s (b \s c)$. 

(v). We will prove that $E((a \s a) \s (\neg a \s \neg a))$, $E((\neg a \s \neg a)  \s (a \s a))$, $E(\neg (a \s a) \s \neg (\neg a \s \neg a))$ and $E(\neg (\neg a \s \neg a)  \s \neg (a \s a))$. The result will then follow by (I1) and (I6). The first two follow immediately by (I1) and (i). As to the other two, we have
\begin{align*}
\neg (a \s a) \s \neg (\neg a \s \neg a) & = a \s (\neg a \s  \neg (\neg a \s \neg a)) & \textrm{by (I4)}   \\
& = \neg a \s ( a \s  \neg (\neg a \s \neg a)) & \textrm{by (I2)}   \\
& = \neg a \s ( \neg \neg a \s  \neg (\neg a \s \neg a)) & \textrm{by (I5)}   \\
& = \neg (\neg  a \s   \neg a) \s  \neg (\neg a \s \neg a) & \textrm{by (I4)}
\end{align*}
and
\begin{align*}
\neg (\neg a \s \neg a) \s \neg (a \s a) & = \neg a \s (\neg \neg a \s  \neg (a \s a) & \textrm{by (I4)}   \\
& = \neg a \s ( a \s  \neg (a \s a)) & \textrm{by (I5)}   \\
& = a \s ( \neg  a \s  \neg (a \s a)) & \textrm{by (I2)}   \\
& = \neg (  a \s   a) \s  \neg (a \s a) & \textrm{by (I4)}.
\end{align*}
Thus the result easily follows.
\end{proof}

\begin{proposition} \label{p:ilem} 
Let $\Al = \langle A, \s, \neg \rangle $ be an $I$-algebra. Then, for all $a, b, c, d, e \in A$: 
\begin{enumerate}[(i)]
\item $E(a \s b)$ and $E(b \s a)$ if and only if $a \s f = b \s f$ for all $f \in A$.
\item $E (((a \s b) \s ((b \s a) \s b)) \s ((b \s a) \s ((a \s b) \s a))).$
\item If $E(a \s b)$, $E(b \s a)$,  $E(c \s d)$ and $E(d \s c)$, then $E((a \s c) \s (b \s d))$ and \\ $E((b \s d) \s (a \s c)).$
\end{enumerate}
\end{proposition}

\begin{proof}

(i). Assume $a \s b = (a \s b) \s (a \s b)$ and $b \s a = (b \s a) \s (b \s a)$. Let $f \in A$. Then, using (i) and (I2), we have 
\begin{align*}
a \s f & =  
(a \s b) \s (a \s f) \\ & =
a \s (b \s f) \\ & =
b \s (a \s f) \\ & =
(b \s a) \s (b \s f) \\ & =
b \s f.
\end{align*}
Conversely, assume $a \s f = b \s f$ for all $f \in A$. Then, using (i), we have 
\begin{align*}
a \s b & =  
b \s b \\ & =
(b \s b) \s (b \s b) \\ & =
(a \s b) \s (a \s b).
\end{align*}
By symmetry, we obtain $b \s a = (b \s a) \s (b \s a) $.

(ii). Using (i) in Proposition \ref{p:ilem-0}  and (I2), we have
\begin{align*}
((a \s b) \s (b \s a)) \s ((a \s b) \s (b \s a))  & = \\
(a \s b) \s  ((b \s a) \s (b \s a))  & =\\
(a \s b) \s  ( ((b \s a) \s b) \s ((b \s a) \s a) )  & =\\
((a \s b) \s ((b \s a) \s b)) \s ((a \s b) \s ((b \s a) \s a) ) & =\\
((a \s b) \s ((b \s a) \s b)) \s ( (b \s a) \s ( (a \s b)  \s a) ).
\end{align*}
Hence, by (i), the result immediately follows.

(iii). Assume that $E(a \s b)$, $E(b \s a)$,  $E(c \s d)$ and $E(d \s c)$. From (i) we have  $a \s f = b \s f$ and $c \s f = d \s f$,  for all $f \in A$.  Then, using (i) and (I2), we have 
\begin{align*}
(a \s c) \s (b \s d) & =
((a \s b) \s (a \s c)) \s ( (b \s a) \s (b \s d))  \\ & =
(a \s (b \s c)) \s ( b \s (a \s d))  \\ & =
(a \s (b \s c)) \s ( a \s (b \s d))  \\ & =
a \s ( (b \s c) \s (b \s d))  \\ & =
a \s ( b \s  (c \s d))  \\ & =
(a \s  b) \s  (a \s (c \s d))  \\ & =
a \s (c \s d)  \\ & =
a \s ((c \s d) \s (c \s d)) \\ & =
(a \s (c \s d)) \s (a \s (c \s d)).
\end{align*}
From this, it easily follows that $$ (a \s c) \s (b \s d) = ((a \s c) \s (b \s d)) \s ((a \s c) \s (b \s d)). $$ By symmetry, we also have $$ (b \s d) \s (a \s c) = (  (b \s d) \s (a \s c)) \s ((b \s d) \s (a \s c)). $$ Now, applying (i), we obtain the result.
\end{proof}

It may perhaps be interesting to observe that, in any  $I$-algebra $\Al = \la A, \s, \neg \ra$, we can define the following relations:
\begin{align*}
\leq_t \; & = \{ \la a, b \ra : E(a \s b) \textrm{ and } E(\neg b \s \neg a) \} \\
\leq_k \; & = \{ \la a, b \ra : E(a \s b) \textrm{ and } E(\neg a \s \neg b) \}.
\end{align*}
It is easy to check that $\leq_t$ and $\leq_k$ are partial orders and that the negation operator is anti-monotonic w.r.t.\ $\leq_t$ and monotonic w.r.t.\ $\leq_k$. Indeed, as the notation suggests, if the $I$-algebra is the reduct of an implicative bilattice, then these relations coincide with the two bilattice orders. 




In the following propositions we describe how, starting from an $I$-algebra,  it is possible to construct a Tarski algebra. This construction will later be employed to prove that any $I$-algebra can be embedded into an implicative bilattice.

Recall that a Tarski algebra is an algebra  $\la A, \s \ra$ satisfying the following identities: 

\begin{enumerate}[ ]
\item (T1) $ (x \s y) \s x \approx x$ 
\item (T2) $x \supset (y \supset z) \approx  y \supset (x \supset z)$
\item (T3) $(x \supset y) \supset y \approx (y \supset x) \supset x $.
\end{enumerate}
Note that in a Tarski algebra the term $x \s x$ is an algebraic constant, that is $x \s x \approx y \s y$ holds in every Tarski algebra. We denote this constant term by 1. The canonical order $\leq$ of a Tarski algebra is defined by
$$a \leq b \; \;\; \text{ iff } \;\;  a \s b = 1$$
and every pair of elements $a, b$ has a supremum in this order defined by
$$a \vee b = (a \s b) \s b.$$

Let us introduce the following:

\begin{definition} \label{d:req}
Let $\Al = \langle A, \s, \neg \rangle $ be an $I$-algebra. The relation $\cc \; \subseteq A \times A$ is defined  as follows: $$a \cc b \;\; \text{ iff } \;\; E(a \s b) \textrm{ and } E(b \s a).$$

\end{definition}

Note that, by Proposition \ref{p:ilem}  (i), we have $$a \cc b  \;\; \text{ iff } \;\;   a \s c = b \s c, \textrm{ for all } c \in A.$$
Note also that, if $\Al$ is the reduct of an implicative bilattice, then  $\cc$ coincides with the relation $\sim_1$ that was introduced in Chapter \ref{ch:int} 
(Definition \ref{d:cong}) in order  to prove the Representation Theorem for interlaced pre-bilattices. 

From Definition  \ref{d:req} and Propositions \ref{p:ilem-0} (i) and (iii) 
 it follows immediately that $\cc$  is an equivalence relation. Moreover, from Proposition \ref{p:ilem}  (iii) it follows that $\cc$ is compatible with the operation $\s$. So we can define an operation $\s$ in the quotient $A/\!\!\sim$ given, for every $a, b \in A$,  by
$$[a] \s [b] = [a\s b].$$ We will first show that the algebra $\Al/{\cc} = \langle A/{\cc}, \s \rangle$  is a Tarski algebra with the property that its canonical order is a lattice order. Then we  will prove that the algebra $\Al$  can be embedded    into an implicative product  bilattice  constructed from this quotient algebra (by embedding we mean here an injective map which is a homomorphism w.r.t. to the operations $\{ \s, \neg \}$).  Hence we will have  shown that $I$-algebras turn out to be subreducts of implicative   bilattices.

\begin{proposition} \label{p:icong} 
Let $\Al = \langle A, \s, \neg \rangle $ be an $I$-algebra. Then the structure 
$ \Al /{\cc} = \la A / {\cc}, \s  \ra $ is a Tarski algebra.
\end{proposition}
\begin{proof}
 Let us check that the equations (T1) to (T3) of the definition of Tarski algebra hold in $\Al /{\cc}$. We denote by $[a]$  the equivalence class of $a \in A$ modulo $\cc$.

(T1) Let $a, b \in A$. In order to show that (T1) holds in $\Al /{\cc}$ we need to prove that $[(a \s b) \s a] = [a]$. Thus, by definition of $\cc$, we have to prove that $E(((a \s b) \s a) \s a)$ and $E(a \s ((a \s b) \s a)$, that is, that  $$((a \s b) \s a) \s a = (((a \s b) \s a) \s a) \s (((a \s b) \s a) \s a)$$ and $$a \s ((a \s b) \s a)= (a \s ((a \s b) \s a)) \s (a \s ((a \s b) \s a)).$$ 
 To prove the former note that by (I1) we have $(a \s a) \s (a \s a) = a \s a$ and by (I3),   $((a \s b) \s a) \s a  = a \s a$. So $$((a \s b) \s a) \s a = (a \s a) \s (a \s a).$$ Therefore, substituting $((a \s b) \s a) \s a$ for $a \s a$ on the right of the equality symbol we obtain $$((a \s b) \s a) \s a = (((a \s b) \s a) \s a) \s (((a \s b) \s a) \s a).$$  Now,  to prove the later note that by (I2) $a \s ((a \s b) \s a) = (a \s b) \s (a \s a)$. By Proposition  \ref{p:ilem-0} (ii), $a \s a = (a \s a) \s (a \s a)$. Thus we have 
 $$a \s ((a \s b) \s a) = (a \s b) \s  ((a \s a) \s (a \s a)).$$ Therefore, by (I2), 
 $$a \s ((a \s b) \s a) = ((a \s b) \s (a \s a)) \s ((a \s b) \s (a \s a)).$$ Hence, since $a \s ((a \s b) \s a) = (a \s b) \s (a \s a)$, we obtain the desired conclusion. 
 
 (T2) Let $a, b, c \in A$. We have to prove that $[a]  \s ([b] \s [c]) = [b] \s ([a] \s [c])$, that is $[a \s (b \s c)] = [b \s (a \s c)]$. So we have to show that $E((a \s (b \s c)) \s (b \s (a \s c)))$ $E((b \s (a \s c)) \s (a \s (b \s c)))$, that is, 
$$(a \s (b \s c)) \s (b \s (a \s c)) = $$ $$((a \s (b \s c)) \s (b \s (a \s c))) \s ((a \s (b \s c)) \s (b \s (a \s c)))$$
and 
$$(b \s (a \s c)) \s (a \s (b \s c)) = $$  $$((b \s (a \s c)) \s (a \s (b \s c))) \s ((b \s (a \s c)) \s (a \s (b \s c))).$$
Note that by (I2)  $a \s (b \s c) = b \s (a \s c)$ and $b \s (a \s c) = a \s (b \s c)$. Moreover, by  Proposition \ref{p:ilem-0} (ii), for every $c \in A$, $c \s c = (c \s c) \s (c \s c)$. So, the two desired  results follow.

(T3) Let $a, b \in A$. We have to show that $([a] \s [b]) \s [b] = ([b] \s [a]) \s [a]$. 
In order to do it we first show that for every $c, d \in A$, $$([c] \s [d]) \s [d]  = [(c \s d) \s ((d \s c) \s d))].$$  Note that using that  (T1) holds, we have $[(c \s d) \s ((d \s c) \s d)] = [c \s d] \s ([d] \s [c]) \s [d])  = [c \s d] \s [d] = [(c \s d) \s d]$.  Now, using the fact just proved and (I2) 
 we have 
\begin{align*}
([a] \s [b]) \s [b] & = 
[((a \s b) \s ((b \s a) \s b))] \\ & =
[((b \s a) \s ((a \s b) \s a))] \\ & =
([b] \s [a]) \s [a].
\end{align*}
\end{proof}

Let $\Al = \langle A, \s, \neg \rangle $ be an $I$-algebra. Since $ \Al /{\cc} = \la A / {\cc}, \s  \ra $ is a Tarski algebra, any two elements $[a], [b] \in A / {\cc}$ have a supremum in the canonical order, defined by $[a] \vee [b] = ([a] \s [b]) \s [b]$. We will show that they also have  an infimum, defined by  $$[a] \wedge [b] : = [\neg (a \s \neg b)].$$
This definition does not depend on the representatives, because if $a_1 \cc b_1$ and $a_2 \cc b_2$, then $a_1 \s c = b_1 \s c$ and $a_2 \s c = b_2 \s c$ for every $c \in A$, so that 
$b_1 \s ( b_2 \s c) = a_1 \s (a_2 \s c)$. Then, using Proposition  \ref{p:ilem-0} (iv), we obtain that for every $c \in A$, $\neg (a_1 \s \neg a_2) \s c =  \neg (b_1 \s \neg b_2) \s c$. Therefore,  $\neg (a_1 \s \neg a_2) \cc  \neg (b_1 \s \neg b_2)$.

\begin{proposition} \label{p:iclass}
Let $\Al = \langle A, \s, \neg \rangle $ be an $I$-algebra. Then the algebra $\la A / {\cc}, $ $\land, \lor, \s  \ra$ is a classical implicative lattice.  
\end{proposition}
\begin{proof}

Let us check that the operations $\vee$ and $\wedge$ satisfy the  lattice axioms. That $\vee$ satisfies the join semi-lattice axioms is known. Let us show that $\wedge$  satisfies the meet semi-lattice axioms.  To prove idempotency we show that for every $c \in A$, $\neg (a \s \neg a) \s c  = a \s c$. By Proposition \ref{p:ilem-0} (iv) and (ii) we have  $\neg (a \s \neg a) \s c = a \s (a \s c) = a \s c$. Commutativity and associativity also follow easily from Proposition \ref{p:ilem-0} (iv) and (ii). As to the absoption laws, we have to prove that
$$ [a] \land ([a] \lor [b]) = [ \neg (a \s \neg ((a \s b) \s b)) ] = [a] $$
and
$$ [a] \lor ([a] \land [b]) = [(\neg (a \s \neg b) \s a) \s a]  = [a]. $$
The first equality holds because, using Proposition \ref{p:ilem-0} (iv), (I2) and (I1), we have, for all $c \in \Al$,
\begin{align*}
\neg (a \s \neg ((a \s b) \s b)) \s c & = 
a \s (((a \s b) \s b) \s c) \\ & = 
(a \s ((a \s b) \s b)) \s (a \s c ) \\ & = 
((a \s b) \s (a \s b)) \s (a \s c ) \\ & = 
(a \s c ).
\end{align*}
The second one is also proved using  Proposition \ref{p:ilem-0} (iv), (I2) and (I1) because
\begin{align*}
((\neg (a \s \neg b) \s a) \s a) \s c & = 
((a \s ( b \s a)) \s a) \s c \\ & = 
((b \s ( a \s a)) \s a) \s c \\ & = 
(((b \s  a) \s (b \s a)) \s a) \s c \\ & = 
a \s c.
\end{align*}

Now we prove that   $\la A / {\cc}, \land, \lor, \s  \ra$ is a classical implicative lattice. We have to show that    $ [a] \land ([a] \s [b]) = [a] \land [b] $ and $ [a] \lor ([a] \s [b]) = [a \s a] $. As to the first, note that by Proposition \ref{p:ilem-0} (iv) and (I2) we have
\begin{align*}
\neg (a \s \neg (a \s b)) \s c & = 
a \s ((a \s b) \s c) \\ & = 
(a \s b) \s (a \s c) \\ & = 
a \s (b \s c) \\ & = 
\neg (a \s \neg b) \s c
\end{align*}
for revery $c \in A$. Hence  $[a] \land ([a] \s [b]) = [a] \land ([a \s b]) = [\neg (a \s \neg (a \s b)]= [\neg (a \s \neg b)] = [a] \land [b] $.

As to the second, recall that, by (T1), we have $([a] \s [b]) \s [a] = [a]$. Therefore we have $$ [a] \lor ([a] \s [b]) = (([a] \s [b]) \s [a]) \s [a] = [a] \s [a] = [a \s a]. $$
\end{proof}

We are now able to prove what we claimed: that any $I$-algebra can be embedded into an implicative bilattice and, moreover, that the embedding we define is in some sense a minimal one:

\begin{thm}
\label{t:icong-2} 
 Let $\Al$ be an $I$-algebra. Then:
\begin{enumerate}[(i)]   
  \item  there is an embedding $h: A \longrightarrow{A/\!\!\sim} \times {A/\!\!\sim} $ of $\Al$ into the implicative   bilattice $\Al[B] = \la {\Al/\!\!\sim}  \odot  {\Al/\!\!\sim}, \s \ra$ defined, for all $a \in A$, as $h(a) = \la [a], [\neg a]\ra$,
  \item  $\Al[B]$ is generated by the set $h[A]$,
  \item  if $f: A \longrightarrow B'$ is a homomorphism from $\Al$ to an implicative   bilattice $\Al[B]'$, then there is a unique map $f': {A/{\cc} \times A/{\cc}} \longrightarrow B'$ which is also a homomorphism from $\Al[B]$ to $\Al[B]'$ such that $f' \cdot h = f$. Moreover, if $f$ is injective, so is $f'$.
\end{enumerate}  

 
\end{thm} 
\begin{proof}
(i) Let $\Al = \langle A, \s^A, \neg^A  \rangle $ and $\Al[B] = \langle {A/\!\!\sim} \times {A/\!\!\sim}, \land^B, \lor^B, \otimes^B, \oplus^B, \s^B, \neg^B  \rangle $. We first prove 
that $h$ is injective. Assume  $h(a) = h(b)$, so that $[a] = [b]$ and $[\neg a] = [\neg b]$, which means that
\begin{align*}
a \s^A b & = (a \s^A b) \s^A (a \s^A b) \\
b \s^A a & = (b \s^A a) \s^A (b \s^A a) \\
\neg a \s^A \neg b & = (\neg a \s^A \neg b) \s^A (\neg a \s^A \neg b) \\
\neg b \s^A \neg a & = (\neg b \s^A \neg a) \s^A (\neg b \s^A \neg a). 
\end{align*}
Using (I1), it easy to see that these conditions imply $p(a,b,a) = a  $ and $p(a,b,b) = b  $. Hence, by (I6), we have $a = b$.
 
Now we prove that $h$ is a $(\s, \neg)$-homomorphism. It is easy to check that $h (\neg^A a) = \neg^B h(a)$.  To prove that $h (a \s^A b) = h(a) \s^B h(b)$, using (I5), 
we have 
\begin{align*}
h (a \s^A b) & = \la [a \s^A b], [\neg (a \s^A b)]  \ra \\ & =
\la [a \s^A b], [\neg (a \s^A \neg \neg b)]  \ra \\ & =
\la [a] \s [b] , [a] \land [\neg b]]   \ra \\ & =
h(a) \s^B h(b).
\end{align*}

To conclude, note that $ \la [a], [b] \ra \in h[A]$ if and only if $a = \neg b$. 

(ii). We will prove that, for every $u \in {A/\!\!\sim}  \times  {A/\!\!\sim}$, there are $u_1, u_2, u_3, u_4  \in h[A]$ such that
$$u = (u_1 \otimes^B u_2 ) \oplus^B (u_3 \otimes^B u_4).$$
Let $u = \la [a], [b] \ra \in {A/\!\!\sim}  \times  {A/\!\!\sim}$. The desired elements are: $u_1 = \la [a], [\neg a] \ra$, $u_2 = \la [b \s a] , [\neg (b \s a)] \ra$, $u_3 = \la [\neg b], [b] \ra$ and $u_4 = \la [\neg (a \s b)], [a \s b] \ra$. It is clear that $u_1, u_2, u_3, u_4 \in h[A]$. Now notice that from the definition of $\wedge$ follows that $[\neg (b \s a)] = [b] \wedge [\neg a]$. Therefore, 
$$u_1 \otimes^B u_2 = \la [a] \wedge [b \s a], [\neg a] \wedge [\neg (b \s a)]\ra =
\la [a], [b] \wedge [\neg a]\ra.$$
Similarly, 
$$u_3 \otimes^B u_4 = \la [\neg b] \wedge [\neg (a \s b)], [b] \wedge [a \s b]\ra =
\la [a] \wedge [\neg b], [b]\ra.$$
Therefore,
$$(u_1 \otimes^B u_2) \oplus^B (u_3 \otimes^B u_4) = \la [a], [b] \wedge [\neg a]\ra \oplus^B 
\la [a] \wedge [\neg b], [b]\ra =$$ $$ \la [a] \vee ([a] \wedge [\neg b]), ([b] \wedge [\neg a]) \vee [b]\ra = \la [a], [b]\ra = u.$$

(iii). Assume $f: A \longrightarrow B'$ is an embedding of $\Al$ into an implicative   bilattice $\Al[B]'$. By Theorem~\ref{t:reprimp2}, we may identify $\Al[B]'$ with its isomorphic image  $\la {\Al[B']/\!\!\sim} \odot {\Al[B']/\!\!\sim}, \s^{B'} \ra$. 
We will prove that the desired embedding is given by the map $f': {A/\!\!\sim}  \odot  {A/\!\!\sim} \longrightarrow  {B'/\!\!\sim} \odot {B'/\!\!\sim}$ defined as follows: for all $a, b \in A$,
$$f'(\langle [a] , [b] \rangle) = \langle [f(a)] , [f(b)] \rangle.$$
To prove that $f'$ is one-to-one, assume $f'(\langle [a] , [b] \rangle) = f'(\langle [c] , [d] \rangle) $ for some $a,b,c,d \in A$. By definition, this means that $ [f(a)] =  [f(c)]$ and $ [f(b)] =  [f(d)]$. By the definition of $\cc$, we have that 
$$f(a) \s^{B'} f(c)  = (f(a) \s^{B'} f(c)) \s^{B'} (f(a) \s^{B'} f(c)).$$
 Since $f$ is a $\{ \s \}$-homomorphism, from the previous equality we obtain 
 $$f(a \s^{A} c)  = f ((a \s^{A} c) \s^{A} (a \s^{A} c)).$$ By the injectivity of $f$, this implies $a \s^{A} c  = (a \s^{A} c) \s^{A} (a \s^{A} c)$, i.e.\ $[a] = [c]$. In a similar way we obtain $[b]=[d]$, so we conclude that $f'$ is one-to-one. It remains to prove that it is indeed a homomorphism. The case of negation is almost immediate:
\begin{align*}
f'(\neg^B \langle [a] , [b] \rangle) & = f'(\langle [b] , [a] \rangle) \\ 
& = \langle [f(b)] , [f(a)] \rangle \\ 
& = \neg^{B'} \langle [f(a)] , [f(b)] \rangle \\ 
& = \neg^{B'} f'(\langle [a] , [b] \rangle).
\end{align*}
The cases of the remaining bilattice connectives are also easy (recall that the relation $\cc$ is compatible with all the connectives except $\neg$). For instance, in the case of conjunction, we have 
\begin{align*}
& f'( \langle [a_1], [a_2] \rangle \land^B \langle [b_1], [b_2] \rangle) = \\
& = f'( \langle [a_1] \land^{{A/\!\sim}} [b_1], [a_2] \lor^{{A/\!\sim}} [b_2] \rangle) = \\
& = f'( \langle [\neg^{A} (a_1 \s^{A} \neg^{A} b_1], [(a_2 \s^{A} b_2) \s^{A} b_2] \rangle) = \\
& = \langle [f (\neg^{A} (a_1 \s^{A} \neg^{A} b_1)], [f((a_2 \s^{A} b_2) \s^{A} b_2)] \rangle = \\
& = \langle [\neg^{A} (f(a_1) \s^{A} \neg^{A} f(b_1))], [(f(a_2) \s^{A} f(b_2)) \s^{A} f(b_2)] \rangle = \\
& = \langle [f(a_1)] \land^{{A/\!\sim}} [f(b_1)], [f(a_2)] \lor^{{A/\!\sim}} [f(b_2)] \rangle = \\
& = \langle [f(a_1)],  [f(a_2)] \ra \land^{B'} \la [f(b_1)], [f(b_2)]  \rangle = \\
& = f' (\langle [(a_1)],  [(a_2)] \ra) \land^{B'} f'(\la [(b_1)], [(b_2)]  \rangle).
\end{align*}
The proofs corresponding to the other  lattice connectives are similar. Let us check the case of implication:
\begin{align*}
f'( \langle [a_1], [a_2] \rangle \supset^B \langle [b_1], [b_2] \rangle) 
& = f'( \langle [a_1] \s^{{A/\!\sim}}  [b_1], [a_1] \land^{{A/\!\sim}} [b_2] \rangle) \\
& = f'( \langle [a_1 \s^{A}  b_1], [\neg^A (a_1 \s^A \neg^{A} b_2] \rangle) \\
& =  \langle [f (a_1 \s^{A}  b_1)], [f ( \neg^A (a_1 \s^A \neg^{A} b_2)] \rangle \\
& =  \langle [f (a_1) \s^{A}  f(b_1)], [\neg^A (f(a_1) \s^A \neg^{A} f(b_2))] \rangle \\
& =  \langle [f (a_1)] \s^{{A/\!\sim}}  [f(b_1)], [(f(a_1)] \land^{{A/\!\sim}} [f(b_2)] \rangle \\
& =  \langle [f(a_1)], [f(a_2)] \rangle \supset^B \langle [f(b_1)], [f(b_2)] \rangle \\
& = f'( \langle [a_1], [a_2] \rangle) \supset^B f'(\langle [b_1], [b_2] \rangle).
\end{align*}
Finally, it is easy to see that $f' \cdot h = f$, for we have, for all $a \in A$,
$$
f' \cdot h (a) = f' (\la [a], [\neg a] \ra) = \la [f(a)], [f(\neg a)] \ra = \la [f(a)], [\neg f (a)] \ra = f(a).
$$
From this it follows that the map $f'$ is unique, and it is also easy to see that, if $f$ is injective (an embedding), then $f'$ is also injective.
Let us note that the previous result may also be proved without relying on Theorem~\ref{t:reprimp2}. In this case we have to define, for all $a, b \in A$,
$$f'(\langle [a] , [b] \rangle) = (f(a) \otimes^{B} f(b \s^{A} a)) \oplus^{B} (\neg^{B} f(b)  \otimes^{B} f (\neg^{A} (a \s^{A} b )) ). $$
 
\end{proof}

The previous theorem enables us to obtain some additional information about the variety of $I$-algebras:

\begin{thm} \label{t:igen}
The variety $\ia$ is generated by $\Al[A_{4}]$, the four-element $I$-algebra which is the $\{ \s, \neg \}$-reduct of the implicative   bilattice $\fours$.
\end{thm} 

\begin{proof}
We will prove that if an equation $\phi \approx \psi$ does not hold in the variety of $I$-algebras, then it does not hold in the four-element $I$-algebra. By assumption, there is some $I$-algebra $\Al$ such that $\phi \approx \psi$ does not hold in $\Al$. By Theorem \ref{t:icong-2}, we know that $\Al$ can be embedded into some implicative   bilattice $\Al[B]$. So $\Al[B]$ does not satisfy $\phi \approx \psi$. By Theorem \ref{t:vargen}, this implies that $\phi \approx \psi$ does not hold in $\fours$. Hence $\phi \approx \psi$ does not hold in the $\{ \s, \neg \}$-reduct of $\fours$.
\end{proof}


We shall now prove a result on the congruences of $I$-algebras that will enable us to characterize the subvarieties of $\ia$. Recall that $E(a)$ is an abbreviation for $a = a \s a$, and we also use $p(a,b,c)$ to abbreviate 
$$(a \s b) \s ((b \s a) \s ((\neg a \s \neg b) \s ((\neg b \s \neg a) \s c))).$$

\begin{lem} \label{lem:edpc}
Let $\Al$ be an $I$-algebra and $a,b,c,d, c', d' \in A$. Then:
\begin{enumerate}[(i)]
  \item $p(a, b, c) = p(a, b, d)$ implies $p(a, b, \neg c) = p(a, b, \neg d)$
  \item $p(a, b, c) = p(a, b, d)$ and  $p(a, b, c') = p(a, b, d')$ imply  $p(a, b, c \s c') = p(a, b, d \s d').$
\end{enumerate}
\end{lem}

\begin{proof}
(i).
Observe that, applying (I4) several times, we have 
\begin{align*}
p(a, b, c) & = (a \s b) \s ((b \s a) \s ((\neg a \s \neg b) \s ((\neg b \s \neg a) \s c)))      \\
    &  = \neg ((a \s b) \s \neg  (b \s a)) \s     ((\neg a \s \neg b) \s ((\neg b \s \neg a) \s c)) \\
    &  = \neg (\neg ((a \s b) \s \neg  (b \s a)) \s \neg (\neg a \s \neg b) ) \s ((\neg b \s \neg a) \s c) \\
    &  = \neg (\neg (\neg ((a \s b) \s \neg  (b \s a)) \s \neg (\neg a \s \neg b) ) \s \neg (\neg b \s \neg a)) \s c.
\end{align*} 
Then we may abbreviate $$m =  \neg (\neg (\neg ((a \s b) \s \neg  (b \s a)) \s \neg (\neg a \s \neg b) ) \s \neg (\neg b \s \neg a)) $$ and refomulate the assumption as $m \s c = m \s d$. We shall prove that $E((m \s \neg c) \s (m \s \neg d))$,  $E(\neg (m\s \neg c) \s \neg (m \s \neg d))$, $E((m \s \neg d) \s (m \s  \neg c))$,   and $E(\neg (m \s \neg d) \s \neg (m \s  \neg c))$. The result will then follow by (I1) and (I6). Clearly, by symmetry, it is sufficient to prove the first two cases. As to the first, we have
\begin{align*}
(m \s \neg c) \s (m \s \neg d)) & = m \s (\neg c \s \neg d)     & \textrm{by (I2)} \\
& = \neg (m \s  c) \s \neg d    & \textrm{by (I4)} \\
& = \neg (m \s  d) \s \neg d    & \textrm{by assumption} \\
& = m \s  (\neg d \s \neg d)     & \textrm{by (I4)} \\
& = (m \s  \neg d)  \s (m \s \neg d) & \textrm{by (I2).}
\end{align*} 
Then, applying Proposition \ref{p:ilem-0} (i), the result easily follows. As to the second, we have
\begin{align*}
\neg (m \s \neg c) \s \neg (m \s \neg d)) & = m \s (c \s \neg (m \s \neg d))     & \textrm{by (I4)} \\
 & = (m \s c) \s (m \s \neg (m \s \neg d))     & \textrm{by (I2)} \\
 & = (m \s d) \s (m \s \neg (m \s \neg d))       & \textrm{by assumption} \\
& = m \s  ( d \s  \neg (m \s \neg d))     & \textrm{by (I2)} \\
& = \neg (m \s  \neg d) \s  \neg (m \s \neg d)     & \textrm{by (I4).}
\end{align*} 
Applying Proposition \ref{p:ilem-0} (i) again we obtain the desired result.

(ii). Using the abbreviation introduced in (i), the assumptions become $m \s c = m \s d$ and $m \s c' = m \s d'$. We have
\begin{align*}
 m \s (c \s c')  & = (m \s c) \s (m \s c')    & \textrm{by (I2)} \\
 & = (m \s d) \s (m \s d')     & \textrm{by assumption} \\
  & = m \s (d \s d')      & \textrm{by (I2).}
\end{align*} 
\end{proof}

Recall that a variety of algebras is said to have \emph{equationally definable principal congruences} (abbreviated EDPC) if there is a finite set $\Sigma$ of equations of the form $t (x, y, z, u) \approx t'(x, y, z, u) $ such that, for any algebra $\Al$ in the variety and for all elements $a, b, c, d \in A$, it holds that $\la c, d \ra \in \Theta(a,b) $ if and only if $t (a, b, c, d) = t'(a, b, c, d) $ for all equations in $\Sigma$. EDPC is a rather strong property: in particular (see \cite[Theorem 1.2]{BP-EDPC-I}) it implies congruence-distributivity and the congruence extension property.

\begin{thm} \label{thm:edpc}
The variety $\ia$ has EDPC. 
\end{thm}

\begin{proof}
For any $\Al \in \ia$ and $a, b \in A$, let us denote by $ \Theta(a, b) $ the congruence generated by $(a, b)$. We shall prove that, for all $c, d \in A$, $\la c, d \ra \in \Theta(a, b) $ if and only if $p(a, b, c) = p(a, b, d)$. Let us then set $\theta = \{\la c, d \ra \in A \times A :  p(a, b, c) = p(a, b, d)\}$. Clearly $\theta$ is an equivalence relation and, by (I6), we have $\la a, b \ra \in \theta$. By Lemma \ref{lem:edpc} it follows that $\theta $ is a congruence of $\Al$. Hence we have that $ \Theta(a, b) \subseteq \theta$. 

To prove the other inclusion, assume $\la c, d \ra \in \theta$. Recall that, by Proposition \ref{p:ilem-0}, we have  $a \s a = \neg a \s \neg a$ for all $a \in A$. Then it is not difficult to see that the assumptions imply that the following elements belong to the same equivalence class modulo $ \Theta(a, b)$: 
$$\{ a \s a, \ a \s b , \ b \s a ,  \ b \s b ,  \neg b \s \neg b , \ \neg a \s \neg b, \ \neg b \s \neg a  \}.$$

Now, using (I1), we easily obtain 
$$ (a \s a ) \s c = c \ \Theta(a, b) \ (\neg b \s \neg a ) \s c
$$
as well as
$$ (a \s a ) \s c  = c \ \Theta(a, b) \ (\neg a \s \neg b) \s ((\neg b \s \neg a ) \s c)
$$
and so forth, so that we may conclude that $ \la c,  \ p(a, b, c) \ra \in \Theta(a, b)$. By symmetry we have $ \la d,  \ p(a, b, d) \ra \in \Theta(a, b)$, hence the assumption implies that  $\la c, d \ra \in \Theta(a, b)$. 
%
%
\end{proof}

We now immediately have the following:

\begin{cor} \label{cor:subdir_i}
Up to isomorphism, there are five subdirectly irreducible algebras in $\ia$, namely: $\Al[A_4]$, the four-element $\{ \s, \neg \}$-reduct of $\fours$, the two three-element subreducts (let us denote them by $\Al[A^{\top}_3]$ and $\Al[A^{\bot}_3]$) whose universes are, respectively, $\{ \false, \true, \top \}$ and $\{ \false, \true, \bot \}$, the two-element one $\Al[A_2]$ with universe  $\{ \false, \true \}$ and the trivial one with universe $\{ \top \}$. Hence, $\ia$ has exactly four proper non-trivial subvarieties, which are generated, respectively, by  $\{ \Al[A^{\top}_3], \Al[A^{\bot}_3] \}$, by $\Al[A^{\top}_3]$, by  $\Al[A^{\bot}_3]$ and by $\Al[A_2]$.
\end{cor}

\begin{proof}
We know, by Theorem \ref{thm:edpc}, that $\ia$ is congruence-distributive. Then, by J\'{o}nsson's Lemma  \cite[Corollary IV.6.10]{BuSa00}, the subdirectly irreducible members of $\ia$ belong to $HS(\Al[A_4])$, and it is not difficult to check that they coincide with the four algebras mentioned in the statement. Moreover, since $\Al[A_2]$ is a subalgebra of both 
$ \Al[A^{\top}_3]$ and $\Al[A^{\bot}_3]$, one easily sees that the only possible combinations for the proper subvarieties of $\ia$ are $V(\{ \Al[A^{\top}_3], \Al[A^{\bot}_3], \Al[A_2] \}) = V(\{ \Al[A^{\top}_3], \Al[A^{\bot}_3]\})$,  $V(\{ \Al[A^{\top}_3],  \Al[A_2] \}) = V( \Al[A^{\top}_3])$, $V(\{  \Al[A^{\bot}_3], \Al[A_2] \}) = V(  \Al[A^{\bot}_3] ) $ and $V( \Al[A_2])$.
\end{proof}

Another consequence of Theorem \ref{thm:edpc} is that the variety $\ia$ is semisimple, for all the subdirectly irreducible algebras we have considered are indeed simple. 

Let us abbreviate  the term  $$((x \s y ) \s ( (\neg y \s \neg x ) \s z )) \s   (((y \s x) \s (\neg x \s \neg y) \s z) \s z)$$ 
as $q(x,y,z)$. Then we may state the following  result that provides a way to axiomatize the subvarieties of $\ia$:



\begin{thm} \label{t:var_ialg}
The varieties  $V(\{ \Al[A^{\top}_3], \Al[A^{\bot}_3] \})$, $V(\Al[A^{\top}_3])$, $V(\Al[A^{\bot}_3])$ and $V(\Al[A_2])$ may be axiomatized by adding the following equations to (I1)-(I6): 

\begin{tabular}{lllll}
\\
   $V(\{ \Al[A^{\top}_3], \Al[A^{\bot}_3] \})$ & &  & $q(x,y,z) \approx z \s z$  \\
   $V(\Al[A^{\top}_3])$ &  & & $(\neg x \s x) \s x \approx x \s x$    \\
    $V(\Al[A^{\bot}_3])$ &  & & $x \s x \approx y \s y$   \\
    $V(\Al[A_2])$  &  & & $x \s y  \approx \neg y \s \neg x.$
\end{tabular}
%
%
%
\end{thm}

\begin{proof}
As to the first claim, it is sufficient to check that both $\Al[A^{\top}_3]$ and $\Al[A^{\bot}_3]$ satisfy the equation, while $\Al[A_4]$ does not. For the second one we need to check that
$$\Al[A^{\top}_3] \vDash (\neg x \s x) \s x \approx x \s x$$ while  
$$\Al[A^{\bot}_3] \nvDash (\neg x \s x) \s x \approx x \s x.$$
The third claim is proved similarly. As to the fourth, we need to check that $\Al[A_2]$ satisfies $x \s y  \approx \neg y \s \neg x$ while neither  $\Al[A^{\top}_3]$ nor $\Al[A^{\bot}_3]$ does.
\end{proof}
%


It is not difficult to prove that the construction described in Proposition \ref{p:icong} can be straighforwardly extended in order to prove results analogous to those of Proposition \ref{p:icong} and Theorem \ref{t:igen} for the other subreducts of implicative   bilattices obtained by expanding the language with the lattice operation corresponding to the two   bilattice orders. 
We can now see that residuated De Morgan lattices are just a particular example of this, namely the $\{\land, \lor, \s,  \neg, \top \}$-subreducts of implicative   bilattices. In the case of the full implicative   bilattice language, we will have, for all elements $a, b$: 
$$ a \land b \cc a \otimes b \cc \neg (a \s \neg b) \cc a * b $$
$$ a \lor b \cc a \oplus b \cc (a \s b) \s b \cc (b \s a) \s a $$
where $\cc$ is the equivalence relation defined in Proposition \ref{p:icong}, which will be compatible with all the operations except negation, and $*$ is the operation defined in Proposition \ref{resid}.   

\section{Categorical equivalences} 
\label{sec:cat}

In \cite{MoPiSlVo00} the representation theorems for bounded interlaced bilattices are used to establish equivalences among various categories of bilattices and lattices. In this section we shall see that these results can be easily generalized to the unbounded case and will develop an analogous study for implicative bilattices.

Let us first recall the main results obtained in \cite{MoPiSlVo00}. We denote by $\lat$ the category of lattices $\Al[L] = \la L, \sqcap, \sqcup \ra$ with morphisms all lattice homomorphisms. Moreover, $\dl$ is the full subcategory of  $\lat$ whose objects are all distributive lattices and $\cl$ is the category of classical implicative lattices $\Al[L] = \la L, \sqcap, \sqcup, \ba, 1 \ra$  with morphism all $\{ \sqcap, \sqcup, \ba, 1 \}$-homomorphisms. Analogously, we denote respectively by $\IPBL$ and $\DPBL$ the categories of  interlaced and distributive pre-bilattices $\Al[B] = \la B, \land, \lor, \otimes, \oplus \ra$ with morphisms all pre-bilattice homomorphisms. $\IBL$ and $\DBL$ denote the corresponding categories of bilattices $\Al[B] = \la B, \land, \lor, \otimes, \oplus, \neg \ra$, with morphisms all bilattice homomorphisms (i.e.\ pre-bilattice homomorphisms that also preserve negation). 


The main result of \cite{MoPiSlVo00} is that, for the case of bounded lattices and bounded (pre-)bilattices, the following categories are naturally equivalent:
\begin{enumerate}[(i)]
  \item  $\IPBL$ and the product category $\lat \times \lat$
  \item $\DPBL$ and the product category $\dl \times \dl$
  \item $\IBL$ and $\lat $
  \item $\DBL$ and $\dl $.
  \end{enumerate}

  Our next aim  is to prove that these equivalences can be generalized to the unbounded case. Moreover, we shall define categories corresponding to some of the other classes of algebras we have considered so far, proving that  equivalences can also be established between:

\begin{enumerate}[(i)]
  \item commutative interlaced bilattices with conflation ($\IBLC$) and involutive lattices ($\IL$) 
  \item commutative distributive bilattices with conflation ($\DBLC$) and De Morgan lattices ($\DML$) 
    \item Kleene bilattices with conflation  ($\KB$) 
    and Kleene lattices ($\KL $)
        \item classical bilattices with conflation ($\CB$)
         and Boolean lattices ($\BOL $)
\item implicative bilattices ($\ib$) and classical implicative lattices ($\cl $).
\end{enumerate}

Let us first consider the case of (unbounded) interlaced pre-bilattices. Given an interlaced pre-bilattice $\Al[B]$, let $\mathsf{L}^2(\Al[B]) = \la \Al[B]/\!\!\sim_1, \Al[B]/\!\!\sim_2 \ra $ (see Proposition \ref{prop:interlaced_pbl}).  Conversely, if $\Al[L_{1}]$ and \Al[L_{2}] are lattices, let $\mathsf{B}(\la \Al[L_{1}], \Al[L_{2}] \ra)$ denote the interlaced pre-bilattice $ \Al[L_{1}] \odot \Al[L_{2}] $. By Proposition \ref{prop:interlaced_pbl}, there is an isomorphism $f_{\Al[B]} : \Al[B] \cong \mathsf{B}(\mathsf{L}^2(\Al[B])) $ defined, for all $a \in B$,  as
\begin{equation}
\label{eq:f_pre}
f_{\Al[B]}(a) = \la [a]_{1}, [a]_{2} \ra
\end{equation} 
where $ [a]_{1}$ and $ [a]_{2}$ denote the equivalence classes of $a$  modulo $\sim_1$ and $\sim_2$ respectively.  It is also easy to see that, given a pair of lattices  $ \Al[L_{1}]$ and $ \Al[L_{2}]$, in the product category $\lat \times \lat$ there is an isomorphism $\la g_{\Al[L_{1}]}, g_{\Al[L_{2}]} \ra$ between $\la \Al[L_{1}], \Al[L_{2}] \ra $ and $\mathsf{L}^2 (\mathsf{B}(\la \Al[L_{1}], \Al[L_{2}] \ra)) $, where 
$ g_{\Al[L_{1}]} :  \Al[L_{1}] \cong  \Al[B]/\!\!\sim_1$ and $g_{\Al[L_{1}]} :  \Al[L_{2}] \cong  \Al[B]/\!\!\sim_2$
are defined, for all $\la a_{1}, a_{2} \ra \in L_{1} \times L_{2}$, as 
\begin{equation}
\label{eq:g_pre}
g_{\Al[L_{1}] } ( a_{1}) = [\la a_{1}, a_{2} \ra]_{1} \quad  \textrm{and} \quad
g_{\Al[L_{2}] } ( a_{2}) = [\la a_{1}, a_{2} \ra]_{2}.
\end{equation} 
Note that the definition of $g_{\Al[L_{1}] } ( a_{1})$ is independent of the element $a_{2}$, for it holds that $ [\la a_{1}, a_{2} \ra]_{1} = [\la a_{1}, b \ra]_{1} $ for any $b \in L_{2}$, and similarly $ [\la a_{1}, a_{2} \ra]_{2} = [\la b, a_{2} \ra]_{2} $ for any $b \in L_{2}$.

%
In order to establish a categorical equivalence, we define two functors $F: \lat \times \lat \longrightarrow\IPBL$ and $G: \IPBL \longrightarrow\lat \times \lat $ as follows. For all $\la \Al[L_{1}], \Al[L_{2}] \ra,  \in \Obj(\lat \times \lat)$, let
 %
 $$F(\la \Al[L_{1}], \Al[L_{2}] \ra ) = \mathsf{B}(\la \Al[L_{1}], \Al[L_{2}] \ra).$$ 
 For all $\la \Al[L_{1}], \Al[L_{2}] \ra, \la \Al[M_{1}], \Al[M_{2}] \ra \in \Obj(\lat \times \lat)$ and all $\la h_{1}, h_{2}  \ra : \la \Al[L_{1}], \Al[L_{2}] \ra \longrightarrow \la \Al[M_{1}], \Al[M_{2}] \ra \in \Mor(\lat \times \lat)$, let 
 $F(\la h_{1}, h_{2} \ra):   \mathsf{B}(\la \Al[L_{1}], \Al[L_{2}] \ra) \longrightarrow \mathsf{B}(\la \Al[M_{1}], \Al[M_{2}] \ra) $ be given, for all $\la a_{1}, a_{2}\ra \in \mathsf{B}(\la \Al[L_{1}], \Al[L_{2}] \ra)$, by 
$$F(\la h_{1}, h_{2} \ra) (\la a_{1}, a_{2} \ra) = \la h_{1}(a_{1}), h_{2}(a_{2}) \ra.$$ It is not difficult to see that $F$ is indeed a functor. The functor $G$ is defined, for all $\Al[B] \in \Obj(\IPBL)$, as
$$ G(\Al[B]) = \mathsf{L}^2(\Al[B]). $$
For all $\Al[B], \Al[C] \in \Obj(\IPBL)$ and $k: \Al[B] \longrightarrow \Al[C] \in \Mor(\IPBL)$, let $G(k):  \mathsf{L}^2(\Al[B])  \longrightarrow  \mathsf{L}^2(\Al[C]) $ be defined as 
$$
G(k) = \la G(k)_{1}, G(k)_{2} \ra
$$
where  $G(k)_{1}([a]_{1}) = [k(a)]_{1}$ and $G(k)_{2}([b]_{2}) = [k(b)]_{2}$ for all $\la [a]_{1}, [b]_{2} \ra \in \mathsf{L}^2(\Al[B]) $.

Using Proposition \ref{prop:symetrization}, it is easy to check that $a \sim_1 b$ implies $k(a) \sim_1 k(b) $ for any $a, b \in B$ and any homomorphism $k: \Al[B]  \longrightarrow \Al[C]$ (and the same holds for $ \sim_2$). Therefore the previous definition is sound.

If we now denote by $I_{\mathsf{C}}$ the identity functor on a given category $\mathsf{C}$, we may prove the following analogue of \cite[Theorem 10]{MoPiSlVo00}:

\begin{thm} \label{t:catpre} 
The family of morphisms $f: I_{\IBL}  \longrightarrow FG $ and $g: I_{\lat \times \lat}  \longrightarrow GF $ defined in \ref{eq:f_pre} and \ref{eq:g_pre}  are natural isomorphisms, so that the categories  $\lat \times \lat$ and  $\IPBL$ are naturally equivalent.
\end{thm}

\begin{proof}
Let $f, g, F, G$ be defined as above. Assume  $\la h_{1}, h_{2} \ra : \la \Al[L_{1}], \Al[L_{2}] \ra  \longrightarrow  \la \Al[M_{1}], \Al[M_{2}] \ra \in \Mor(\lat \times \lat)$ and $\la a_{1}, a_{2}\ra \in  L_{1} \times L_{2} $. We have to prove that the following diagram commutes:

\vspace{0.5cm}
\xymatrix{ 
& & & \la \Al[L_{1}], \Al[L_{2}] \ra \ar[rrr]^-*{\la g_{\Al[L_{1}]}, g_{\Al[L_{2}]} \ra}  \ar[dd]_-*{\la h_{1}, h_{2} \ra}  & & & G(F(\la \Al[L_{1}], \Al[L_{2}] \ra)) \ar[dd]^-*{G(F(\la h_{1}, h_{2} \ra))} \\ \\
& & &\la \Al[M_{1}], \Al[M_{2}] \ra \ar[rrr]_-*{\la g_{\Al[M_{1}]}, g_{\Al[M_{2}]} \ra} & & & G(F(\la \Al[M_{1}], \Al[M_{2}] \ra)) }
\vspace{0.5cm}

Applying our definitions, we have
\begin{align*}
& G(F(\la h_{1}, h_{2} \ra)) \cdot \la g_{\Al[L_{1}]}, g_{\Al[L_{2}]} \ra (\la a_{1}, a_{2}\ra) = \\
& = G(F(\la h_{1}, h_{2} \ra)) \la [\la a_{1}, a_{2} \ra]_{1}, [\la a_{1}, a_{2} \ra]_{2} \ra   =  \\
& =  \la [F(\la h_{1}, h_{2} \ra) (\la a_{1}, a_{2} \ra)]_{1}, [F(\la h_{1}, h_{2} \ra) (\la a_{1}, a_{2} \ra)]_{2} \ra = \\
& =  \la [\la h_{1}(a_{1}), h_{2}(a_{2}) \ra]_{1}, [\la h_{1}(a_{1}), h_{2}(a_{2}) \ra]_{2} \ra = \\
& =  \la [\la h_{1}(a_{1}), h_{2}(a_{2}) \ra]_{1}, [\la h_{1}(a_{1}), h_{2}(a_{2}) \ra]_{2} \ra = \\
& = \la g_{\Al[M_{1}]}, g_{\Al[M_{2}]} \ra   (\la h_{1}(a_{1}), h_{2}(a_{2}) \ra) = \\
& =\la g_{\Al[M_{1}]}, g_{\Al[M_{2}]} \ra  \cdot \la h_{1}, h_{2} \ra (\la a_{1}, a_{2}\ra).
\end{align*}

Assume now $k: \Al[B]  \longrightarrow \Al[C] \in \Mor(\IPBL)$ and $a \in B$. We have to prove that the following diagram commutes:

\vspace{0.5cm}
\xymatrix{ 
&&& & & \Al[B] \ar[rr]^-*{f_{\Al[B]}}  \ar[dd]_-*{k}  & & F(G(\Al[B])) \ar[dd]^-*{F(G(k))} \\ \\
&&& & &\Al[C] \ar[rr]_-*{f_{\Al[C]}} & & F(G(\Al[C])) }
\vspace{0.5cm}

Applying again the definitions, we obtain
\begin{align*}
F(G(k)) \cdot f_{\Al[B]} (a) 
& = F(G(k)) \la [a]_{1}, [a]_{2} \ra \\
& =  \la [k(a)]_{1}, [k(a)]_{2} \ra \\
& =  f_{\Al[C]} \cdot k(a). %
\end{align*}

We have thus proved that $f$ and $g$ are natural transformations. Since, as we have noted, $f_{\Al[B]}: \Al[B]  \longrightarrow F(G (\Al[B]) )$ and $g_{\Al[L]}: \Al[L]  \longrightarrow G(F (\Al[L]))$ are isomorphisms, we conclude that  $f$ and $g$ are natural isomorphisms.
\end{proof}

From the previous theorem we immediately obtain the following:

\begin{cor} \label{c:catpredist}
The category $\dl \times \dl$ and $\DPBL$ are naturally equivalent.
\end{cor}

Let us now consider the case of interlaced bilattices. As we have seen in Section \ref{sec:repbil}, in the presence of negation we can establish an isomorphism between an interlaced bilattice $\Al[B]$ and the product bilattice $\Reg(\Al[B]) \odot \Reg(\Al[B])$, where $\Reg(\Al[B])$ denotes the sublattice of the k-lattice whose universe  is the set of regular elements of  $\Al[B]$ (i.e.\ the fixed points of the negation operator). Given an interlaced bilattice $\Al[B]$, we may then set $ \mathsf{L}(\Al[B]) =  \Reg(\Al[B])$. Conversely, given a lattice $\Al[L]$, we denote by $\mathsf{B}(\Al[L])$ the interlaced bilattice $ \Al[L] \odot \Al[L] $. The isomorphism $f_{\Al[B]} : \Al[B] \cong \mathsf{B}(\mathsf{L}(\Al[B])) $ is then defined, for all $a \in B$, as 
\begin{equation}
\label{eq:f_bil}
f_{\Al[B]} (a) = \la \reg(a), \reg(\neg a) \ra.
\end{equation}
%
%
Given a lattice $\Al[L] $, we have an isomorphism $g_{\Al[L]}: \Al[L] \cong \mathsf{L} (\mathsf{B} (\Al[L] ) )$ given, for all $a \in L$, by 
%
%
\begin{equation}
\label{eq:g_bil}
g_{\Al[L]} (a) = \la a, a \ra. 
\end{equation}
We now define the functors $F: \lat   \longrightarrow \IBL$ and $G: \IBL  \longrightarrow \lat  $ as follows. For every $\Al[L] \in \Obj (\lat)$, set 
$$F(\Al[L] ) = \mathsf{B}(\Al[L])$$
and for all $h: \Al[L]  \longrightarrow \Al[M] \in \Mor(\lat)$, $F(h): \mathsf{B}(\Al[L])  \longrightarrow \mathsf{B}(\Al[M])$ is given, for all $a, b \in \mathsf{B}(\Al[L]) $,  by
$$ F(h)(\la a, b \ra) = \la h(a), h(b) \ra. $$
Note that $F$ preserves surjections, i.e.\ if $h: L  \longrightarrow M$ is surjective, then so is $F(h): \mathsf{B}(\Al[L])  \longrightarrow \mathsf{B}(\Al[M])$.
For any $\Al[B] \in \Obj (\IBL)$, we set
$$ G(\Al[B]) =   \mathsf{L}(\Al[B])$$
and for every $\Al[B], \Al[C] \in \Obj(\IBL)$ and $k: \Al[B]  \longrightarrow \Al[C] \in \Mor(\IBL)$, the functor $G(k):  \mathsf{L}(\Al[B])  \longrightarrow  \mathsf{L}(\Al[C]) $ is defined as
$$ G(k)(a) = k(a). 
$$

We are now able to state an analogue of \cite[Theorem 13]{MoPiSlVo00}:

\begin{thm} \label{t:catbil} 
The family of morphisms $f: I_{\IBL}  \longrightarrow FG $ and $g: I_{\lat}  \longrightarrow GF $ defined in \ref{eq:f_bil} and \ref{eq:g_bil}  are natural isomorphisms, so that the categories $\lat$ and  $\IBL$ are naturally equivalent.
\end{thm}

\begin{proof}
Let $f, g, F, G$ be defined as above. Assume  $h :  \Al[L]  \longrightarrow   \Al[M] \in \Mor(\lat )$ for some $\Al[L], \Al[M] \in \Obj(\lat) $ and $ a \in  L$. We have to prove that the following diagram commutes:


\vspace{0.2cm}
\xymatrix{
&&& & & \Al[L] \ar[rr]^-*{g_{\Al[L]}}  \ar[dd]_-*{h}  & & G(F(\Al[L])) \ar[dd]^-*{G(F(h))} \\ \\
&&& & &\Al[M] \ar[rr]_-*{g_{\Al[M]}} & & G(F(\Al[M])) }
\vspace{0.5cm}

%
%
%
%
Applying our definitions, we have
\begin{align*}
G(F(h)) \cdot g_{\Al[L]} (a) & = G(F(h))  (\la a, a \ra) \\
& = F(h)  (\la a, a \ra) \\
& = \la h(a), h(a) \ra \\
& = g_{\Al[M]} \cdot  h  (a).
\end{align*}

Let now  $k :  \Al[B]  \longrightarrow   \Al[C] \in \Mor(\lat )$ for some $\Al[B], \Al[C] \in \Obj(\IBL)$ and $ a \in  B$. We have to show that the following diagram commutes:

\vspace{0.5cm}
\xymatrix{ 
&&& & & \Al[B] \ar[rr]^-*{f_{\Al[B]}}  \ar[dd]_-*{k}  & & F(G(\Al[B])) \ar[dd]^-*{F(G(k))} \\ \\
&&& & &\Al[C] \ar[rr]_-*{f_{\Al[C]}} & & F(G(\Al[C])) }
\vspace{0.5cm}
%
In order to see this, recall that $\reg(a) = (a \lor (a \otimes \neg a)) \oplus \neg (a \lor (a \otimes \neg a))$. It is then obvious that $k(\reg(a)) =  \reg(k(a))$ and $k(\reg(\neg a)) = \reg(\neg k(a))$.
We may now apply our definitions to obtain
\begin{align*}
 F(G(k)) \cdot f_{\Al[B]} (a) & = F(G(k))  \la \reg(a), \reg(\neg a) \ra \\
 & = \la k(\reg(a)), k(\reg(\neg a)) \ra \\
 & = \la \reg(k(a)), \reg(\neg k(a)) \ra \\
 & =  f_{\Al[C]} \cdot  k  (a).
\end{align*}

This shows that $f$ and $g$ are natural transformations. Since, as we have observed, $f_{\Al[B]}: \Al[B]  \longrightarrow F(G (\Al[B]) )$ and $g_{\Al[L]}: \Al[L]  \longrightarrow G(F (\Al[L]))$ are isomorphisms, we conclude that  $f$ and $g$ are natural isomorphisms.
\end{proof}

From the previous theorem we immediately obtain the following:

\begin{cor} \label{c:catdist}
The category $\dl$ and $\DBL$ are naturally equivalent.
\end{cor}

It is sufficient to examine the proof of  Theorem \ref{t:catbil} to see that, using the same definitions, we may obtain an analogous result concerning bilattices with conflation. Let us denote by $\IBLC$ the category of commutative interlaced bilattices with conflation with morphisms all bilattice homomorphisms that preserve also the conflation operator. Let $\IL$ denote the category of lattices with involution as defined in Section~ \ref{sec:confl}, with morphisms all lattice homomorphisms that also preserve the involution. Then we may state the following:

\begin{thm} \label{t:catbilconf} 
The categories $\IL$ and  $\IBLC$ are naturally equivalent.
\end{thm}

Let us denote by $\DBLC$ the subcategory of commutative distributive bilattices with conflation and by $\DML$ the category of De Morgan lattices with morphisms all lattice homomorphisms that also preserve the involution. Then from the previous theorem we may obtain the following:

\begin{cor} \label{c:cat_dist_conf} 
The categories $\DML$ and  $\DBLC$ are naturally equivalent.
\end{cor}

Analogous results may be obtained for the categories associated with the other two subvarieties of $\IBLC$ considered in Section~\ref{sec:confl}, namely $\KB$ (Kleene bilattices with conflation) and $\CB$ (classical bilattices with conflation), which correspond to the subvarieties of $\DML$ that we denote by $\KL$ (Kleene lattices) and $\BOL$ (Boolean lattices).

\begin{cor} \label{c:catkle}
The category $\KL$ and $\KB$ are naturally equivalent.
\end{cor}

\begin{cor} \label{c:catkle}
The category $\BOL$ and $\CB$ are naturally equivalent.
\end{cor}

The proof of Theorem \ref{t:catbil}  can be adapted in order to obtain a similar result  about implicative bilattices. 

Let us denote by $\ib$ be the category of implicative bilattices $\Al[B] = \la B, \land, \lor, \otimes,$ $\oplus, \neg, {\s,} \top \ra$ with morphisms all bilattice homomorphisms that also preserve implication, and let $\cl$ be the category of classical implicative lattices  $\Al[L] = \la L, \sqcap, \sqcup, \ba, 1 \ra$ with morphisms all lattice homomorphisms that also preserve the operation $\ba$. For any implicative bilattice $\Al[B]$, let $\mathsf{L}(\Al[B]) = \Al[B^-]$, where $\Al[B^-]$ is defined as in Proposition \ref{reprimp}.
Conversely, to any classical implicative lattice $\Al[L]$ we associate the implicative bilattice  $\mathsf{B}(\Al[L]) = \la \Al[L] \odot \Al[L], \s \ra$ defined as in Section \ref{sec:repimp}. By Theorem \ref{t:reprimp2}, we know that there is an isomorphism   $f_{\Al[B]} : \Al[B] \cong \mathsf{B}(\mathsf{L}(\Al[B])) $ defined, for all $a \in B$, as 
\begin{equation}
\label{eq:f_imp}
f_{\Al[B]} (a) = \la a \land \top, \neg a \land \top \ra.
\end{equation}
Moreover, given a classical implicative lattice $\Al[L] $, we have an isomorphism $g_{\Al[L]}: \Al[L] \cong \mathsf{L} (\mathsf{B} (\Al[L] ) )$ given, for all $a \in L$, by 
\begin{equation}
\label{eq:g_imp}
g_{\Al[L]} (a) = \la a, 1 \ra. 
\end{equation}
The functors $F: \lat   \longrightarrow \IBL$ and $G: \IBL  \longrightarrow \lat  $ are defined as in the case of interlaced bilattices. For every $\Al[L] \in \Obj (\lat)$, we set 
$$F(\Al[L] ) = \mathsf{B}(\Al[L])$$
and, for all $h: \Al[L]  \longrightarrow \Al[M] \in \Mor(\lat)$, $F(h): \mathsf{B}(\Al[L])  \longrightarrow \mathsf{B}(\Al[M])$ is given, for all $a, b \in \mathsf{B}(\Al[L]) $,  by
$$ F(h)(\la a, b \ra) = \la h(a), h(b) \ra. $$
For any $\Al[B] \in \Obj (\IBL)$, set
$$ G(\Al[B]) =   \mathsf{L}(\Al[B])$$
and for every $\Al[B], \Al[C] \in \Obj(\IBL)$ and $k: \Al[B]  \longrightarrow \Al[C] \in \Mor(\IBL)$, the functor $G(k):  \mathsf{L}(\Al[B])  \longrightarrow  \mathsf{L}(\Al[C]) $ is defined as
$$ G(k)(a) = k(a). 
$$

We have then the following:
\begin{thm} \label{t:catimp} 
The family of morphisms $f: I_{\ib}  \longrightarrow FG $ and $g: I_{\cl}  \longrightarrow GF $ defined in \ref{eq:f_imp} and \ref{eq:g_imp}  are natural isomorphisms, so that the categories $\cl$ and  $\ib$ are naturally equivalent.
\end{thm}

\begin{proof}
Similar to the proof of Theorem \ref{t:catbil}. On the one hand, we have
\begin{align*}
G(F(h)) \cdot g_{\Al[L]} (a) & = G(F(h))  (\la a, 1 \ra) \\
& = F(h)  (\la a, 1 \ra) \\
& = \la h(a), 1 \ra \\
& = g_{\Al[M]} \cdot  h  (a).
\end{align*}
On the other hand:
\begin{align*}
 F(G(k)) \cdot f_{\Al[B]} (a) & = F(G(k))  \la a \land \top, \neg a \land \top \ra \\
 & = \la k(a \land \top), k(\neg a \land \top) \ra \\
 & = \la k(a) \land \top, \neg k( a) \land \top \ra \\
 & =  f_{\Al[C]} \cdot  k  (a).
 \end{align*}
Hence $f$ and $g$ are natural transformations and since $f_{\Al[B]}: \Al[B]  \longrightarrow F(G (\Al[B]) )$ and $g_{\Al[L]}: \Al[L]  \longrightarrow G(F (\Al[L]))$ are isomorphisms, we conclude that  $f$ and $g$ are natural isomorphisms.
\end{proof}

To close the section, we will study from a categorical point of view the relationship between implicative bilattices and $I$-algebras, the $\{ \s, \neg \}$-subreducts considered in Section \ref{sec:other}. Let us denote by $ \ia$ the category of $I$-algebras $\Al = \la A, {\s,}  \neg \ra$ with morphisms all $\{ \s, \neg \}$-homomorphisms, and let $\ib$ be the category of implicative bilattices defined as before. For any  $I$-algebra $\Al$, let 
$$\mathsf{B}(\Al) = \la {\Al/\!\! \sim} \odot {\Al/\!\! \sim}, \s \ra $$ 
where $\la {\Al/\!\! \sim} \odot {\Al/\!\! \sim}, \s \ra$ is the implicative bilattice obtained through the construction described in Section \ref{sec:other} (see Theorem \ref{t:icong-2}).
For any $a \in A$, we denote by $[a]$ the equivalence class of $a$ modulo the relation $\sim$ introduced in Definition~\ref{d:req}. We may now define a functor $F: \ia  \longrightarrow \ib $ as follows. For any $\Al \in \Obj(\ia)$, we set 
$$F(\Al) = \mathsf{B}(\Al).$$ 
For any $h: \Al  \longrightarrow \Al' \in \Mor(\ia)$, we define $F(h): \mathsf{B}(\Al)  \longrightarrow \mathsf{B}(\Al') $, for any $a, b \in A$, as $$F(h)(\la [a], [b] \ra)= \la [h(a)], [h(b)] \ra.$$ It is not difficult to see that the previous definition is sound (see Definition \ref{d:req}) and that $F$ is indeed a functor. Note also that $F$ preserves surjections. In fact, if $h:  \Al  \longrightarrow \Al'$ is onto, then for all $\la [a'], [b'] \ra \in \mathsf{B}(\Al') $ it holds that $a' = h(a)$ and $b' = h(b)$ for some $a, b \in \Al$, so that $\la [a'], [b'] \ra = F(h)(\la [a], [b] \ra)$.
%

Conversely, from any implicative bilattice $\Al[B]$ we may obtain an $I$-algebra through a forgetful functor that associates to $\Al[B] = \la B, \land, \lor, \otimes, \oplus, \neg, \s \ra$  the reduct $ \mathsf{A}(\Al[B]) = \la B, \neg, \s \ra$. Let then $G: \ib   \longrightarrow \ia$ be the functor defined as follows. For any $\Al[B] \in  \Obj(\ib)$, we set 
$$G(\Al[B])= \mathsf{A}(\Al[B]).$$ 
For any $k: \Al[B]  \longrightarrow \Al[B'] \in \Mor(\ib)$, we define $G(k): \mathsf{A}(\Al[B])  \longrightarrow \mathsf{A}(\Al[B'])$, for all $a \in B$, as 
$$G(k) (a) = k(a) .$$ 
Again, it is easy to check that $G$ is a functor, that it is faithful and preserves both injections and surjections. To be faithful means that, for all $\Al[B], \Al[B']$ and all $k_{1}, k_{2}: \Al[B]  \longrightarrow \Al[B']$,  if $G(k_{1}) = G(k_{2})$, then $k_{1} = k_{2}$, which in this case is obvious.

The relationship between the two functors defined may be formalized through the following result:

\begin{thm} \label{t:adj} 
The functor $F: \ia  \longrightarrow \ib $ and $G: \ib   \longrightarrow \ia$, 
form an adjoint pair. More precisely,  $F$  is left adjoint to $G$.
\end{thm}

\begin{proof}
For any  $I$-algebra $\Al$, let $f_{\Al} : A  \longrightarrow \mathsf{B}(\Al[A]) $ be defined, for all $a \in A$, as $f_{\Al} (a) = \la [a], [\neg a] \ra$. We have proved that this map is an embedding (Theorem~\ref{t:icong-2}). Let us check that $f: I_{\ia}  \longrightarrow GF $ is a natural transformation. We have to prove that the following diagram commutes:

\vspace{0.5cm}
\xymatrix{
&&& & & \Al[A] \ar[rr]^-*{f_{\Al[A]}}  \ar[dd]_-*{h}  & & G(F(\Al[A])) \ar[dd]^-*{G(F(h))} \\ \\
&&& & &\Al[A'] \ar[rr]_-*{f_{\Al[A']}} & & G(F(\Al[A'])) }
\vspace{0.5cm}

We have 
\begin{align*}
G(F(h)) \cdot f_{\Al[A]} (a) 
& =  G(F(h))   (\la [a], [\neg a] \ra) \\
& = F(h)   (\la [a], [\neg a] \ra) \\
& = \la [h(a)], [h(\neg a)] \ra \\
& = \la [h(a)], [\neg h(a)] \ra \\
& = f_{\Al[A']}  \cdot h(a).
\end{align*}
It remains to prove that, for all objects $\Al \in \Obj(\ia)$, $\Al[B] \in \Obj(\ib)$ and any morphism $h: \Al  \longrightarrow G(\Al[B] )$, there is a unique $g: F(\Al)  \longrightarrow \Al[B] $ that makes the following diagram commute:

\vspace{0.5cm}
\xymatrix{
&&& & & \Al[A] \ar[rr]^-*{f_{\Al[A]}}  \ar[ddrr]_-*{h}  & & G(F(\Al[A])) \ar[dd]^-*{G(g))} \\ \\
&&& & & & & G(\Al[B]) }
\vspace{0.5cm}

Observe that, following the proof of Theorem~\ref{t:icong-2}, we may identify any  $\Al[B] \in \Obj(\ib)$ with its isomorphic image $\la {\Al[B]/\!\!\sim} \odot {\Al[B]/\!\!\sim}, \s^{B} \ra$. In this way we have  $h(a) = \la [h(a)], [\neg h(a)]\ra$ for all $a \in A$, and we may define the morphism $g$ as $g (\la[a], [b] \ra)  = \la [h(a)], [h(b)] \ra $ for any $a, b \in A$. Thus we obtain 
\begin{align*}
G(g) \cdot f_{\Al[A]} (a) 
& = g (\la [a], [\neg a] \ra) \\
& = \la [h(a)], [ h(\neg a)] \ra \\
& = \la [h(a)], [\neg h(a)] \ra \\
& = h(a).
\end{align*}

\end{proof}

\bibliography{logalg}{}
\bibliographystyle{plain}





\resumen

\selectlanguage{spanish}

El objetivo de la presente memoria es desarrollar un estudio desde el punto de vista de la Lógica Algebraica Abstracta de algunos sistemas deductivos, basados en estructuras algebraicas llamadas ``birretículos'', que fueron introducidos en los años noventa por Ofer Arieli y Arnon Avron. El interés de dicho estudio procede principalmente de dos ámbitos.

Por un lado, la teoría de birretículos constituye un formalismo elegante que en las últimas dos décadas ha originado diversas aplicaciones, especialmente en el ámbito de la Informática Teórica y de la Inteligencia Artificial. En este respecto, la presente memoria pretende ser una contribución a una mejor comprensión de la estructura matemática y lógica subyacente a dichas aplicaciones. 

Por otro lado, nuestro interés en las lógicas basadas en birretículos procede de la Lógica Algebraica Abstracta. En términos muy generales, la lógica algebraica  se puede describir como el estudio de las relaciones entre álgebra y lógica. Una de las  razones principales que motivan dicho estudio es la posibilidad de aplicar métodos algebraicos a problemas lógicos y viceversa: esto  se realiza asociando a cada sistema deductivo una clase de modelos algebraicos que puede considerarse la contrapartida algebraica de esa lógica. Empezando con la obra de Tarski y de sus colaboradores, el método de algebraización de las lógicas fue  constantemente desarrollado y generalizado. En las últimas dos décadas, los lógicos algebraicos han ido concentrando su atención sobre el proceso de algebraización en si mismo. Éste tipo de investigaciones forma ahora una rama de la lógica algebraica conocida como Lógica Algebraica Abstracta.

Un tema importante en Lógica Algebraica Abstracta es la posibilidad de aplicar los métodos de la teoría general de la algebraización de las lógicas a una gama cada vez más amplia de sistemas deductivos. En este respecto, algunas de las lógicas basadas en los birretículos resultan especialmente interesantes en cuanto ejemplos naturales de las llamadas \emph{lógicas no protoalgebraicas}, una clase que incluye los sistemas lógicos que resultan más difíciles de tratar con herramientas algebraicas.

Hasta años recientes  relativamente pocas lógicas no protoalgebraicas habían sido estudiadas. Posiblemente también a causa de esa falta de ejemplos, los resultados generales que se conocen sobre esta clase de lógicas no son todavía comparables en número ni en profundidad con los que se obtuvieron acerca de los sistemas lógicos que muestran un buen comportamiento desde el punto de vista algebraico, las llamadas \emph{lógicas protoalgebraicas}. En este respecto, la presente memoria pretende ser una contribución al objetivo  de extender la teoría general de la algebraización de las lógicas más allá de sus fronteras actuales.

Vamos ahora a introducir informalmente las ideas principales subyacentes a los birretículos y algunas de sus aplicaciones.

Los birretículos son estructuras algebraicas propuestas por Matthew Ginsbgerg  \cite{Gi88} como un formalismo uniforme para la deducción en Inteligencia Artificial, en particular en el ámbito del razonamiento por defecto (\textit{default reasoning}) y del razonamiento no monótono. En las últimas dos décadas la teoría de birretículos ha resultado  útil en diversos ámbitos, a veces harto distintos del que los originó;  a continuación mencionaremos tan sólo algunos. 

Observa Ginsberg  \cite{Gi88} que muchos sistemas de deducción usados en la Inteligencia Artificial se pueden unificar bajo la perspectiva de una lógica multivalorada cuyo espacio de valores de verdad es un conjunto dotado de una doble estructura reticular. La idea de que deba haber un orden entre los valores de verdad es muy común, casi estándar, en el ámbito de las lógicas multivaloradas: por ejemplo, en las lógicas borrosas los valores están ordenados según su ``grado de verdad.'' En este respecto, la original idea  de Ginsberg fue que, además del orden asociado al grado de verdad, hay otro orden que es natural considerar. Dicha relación, que Ginsberg llamó ``orden del conocimiento'' (\emph{knowledge ordering}),  pretende reflejar el grado de conocimiento o información asociado a una oración: por ejemplo, en el contexto de la deducción automática, es posible etiquetar una oración como ``desconocida'' cuando el agente epistémico no posee ninguna información acerca de la verdad o de la falsedad de la oración. Dicha idea, nota Ginsberg, se puede encontrar ya en los trabajos de Belnap \cite{Be76}, \cite{Be77}, quien propuso una interpretación análoga para la lógica de cuatro valores de Belnap-Dunn. Desde un punto de vista matemático, el aporte principal de Ginsberg fue el desarrollo de un marco general que permite  manejar  conjuntos doblemente ordenados de valores de verdad de tamaño arbitrario. 

Según la notación introducida por Ginsberg, en el ámbito de los birretículos las dos relaciones de orden se denotan usualmente con
$\leq_t$ ($t$ de ``truth'') y $\leq_k$ ($k$ de ``knowledge''). Observa Fitting~\cite{Fi06b} que el orden $\leq_k$ debería más bien pensarse como asociado al grado de información y, por tanto, debería usarse la notación $\leq_i$. Dicha observación nos parece correcta: sin embargo, el uso de $\leq_k$, que adoptamos también en esta memoria, es ya estándar en la literatura sobre birretículos así como en los trabajos de Fitting mismo (véase \cite{Fi06b}: ``but I have always written $\leq_k$, and now I'm stuck with it.'').

Después de los trabajos iniciales de Ginsberg (\cite{Gi88}, \cite{Gi90a}, \cite{Gi95a}), los birretículos fueron extensamente investigados por Fitting, que considera aplicaciones a la Programación Lógica (\cite{Fi90}, \cite{Fi91}; sobre el tema véase tambíen \cite{Ko01a} y \cite{LoSt04}), a problemas filosóficos como la teoria de la verdad (\cite{Fi89}, \cite{Fi06b}) y además estudia la relación entre los birretículos y una familia de sistemas multivalorados que generalizan la lógica de tres valores de Kleene (\cite{Fi91b}, \cite{Fi94}). 
Otras interesantes aplicaciones incluyen el análisis de la implicación, la implicatura y de la presuposición en el lenguaje natural \cite{Sc96}, la semántica de las preguntas en el lenguaje natural  \cite{NeFr02} y
la lógica epistémica \cite{Si94}.
 
En los años noventa los birretículos fueron también estudiados en profundidad por Arieli y Avron, tanto desde un punto de vista algebraico (\cite{Av95}, \cite{Av96a}) como lógico (\cite{ArAv94}, \cite{ArAv98}). Para tratar la paraconsitencia y la deducción no monótona en la Inteligencia Artificial, Arieli y Avron \cite{ArAv96} desarrollaron los primeros sistemas lógicos en sentido tradicional basados en birretículos. La más sencilla de dichas lógicas, que vamos a llamar $\mathcal{LB}$, está definida semánticamente por una clase de matrices llamadas ``birretículos lógicos'' (\emph{logical bilattices}) y es una expansión de la sobredicha lógica de Belnap-Dunn al lenguaje estándar de los birretículos. En \cite{ArAv96} los autores introducen un sistema Gentzen como contrapartida sintáctica de la lógica $\mathcal{LB}$ y prueban la completitud y la eliminación del corte (\emph{cut elimination}). En el mismo trabajo, Arieli y Avron consideran también una expansión de $\lb$, obtenida añadiéndole dos implicaciones (interdefinibles). Dicha lógica, que vamos a denotar $\lbs$, también está definida semánticamente a través del concepto de birretículo lógico (\emph{logical bilattice}). En \cite{ArAv96} los autores introducen tanto un cálculo estilo Gentzen como un cálculo estilo Hilbert para $\lbs$ y prueban los teoremas de completitud y de eliminación del corte  para el cálculo Gentzen.

El objetivo principal de la presente memoria es el estudio de estos dos sistemas lógicos desde el punto de vista de la Lógica Algebraica Abstracta. Dicha investigación revela  interesantes  aspectos tanto algebraicos como lógicos de los birretículos.

Presentamos a continuación un resumen de los principales resultados contenidos en la presente memoria, organizados según la estructura en capítulos y secciones. 

El capítulo \ref{ch:intro} contiene una introducción a la presente memoria y presenta algunos resultados conocidos sobre los birretículos. 

En la sección \ref{sec:intro} presentamos las ideas que llevaron a la introducción de los birretículos, los principales motivos de interés por el estudio de las lógicas basadas en birretículos, mencionamos algunas aplicaciones y damos un resumen de los contenidos de la presente memoria.

En la sección \ref{sec:aal} 
presentamos algunas definiciones y resultados fundamentales de Lógica Algebraica Abstracta que utilizamos a lo largo de todo nuestro trabajo. 

Introducimos la noción de \emph{matriz lógica} como modelo algebraico de una lógica proposicional y las definiciones relacionadas de \emph{congruencia de Leibniz} de una matriz y de \emph{operador de Leibniz}. Mencionamos algunas de las clases de lógicas que pertenecen a la clasificación llamada \emph{jerarquía de Leibniz}, que se basa en las propiedades del operador de Leibniz, en particular  las  \emph{lógicas protoalgebraicas} y  las \emph{lógicas algebraizables}, dos clases de sistemas deductivos que tienen especial importancia en nuestro estudio de las lógicas basadas en los birretículos.  Gracias a las definiciones anteriores, podemos introducir la noción de \emph{modelo reducido} de una lógica proposicional  $\mathcal{L}$, que permite definir la clase $\algstar \mathcal{L}$ de los reductos algebraicos de los modelos reducidos de $\mathcal{L}$.

Introducimos a continuación la noción de \emph{matriz generalizada} (junto con la, equivalente, de \emph{lógica abstracta}) como modelo de una lógica proposicional, 
un concepto de fundamental importancia  para el estudio de las lógicas no protoalgebraicas (a las que pertenecen algunas de las lógicas basadas en los birretículos). 

Definimos la \emph{relación de Frege} y la \emph{congruencia de Tarski} asociadas a una matriz generalizada, que nos permiten introducir el concepto de \emph{modelo generalizado reducido}. Dada una lógica proposicional $\mathcal{L}$, podemos entonces estudiar la clase $\alg \mathcal{L}$ de los reductos algebraicos de los modelos reducidos de  $\mathcal{L}$. Recordamos también algunas nociones de la teoría de las matrices generalizadas que usaremos en nuestro estudio de las lógicas basadas en los birretículos, entre ellas la de \emph{morfismo bilógico} y de  \emph{modelo pleno}. 

Acabamos la sección mencionando la teoría de la algebraizabilidad de sistemas de Gentzen, que también permite obtener interesantes resultados en el estudio de lógicas no protoalgebraicas, como el que presentamos en la sección \ref{sec:gentz}.

En la sección \ref{sec:bil} 
introducimos las definiciones básicas y algunos resultados conocidos acerca de los birretículos. En particular, presentamos la definición de las clases de álgebras llamadas \emph{pre-birretículos} (\emph{pre-bilattices}) $\PBL$, \emph{pre-birretículos entrelazados} (\textit{interlaced pre-bilattices}) $\IPBL$ y \emph{pre-birretículos distributivos} $\DPBL$. 

Un pre-birretículo es un álgebra $\Al[B] = \la B, \land, \lor, \otimes, \oplus \ra $ tal que los reductos $\la B, \land, \lor \ra $  y $ \la B,  \otimes, \oplus \ra $ son retículos, cuyos órdenes asociados se denotan, respectivamente, $\leq_{t} $ y $\leq_{k}$. 

Un pre-birretículo $\Al[B]$ es \emph{entrelazado} si cada una de las cuatro operaciones reticulares es monótona con respecto a ambos ordenes $\leq_{t} $ y $\leq_{k}$, es decir, si satisface las siguientes propiedades: para todo $a, b, c \in B$,
\begin{equation*}
\begin{split}
 a \leq_t b \: & \Rightarrow \: a \otimes c \leq_t b \otimes c  \qquad
 \qquad a \leq_t b \: \Rightarrow \: a \oplus c \leq_t b \oplus c \\
  a \leq_k b \: & \Rightarrow \: a \land c \leq_k b \land c  
  \qquad
\qquad a \leq_k b \: \Rightarrow \: a \lor c \leq_k b \lor c.  
\end{split}
\end{equation*}

Un pre-birretículo $\Al[B]$ es \emph{distributivo} si satisface  las doce posibles leyes distributivas entre las cuatro operaciones $\{ \land, \lor, \otimes, \oplus \}$, es decir, si, para todo $a, b, c \in B$:
\begin{equation*}
\begin{split}
 a \circ ( b \bullet c )  \approx (a \circ b) \bullet (a
 \circ c) \mbox{\quad para todo } \ \circ, \bullet \in \{ \land, \lor,
 \otimes, \oplus \} \mbox{ con } \circ \neq \bullet.
\end{split}
\end{equation*}
%
%
%
%
\begin{center}
\begin{figure}[t]
\vspace{15pt}

\begin{center}
\begin{tabular}{cccc}

\vspace{5pt}

\begin{minipage}{2cm}
\setlength{\unitlength}{1.2cm}
\begin{center}
\begin{picture}(2,2)(0.15,0)
\put(1,0){\makebox(0,0)[l]{ $\bot$}} 
\put(0,1){\makebox(0,0)[r]{$\false$ }} 
\put(2,1){\makebox(0,0)[l]{ $\true$}} 
\put(1,2){\makebox(0,0)[r]{$\top$ }} 

\put(1,0){\circle*{0.2}} 
\put(0,1){\circle*{0.2}} 
\put(2,1){\circle*{0.2}}
\put(1,2){\circle*{0.2}}

\put(1,0){\line(1,1){1}} 
\put(1,0){\line(-1,1){1}} 
\put(0,1){\line(1,1){1}} 
\put(2,1){\line(-1,1){1}}

\end{picture}
\end{center}
\end{minipage}

&

\begin{minipage}{3cm}
\setlength{\unitlength}{0.8cm}
\begin{center}
\begin{picture}(2,2)(0,-0.25)
\put(1,-1){\makebox(0,0)[l]{ $\bot$}} 
\put(1,0){\makebox(0,0)[l]{ $a$}} 
\put(0,1){\makebox(0,0)[r]{$\false$ }} 
\put(2,1){\makebox(0,0)[l]{ $\true$}} 
\put(1,2){\makebox(0,0)[r]{$\top$ }} 

\put(1,-1){\circle*{0.2}} 
\put(1,0){\circle*{0.2}} 
\put(0,1){\circle*{0.2}} 
\put(2,1){\circle*{0.2}}
\put(1,2){\circle*{0.2}}

\put(1,0){\line(1,1){1}} 
\put(1,0){\line(-1,1){1}} 
\put(0,1){\line(1,1){1}} 
\put(2,1){\line(-1,1){1}} 

\put(1,-1){\line(0,1){1}} 
\put(1,-1){\line(-1,2){1}} 
\put(1,-1){\line(1,2){1}} 

\end{picture}
\end{center}
\end{minipage}

&

\begin{minipage}{3cm}
\setlength{\unitlength}{0.7cm}
\begin{center}
\begin{picture}(3,4)(-0.25,0)
\multiput(1,0)(1,1){3}{\circle*{0.3}}
\multiput(0,1)(1,1){3}{\circle*{0.3}}
\multiput(-1,2)(1,1){3}{\circle*{0.3}}

\put(1,0){\line(1,1){2}} 
\put(0,1){\line(1,1){2}} 
\put(-1,2){\line(1,1){2}} 
\put(1,0){\line(-1,1){2}} 
\put(2,1){\line(-1,1){2}} 
\put(3,2){\line(-1,1){2}} 

\put(1,0){\makebox(0,0)[l]{ $\bot$}} 
\put(-1,2){\makebox(0,0)[r]{$\false$ }} 
\put(3,2){\makebox(0,0)[l]{ $\true$}} 
\put(1,4){\makebox(0,0)[r]{$\top$ }} 

\end{picture}
\end{center}
\end{minipage}

&

\begin{minipage}{3cm}
\setlength{\unitlength}{0.9cm}
\begin{center}
\begin{picture}(2,2)(0,-0.25)
\put(1,-1){\makebox(0,0)[l]{ $\bot$}} 
\put(0.5,-0.5){\makebox(0,0)[r]{$a$ }} 
\put(1.5,-0.5){\makebox(0,0)[l]{ $b$}} 
\put(1,0){\makebox(0,0)[l]{ $c$}} 
\put(0,1){\makebox(0,0)[r]{$\false$ }} 
\put(2,1){\makebox(0,0)[l]{ $\true$}} 
\put(1,2){\makebox(0,0)[r]{$\top$ }} 

\put(1,-1){\circle*{0.2}} 
\put(0.5,-0.5){\circle*{0.2}}
\put(1.5,-0.5){\circle*{0.2}}
\put(1,0){\circle*{0.2}} 
\put(0,1){\circle*{0.2}} 
\put(2,1){\circle*{0.2}}
\put(1,2){\circle*{0.2}}

\put(1,0){\line(1,1){1}} 
\put(1,0){\line(-1,1){1}} 
\put(0,1){\line(1,1){1}} 
\put(2,1){\line(-1,1){1}} 

\put(1,-1){\line(-1,1){0.5}} 
\put(1,-1){\line(1,1){0.5}} 

\put(1,0){\line(1,-1){0.5}} 
\put(1,0){\line(-1,-1){0.5}} 

\put(0,1){\line(1,-3){0.5}} 
\put(2,1){\line(-1,-3){0.5}} 

\end{picture}
\end{center}
\end{minipage}
\\ \\

$\four$ & $\five$ & $\nine$ & $\7$
\end{tabular}

\caption{Algunos ejemplos de (pre-)birretículos} \label{fig:hasse1_bis}
\end{center}
\end{figure}
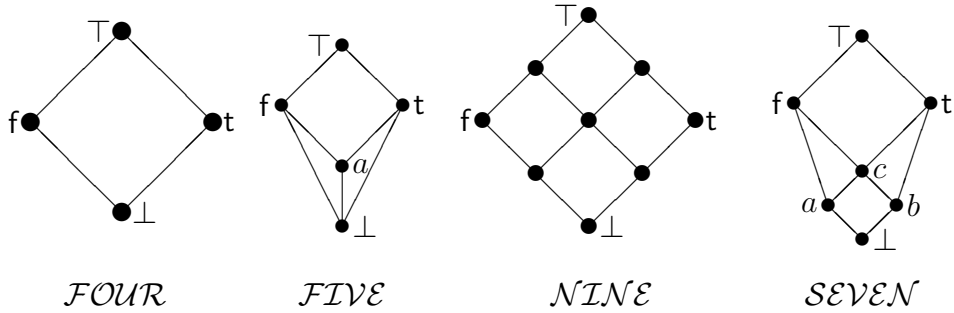
\end{center}

Observamos que las tres clases de pre-birretículos consideradas son ecuacionales y se da la siguiente cadena de inclusiones estrictas: $$\DPBL \varsubsetneq \IPBL \varsubsetneq \PBL.$$ Observamos también que de las definiciones se sigue que hay una dualidad entre los dos órdenes de todo pre-birretículo, análoga a la dualidad que hay entre ínfimo y supremo en los retículos: para simplificar las pruebas utilizamos frecuentemente este hecho, que llamamos Principio de Dualidad.

Presentamos algunas propiedades básicas de los pre-birretículos acotados y el interesante resultado que todo pre-birretículo entrelazado acotado se puede obtener a partir de un retículo acotado que posea dos elementos que satisfacen ciertas propiedades.

 Introducimos a continuación la definición de birretículo, que es un álgebra $\Al[B] = \la B, \land, \lor, \otimes, \oplus, \neg \ra $ tal que el reducto $ \la B, \land, \lor, \otimes, \oplus \ra $ es un pre-birretículo y la operación unaria $\neg: B \longrightarrow B$, llamada \emph{negación}, es involutiva, monótona con respecto al orden $\leq_{k}$ y antimonótona con respecto a $\leq_{t}$, es decir, satisface las siguientes condiciones: para todo $a, b \in B$,
 \begin{enumerate}[]
 \item \quad \textbf{(neg1)} \qquad si $a \leq_{t} b$,  entonces  $\neg b \leq_{t} \neg a$
 \item \quad \textbf{(neg2)} \qquad  si $a \leq_{k} b$,  entonces  $\neg a \leq_{k} \neg b$ 
 \item \quad \textbf{(neg3)} \qquad  $ a = \neg \neg a  $.
\end{enumerate}

 Damos algunos ejemplos de pre-birretículos y birretículos importantes que se pueden representar mediante dobles diagramas de Hasse (Figura \ref{fig:hasse1_bis}), en particular destacamos $\four$, el mínimo birretículo no trivial, que desarrolla un papel fundamental a nivel lógico. 
 
Presentamos a continuación una importante construcción, introducida por Ginsberg y extensamente estudiada por Fitting y Avron, que permite obtener un pre-birretículo entrelazado (que escribimos $\Al[L_{1}] \odot \Al[L_{2}]$) como un producto, análogo a un producto directo, de dos retículos cualesquiera $\Al[L_{1}]$ y $\Al[L_{2}]$; el mismo tipo de producto permite  construir un birretículo entrelazado $\Al[L] \odot \Al[L]$ a partir de dos copias isomorfas de un retículo $\Al[L]$ cualquiera. 

Acabamos la sección con un  un teorema de representación, fundamental, debido a Fitting y generalizado por Avron:  a saber,  que todo pre-birretículo entrelazado acotado $\Al[B]$ es isomorfo a un producto $\Al[L_{1}] \odot \Al[L_{2}]$ de dos retículos acotados $\Al[L_{1}]$ y $\Al[L_{2}]$ obtenido mediante la sobredicha construcción, y análogamente que todo birretículo entrelazado acotado se puede obtener como un producto  $\Al[L] \odot \Al[L]$ a partir de dos copias isomorfas de un retículo acotado $\Al[L]$. Un corolario de este resultado es una caracterización del retículo de las congruencias de todo  \mbox{(pre-)}\-birretículo acotado en términos de los retículos de las congruencias de los dos retículos factores asociados a él mediante la construcción que hemos descrito.

El capítulo 
\ref{ch:int} presenta algunos nuevos resultados algebraicos  sobre los  \mbox{(pre-)}\-birretículos entrelazados que se usan  en los siguientes capítulos para desarrollar nuestro estudio de las lógicas basadas en los birretículos.
 
El resultado principal de la sección \ref{sec:reppre} 
es una generalización del teorema de representación para pre-birretículos entrelazados acotados al caso de pre-birretículos entrelazados cualesquiera. 

La demostración que presentamos difiere esencialmente de las conocidas en la literatura, que se basan en la presencia de las cotas. 

Dado un birretículo entrelazado $\Al[B]$, definimos dos quasi-ordenes $\leq_t \circ \leq_k$ y $\geq_t \circ \leq_k$ dados por la composición de los dos ordenes reticulares y consideramos las relaciones de equivalencia $\sim_{1}$ y $\sim_{2}$ asociadas a dichos quasi-ordenes. Probamos que $\sim_{1}$ y $\sim_{2}$ son congruencias factores de $\Al[B]$ y que por tanto $\Al[B]$ es isomorfo al producto directo $\Al[B]/\!\!\sim_1 \times \: \Al[B]/\!\!\sim_2$. Observamos que, en el caso de pre-birretículos, la construcción producto $\Al[L_{1}] \odot \Al[L_{2}]$ se puede ver como un caso particular de producto directo, y que por tanto el resultado anterior implica que todo pre-birretículo entrelazado $\Al[B]$ es isomorfo a un producto $\Al[L_{1}] \odot \Al[L_{2}]$ de dos retículos $\Al[L_{1}]$ y $\Al[L_{2}]$ (que podemos obtener como cocientes de uno de los dos reductos reticulares de $\Al[B]$). 

Obtenemos, como corolarios, que el retículo de congruencias $\Con(\Al[L_{1}] \odot \Al[L_{2}])$ es isomorfo al producto directo $\Con(\Al[L_{1}]) \times \Con(\Al[L_{2}])$ y que, en todo pre-birretículo entrelazado $\Al[B]$, las congruencias de $\Al[B]$ coinciden con las congruencias de cada uno de sus dos reductos reticulares, es decir que 
$$\Con(\Al[B]) = \Con (\la B, \land, \lor \ra) = \Con (\la B, \otimes, \oplus \ra). $$ 

Otro interesante corolario es un teorema de representación análogo al conocido teorema de representación para retículos distributivos: todo pre-birretículo distributivo se puede representar como un pre-birretículo de conjuntos.

En la sección \ref{sec:repbil} demostramos el teorema de representación para birretículos entrelazados, que se obtiene fácilmente a partir del teorema de representación para pre-birretículos. En este caso vemos que, dado cualquier birretículo entrelazado  $\Al[B]$,   es suficiente considerar la relación $\sim_{1}$. Dicha relación ya no es una congruencia de $\Al[B]$ (porque no es compatible con la negación), pero nos permite obtener  como cociente de uno cualquiera de los reductos reticulares de $\Al[B]$ un retículo $\Al[L]$ tal que $\Al[B]$ resulta ser isomorfo a $\Al[L] \odot \Al[L]$. 

Como corolario, obtenemos una caracterización de las congruencias de todo birretículo entrelazado $\Al[B] \cong \Al[L] \odot \Al[L]$: tenemos que 
$\Con (\Al[B])$ es isomorfo a $\Con (\Al[L]).$ Probamos además  que $$\Con(\Al[B]) = \Con (\la B, \land, \neg \ra) =  \Con (\la B, \lor, \neg \ra) =  \Con (\la B, \otimes, \oplus, \neg \ra). $$ 

Acabamos la sección presentando una prueba alternativa del teorema de representación, que se basa en la consideración de los elementos que son puntos fijos del operador de negación, a los que llamamos elementos \emph{regulares}. Obtenemos así que todo birretículo entrelazado $\Al[B]$ es isomorfo al producto $\la \Reg(\Al[B]), \otimes, \oplus \ra \odot \la \Reg(\Al[B]), \otimes, \oplus \ra$, donde $\la \Reg(\Al[B]), \otimes, \oplus \ra $ es el subretículo del reducto $\la B, \otimes, \oplus \ra$ cuyo universo es el conjunto $\Reg(\Al[B])$ de los elementos regulares de $\Al[B]$. 

En la sección \ref{sec:bif} estudiamos los filtros de retículo en los (pre-)birretículos entrelazados. Puesto que en cada pre-birretículo $\Al[B]$ hay dos órdenes reticulares, es natural considerar cuatro tipos de subconjuntos de $B$, es decir: los subconjuntos que son filtros de retículo en ambos órdenes, los ideales en ambos ordenes, los $\leq_{t}$-filtros y $\leq_{k}$-ideales, y los  $\leq_{t}$-ideales y $\leq_{k}$-filtros. 

 Sin embargo es fácil ver que, por el Principio de Dualidad, es suficiente estudiar uno cualquiera de esos tipos de subconjuntos: nos concentramos, por tanto, en el estudio de los que son filtros en ambos ordenes, ya considerados por Arieli y Avron \cite{ArAv96}, que los llaman \emph{bifiltros}.
 
 Definimos el operador de clausura asociado a la generación de bifiltros y, dado un pre-birretículo entrelazado $\Al[B]$, damos una caracterización del bifiltro generado por cualquier conjunto $X \subseteq B$ análoga a la conocida caracterización del filtro generado por un subconjunto cualquiera de un retículo. 
 
 Observamos que las relaciones $\sim_{1}$ y $\sim_{2}$ introducidas en la sección \ref{sec:reppre} se pueden caracterizar de la manera siguiente. Dado un pre-birretículo entrelazado $\Al[B]$ y elementos $a, b \in B$, tenemos que $a \sim_{1} b$ si y solamente si el bifiltro generado por $a$ coincide con el bifiltro generado por $b$ (análogamente se puede caracterizar $\sim_{2}$ en términos de los operadores de generación de filtros-ideales o de ideales-filtros).

Acabamos la sección con un resultado especialmente importante desde el punto de vista de las lógicas asociadas a birretículos: el retículo de bifiltros de todo pre-birretículo entrelazado $\Al[L_{1}] \odot \Al[L_{2}]$ es isomorfo al retículo de filtros del primer factor $\Al[L_{1}] $. 

En la sección \ref{sec:dis} nos concentramos en las variedades $\DPBL$ y $\DBL$ de pre-birretículos y birretículos distributivos. Gracias a los teoremas de representación y a la caracterización de las congruencias de todo birretículo entrelazado obtenidos en las secciones anteriores, podemos caracterizar la variedad $\DPBL$ como generada por sus dos miembros de dos elementos y la variedad $\DBL$ como generada por su mínimo miembro no trivial (el birretículo de cuatro elementos $\four$). 

Estudiamos, a continuación, la estructura de los bifiltros en los pre-birretículos distributivos. Obtenemos así algunos resultados análogos a conocidos teoremas sobre retículos distributivos.  En particular, probamos un teorema de extensión del bifiltro y un teorema del bifiltro primo (decimos que un bifiltro es primo si es un filtro primo en ambos órdenes). Estos resultados nos permiten dar una demostración directa (y alternativa a la de la sección \ref{sec:reppre}) del teorema de representación de todo (pre-)birretículo distributivo como un (pre-)birretículo de conjuntos. 

Acabamos el capítulo (sección \ref{sec:confl}) considerando una expansión del lenguaje de los birretículos obtenida añadiendo una operación unaria dual de la negación, es decir involutiva, monótona con respecto al orden $\leq_{t}$ y antimonótona con respecto a $\leq_{k}$. Esta operación fue introducida por Fitting \cite{Fi94}, que la llama ``conflación'' (\emph{conflation}). Llamamos por tanto \emph{birretículo con conflación} a un álgebra $\Al[B] = \la B, \land, \lor, \otimes, \oplus, \neg, \co \ra $ tal que el reducto $ \la B, \land, \lor, \otimes, \oplus, \neg \ra $ es un birretículo y la operación $\co: B \longrightarrow B$ satisface, para todo $a, b \in B$, las siguientes condiciones:
\begin{enumerate}[]
 \item \quad \textbf{(con1)} \qquad si $a \leq_{k} b$,  entonces  $- b \leq_{k} - a$
 \item \quad \textbf{(con2)} \qquad  si $a \leq_{t} b$,  entonces  $- a \leq_{t} - b$ 
 \item \quad \textbf{(con3)} \qquad  $ a = - - a  $.
\end{enumerate}

Decimos que un birretículo con conflación es conmutativo si negación y conflación conmutan, es decir si,  para todo $a \in B$,
$$\neg \co  a = \co  \neg \: a.$$
Observamos que evidentemente los birretículos con conflación forman una variedad (y así los birretículos entrelazados con conflación, etc.).

Damos un teorema de representación, análogo al teorema de representación para birretículos, para los birretículos entrelazados conmutativos con conflación. En este caso tenemos que cada álgebra $\Al[B]$ perteneciente a  dicha variedad  es isomorfa a un producto $\Al[L] \odot \Al[L]$ de dos copias de un álgebra $\Al[L] = \la L, \sqcap, \sqcup, ' \ra$, donde $\la L, \sqcap, \sqcup \ra$ es un retículo y $': L \longrightarrow L$  es una operación unaria involutiva y antimonótona con respecto al orden reticular (que llamamos \emph{involución}).  

Demostramos que, análogamente al caso de los birretículos, hay un isomorfismo entre las congruencias de todo birretículo entrelazado conmutativo con conflación $\Al[L] \odot \Al[L]$ y las congruencias de  $\Al[L]$. Dicho resultado nos permite obtener una caracterización de las subvariedades de la variedad de los birretículos distributivos conmutativos con conflación en términos de las correspondientes variedades de retículos distributivos con involución (llamados \emph{retículos de De Morgan}).

En el capítulo \ref{ch:log} estudiamos, desde el punto de vista de la Lógica Algebraica Abstracta, la lógica sin implicación $\lb$, introducida por Arieli y Avron \cite{ArAv96} a partir de una clase de matrices llamadas birretículos lógicos, que consisten en un par $\la \Al[B], F \ra$ donde $\Al[B]$ es un birretículo y $F \subseteq B$ un bifiltro primo.  

En la sección \ref{sec:sem} introducimos semánticamente $\lb$ como la lógica definida por la matriz $\la \four, \Tr \ra$, donde $\Tr = \{ \top, \true \}$. Presentamos a continuación algunos resultados importantes obtenidos por Arieli y Avron: entre ellos, el hecho de que la lógica definida por cualquier birretículo lógico  $\la \Al[B], F \ra$ coincide con la definida por la matriz $\la \four, \Tr \ra$ (y por tanto con $\lb$) y la introducción de un cálculo Gentzen completo para $\lb$ (Cuadro \ref{tab:Gentz}). 

En la sección \ref{sec:hil} introducimos una presentación de $\lb$ mediante un cálculo estilo Hilbert (Cuadro \ref{tab:hilbert}), que usamos en las secciones siguientes para estudiar $\lb$ desde el punto de vista de la Lógica Algebraica Abstracta. 

Demostramos que cada fórmula se puede reducir a una forma normal y, gracias a dicho resultado, obtenemos para nuestro  cálculo un teorema de completitud  con respecto a la semántica de $\lb$ introducida en la sección anterior.

En la sección \ref{sec:tar} caracterizamos la lógica $\lb$ en términos de algunas propiedades metalógicas (a veces llamadas \emph{estilo Tarski}); probamos, además, que   $\lb$ no tiene extensiones consistentes.

\begin{table}[ht!]
\fbox{\parbox{\textwidth}{
\vspace{0,15cm}
\begin{enumerate}[ ]
  \item \quad \quad \quad \textbf{Axioma:} \qquad $(Ax) \quad \Gamma, \varphi \rhd \varphi,
    \Delta$.
    \item
  \item \quad \quad \quad \textbf{Reglas:} \qquad \;  Regla de Corte más las siguientes reglas lógicas:
\end{enumerate}
\begin{center} {
  \begin{align*}
&(\land \rhd) \quad \frac{\Gamma, \varphi, \psi \rhd \Delta}{\Gamma, \varphi \land \psi \rhd \Delta} 
& & (\rhd \land) \quad \frac{\Gamma \rhd \Delta, \varphi \quad \Gamma \rhd \Delta, \psi }{\Gamma \rhd \Delta, \varphi \land \psi} \\ \\
& (\neg \land \rhd) \quad \frac{\Gamma, \neg \varphi \rhd \Delta \quad \Gamma, \neg \psi \rhd \Delta }{\Gamma, \neg (\varphi \land \psi) \rhd \Delta}  
& &  (\rhd \neg \land ) \quad \frac{\Gamma \rhd \Delta, \neg \varphi,\neg \psi}{\Gamma \rhd \Delta, \neg (\varphi \land \psi)} \\ \\
& (\lor \rhd) \quad \frac{\Gamma, \varphi \rhd \Delta \quad \Gamma, \psi \rhd \Delta }{\Gamma, \varphi \lor \psi \rhd \Delta}  & &  (\rhd \lor ) \quad \frac{\Gamma \rhd \Delta, \varphi, \psi}{\Gamma \rhd \Delta, \varphi \lor \psi} \\ \\
&(\neg \lor \rhd) \quad \frac{\Gamma, \neg \varphi, \neg \psi \rhd \Delta}{\Gamma, \neg (\varphi \lor \psi) \rhd \Delta} & & (\rhd \neg \lor) \quad \frac{\Gamma \rhd \Delta, \neg \varphi \quad \Gamma \rhd \Delta, \neg \psi }{\Gamma \rhd \Delta, \neg (\varphi \lor \psi)} \\ \\
&(\otimes \rhd) \quad \frac{\Gamma, \varphi, \psi \rhd \Delta}{\Gamma, \varphi \otimes \psi \rhd \Delta} 
& & (\rhd \otimes) \quad \frac{\Gamma \rhd \Delta, \varphi \quad \Gamma \rhd \Delta, \psi }{\Gamma \rhd \Delta, \varphi \otimes \psi} \\ \\
& (\neg \otimes \rhd) \quad \frac{\Gamma, \neg \varphi, \neg \psi \rhd \Delta}{\Gamma, \neg (\varphi \otimes \psi) \rhd \Delta}  
& & (\rhd \neg \otimes) \quad \frac{\Gamma \rhd \Delta, \neg \varphi \quad \Gamma \rhd \Delta, \neg \psi }{\Gamma \rhd \Delta, \neg (\varphi \otimes \psi)}  \\ \\
& (\oplus \rhd) \quad \frac{\Gamma, \varphi \rhd \Delta \quad \Gamma, \psi \rhd \Delta }{\Gamma, \varphi \oplus \psi \rhd \Delta}  & &  (\rhd \oplus ) \quad \frac{\Gamma \rhd \Delta, \varphi, \psi}{\Gamma \rhd \Delta, \varphi \oplus \psi} \\ \\
&(\neg \oplus \rhd) \quad \frac{\Gamma, \neg \varphi \rhd \Delta \quad
\Gamma, \neg \psi \rhd \Delta }{\Gamma, \neg (\varphi \oplus \psi) \rhd
\Delta} & & (\rhd \neg \oplus ) \quad \frac{\Gamma \rhd \Delta, \neg
\varphi,\neg \psi}{\Gamma \rhd \Delta, \neg (\varphi \oplus \psi)} \\ \\
& (\neg \neg \rhd) \quad \frac{\Gamma, \varphi \rhd \Delta }{\Gamma, \neg \neg \varphi \rhd \Delta}  & &  (\rhd \neg \neg) \quad \frac{\Gamma \rhd \Delta, \varphi}{\Gamma \rhd \Delta, \neg \neg \varphi}
\end{align*}
}
\end{center}
}}

\caption{Un cálculo de secuentes completo para la lógica
$\mathcal{LB}$} \label{tab:Gentz}
\end{table}

\begin{table}[ht!]
\fbox{\parbox{\textwidth}{
\begin{center} {
  \hspace{-0.7cm}
  \begin{tabular}{ccc}
    \AxiomC{$p \land q$}
    \LeftLabel{\,(R1)}
    \UnaryInfC{$p$}
    \DisplayProof &
    \AxiomC{$p \land q$}
    \LeftLabel{\,(R2)}
    \UnaryInfC{$q$}
    \DisplayProof &
    \AxiomC{$p \!$}
    \AxiomC{$\! q$}
    \LeftLabel{\,(R3)}
    \BinaryInfC{$p \land q$}
    \DisplayProof \\ \\ 

    \AxiomC{$p$}
    \LeftLabel{\,(R4)}
    \UnaryInfC{$p \lor q$}
    \DisplayProof &
    \AxiomC{$p \lor q$}
    \LeftLabel{\,(R5)}
    \UnaryInfC{$q \lor p$}
    \DisplayProof &
    \AxiomC{$p \lor p$}
    \LeftLabel{\,(R6)}
    \UnaryInfC{$p$}
    \DisplayProof \\ \\ 
 
    \AxiomC{$p \lor (q \lor r)$}
    \LeftLabel{\,(R7)}
    \UnaryInfC{$(p \lor q) \lor r$}
    \DisplayProof &
    \AxiomC{$p \lor (q \land r)$}
    \LeftLabel{\,(R8)}
    \UnaryInfC{$(p \lor q) \land (p \lor r)$}
    \DisplayProof &
    \AxiomC{$(p \lor q) \land (p \lor r)$}
    \LeftLabel{\,(R9)}
    \UnaryInfC{$p \lor (q \land r)$}
    \DisplayProof \\ \\ 
 
    \AxiomC{$p \lor r$}
    \LeftLabel{\,(R10)}
    \UnaryInfC{$\neg \neg p \lor r$}
    \DisplayProof &
    \AxiomC{$\neg \neg p \lor r$}
    \LeftLabel{\,(R11)}
    \UnaryInfC{$p \lor r$}
    \DisplayProof &
    \AxiomC{$\neg (p \lor q) \lor r$}
    \LeftLabel{\,(R12)}
    \UnaryInfC{$ (\neg p \land \neg q) \lor r$}
    \DisplayProof \\ \\ 
 
    \AxiomC{$(\neg p \land \neg q) \lor r$}
    \LeftLabel{\,(R13)}
    \UnaryInfC{$\neg (p \lor q) \lor r$}
    \DisplayProof &
    \AxiomC{$\neg (p \land q) \lor r$}
    \LeftLabel{\,(R14)}
    \UnaryInfC{$(\neg p \lor \neg q) \lor r$}
    \DisplayProof &
    \AxiomC{$(\neg p \lor \neg q) \lor r$}
    \LeftLabel{\,(R15)}
    \UnaryInfC{$\neg (p \land q) \lor r$}
    \DisplayProof \\ \\ \\ 
 
    \AxiomC{$( p \otimes q) \lor r$}
    \LeftLabel{\,(R16)}
    \UnaryInfC{$(p \land q) \lor r$}
    \DisplayProof &
    \AxiomC{$(p \land q) \lor r$}
    \LeftLabel{\,(R17)}
    \UnaryInfC{$(p \otimes q) \lor r$}
    \DisplayProof &
    \AxiomC{$(p \oplus q) \lor r$}
    \LeftLabel{\,(R18)}
    \UnaryInfC{$(p \lor q) \lor r$}
    \DisplayProof \\ \\ 
 
    \AxiomC{$( p \lor q) \lor r$}
    \LeftLabel{\,(R19)}
    \UnaryInfC{$(p \oplus q) \lor r$}
    \DisplayProof &
    \AxiomC{$(\neg p \otimes \neg q) \lor r$}
    \LeftLabel{\,(R20)}
    \UnaryInfC{$\neg (p \otimes q) \lor r$}
    \DisplayProof &
    \AxiomC{$\neg (p \otimes q) \lor r$}
    \LeftLabel{\,(R21)}
    \UnaryInfC{$(\neg p \otimes \neg q) \lor r$}
    \DisplayProof \\ \\ 
 
    \AxiomC{$(\neg p \oplus \neg q) \lor r$}
    \LeftLabel{\,(R22)}
    \UnaryInfC{$\neg (p \oplus q) \lor r$}
    \DisplayProof &
    \AxiomC{$\neg (p \oplus q) \lor r$}
    \LeftLabel{\,(R23)}
    \UnaryInfC{$(\neg p \oplus \neg q) \lor r$}
    \DisplayProof & \\
  \end{tabular} } 
\end{center} 
}} \vspace{0,4cm}

\caption{Un cálculo estilo Hilbert completo para la lógica $\LB$}
\label{tab:hilbert}
\end{table}

A continuación (sección \ref{sec:lb}) comenzamos el verdadero estudio de $\lb$ desde el punto de vista de la Lógica Algebraica Abstracta. En primer lugar, clasificamos dicha lógica como no protoalgebraica y no autoextensional. Caracterizamos luego la congruencia de Tarski asociada a $\lb$ y, gracias a dicho resultado, demostramos que  la clase $\alg \lb$ de los reductos algebraicos de los modelos generalizados reducidos de $\lb$ es la variedad generada por el birretículo $\four$ (es decir la variedad $\DBL$ de los birretículos distributivos).  

Observamos que, al contrario por ejemplo de las clases de los retículos distributivos y de los retículos de De Morgan, a la clase de los birretículos distributivos se puede asociar una lógica algebraizable  $\mathcal{L}$ (por tanto, distinta de $\lb$) tal que $\algstar \mathcal{L} = \DBL$.

Caracterizamos los modelos plenos de $\lb$ en términos de las propiedades estudiadas en la sección \ref{sec:tar}. Gracias a dicho resultado, podemos también demostrar que el cálculo Gentzen mostrado en el Cuadro \ref{tab:Gentz} es \emph{plenamente adecuado} para la lógica $\lb$. 

Estudiamos a continuación los modelos reducidos de $\lb$ y la clase de sus reductos algebraicos $\algstar \lb$. Probamos que dicha clase no es una cuasivariedad y caracterizamos sus miembros como birretículos distributivos superiormente acotados en el orden $\leq_{k}$ que satisfacen cierta propiedad. En particular, demostramos que $\algstar \lb$ está formada por los birretículos distributivos $\Al[B]$ tales que $\Al[B] \cong \Al[L] \odot \Al[L]$, donde $\Al[L]$ es un ``retículo disyuntivo dual'' (\emph{dual disjunctive lattice}), es decir
un retículo distributivo que satisface cierta propiedad dual de la propiedad disyuntiva considerada en  \cite{Wa38}  y \cite{Ci91}. 

Acabamos el capítulo (sección \ref{sec:gentz}) con la demostración de que el cálculo Gentzen introducido por Arieli y Avron es algebraizable en el sentido de Rebagliato y Verd\'{u} \cite{ReV93b}, y que su semántica algebraica equivalente es la variedad de los birretículos distributivos.

En el capítulo \ref{ch:add} nos ocupamos de una expansión de la lógica $\lb$ también introducida por Arieli y Avron \cite{ArAv96}, que denominamos $\lbs$, obtenida añadiendo al lenguaje $\{ \land, \lor, \otimes, \oplus, \neg \}$ dos conectivas de implicación interdefinibles, una \emph{implicación débil} $\s$ y una \emph{implicación fuerte} $\ta$. Adoptamos la primera como primitiva, y definimos $$p \ta q : = (p \s q) \land (\neg q \s \neg p).$$ Usamos también las siguientes abreviaciones: 
\begin{align*}
  p \leftrightarrow q \   & : = \ (p \ta q) \land (q \ta p)  \\
   p \equiv q   \ & : = \ (p \s q) \land (q \s p).
\end{align*}

En la sección  \ref{sec:add} definimos semánticamente la lógica  $\lbs$ y presentamos el cálculo estilo Hilbert  $H_{\s}$ (Cuadro \ref{tab:hilbs}) introducido por Arieli y Avron. Citamos algunos de los resultados fundamentales obtenidos en \cite{ArAv96}, en particular el teorema de completitud del cálculo $H_{\s}$  con respecto a la semántica de $\lbs$.

\begin{table}[ht!]

\begin{align*}
& \mathbf{Axiomas\!:} \\ \\
&(\supset 1)  & &p \supset (q \supset p) \\
&(\supset 2)  & &  (p \supset (q \supset r)) \supset ((p\supset q) \supset (p\supset r)) \\
&(\supset 3)  & &  ((p \supset q) \supset p) \supset p \\
&(\land \supset )  & &  (p \land q) \supset p \qquad \qquad (p \land q) \supset q \\
&(\supset \land )  & &  p \supset (q \supset (p \land q)) \\
&(\otimes \supset )  & &  (p \otimes q) \supset p \qquad \qquad (p \otimes q) \supset q \\
&(\supset \otimes )  & &  p \supset (q \supset (p \otimes q)) \\
&(\supset \lor)  & &  p  \supset (p \lor q) \qquad \qquad q \supset (p \lor q) \\
&(\lor \supset )  & &  (p \supset r) \supset ((q \supset r) \supset ((p \lor q) \supset r)) \\
&(\supset \oplus)  & &  p  \supset (p \oplus q) \qquad  \qquad q \supset (p \oplus q) \\
&(\oplus \supset )  & &  (p \supset r) \supset ((q \supset r) \supset ((p \oplus q) \supset r)) \\
&(\neg \land )  & &   \neg (p \land q ) \equiv (\neg p \vee \neg q)  \\
&(\neg \lor )  & &  \neg (p \vee q) \equiv (\neg p \land \neg q ) \\
&(\neg \otimes )  & &   \neg (p \otimes q ) \equiv (\neg p \otimes \neg q)  \\
&(\neg \oplus )  & &   \neg (p \oplus q ) \equiv (\neg p \oplus \neg q)  \\
&(\neg \supset )  & &   \neg (p \supset q ) \equiv (p \land \neg q)  \\
&(\neg \neg) & & p \equiv \neg \neg p \\ \\
& \mathbf{Regla\!:}
\end{align*}
$$
\frac{p \quad p \supset q }{q}
$$

\caption{Un cálculo estilo Hilbert completo para la lógica $\lbs$}
\label{tab:hilbs}
\end{table}

En la sección \ref{sec:lbs} demostramos varias propiedades sintácticas del cálculo $H_{\s}$ que nos permiten obtener el resultado siguiente: el cálculo  $H_{\s}$  es algebraizable, con fórmula de equivalencia $\phi \leftrightarrow \psi$ y ecuación definitoria $\phi \approx \phi \s \phi$. Por el teorema de completitud podemos concluir que la lógica $\lbs$ es algebraizable.

A continuación (sección \ref{sec:algstar}) nos ocupamos por tanto de individuar y estudiar la semántica algebraica equivalente de  $\lbs$. Introducimos mediante una presentación ecuacional la variedad $\ib$, cuyos miembros llamamos ``birretículos implicativos'' (\emph{implicative bilattices}), es decir estructuras $\Al[B] = \la B, \land, \lor, \otimes, \oplus, \s \neg \ra$ tales que el reducto $\la B, \land, \lor, \otimes, \oplus, \neg \ra$ es un birretículo y la operación binaria $\s: B \times B \longrightarrow B$ es tal que $\Al[B]$ satisface las siguientes ecuaciones:
\begin{enumerate}[ ]
\item (IB1) \hspace{2mm} $(x \supset x) \supset y \approx y$ 
\item (IB2)  \hspace{2mm}  $x \supset (y \supset z) \approx (x \land y) \supset z \approx (x \otimes y) \supset z$
\item(IB3)   \hspace{2mm}  $((x \supset y) \supset x) \supset x  \approx x \s x $ 
\item(IB4)  \hspace{2mm}  $(x \lor y) \supset z \approx (x \supset z) \land  (y \supset z) \approx (x \oplus y) \supset z  $ 
\item (IB5)  \hspace{2mm}  $x \land ((x \supset y) \supset (x \otimes y)) \approx x $
\item (IB6)  \hspace{2mm}  $\neg (x \supset y ) \supset z  \approx (x \land \neg y)  \supset z.$
\end{enumerate}

 Demostramos a continuación varias propiedades aritméticas de la variedad $\ib$, que nos permiten obtener el resultado que  dicha variedad   es la semántica algebraica equivalente de la lógica $\lbs$. También probamos que todo reducto de un birretículo implicativo es un birretículo distributivo, un hecho que usamos en el capítulo siguiente, y demostramos  que la lógica $\lbs$, así como su fragmento $\lb$, no tiene extensiones consistentes. 
 
En el capítulo  \ref{ch:imp} presentamos un estudio algebraico de los birretículos implicativos y algunas estructuras algebraicas relacionadas con ellos. 

Comenzamos el capítulo (sección \ref{sec:repimp}) demostrando un teorema de representación para los birretículos implicativos análogo al teorema de representación para los  birretículos. Por los resultados anteriores sabemos que, para todo birretículo implicativo $\Al[B] = \la B, \land, \lor, \otimes, \oplus, \s \neg \ra$, el reducto $\la B, \land, \lor, \otimes, \oplus, \neg \ra$ es isomorfo al producto $\Al[L] \odot \Al[L]$, donde $\Al[L]$ es un retículo distributivo superiormente acotado. En el caso de los birretículos implicativos, demostramos que  además  $\Al[L]$ cumple una propiedad adicional, es decir es un retículo \emph{relativamente complementado} (todo elemento tiene un complemento en todo intervalo de $\Al[L]$). 

Dado un retículo relativamente complementado y superiormente acotado $\Al[L] = \la L, \sqcap, \sqcup \ra $ cuyo elemento máximo es 1, consideramos la operación $\ba : L \times L \longrightarrow L$ que a todo par de elementos $a, b \in L$ asocia el complemento relativo de  $a$ en el intervalo $[a \sqcap b, 1]$, que denotamos $a \ba b$. Observamos que esta clase de retículos, considerados como álgebras en el lenguaje $\{ \sqcap, \sqcup, \ba \}$, forma una variedad. Siguiendo la nomenclatura usada en \cite{Cu77}, llamamos a los miembros de dicha variedad ``retículos implicativos clásicos'' (\emph{classical implicative lattices}).

Demostramos entonces que a partir de cualquier retículo implicativo clásico  $\Al[L]$ es posible construir un birretículo implicativo mediante una construcción que, para el reducto reticular, coincide con el producto $\Al[L] \odot \Al[L]$ y además, usando la operación $\ba$, permite definir una implicación $\s$ que satisface las ecuaciones que definen la variedad de los birretículos implicativos. 

Tenemos, por tanto, que todo birretículo implicativo $\Al[B]$ es isomorfo a un producto de este tipo (que podemos denotar también $\Al[L] \odot \Al[L]$) de dos copias de un  retículo implicativo clásico  $\Al[L]$. 

Nos ocupamos a continuación de las congruencias de los birretículos implicativos. Gracias al teorema de representación para birretículos implicativos, demostramos que las congruencias de todo birretículo implicativo $\Al[L] \odot \Al[L]$ son isomorfas a las del retículo implicativo clásico  $\Al[L]$. Puesto que las congruencias de todo retículo implicativo clásico  $\Al[L] = \la L, \sqcap, \sqcup, \ba \ra $ coinciden con las congruencias de su reducto reticular $ \la L, \sqcap, \sqcup \ra $, obtenemos el siguiente resultado: las congruencias de todo birretículo implicativo  $\Al[B] = \la B, \land, \lor, \otimes, \oplus, \s \neg \ra$, coinciden con las de su reducto $\la B, \land, \lor, \otimes, \oplus, \neg \ra$ (que, por los resultados anteriores, también coinciden con las congruencias del reducto  $\la B, \land, \neg \ra$). 

En la siguiente sección (\ref{sec:imp}) desarrollamos un estudio más extenso de la variedad $\ib$ de los birretículos implicativos. Usando los resultados de la sección anterior, probamos que la única álgebra subdirectamente irreducible en $\ib$ es $\fours$, su miembro de cuatro elementos, cuyo reducto birreticular es $\four$. Por tanto, dicha álgebra genera la variedad de birretículos implicativos. Demostramos, además, que $\ib$ es una variedad con término  discriminador y que sus miembros finitos son isomorfos a potencias directas  de $\fours$.

Obtenemos también el interesante resultado de que en un birretículo implicativo cada uno de los órdenes reticulares se puede definir explícitamente usando sólo la implicación y las conectivas que corresponden al otro orden.

En la sección \ref{sec:dua} estudiamos la relación entre los retículos implicativos clásicos y los retículos disyuntivos duales considerados en la sección \ref{sec:lb}. Probamos, en particular, que la clase de los  retículos implicativos clásicos (considerados en el puro lenguaje reticular) está propiamente incluida en la de los retículos disyuntivos duales e individuamos una propiedad necesaria y suficiente para que un  retículo disyuntivo dual pertenezca a la clase de los retículos implicativos clásicos. 

En las dos secciones siguientes nos ocupamos de algunos subreductos de los birretículos implicativos que resultan especialmente interesantes desde un punto de vista lógico. 

Comenzamos, en la sección \ref{sec:sub}, observando que en todo birretículo implicativo  $\Al[B] = \la B, \land, \lor, \otimes, \oplus, \s \neg \ra$ es posible definir explícitamente una operación binaria $* : B \times B \longrightarrow B$ tal que el par $\{ *, \ta \}$ es residuado con respecto al orden $\leq_{t}$. La definición es la siguiente: para todo par de elementos $a, b \in B$,
$$a * b \ :=  \ \neg (a \ta \neg b).$$ 
Demostramos entonces que el álgebra  $\la B, \land, \lor, *, \ta \neg, \top \ra$ es, usando la nomenclatura de \cite{GaRa04}, un  ``retículo residuado conmutativo distributivo con involución'' (\emph{involutive commutative distributive residuated lattice}).

Introducimos a continuación, mediante una presentación ecuacional, una clase de álgebras que llamamos ``retículos residuados de De Morgan'' (\emph{residuated De Morgan lattices}), con el intento de probar que dichas estructuras corresponden a los $\{  \land, \lor,  \s \neg, \top \}$-subreductos de los birretículos implicativos. 

Un retículo residuado de De Morgan es un álgebra $\Al = \langle A, \land, \lor, {\supset,}$ $\neg, \top  \rangle$ tal que el reducto $\langle A, \land, \lor, \neg \rangle$ es un retículo de De Morgan y las siguientes ecuaciones son satisfechas:
\begin{enumerate}[ ]
\item (RD0) $\quad \top  \approx \neg \top $ 
\item (RD1) $\quad \top \supset x \approx x$ 
\item (RD2) $\quad  x \supset (y \supset z) \approx (x \land y) \supset z$
\item (RD3) $\quad \top \land (((x \supset y) \supset x) \supset x) \approx \top  $ 
\item (RD4) $ \quad (x \lor y) \supset z \approx (x \supset z) \land  (y \supset z)$ 
\item (RD5) $\quad  x \land (((x \supset y) \land (\neg y \supset \neg x)) \supset y) \approx x$  
\item (RD6) $\quad \neg (x \supset y ) \supset z  \approx (x \land \neg y)  \supset z$.
\end{enumerate}

Comenzamos demostrando algunas propiedades aritméticas de los retículos  residuados de De Morgan, que nos permiten probar que, para todo retículo de este tipo $\Al = \la A, \land, \lor, \s,  \neg, \top \ra$, la estructura $\la A, \land, \lor, *, \ta \neg, \top \ra$ es efectivamente un retículo residuado conmutativo distributivo con involución.

Demostramos a continuación que todo retículo  residuado de De Morgan $\Al$ contiene como subretículos dos retículos relativamente complementados (que denotamos $\Al[A^-]$ y $\Al[A^+]$) tales que $\Al[A^-] \cong \Al[A^+]$. Sabemos por tanto que es posible construir un birretículo implicativo como un producto $\Al[A^-] \odot \Al[A^-]$ (o $\Al[A^+] \odot \Al[A^+]$). 

Gracias al resultado anterior podemos definir, para todo retículo  residuado de De Morgan $\Al$, una inmersión $h: A \longrightarrow A^- \times A^-$, es decir una función inyectiva que es un homomorfismo en el lenguaje $\{ \land, \lor, \s,  \neg, \top \}$. Demostramos, además, que dicha inmersión es  canónica, en el siguiente sentido: si existe un homomorfismo $f : A \longrightarrow B$ para algún birretículo implicativo $\Al[B]$, entonces existe un único homomorfismo $f': A^- \times A^- \longrightarrow B $ tal que $f' \cdot h = f$.

Obtenemos por tanto que los retículos residuados de De Morgan coinciden con los 
$\{  \land, \lor,  \s, \neg, \top \}$-subreductos de los birretículos implicativos y que la variedad de los retículos residuados de De Morgan es generada por su miembro de cuatro elementos que es el reducto del birretículo implicativo $\fours$.

En la sección \ref{sec:other} generalizamos la construcción introducida en la sección anterior para estudiar una clase más amplia de subreductos de los birretículos implicativos. 

Nos concentramos, en particular, en los $\{ \s, \neg \}$-subreductos. Éstos tienen, a nuestro parecer, un interés particular a nivel lógico, en cuanto se trata del fragmento mínimo del lenguaje de los birretículos implicativos que se necesita para definir las dos traducciones de fórmulas en ecuaciones y viceversa que nos permiten demostrar la algebraizabilidad de la lógica $\lbs$ con respecto a la variedad de los birretículos implicativos.

Introducimos a continuación, mediante una presentación ecuacional, una variedad de álgebras en el lenguaje $\{ \s, \neg \}$, a las que llamamos $I$\emph{-álgebras}. 

Una $I$-álgebra es una estructura $\Al = \la A, \s, \neg \ra$ que satisface las siguientes ecuaciones:
\begin{enumerate}[ ]
\item (I1) $ \quad (x \supset x) \supset y \approx y$ 
\item (I2) $ \quad  x \supset (y \supset z) \approx (x \s y) \s (x \s z) \approx y \supset (x \supset z)$
\item (I3) $ \quad ((x \supset y) \supset x) \supset x  \approx x \s x $ 
\item (I4) $\quad  x \supset (\neg y \supset z) \approx \neg (x \s y) \s z$
\item (I5) $\quad \neg \neg x \approx x$
\item  (I6) $ \quad  p(x,y,x) \approx p(x,y,y)$
\end{enumerate}
donde $p(x,y,z)$ es una abreviación de $$(x \s y) \s ((y \s x) \s ((\neg x \s \neg y) \s ((\neg y \s \neg x) \s z))).$$

Demostramos algunas propiedades aritméticas de dicha variedad, que nos permiten probar que a toda 
$I$-álgebra $\Al = \la A, \s, \neg \ra$ es posible asociar un álgebra de Tarski $\Al[A]/\!\!\sim \ = \la A/\!\!\sim, \s \ra$ obtenida como un cociente del reducto $\la A, \s \ra$ de $\Al$ por una relación de equivalencia $\sim$ que es compatible con la operación $\s$ (no lo es con respecto a $\neg$). 

Demostramos que, en toda álgebra de Tarski $\Al[A]/\!\!\sim \ = \la A/\!\!\sim, \s \ra$ obtenida de la manera descrita, es posible definir operaciones de ínfimo y supremo $\{\sqcap, \sqcup \}$ con respecto al orden natural, de forma que el álgebra  $ \la A/\!\!\sim, \sqcap, \sqcup, \s \ra$ resulta ser un retículo implicativo clásico. 

Podemos por tanto introducir una construcción análoga a la descrita en la sección \ref{sec:sub}, obteniendo el correspondiente resultado que, para toda $I$-álgebra $\Al$, existe una inmersión $h: A \longrightarrow A/\!\!\sim \times \ A/\!\!\sim$ de $\Al$ en el birretículo implicativo $ \la A/\!\!\sim, \sqcap, \sqcup, \s \ra \odot  \la A/\!\!\sim, \sqcap, \sqcup, \s \ra$. 

También en este caso demostramos que la función $h$ que definimos es canónica, en el sentido de que, si existe un homomorfismo $f : A \longrightarrow B$ para algún birretículo implicativo $\Al[B]$, entonces existe un único homomorfismo $f': A/\!\!\sim \times \ A/\!\!\sim \ \longrightarrow B $ tal que $f' \cdot h = f$.

Obtenemos por tanto que las $I$-álgebras coinciden con los $\{  \s, \neg \}$-subreductos de los birretículos implicativos y que la variedad de las $I$-álgebras es generada por su miembro de cuatro elementos que es el reducto del birretículo implicativo $\fours$. También damos una caracterización de las subvariedades de $I$-álgebras como generadas por los $\{  \s, \neg \}$-subreductos de $\fours$ y una presentación ecuacional de dichas subvariedades. 

Observamos que la construcción descrita se puede fácilmente aplicar a todos los otros subreductos que corresponden a fragmentos del lenguaje de los birretículos implicativos  que contienen $\{  \s, \neg \}$, obteniendo análogas caracterizaciones de dichas clases de álgebras.

En la última sección (\ref{sec:cat}) presentamos una formulación de algunos de los resultados obtenidos en la presente memoria en términos categoriales. Definimos por tanto categorías asociadas a las diversas clases de (pre-)birretículos y retículos considerados a lo largo de la presente memoria, cuyos objetos son las álgebras de la variedad correspondiente y cuyos morfismos son los homomorfismos entre álgebras. 

Obtenemos así equivalencias categoriales entre las siguientes  categorías:
\begin{enumerate}[(i)]
  \item  pre-birretículos entrelazados ($\IPBL$) y la categoría producto $\lat \times \lat$ cuyos objetos son pares de retículos,
  \item pre-birretículos distributivos ($\DPBL$) y la categoría producto $\dl \times \dl$ cuyos objetos son pares de retículos distributivos,
  \item birretículos entrelazados  ($\IBL$) y retículos ($\lat$),
  \item birretículos  distributivos  ($\DBL$) y retículos  distributivos ($\dl $),
  \item birretículos entrelazados conmutativos con conflación ($\IBLC$) y retículos con involución ($\IL$),
  \item birretículos distributivos conmutativos con conflación ($\DBLC$) y retículos de De Morgan ($\DML$),
    \item birretículos  de Kleene con conflación  ($\KB$, una subvariedad de $\DBLC$)
    y retículos de Kleene ($\KL $),
        \item birretículos  clásicos con conflación ($\CB$, una subvariedad de $\KB$)
         y retículos de Boole ($\BOL $),
\item birretículos implicativos ($\ib$) y retículos implicativos clásicos ($\cl $).
\end{enumerate}

Acabamos la sección y el capítulo demostrando también que es posible definir funtores $F $ y $G$ entre la categoría correspondiente a las $I$-álgebras y la correspondiente a los birretículos implicativos de manera que $\la F, G \ra$ forma un par adjunto.

\sommario

\selectlanguage{italian}

L'obiettivo del presente lavoro è quello di condurre uno studio dal punto di vista della Logica Algebrica Astratta di alcuni sistemi deduttivi basati su strutture algebriche chiamate ``bireticoli'', che furono introdotti negli anni novanta da Ofer Arieli e Arnon Avron. Le motivazioni dell'interesse per tale studio hanno una duplice radice.

Da un lato, la teoria dei bireticoli costituisce un elegante formalismo che negli ultimi due decenni ha dato origine a diverse interessanti applicazioni, in particolar modo nell'ambito dell'Informatica Teorica e dell'Intelligenza Artificiale. Da questo punto di vista il presente lavoro intende dare un contributo a una migliore comprensione della struttura matematica e logica che sottostà a tali applicazioni. 
 
D'altra canto, il nostro interesse per le logiche basate su bireticoli proviene dalla Logica Algebrica Astratta. In termini estremamente generali, la logica algebrica può essere descritta come lo studio delle relazioni fra algebra e logica. Una delle principali motivazioni di tale studio è la possibilità di applicare metodi algebrici a problemi di natura logica e viceversa: ciò si realizza associando ad ogni sistema deduttivo una classe di modelli algebrici che può essere considerata la controparte algebrica di tale logica. A partire dall'opera di Tarski e dei suoi collaboratori, il metodo di algebrizzazione  delle logiche è stato costantemente sviluppato e generalizzato. Negli ultimi due decenni, i logici algebrici hanno concentrato la loro attenzione sul processo di algebrizzazione in sé stesso. Tale tipo di ricerche forma oggi una branca della logica algebrica conosciuta come Logica Algebrica Astratta.

Uno dei temi centrali della Logica Algebrica Astratta è il tentativo di applicare i metodi della teoria generale dell'algebrizzazione delle logiche a uno spettro sempre più vasto di sistemi deduttivi. A tal proposito, alcune delle logiche basate su bireticoli risultano particolarmente interessanti in quanto esempi naturali delle cosiddette  \emph{logiche non protoalgebriche}, una classe che include i sistemi logici che risultano più difficili da trattare con strumenti di tipo algebrico. 

Fino a pochi anni fa erano state studiate relativamente poche logiche non protoalgebriche. In parte anche a causa di tale mancanza di esempi, i risultati generali che si conoscono su questa classe di logiche non sono tuttora paragonabili in numero e profondità ai risultati ottenuti sui sistemi logici che mostrano un buon comportamento dal punto di vista algebrico, le cosiddette \emph{logiche protoalgebriche}. A tal proposito il presente lavoro intende dunque dare un contributo all'obiettivo a lungo termine di estendere la teoria generale dell'algebrizzazione delle logiche oltre i suoi limiti attuali.

Introdurremo ora informalmente le principali idee che sottostanno al formalismo dei bireticoli e alcune delle applicazioni.

I bireticoli sono strutture algebriche proposte da Matthew Ginsberg
\cite{Gi88} come un formalismo uniforme per la deduzione in Intelligenza Artificiale, in particolare nell'ambito del ragionamento per \emph{default} e del ragionamento non monotono. Negli ultimi due decenni, la teoria dei bireticoli si è dimostrata applicabile a vari ambiti, talvolta assai differenti da quello originale; qui di seguito ci limiteremo a menzionarne  alcuni. 

Osserva Ginsberg \cite{Gi88}  che molti dei sistemi di deduzione usati in Intelligenza Artificiale si possono unificare nella prospettiva di una logica polivalente il cui spazio di valori di verità è un insieme dotato di una doppia struttura reticolare. L'idea che i valori di verità vadano ordinati è assai comune, pressoché standard, in logica polivalente: per esempio, nelle logiche \emph{fuzzy} i valori sono ordinati secondo il loro ``grado di verità''. A tal proposito, l'intuizione originale di Ginsberg è che vi è un altro ordine naturale che possiamo considerare. Tale relazione, che Ginsberg chiama ``ordine della conoscenza'' (\emph{knowledge ordering}), intende rispecchiare il grado di conoscenza o informazione associato a una proposizione: per esempio, nel contesto della deduzione automatica, è possibile etichettare una proposizione come ``sconosciuta'' qualora l'agente epistemico non disponga di alcuna informazione circa la verità o la falsità della proposizione in questione. Quest'idea, osserva ancora Ginsberg, si può ritrovare già nei lavori di Belnap \cite{Be76}, \cite{Be77}, il quale propose un'analoga interpretazione per la logica tetravalente di Belnap-Dunn. Da un punto di vista matematico, l'apporto principale del lavoro di  Ginsberg fu l'introduzione di un formalismo generale che permette di trattare insiemi di valori di verità doppiamente ordinati di dimensione arbitraria. 

Secondo la notazione introdotta da Ginsberg, nel contesto dei bireticoli le due relazioni d'ordine si indicano usualmente con i simboli 
$\leq_t$ ($t$ di ``truth'') e $\leq_k$ ($k$ di ``knowledge''). Osserva Fitting~\cite{Fi06b} che l'ordine $\leq_k$ andrebbe pensato come associato al grado di informazione e, pertanto, occorrerebbe usare piuttosto il simbolo  $\leq_i$. Tale osservazione ci pare corretta: tuttavia l'uso di $\leq_k$, che adottiamo anche nel presente lavoro, è ormai divenuto standard nella letteratura sui bireticoli, così come nei lavori dello stesso Fitting (si veda \cite{Fi06b}: ``but I have always written $\leq_k$, and now I'm stuck with it.'').

Dopo gli iniziali lavori di Ginsberg (\cite{Gi88}, \cite{Gi90a}, \cite{Gi95a}), i bireticoli furono studiati approfonditamente da Fitting, che considera applicazioni alla Programmazione Logica (\cite{Fi90}, \cite{Fi91}; in proposito si vedano anche \cite{Ko01a} y \cite{LoSt04}), a
problemi di natura filosofica come la teoria della verità (\cite{Fi89}, \cite{Fi06b}) e studia inoltre la relazione fra i bireticoli e una famiglia di sistemi polivalenti che generalizzano la logica trivalente  di Kleene (\cite{Fi91b}, \cite{Fi94}). 

Altre interessanti applicazioni includono l'analisi dell'implicazione, l'implicatura e la presupposizione nel linguaggio naturale \cite{Sc96}, la semantica delle domande nel linguaggio naturale  \cite{NeFr02} e la logica epistemica \cite{Si94}.

Negli anni novanta i bireticoli furono altresì studiati approfonditamente da Arieli e Avron, sia dal punto di vista algebrico (\cite{Av95}, \cite{Av96a}) che da quello logico (\cite{ArAv94}, \cite{ArAv98}). Per affrontare i problemi della paraconsistenza e della deduzione non monotona in Intelligenza Artificiale, Arieli e Avron   \cite{ArAv96} hanno sviluppato i primi sistemi logici in senso stretto basati su bireticoli. La più semplice di queste logiche, che chiameremo $\mathcal{LB}$, viene definita semanticamente a partire da una classe di matrici chiamate ``bireticoli logici'' (\emph{logical bilattices}) e consiste in un'espansione della logica di Belnap-Dunn al linguaggio standard dei bireticoli. In \cite{ArAv96} gli autori introducono un sistema di tipo Gentzen come controparte sintattica della logica $\lb$ e dimostrano la completezza e il teorema di eliminazione del taglio (\emph{cut elimination}). Nello stesso lavoro Arieli e Avron considerano anche un'espansione della logica $\lb$, ottenuta introducendo due connettivi di implicazione interdefinibili. Tale logica, che indicheremo con il simbolo $\lbs$, viene pure definita semanticamente mediante il concetto di bireticolo logico (\emph{logical bilattice}). In \cite{ArAv96} gli autori introducono sia un calcolo alla Gentzen che un calcolo alla Hilbert per $\lbs$ e dimostrano i relativi teoremi di completezza e, per il calcolo Gentzen, anche il teorema di eliminazione del taglio.

L'obiettivo principale del presente lavoro è lo studio di questi due sistemi logici dal punto di vista della Logica Algebrica Astratta. Tale studio rivela interessanti aspetti sia algebrici che logici del formalismo dei bireticoli. 

Presentiamo qui di seguito una sintesi dei principali risultati ottenuti in questo  lavoro, organizzati secondo la disposizione in capitoli e sezioni.

Il capitolo \ref{ch:intro} contiene una introduzione generale, insieme con alcuni risultati noti nella letteratura sui bireticoli.

Nella sezione \ref{sec:intro} presentiamo le idee che portarono all'introduzione dei bireticoli, i principali motivi di interesse per lo studio delle logiche basate su bireticoli, menzioniamo alcune applicazioni e diamo un riassunto dei contenuti dei successivi capitoli.

Nella sezione  \ref{sec:aal} presentiamo alcune definizioni e risultati fondamentali di Logica Algebrica Astratta che vengono utilizzati nel corso del presente lavoro.

Introduciamo la nozione di \emph{matrice logica} come modello algebrico di una logica proposizionale e le correlate definizioni di  \emph{congruenza di Leibniz} di una matrice e di  \emph{operatore di Leibniz}. Citiamo alcune delle classi di logiche che fanno parte della classificazione chiamata  \emph{gerarchia di Leibniz}, che si basa su proprietà dell'operatore di Leibniz; in particolare consideriamo le  \emph{logiche protoalgebriche} e le  \emph{logiche algebrizzabili}, due classi di sistemi che rivestono particolare importanza nell'ambito del nostro studio delle logiche basate su bireticoli. Grazie alle precedenti definizioni, possiamo introdurre la nozione di  \emph{modello ridotto} di una logica proposizionale $\mathcal{L}$, che permette a sua volta di definire la classe  $\algstar \mathcal{L}$ dei ridotti algebrici dei modelli ridotti.

Introduciamo di seguito la nozione di \emph{matrice generalizzata} (insieme a quella, equivalente, di \emph{logica astratta}) come modello di una logica proposizionale, un concetto di fondamentale importanza per lo studio delle logiche non protoalgebriche (alle quali appartengono alcune delle logiche basate sui bireticoli).

Definiamo la \emph{relazione di Frege} e la \emph{congruenza di Tarski} associate a una matrice generalizzata, che ci permettono di introdurre il concetto di \emph{modello generalizzato ridotto}. Data una logica proposizionale $\mathcal{L}$, possiamo in tal modo studiare la classe $\alg \mathcal{L}$ dei ridotti algebrici dei modelli generalizzati ridotti di  $\mathcal{L}$. Ricordiamo altresì alcune nozioni della teoria delle matrici generalizzate che usiamo nel nostro studio delle logiche basate su bireticoli, tra cui quella di  \emph{morfismo bilogico} e di  \emph{modello pieno}.

Concludiamo la sezione menzionando la teoria dell'algebrizzabilità dei sistemi Gentzen, che pure permette di ottenere interessanti risultati nello studio delle logiche non protoalgebriche, come ad esempio quello da noi presentato nella sezione \ref{sec:gentz}.

Nella sezione  \ref{sec:bil} introduciamo le definizioni fondamentali e alcuni risultati noti circa i bireticoli. In particolare, diamo la definizione delle classi di algebre chiamate \emph{pre-bireticoli} (\emph{pre-bilattices}) $\PBL$, \emph{pre-bireticoli intrecciati} (\textit{interlaced pre-bilattices}) $\IPBL$ e \emph{pre-bireticoli distributivi} $\DPBL$. 

Un pre-bireticolo è un'algebra $\Al[B] = \la B, \land, \lor, \otimes, \oplus \ra $ tale che i ridotti $\la B, \land, \lor \ra $  y $ \la B,  \otimes, \oplus \ra $ sono entrambi reticoli, le cui relazioni d'ordine indichiamo, rispettivamente, con $\leq_{t} $ e $\leq_{k}$.

Un pre-bireticolo $\Al[B]$ si dice \emph{intrecciato} se ognuna delle quattro operazioni reticolari è monotona rispetto a entrambe le relazioni d'ordine  $\leq_{t} $ y $\leq_{k}$, vale a dire se valgono le seguenti proprietà: per ogni $a, b, c \in B$,
\begin{equation*}
\begin{split}
 a \leq_t b \: & \Rightarrow \: a \otimes c \leq_t b \otimes c  \qquad
 \qquad a \leq_t b \: \Rightarrow \: a \oplus c \leq_t b \oplus c \\
  a \leq_k b \: & \Rightarrow \: a \land c \leq_k b \land c  
  \qquad
\qquad a \leq_k b \: \Rightarrow \: a \lor c \leq_k b \lor c.  
\end{split}
\end{equation*}

Un pre-bireticolo $\Al[B]$ è \emph{distributivo} se soddisfa tutte le dodici possibili identità distributive fra le quattro operazioni reticolari $\{ \land, \lor, \otimes, \oplus \}$, vale a dire se, per ogni $a, b, c \in B$:
\begin{equation*}
\begin{split}
 a \circ ( b \bullet c )  \approx (a \circ b) \bullet (a
 \circ c) \mbox{\quad per  ogni } \ \circ, \bullet \in \{ \land, \lor,
 \otimes, \oplus \} \mbox{ con } \circ \neq \bullet.
\end{split}
\end{equation*}
%
%
%
\begin{center}
\begin{figure}[t]
\vspace{15pt}

\begin{center}
\begin{tabular}{cccc}

\vspace{5pt}

\begin{minipage}{2cm}
\setlength{\unitlength}{1.2cm}
\begin{center}
\begin{picture}(2,2)(0.15,0)
\put(1,0){\makebox(0,0)[l]{ $\bot$}} 
\put(0,1){\makebox(0,0)[r]{$\false$ }} 
\put(2,1){\makebox(0,0)[l]{ $\true$}} 
\put(1,2){\makebox(0,0)[r]{$\top$ }} 

\put(1,0){\circle*{0.2}} 
\put(0,1){\circle*{0.2}} 
\put(2,1){\circle*{0.2}}
\put(1,2){\circle*{0.2}}

\put(1,0){\line(1,1){1}} 
\put(1,0){\line(-1,1){1}} 
\put(0,1){\line(1,1){1}} 
\put(2,1){\line(-1,1){1}}

\end{picture}
\end{center}
\end{minipage}

&

\begin{minipage}{3cm}
\setlength{\unitlength}{0.8cm}
\begin{center}
\begin{picture}(2,2)(0,-0.25)
\put(1,-1){\makebox(0,0)[l]{ $\bot$}} 
\put(1,0){\makebox(0,0)[l]{ $a$}} 
\put(0,1){\makebox(0,0)[r]{$\false$ }} 
\put(2,1){\makebox(0,0)[l]{ $\true$}} 
\put(1,2){\makebox(0,0)[r]{$\top$ }} 

\put(1,-1){\circle*{0.2}} 
\put(1,0){\circle*{0.2}} 
\put(0,1){\circle*{0.2}} 
\put(2,1){\circle*{0.2}}
\put(1,2){\circle*{0.2}}

\put(1,0){\line(1,1){1}} 
\put(1,0){\line(-1,1){1}} 
\put(0,1){\line(1,1){1}} 
\put(2,1){\line(-1,1){1}} 

\put(1,-1){\line(0,1){1}} 
\put(1,-1){\line(-1,2){1}} 
\put(1,-1){\line(1,2){1}} 

\end{picture}
\end{center}
\end{minipage}

&

\begin{minipage}{3cm}
\setlength{\unitlength}{0.7cm}
\begin{center}
\begin{picture}(3,4)(-0.25,0)
\multiput(1,0)(1,1){3}{\circle*{0.3}}
\multiput(0,1)(1,1){3}{\circle*{0.3}}
\multiput(-1,2)(1,1){3}{\circle*{0.3}}

\put(1,0){\line(1,1){2}} 
\put(0,1){\line(1,1){2}} 
\put(-1,2){\line(1,1){2}} 
\put(1,0){\line(-1,1){2}} 
\put(2,1){\line(-1,1){2}} 
\put(3,2){\line(-1,1){2}} 

\put(1,0){\makebox(0,0)[l]{ $\bot$}} 
\put(-1,2){\makebox(0,0)[r]{$\false$ }} 
\put(3,2){\makebox(0,0)[l]{ $\true$}} 
\put(1,4){\makebox(0,0)[r]{$\top$ }} 

\end{picture}
\end{center}
\end{minipage}

&

\begin{minipage}{3cm}
\setlength{\unitlength}{0.9cm}
\begin{center}
\begin{picture}(2,2)(0,-0.25)
\put(1,-1){\makebox(0,0)[l]{ $\bot$}} 
\put(0.5,-0.5){\makebox(0,0)[r]{$a$ }} 
\put(1.5,-0.5){\makebox(0,0)[l]{ $b$}} 
\put(1,0){\makebox(0,0)[l]{ $c$}} 
\put(0,1){\makebox(0,0)[r]{$\false$ }} 
\put(2,1){\makebox(0,0)[l]{ $\true$}} 
\put(1,2){\makebox(0,0)[r]{$\top$ }} 

\put(1,-1){\circle*{0.2}} 
\put(0.5,-0.5){\circle*{0.2}}
\put(1.5,-0.5){\circle*{0.2}}
\put(1,0){\circle*{0.2}} 
\put(0,1){\circle*{0.2}} 
\put(2,1){\circle*{0.2}}
\put(1,2){\circle*{0.2}}

\put(1,0){\line(1,1){1}} 
\put(1,0){\line(-1,1){1}} 
\put(0,1){\line(1,1){1}} 
\put(2,1){\line(-1,1){1}} 

\put(1,-1){\line(-1,1){0.5}} 
\put(1,-1){\line(1,1){0.5}} 

\put(1,0){\line(1,-1){0.5}} 
\put(1,0){\line(-1,-1){0.5}} 

\put(0,1){\line(1,-3){0.5}} 
\put(2,1){\line(-1,-3){0.5}} 

\end{picture}
\end{center}
\end{minipage}
\\ \\

$\four$ & $\five$ & $\nine$ & $\7$
\end{tabular}

\caption{Alcuni esempi di (pre-)bireticoli} \label{fig:hasse1_it}
\end{center}
\end{figure}
\end{center}

Osserviamo che le tre classi di pre-bireticoli considerate sono equazionali e che sussiste la seguente catena di inclusioni proprie:
$$\DBL \varsubsetneq \IPBL \varsubsetneq \PBL.$$ 
Osserviamo altresì che dalle definizioni date segue che vi è una dualità tra i due ordini di ogni pre-bireticolo, analoga alla dualità fra infimo e supremo nei reticoli: per semplificare le dimostrazioni usiamo frequentemente questo fatto, che chiamiamo Principio di Dualità. 

Presentiamo alcune proprietà fondamentali dei pre-bireticoli limitati e l'interessante risultato che ogni pre-bireticolo intrecciato limitato si può ottenere a partire da un reticolo limitato che possieda due elementi che soddisfano certe proprietà.

Introduciamo di seguito la definizione di bireticolo, che è un'algebra $\Al[B] = \la B, \land, \lor, \otimes, \oplus, \neg \ra $ tale che il ridotto $ \la B, \land, \lor, \otimes, \oplus \ra $ è un pre-bireticolo e l'operazione unaria $\neg: B \longrightarrow B$, chiamata \emph{negazione}, risulta involutiva, monotona rispetto a $\leq_{k}$ e antimonotona rispetto a $\leq_{t}$, ovvero soddisfa le seguenti condizioni: per ogni $a, b \in B$,
 \begin{enumerate}[]
 \item \quad \textbf{(neg1)} \qquad se $a \leq_{t} b$,  allora  $\neg b \leq_{t} \neg a$
 \item \quad \textbf{(neg2)} \qquad  se $a \leq_{k} b$,  allora  $\neg a \leq_{k} \neg b$ 
 \item \quad \textbf{(neg3)} \qquad  $ a = \neg \neg a  $.
\end{enumerate}

Diamo alcuni esempi di pre-bireticoli e reticoli notevoli che si possono rappresentare mediante doppi diagrammi di Hasse (Figura \ref{fig:hasse1_it}), in particolare $\four$, il più piccolo bireticolo non degenere, che ha un ruolo fondamentale a livello logico.

Presentiamo di seguito una costruzione, introdotta da Ginsberg e poi studiata ampiamente da Fitting e Avron, che permette di ottenere un pre-bireticolo intrecciato (che indichiamo con  $\Al[L_{1}] \odot \Al[L_{2}]$)  come un prodotto, analogo a un prodotto diretto, di due reticoli qualsiasi  $\Al[L_{1}]$ e $\Al[L_{2}]$; una analoga costruzione permette di ottenere un bireticolo intrecciato $\Al[L] \odot \Al[L]$ a partire da due copie isomorfe di un qualsiasi reticolo $\Al[L]$. 

Concludiamo la sezione con un fondamentale teorema di rappresentazione, dimostrato per primo da Fitting e in seguito generalizzato da Avron, secondo il quale ogni pre-bireticolo intrecciato e limitato $\Al[B]$ è isomorfo a un prodotto $\Al[L_{1}] \odot \Al[L_{2}]$ di due reticoli limitati $\Al[L_{1}]$ e $\Al[L_{2}]$ ottenuto mediante la costruzione che abbiamo menzionato, e analogamente ogni bireticolo intrecciato e limitato si può ottenere come un prodotto  $\Al[L] \odot \Al[L]$ a partire da due copie isomorfe di un reticolo limitato  $\Al[L]$. Un corollario di tale risultato è una caratterizzazione del reticolo delle congruenze di ogni pre-bireticolo limitato nei termini dei reticoli delle congruenze dei due reticoli fattori associati ad esso mediante la costruzione descritta.

Il capitolo \ref{ch:int} contiene alcuni nuovi risultati algebrici sui pre-bireticoli intrecciati che vengono usati nei capitoli seguenti per condurre il nostro studio delle logiche basate sui bireticoli.

Il risultato principale della sezione  \ref{sec:reppre} è una generalizzazione del teorema di rappresentazione dei pre-bireticoli intrecciati e limitati al caso di pre-bireticoli intrecciati arbitrari. 

La dimostrazione che presentiamo differisce essenzialmente da quelle conosciute nella letteratura, che si basano sulla presenza delle costanti (gli elementi massimo e minimo relativi a ciascuno dei due ordini del pre-bireticolo).

Dato un bireticolo intrecciato $\Al[B]$, definiamo due quasiordini $\leq_t \circ \leq_k$ e $\geq_t \circ \leq_k$ determinati dalla composizione dei due ordini reticolari e consideriamo le relazioni di equivalenza  $\sim_{1}$ e $\sim_{2}$ associate a tali quasiordini. Dimostriamo che  $\sim_{1}$ e $\sim_{2}$ sono congruenze-fattore di $\Al[B]$ e che pertanto  $\Al[B]$ risulta isomorfo al prodotto diretto $\Al[B]/\!\!\sim_1 \times \: \Al[B]/\!\!\sim_2$. Osserviamo che, nel caso dei pre-bireticoli, la costruzione del prodotto $\Al[L_{1}] \odot \Al[L_{2}]$ può essere vista come un caso particolare di prodotto diretto, e che pertanto il precedente risultato implica che ogni pre-bireticolo intrecciato $\Al[B]$ è isomorfo a un prodotto $\Al[L_{1}] \odot \Al[L_{2}]$ di due reticoli $\Al[L_{1}]$ e $\Al[L_{2}]$ (che possiamo ottenere come quozienti di uno qualsiasi dei due ridotti reticolari di $\Al[B]$). 

Otteniamo, come corollario, che il reticolo delle congruenze  $\Con(\Al[L_{1}] \odot \Al[L_{2}])$ risulta isomorfo al prodotto diretto $\Con(\Al[L_{1}]) \times \Con(\Al[L_{2}])$ e che, in ogni pre-bireticolo intrecciato $\Al[B]$, le congruenze di $\Al[B]$ coincidono con le congruenze di ciascuno dei suoi due ridotti reticolari, ovvero si hanno le seguenti uguaglianze
$$\Con(\Al[B]) = \Con (\la B, \land, \lor \ra) = \Con (\la B, \otimes, \oplus \ra). $$ 

Un altro interessante corollario è un teorema di rappresentazione analogo al noto teorema di rappresentazione per i reticoli distributivi: ogni pre-bireticolo distributivo si può rappresentare come un pre-bireticolo di insiemi.

Nella sezione  \ref{sec:repbil} dimostriamo il teorema di rappresentazione per i bireticoli intrecciati, che si ottiene facilmente a partire dal teorema di rappresentazione per i pre-bireticoli. In questo caso, dato un qualsiasi bireticolo intrecciato   $\Al[B]$,   è sufficiente considerare la relazione  $\sim_{1}$. Tale relazione non risulta essere una congruenza di $\Al[B]$ (in quanto non è compatibile con la negazione), ma ci permette di ottenere come quoziente di uno qualsiasi dei ridotti reticolari di $\Al[B]$ un reticolo $\Al[L]$ tale che $\Al[B]$ risulta isomorfo a $\Al[L] \odot \Al[L]$. 

Otteniamo come corollario una caratterizzazione delle congruenze di un qualsiasi  bireticolo intrecciato $\Al[B] \cong \Al[L] \odot \Al[L]$: si ha che 
$\Con (\Al[B])$ è isomorfo a $\Con (\Al[L]).$ Dimostriamo inoltre che  valgono le seguenti uguaglianze: $$\Con(\Al[B]) = \Con (\la B, \land, \neg \ra) =  \Con (\la B, \lor, \neg \ra) =  \Con (\la B, \otimes, \oplus, \neg \ra). $$ 

Concludiamo la sezione presentando una dimostrazione alternativa del teorema di rappresentazione, che si basa sulla considerazione degli elementi di un bireticolo che sono punti fissi dell'operatore di negazione, che denominiamo elementi \emph{regolari}. Dimostriamo che ogni bireticolo intrecciato  $\Al[B]$ è isomorfo al prodotto $\la \Reg(\Al[B]), \otimes, \oplus \ra \odot \la \Reg(\Al[B]), \otimes, \oplus \ra$, dove $\la \Reg(\Al[B]), \otimes, \oplus \ra $ è il subreticolo del ridotto $\la B, \otimes, \oplus \ra$ il cui universo è dato dall'insieme $\Reg(\Al[B])$ degli elementi regolari di  $\Al[B]$. 

Nella sezione  \ref{sec:bif} studiamo i filtri di reticolo nei (pre-)bireticoli intrecciati. Dato che in ogni pre-bireticolo vi sono due relazioni d'ordine reticolare, risulta naturale considerare quattro tipi di sottoinsiemi di $B$, vale a dire: i sottoinsiemi che sono filtri di reticolo in entrambi gli ordini, quelli che sono ideali in ambo gli ordini, quelli che sono $\leq_{t}$-filtri e $\leq_{k}$-ideali e infine gli  $\leq_{t}$-ideali e $\leq_{k}$-filtri.

È tuttavia facile rendersi conto che, per il Principio di Dualità, è sufficiente studiare uno qualsiasi di tali tipi di sottoinsiemi: ci concentriamo pertanto sullo studio di quelli che sono filtri in entrambi gli ordini, già considerati da  Arieli e Avron \cite{ArAv96}, che li chiamano \emph{bifiltri}.
 
Definiamo l'operatore di chiusura associato alla generazione dei bifiltri e, dato un pre-bireticolo intrecciato  $\Al[B]$, diamo una caratterizzazione del bifiltro generato da un qualsiasi insieme $X \subseteq B$ analoga alla nota caratterizzazione del filtro generato da un sottoinsieme qualsiasi di un reticolo.

Osserviamo che le relazioni $\sim_{1}$ e $\sim_{2}$ introdotte nella sezione \ref{sec:reppre} si possono caratterizzare nel seguente modo. Dato un pre-bireticolo intrecciato $\Al[B]$ ed elementi $a, b \in B$, si ha che  $a \sim_{1} b$ se e solo se il bifiltro generato da  $a$ coincide con quello generato da $b$ (analogamente si può caratterizzare la relazione  $\sim_{2}$ mediante gli operatori di generazione di filtri-ideali o di ideali-filtri).

Concludiamo la sezione con un risultato particolarmente importante dal punto di vista delle logiche associate ai bireticoli: il reticolo dei bifiltri di ogni pre-bireticolo intrecciato $\Al[L_{1}] \odot \Al[L_{2}]$ risulta essere isomorfo al reticolo dei filtri del primo fattore $\Al[L_{1}] $. 

Nella sezione \ref{sec:dis} concentriamo la nostra attenzione sulle varietà $\DPBL$ e $\DBL$ dei pre-bireticoli e bireticoli distributivi. Grazie ai teoremi di rappresentazione e alla caratterizzazione delle congruenze dei bireticoli intrecciati ottenuti nelle sezioni precedenti, possiamo caratterizzare la varietà  $\DPBL$ come generata dai suoi due membri di due elementi e la varietà $\DBL$ 
come generata dal suo minimo membro non degenere (il bireticolo di quattro elementi $\four$). 

Studiamo di seguito la struttura dei bifiltri nei pre-bireticoli distributivi. Otteniamo in tal modo alcuni risultati analoghi a noti teoremi riguardanti i reticoli distributivi. In particolare, dimostriamo un teorema di estensione del bifiltro e un teorema del bifiltro primo  (si dice che un bifiltro è primo quando risulta essere un filtro primo rispetto a entrambi gli ordini del pre-bireticolo). Tali risultati ci permettono di dare una dimostrazione diretta (e alternativa a quella presentata nella sezione \ref{sec:reppre}) del teorema di rappresentazione dei (pre-)bireticoli distributivi come  (pre-)bireticoli de insiemi.

Concludiamo il capitolo   (sezione \ref{sec:confl}) considerando una espansione del linguaggio dei bireticoli ottenuta mediante l'aggiunta di una operazione unaria duale della negazione, vale a dire involutiva, monotona rispetto a $\leq_{t}$ e antimonotona rispetto a $\leq_{k}$. Tale operatore fu introdotto da  Fitting \cite{Fi94}, che lo denomina  ``conflazione'' (\emph{conflation}). Chiameremo pertanto  \emph{bireticolo con conflazione} un'algebra $\Al[B] = \la B, \land, \lor, \otimes, \oplus, \neg, \co \ra $ tale che il ridotto $ \la B, \land, \lor, \otimes, \oplus, \neg \ra $ è un bireticolo e l'operazione  $\co: B \longrightarrow B$ verifica, per ogni $a, b \in B$, le seguenti condizioni:
\begin{enumerate}[]
 \item \quad \textbf{(con1)} \qquad se $a \leq_{k} b$,  allora  $- b \leq_{k} - a$
 \item \quad \textbf{(con2)} \qquad  se $a \leq_{t} b$,  allora  $- a \leq_{t} - b$ 
 \item \quad \textbf{(con3)} \qquad  $ a = - - a  $.
\end{enumerate}

Diremo che un bireticolo con conflazione è \emph{commutativo} se negazione e conflazione commutano,  ovvero se, per ogni $a \in B$, si ha
$$\neg \co  a = \co  \neg \: a.$$
Osserviamo che evidentemente i bireticoli con conflazione formano una classe equazionale (così come i bireticoli intrecciati con conflazione, ecc.).

Presentiamo un teorema di rappresentazione, analogo al teorema di rappresentazione per i bireticoli, relativo ai bireticoli intrecciati commutativi con conflazione. In questo caso mostriamo che ogni algebra  $\Al[B]$ appartenente alla suddetta varietà risulta isomorfa a un prodotto  $\Al[L] \odot \Al[L]$ di due copie di un'algebra $\Al[L] = \la L, \sqcap, \sqcup, ' \ra$, dove $\la L, \sqcap, \sqcup \ra$ è un reticolo  e $': L \longrightarrow L$  è un'operazione unaria involutiva e antimonotona rispetto all'ordine reticolare (che chiamiamo \emph{involuzione}).

Mostriamo che, analogamente al caso dei bireticoli, vi è un isomorfismo fra le congruenze di ogni bireticolo intrecciato commutativo con conflazione   $\Al[L] \odot \Al[L]$ e le congruenze di $\Al[L]$. Tale risultato ci permette di ottenere una caratterizzazione delle sottovarietà della varietà dei bireticoli distributivi commutativi con conflazione a partire dalle corrispondenti varietà di reticoli distributivi con involuzione (detti \emph{reticoli di De Morgan}).

Nel capitolo \ref{ch:log} studiamo, dal punto di vista della Logica Algebrica Astratta, la logica senza implicazione $\lb$, introdotta da  Arieli e Avron \cite{ArAv96} attraverso una classe di matrici dette bireticoli logici, che consistono in una coppia $\la \Al[B], F \ra$ dove $\Al[B]$ è un bireticolo e  $F \subseteq B$ un bifiltro primo.  

Nella sezione  \ref{sec:sem} introduciamo semanticamente $\lb$ come la logica definita dalla matrice $\la \four, \Tr \ra$, dove $\Tr = \{ \top, \true \}$. Presentiamo di seguito alcuni risultati importanti ottenuti da Arieli e Avron: tra questi il fatto che la logica definita da qualsiasi bireticolo logico   $\la \Al[B], F \ra$ coincide con quella definita dalla matrice $\la \four, \Tr \ra$ (e quindi con $\lb$) e l'introduzione di un calcolo alla Gentzen che risulta completo rispetto a $\lb$ (Tabella \ref{tab:Gentz_it}). 

Nella sezione \ref{sec:hil} introduciamo una presentazione di  $\lb$ mediante un calcolo alla Hilbert (Tabella \ref{tab:hilbert_it}), che utilizziamo nelle seguenti sezioni per studiare $\lb$ dal punto di vista della Logica Algebrica Astratta.

Dimostriamo che ogni formula si può ridurre a una forma normale e, grazie a tale risultato, otteniamo per il nostro calcolo un teorema di completezza rispetto alla semantica di $\lb$ introdotta nella sezione precedente.

Nella sezione  \ref{sec:tar} caratterizziamo la logica $\lb$ mediante alcune proprietà metalogiche (talvolta dette \emph{alla Tarski}); mostriamo inoltre che $\lb$ non ammette estensioni consistenti.

\begin{table}[ht!]
\fbox{\parbox{\textwidth}{
\vspace{0,15cm}
\begin{enumerate}[ ]
  \item \quad \quad \quad \textbf{Assioma:} \qquad $(Ax) \quad \Gamma, \varphi \rhd \varphi,
    \Delta$.
    \item
  \item \quad \quad \quad \textbf{Regole:} \qquad \; Regola del Taglio più le seguenti regole logiche:
\end{enumerate}
\begin{center} {
  \begin{align*}
&(\land \rhd) \quad \frac{\Gamma, \varphi, \psi \rhd \Delta}{\Gamma, \varphi \land \psi \rhd \Delta} 
& & (\rhd \land) \quad \frac{\Gamma \rhd \Delta, \varphi \quad \Gamma \rhd \Delta, \psi }{\Gamma \rhd \Delta, \varphi \land \psi} \\ \\
& (\neg \land \rhd) \quad \frac{\Gamma, \neg \varphi \rhd \Delta \quad \Gamma, \neg \psi \rhd \Delta }{\Gamma, \neg (\varphi \land \psi) \rhd \Delta}  
& &  (\rhd \neg \land ) \quad \frac{\Gamma \rhd \Delta, \neg \varphi,\neg \psi}{\Gamma \rhd \Delta, \neg (\varphi \land \psi)} \\ \\
& (\lor \rhd) \quad \frac{\Gamma, \varphi \rhd \Delta \quad \Gamma, \psi \rhd \Delta }{\Gamma, \varphi \lor \psi \rhd \Delta}  & &  (\rhd \lor ) \quad \frac{\Gamma \rhd \Delta, \varphi, \psi}{\Gamma \rhd \Delta, \varphi \lor \psi} \\ \\
&(\neg \lor \rhd) \quad \frac{\Gamma, \neg \varphi, \neg \psi \rhd \Delta}{\Gamma, \neg (\varphi \lor \psi) \rhd \Delta} & & (\rhd \neg \lor) \quad \frac{\Gamma \rhd \Delta, \neg \varphi \quad \Gamma \rhd \Delta, \neg \psi }{\Gamma \rhd \Delta, \neg (\varphi \lor \psi)} \\ \\
&(\otimes \rhd) \quad \frac{\Gamma, \varphi, \psi \rhd \Delta}{\Gamma, \varphi \otimes \psi \rhd \Delta} 
& & (\rhd \otimes) \quad \frac{\Gamma \rhd \Delta, \varphi \quad \Gamma \rhd \Delta, \psi }{\Gamma \rhd \Delta, \varphi \otimes \psi} \\ \\
& (\neg \otimes \rhd) \quad \frac{\Gamma, \neg \varphi, \neg \psi \rhd \Delta}{\Gamma, \neg (\varphi \otimes \psi) \rhd \Delta}  
& & (\rhd \neg \otimes) \quad \frac{\Gamma \rhd \Delta, \neg \varphi \quad \Gamma \rhd \Delta, \neg \psi }{\Gamma \rhd \Delta, \neg (\varphi \otimes \psi)}  \\ \\
& (\oplus \rhd) \quad \frac{\Gamma, \varphi \rhd \Delta \quad \Gamma, \psi \rhd \Delta }{\Gamma, \varphi \oplus \psi \rhd \Delta}  & &  (\rhd \oplus ) \quad \frac{\Gamma \rhd \Delta, \varphi, \psi}{\Gamma \rhd \Delta, \varphi \oplus \psi} \\ \\
&(\neg \oplus \rhd) \quad \frac{\Gamma, \neg \varphi \rhd \Delta \quad
\Gamma, \neg \psi \rhd \Delta }{\Gamma, \neg (\varphi \oplus \psi) \rhd
\Delta} & & (\rhd \neg \oplus ) \quad \frac{\Gamma \rhd \Delta, \neg
\varphi,\neg \psi}{\Gamma \rhd \Delta, \neg (\varphi \oplus \psi)} \\ \\
& (\neg \neg \rhd) \quad \frac{\Gamma, \varphi \rhd \Delta }{\Gamma, \neg \neg \varphi \rhd \Delta}  & &  (\rhd \neg \neg) \quad \frac{\Gamma \rhd \Delta, \varphi}{\Gamma \rhd \Delta, \neg \neg \varphi}
\end{align*}
}
\end{center}
}}

\caption{Un calcolo di sequenti completo per la logica 
$\mathcal{LB}$} \label{tab:Gentz_it}
\end{table}

\begin{table}[ht!]
\fbox{\parbox{\textwidth}{
\begin{center} {
  \hspace{-0.7cm}
  \begin{tabular}{ccc}
    \AxiomC{$p \land q$}
    \LeftLabel{\,(R1)}
    \UnaryInfC{$p$}
    \DisplayProof &
    \AxiomC{$p \land q$}
    \LeftLabel{\,(R2)}
    \UnaryInfC{$q$}
    \DisplayProof &
    \AxiomC{$p \!$}
    \AxiomC{$\! q$}
    \LeftLabel{\,(R3)}
    \BinaryInfC{$p \land q$}
    \DisplayProof \\ \\ 

    \AxiomC{$p$}
    \LeftLabel{\,(R4)}
    \UnaryInfC{$p \lor q$}
    \DisplayProof &
    \AxiomC{$p \lor q$}
    \LeftLabel{\,(R5)}
    \UnaryInfC{$q \lor p$}
    \DisplayProof &
    \AxiomC{$p \lor p$}
    \LeftLabel{\,(R6)}
    \UnaryInfC{$p$}
    \DisplayProof \\ \\ 
 
    \AxiomC{$p \lor (q \lor r)$}
    \LeftLabel{\,(R7)}
    \UnaryInfC{$(p \lor q) \lor r$}
    \DisplayProof &
    \AxiomC{$p \lor (q \land r)$}
    \LeftLabel{\,(R8)}
    \UnaryInfC{$(p \lor q) \land (p \lor r)$}
    \DisplayProof &
    \AxiomC{$(p \lor q) \land (p \lor r)$}
    \LeftLabel{\,(R9)}
    \UnaryInfC{$p \lor (q \land r)$}
    \DisplayProof \\ \\ 
 
    \AxiomC{$p \lor r$}
    \LeftLabel{\,(R10)}
    \UnaryInfC{$\neg \neg p \lor r$}
    \DisplayProof &
    \AxiomC{$\neg \neg p \lor r$}
    \LeftLabel{\,(R11)}
    \UnaryInfC{$p \lor r$}
    \DisplayProof &
    \AxiomC{$\neg (p \lor q) \lor r$}
    \LeftLabel{\,(R12)}
    \UnaryInfC{$ (\neg p \land \neg q) \lor r$}
    \DisplayProof \\ \\ 
 
    \AxiomC{$(\neg p \land \neg q) \lor r$}
    \LeftLabel{\,(R13)}
    \UnaryInfC{$\neg (p \lor q) \lor r$}
    \DisplayProof &
    \AxiomC{$\neg (p \land q) \lor r$}
    \LeftLabel{\,(R14)}
    \UnaryInfC{$(\neg p \lor \neg q) \lor r$}
    \DisplayProof &
    \AxiomC{$(\neg p \lor \neg q) \lor r$}
    \LeftLabel{\,(R15)}
    \UnaryInfC{$\neg (p \land q) \lor r$}
    \DisplayProof \\ \\ \\ 
 
    \AxiomC{$( p \otimes q) \lor r$}
    \LeftLabel{\,(R16)}
    \UnaryInfC{$(p \land q) \lor r$}
    \DisplayProof &
    \AxiomC{$(p \land q) \lor r$}
    \LeftLabel{\,(R17)}
    \UnaryInfC{$(p \otimes q) \lor r$}
    \DisplayProof &
    \AxiomC{$(p \oplus q) \lor r$}
    \LeftLabel{\,(R18)}
    \UnaryInfC{$(p \lor q) \lor r$}
    \DisplayProof \\ \\ 
 
    \AxiomC{$( p \lor q) \lor r$}
    \LeftLabel{\,(R19)}
    \UnaryInfC{$(p \oplus q) \lor r$}
    \DisplayProof &
    \AxiomC{$(\neg p \otimes \neg q) \lor r$}
    \LeftLabel{\,(R20)}
    \UnaryInfC{$\neg (p \otimes q) \lor r$}
    \DisplayProof &
    \AxiomC{$\neg (p \otimes q) \lor r$}
    \LeftLabel{\,(R21)}
    \UnaryInfC{$(\neg p \otimes \neg q) \lor r$}
    \DisplayProof \\ \\ 
 
    \AxiomC{$(\neg p \oplus \neg q) \lor r$}
    \LeftLabel{\,(R22)}
    \UnaryInfC{$\neg (p \oplus q) \lor r$}
    \DisplayProof &
    \AxiomC{$\neg (p \oplus q) \lor r$}
    \LeftLabel{\,(R23)}
    \UnaryInfC{$(\neg p \oplus \neg q) \lor r$}
    \DisplayProof & \\
  \end{tabular} } 
\end{center} 
}} \vspace{0,4cm}

\caption{Un calcolo alla Hilbert completo rispetto alla logica  $\LB$}
\label{tab:hilbert_it}
\end{table}

Cominciamo quindi (sezione \ref{sec:lb}) il vero e proprio studio di  $\lb$ dal punto di vista della Logica Algebrica Astratta. In primo luogo classifichiamo tale logica come non protoalgebrica e non autoestensionale. Caratterizziamo quindi la congruenza di Tarski associata a $\lb$ e, grazie a tale risultato, dimostriamo che la classe $\alg \lb$ dei ridotti algebrici dei modelli generalizzati ridotti di $\lb$ coincide con la varietà generata dal bireticolo $\four$ (vale a dire la varietà $\DBL$ dei bireticoli distributivi).

Osserviamo che, al contrario ad esempio di quanto accade con i reticoli distributivi o i reticoli di De Morgan, alla classe dei bireticoli distributivi è possibile associare una logica algebrizzabile  $\mathcal{L}$ (pertanto diversa da $\lb$) tale che $\algstar \mathcal{L} = \DBL$.

Caratterizziamo quindi i modelli pieni di  $\lb$ in funzione delle proprietà studiate nella sezione \ref{sec:tar}. Grazie a tale risultato siamo inoltre in grado di dimostrare che il calcolo alla Gentzen mostrato nella Tabella \ref{tab:Gentz_it} è \emph{pienamente adeguato} alla logica $\lb$. 

Consideriamo di seguito i modelli ridotti di  $\lb$ e la classe dei relativi ridotti algebrici $\algstar \lb$. Dimostriamo che tale classe non è una quasivarietà e ne caratterizziamo i membri come bireticoli distributivi superiormente limitati rispetto all'ordinamento  $\leq_{k}$ che soddisfano una certa proprietà addizionale. In particolare, mostriamo che  $\algstar \lb$ è costituita dai bireticoli distributivi $\Al[B]$ tali che $\Al[B] \cong \Al[L] \odot \Al[L]$, dove $\Al[L]$ è un  ``reticolo disgiuntivo duale''(\emph{dual disjunctive lattice}), ovvero un reticolo distributivo che soddisfa una proprietà che risulta essere duale rispetto alla cosiddetta \emph{proprietà disgiuntiva} considerata ad esempio in  \cite{Wa38} e\cite{Ci91}. 

Concludiamo il capitolo (sezione \ref{sec:gentz}) dimostrando che il calcolo alla Gentzen introdotto da Arieli e Avron è algebrizzabile nel senso di Rebagliato e Verd\'{u} \cite{ReV93b}, e che la sua semantica algebrica equivalente è precisamente la varietà dei bireticoli distributivi.

Nel capitolo \ref{ch:add} ci occupiamo di una espansione della logica $\lb$ introdotta degli stessi Arieli e Avron \cite{ArAv96}, che denominiamo $\lbs$, ottenuta aggiungendo al linguaggio proposizionale $\{ \land, \lor, \otimes, \oplus, \neg \}$ due connettivi di implicazione interdefinibili, una  \emph{implicazione debole} $\s$ e una \emph{implicazione forte} $\ta$. Adottiamo la prima come primitiva e definiamo $$p \ta q : = (p \s q) \land (\neg q \s \neg p).$$ Ci serviamo  inoltre delle seguenti abbreviazioni:
\begin{align*}
  p \leftrightarrow q \   & : = \ (p \ta q) \land (q \ta p)  \\
   p \equiv q   \ & : = \ (p \s q) \land (q \s p).
\end{align*}

Nella sezione  \ref{sec:add} definiamo semanticamente la logica $\lbs$ e presentiamo il calcolo alla Hilbert $H_{\s}$ introdotto da Arieli e Avron (Tabella \ref{tab:hilbs_it}). Citiamo alcuni dei principali risultati ottenuti in  \cite{ArAv96}, in particolare il teorema di completezza del calcolo $H_{\s}$ rispetto alla semantica di $\lbs$.

Nella sezione \ref{sec:lbs} dimostriamo diverse proprietà sintattiche del calcolo $H_{\s}$, che ci permettono di giungere  al seguente risultato: il calcolo   $H_{\s}$ è algebrizzabile, con formula di equivalenza $\phi \leftrightarrow \psi$ ed equazione definitoria  $\phi \approx \phi \s \phi$. Grazie  al teorema di completezza, possiamo dunque concludere che la stessa logica $\lbs$ è algebrizzabile.

\begin{table}[ht!]

\begin{align*}
& \mathbf{Assiomi\!:} \\ \\
&(\supset 1)  & &p \supset (q \supset p) \\
&(\supset 2)  & &  (p \supset (q \supset r)) \supset ((p\supset q) \supset (p\supset r)) \\
&(\supset 3)  & &  ((p \supset q) \supset p) \supset p \\
&(\land \supset )  & &  (p \land q) \supset p \qquad \qquad (p \land q) \supset q \\
&(\supset \land )  & &  p \supset (q \supset (p \land q)) \\
&(\otimes \supset )  & &  (p \otimes q) \supset p \qquad \qquad (p \otimes q) \supset q \\
&(\supset \otimes )  & &  p \supset (q \supset (p \otimes q)) \\
&(\supset \lor)  & &  p  \supset (p \lor q) \qquad \qquad q \supset (p \lor q) \\
&(\lor \supset )  & &  (p \supset r) \supset ((q \supset r) \supset ((p \lor q) \supset r)) \\
&(\supset \oplus)  & &  p  \supset (p \oplus q) \qquad  \qquad q \supset (p \oplus q) \\
&(\oplus \supset )  & &  (p \supset r) \supset ((q \supset r) \supset ((p \oplus q) \supset r)) \\
&(\neg \land )  & &   \neg (p \land q ) \equiv (\neg p \vee \neg q)  \\
&(\neg \lor )  & &  \neg (p \vee q) \equiv (\neg p \land \neg q ) \\
&(\neg \otimes )  & &   \neg (p \otimes q ) \equiv (\neg p \otimes \neg q)  \\
&(\neg \oplus )  & &   \neg (p \oplus q ) \equiv (\neg p \oplus \neg q)  \\
&(\neg \supset )  & &   \neg (p \supset q ) \equiv (p \land \neg q)  \\
&(\neg \neg) & & p \equiv \neg \neg p \\ \\
& \mathbf{Regola\!:}
\end{align*}
$$
\frac{p \quad p \supset q }{q}
$$

\caption{Un calcolo alla Hilbert completo rispetto alla logica $\lbs$}
\label{tab:hilbs_it}
\end{table}

Di seguito (sezione \ref{sec:algstar}) ci occupiamo pertanto di individuare e studiare la semantica algebrica equivalente di   $\lbs$. A tal fine introduciamo mediante una presentazione equazionale la varietà $\ib$, i cui membri denominiamo ``bireticoli implicativi'' (\emph{implicative bilattices}), ovvero strutture $\Al[B] = \la B, \land, \lor, \otimes, \oplus, \s \neg \ra$ tali che $\la B, \land, \lor, \otimes, \oplus, \neg \ra$ è un bireticolo e l'operazione binaria $\s: B \times B \longrightarrow B$ è tale che le seguenti equazioni sono verificate:
\begin{enumerate}[ ]
\item (IB1) \hspace{2mm} $(x \supset x) \supset y \approx y$ 
\item (IB2)  \hspace{2mm}  $x \supset (y \supset z) \approx (x \land y) \supset z \approx (x \otimes y) \supset z$
\item(IB3)   \hspace{2mm}  $((x \supset y) \supset x) \supset x  \approx x \s x $ 
\item(IB4)  \hspace{2mm}  $(x \lor y) \supset z \approx (x \supset z) \land  (y \supset z) \approx (x \oplus y) \supset z  $ 
\item (IB5)  \hspace{2mm}  $x \land ((x \supset y) \supset (x \otimes y)) \approx x $
\item (IB6)  \hspace{2mm}  $\neg (x \supset y ) \supset z  \approx (x \land \neg y)  \supset z.$
\end{enumerate}

Dimostriamo quindi diverse proprietà aritmetiche della varietà $\ib$, che ci permettono concludere che tale varietà costituisce la semantica algebrica equivalente della logica $\lbs$. Mostriamo inoltre che il ridotto di ogni bireticolo implicativo è un bireticolo distributivo, un fatto che utilizziamo nel capitolo seguente, e che la logica $\lbs$, così come il suo frammento $\lb$, non ammette estensioni consistenti.

Nel capitolo \ref{ch:imp} presentiamo uno studio algebrico dei bireticoli implicativi e di alcune strutture algebriche relazionate con essi. 

Cominciamo il capitolo (sezione \ref{sec:repimp}) dimostrando un teorema di rappresentazione per i bireticoli implicativi analogo al teorema di rappresentazione per i bireticoli. Per i risultati ottenuti in precedenza sappiamo che, per ogni bireticolo implicativo $\Al[B] = \la B, \land, \lor, \otimes, \oplus, \s \neg \ra$, la struttura $\la B, \land, \lor, \otimes, \oplus, \neg \ra$ risulta isomorfa al prodotto $\Al[L] \odot \Al[L]$, dove $\Al[L]$ è un reticolo distributivo superiormente limitato. Nel caso dei bireticoli implicativi, mostriamo inoltre che $\Al[L]$ soddisfa una proprietà supplementare, ovvero è un reticolo \emph{relativamente complementato} (ogni elemento possiede un complemento in ogni intervallo di $\Al[L]$). 

 Dato un reticolo relativamente complementato e superiormente limitato  $\Al[L] = \la L, \sqcap, \sqcup \ra $, il cui elemento massimo indichiamo con 1, consideriamo l'operazione  $\ba : L \times L \longrightarrow L$ che ad ogni coppia di elementi $a, b \in L$ associa il complemento relativo di $a$ nell'intervallo  $[a \sqcap b, 1]$, che indichiamo con $a \ba b$. Osserviamo che tale classe di reticoli, considerati come algebre nel linguaggio $\{ \sqcap, \sqcup, \ba \}$, forma una varietà. Adottando la terminologia usata in  \cite{Cu77}, denominiamo i membri di tale varietà ``reticoli implicativi classici'' (\emph{classical implicative lattices}).

Dimostriamo dunque che a partire da un qualsiasi reticolo implicativo classico $\Al[L]$ è possibile definire un bireticolo implicativo mediante una costruzione che, per quanto riguarda il ridotto reticolare, coincide con il prodotto $\Al[L] \odot \Al[L]$ e inoltre, usando l'operazione $\ba$, ci permette di definire un'implicazione  $\s$ che soddisfa le equazioni che definiscono la varietà dei bireticoli implicativi. 

Abbiamo, pertanto, che ogni bireticolo implicativo $\Al[B]$ è isomorfo a un prodotto di tale tipo (che possiamo pure  indicare, ove non vi sia pericolo di fraintendimento, con $\Al[L] \odot \Al[L]$) di due copie di un reticolo implicativo classico $\Al[L]$. 

Ci occupiamo di seguito delle congruenze dei bireticoli implicativi. Usando il teorema di rappresentazione per i bireticoli implicativi, dimostriamo che il reticolo delle congruenze di un bireticolo implicativo  $\Al[L] \odot \Al[L]$ è isomorfo a quello delle congruenze del reticolo implicativo classico  $\Al[L]$. Dato che le congruenze di un reticolo implicativo classico   $\Al[L] = \la L, \sqcap, \sqcup, \ba \ra $ coincidono con quelle del suo ridotto reticolare $ \la L, \sqcap, \sqcup \ra $, giungiamo al seguente risultato: le congruenze di ogni bireticolo implicativo $\Al[B] = \la B, \land, \lor, \otimes, \oplus, \s \neg \ra$ coincidono con quelle del suo ridotto $\la B, \land, \lor, \otimes, \oplus, \neg \ra$ (le quali a loro volta, per i risultati ottenuti in precedenza, coincidono con quelle del ridotto $\la B, \land, \neg \ra$). 

Nella seguente  sezione  (\ref{sec:imp}) intraprendiamo uno studio più approfondito della varietà $\ib$ dei bireticoli implicativi. Utilizzando i risultati ottenuti nella sezione precedente, dimostriamo che l'unica algebra sottodirettamente irriducibile in $\ib$ è $\fours$, il suo membro di quattro elementi, il cui ridotto bireticolare è $\four$. Tale algebra genera pertanto la varietà dei bireticoli implicativi. Dimostriamo inoltre che $\ib$ è una varietà con termine discriminatore e che i suoi membri finiti sono isomorfi a potenze di $\fours$.

Otteniamo inoltre l'interessante risultato che, in un bireticolo implicativo, ciascuna delle due relazioni d'ordine reticolare può definirsi esplicitamente (tramite un'equazione) usando solo la implicazione e i connettivi corrispondenti all'altro ordine.

Nella sezione \ref{sec:dua} studiamo le relazioni fra i bireticoli implicativi classici e i reticoli disgiuntivi duali considerati nella sezione
\ref{sec:lb}. Dimostriamo, in particolare, che la classe dei reticoli implicativi classici (considerata nel mero linguaggio reticolare) è propriamente inclusa in quella dei reticoli disgiuntivi duali, e individuiamo una proprietà necessaria e sufficiente affinché un reticolo disgiuntivo duale appartenga alla classe dei reticoli implicativi classici.

Nelle due sezioni seguenti ci occupiamo di alcuni sottoridotti dei bireticoli implicativi che risultano particolarmente interessanti da un punto di vista logico.

Cominciamo, nella sezione  \ref{sec:sub}, osservando che in ogni bireticolo implicativo $\Al[B] = \la B, \land, \lor, \otimes, \oplus, \s \neg \ra$ risulta possibile definire esplicitamente un'operazione binaria $* : B \times B \longrightarrow B$ tale che la coppia $\{ *, \ta \}$ è residuata rispetto all'ordine $\leq_{t}$. La definizione è la seguente: per ogni coppia di elementi $a, b \in B$,
$$a * b \ :=  \ \neg (a \ta \neg b).$$ 
Mostriamo quindi che l'algebra  $\la B, \land, \lor, *, \ta \neg, \top \ra$ è, usando la terminologia di  \cite{GaRa04}, un ``reticolo residuato commutativo distributivo con involuzione'' (\emph{involutive commutative distributive residuated lattice}).

Introduciamo di seguito, mediante una presentazione equazionale, una classe di algebre che denominiamo ``reticoli residuati di De Morgan''  (\emph{residuated De Morgan lattices}), con lo scopo di dimostrare che tali strutture corrispondono ai $\{  \land, \lor,  \s \neg, \top \}$-sottoridotti dei bireticoli implicativi. 

Un reticolo residuato di De Morgan è un'algebra $\Al = \langle A, \land, \lor, {\supset,}$ $\neg, \top  \rangle$ tale che il ridotto $\langle A, \land, \lor, \neg \rangle$ è un reticolo di De Morgan e le seguenti condizioni sono verificate:
\begin{enumerate}[ ]
\item (RD0) $\quad \top  \approx \neg \top $ 
\item (RD1) $\quad \top \supset x \approx x$ 
\item (RD2) $\quad  x \supset (y \supset z) \approx (x \land y) \supset z$
\item (RD3) $\quad \top \land (((x \supset y) \supset x) \supset x) \approx \top  $ 
\item (RD4) $ \quad (x \lor y) \supset z \approx (x \supset z) \land  (y \supset z)$ 
\item (RD5) $\quad  x \land (((x \supset y) \land (\neg y \supset \neg x)) \supset y) \approx x$  
\item (RD6) $\quad \neg (x \supset y ) \supset z  \approx (x \land \neg y)  \supset z$.
\end{enumerate}

Cominciamo dimostrando alcune proprietà aritmetiche dei reticoli residuati di De Morgan, che ci permettono di mostrare che, per ogni reticolo di tale tipo $\Al = \la A, \land, \lor, \s,  \neg, \top \ra$, la struttura $\la A, \land, \lor, *, \ta \neg, \top \ra$ risulta effettivamente essere un reticolo residuato commutativo distributivo con involuzione.

Dimostriamo di seguito che ogni reticolo residuato di De Morgan $\Al$ contiene come sottoreticoli due reticoli relativamente complementati (che indichiamo con $\Al[A^-]$ e $\Al[A^+]$) tali che $\Al[A^-] \cong \Al[A^+]$. Sappiamo pertanto che è possibile costruire un bireticolo implicativo come prodotto $\Al[A^-] \odot \Al[A^-]$ (oppure $\Al[A^+] \odot \Al[A^+]$). 

Grazie a questo risultato possiamo definire, per ogni reticolo residuato di De Morgan $\Al$, una immersione  $h: A \longrightarrow A^- \times A^-$, vale a dire una funzione iniettiva che risulta essere un omomorfismo rispetto al linguaggio  $\{ \land, \lor, \s,  \neg, \top \}$. Mostriamo, inoltre, che tale immersione è canonica nel senso seguente: se esiste un omomorfismo  $f : A \longrightarrow B$ in un bireticolo implicativo   $\Al[B]$, allora esiste un unico omomorfismo $f': A^- \times A^- \longrightarrow B $ tale che $f' \cdot h = f$.

Otteniamo pertanto il risultato che i reticoli residuati di De Morgan coincidono con i
$\{  \land, \lor,  \s, \neg, \top \}$-sottoridotti dei bireticoli implicativi, e che la varietà dei reticoli residuati di De Morgan risulta generata dal suo membro di quattro elementi che è il ridotto del bireticolo implicativo $\fours$.

Nella sezione \ref{sec:other} generalizziamo la costruzione introdotta nella sezione precedente per studiare una classe più ampia di sottoridotti dei bireticoli implicativi. 

Ci soffermiamo, in particolare, sui   $\{ \s, \neg \}$-sottoridotti. Questi rivestono infatti, a nostro parere, un interesse particolare a livello logico, in quanto si tratta del minimo frammento del linguaggio dei bireticoli implicativi che risulta necessario per definire le due traduzioni da formule in equazioni e viceversa che permettono di dimostrare l'algebrizzabilità della logica $\lbs$ rispetto alla varietà dei bireticoli implicativi.

Introduciamo di seguito, mediante una presentazione equazionale, una varietà di algebre nel linguaggio $\{ \s, \neg \}$, che denominiamo $I$\emph{-algebre}.

Una $I$-algebra è una struttura $\Al = \la A, \s, \neg \ra$ che verifica le seguenti equazioni:
\begin{enumerate}[ ]
\item (I1) $ \quad (x \supset x) \supset y \approx y$ 
\item (I2) $ \quad  x \supset (y \supset z) \approx (x \s y) \s (x \s z) \approx y \supset (x \supset z)$
\item (I3) $ \quad ((x \supset y) \supset x) \supset x  \approx x \s x $ 
\item (I4) $\quad  x \supset (\neg y \supset z) \approx \neg (x \s y) \s z$
\item (I5) $\quad \neg \neg x \approx x$
\item  (I6) $ \quad  p(x,y,x) \approx p(x,y,y)$
\end{enumerate}
dove $p(x,y,z)$ è una abbreviazione di $$(x \s y) \s ((y \s x) \s ((\neg x \s \neg y) \s ((\neg y \s \neg x) \s z))).$$

Dimostriamo alcune proprietà aritmetiche di questa varietà, che ci permettono di mostrare che a una qualsiasi $I$-algebra $\Al = \la A, \s, \neg \ra$ è possibile associare un'algebra di Tarski $\Al[A]/\!\!\sim \ = \la A/\!\!\sim, \s \ra$ ottenuta come un quoziente del ridotto $\la A, \s \ra$ di $\Al$ modulo la relazione di equivalenza $\sim$, che risulta compatibile con l'operazione $\s$ (ma non con $\neg$). 

Mostriamo inoltre che, in ogni algebra di Tarski $\Al[A]/\!\!\sim \ = \la A/\!\!\sim, \s \ra$ 
ottenuta nella suddetta maniera, è possibile definire operazioni di infimo e supremo  $\{\sqcap, \sqcup \}$ rispetto all'ordine naturale, in maniera tale che l'algebra  $ \la A/\!\!\sim, \sqcap, \sqcup, \s \ra$
risulta essere un reticolo implicativo classico.

Possiamo pertanto sviluppare una costruzione analoga a quella descritta nella sezione \ref{sec:sub}, ottenendo il corrispondente risultato che, per ogni $I$-algebra $\Al$, esiste una immersione $h: A \longrightarrow A/\!\!\sim \times \ A/\!\!\sim$ di $\Al$ nel bireticolo implicativo $$ \la A/\!\!\sim, \sqcap, \sqcup, \s \ra \odot  \la A/\!\!\sim, \sqcap, \sqcup, \s \ra.$$

Anche in questo caso dimostriamo che la funzione  $h$ da noi definita è canonica nel senso che, se esiste un omomorfismo
$f : A \longrightarrow B$ 
per un qualche bireticolo implicativo $\Al[B]$, 
allora esiste un unico omomorfismo
$f': A/\!\!\sim \times \ A/\!\!\sim \ \longrightarrow B $ tale che $f' \cdot h = f$.

Abbiamo pertanto che le $I$-algebre coincidono con i 
$\{  \s, \neg \}$-sottoridotti dei bireticoli implicativi, e che la varietà delle $I$-algebre risulta generata dal suo membro di quattro elementi che è il ridotto del bireticolo implicativo
$\fours$. Otteniamo altresì una caratterizzazione di tutte le sottovarietà delle $I$-algebre, che risultano generate dai  $\{  \s, \neg \}$-sottoridotti di $\fours$, e diamo una presentazione equazionale di tali sottovarietà.

Facciamo notare che la costruzione descritta si può agevolmente applicare a tutti gli altri sottoridotti che corrispondono a frammenti del linguaggio dei bireticoli implicativi che contengono  $\{  \s, \neg \}$,  ottenendo analoghe caratterizzazioni per tali classi di algebre.

Nell'ultima sezione 
(\ref{sec:cat}) diamo una formulazione di alcuni dei risultati ottenuti in termini di teoria delle categorie. Definiamo pertanto categorie associate alle diverse classi di (pre-)bireticoli e reticoli considerati nel corso del presente lavoro, i cui oggetti sono le algebre della relativa varietà e i cui morfismi sono gli omomorfismi tra algebre.

Otteniamo in tal modo equivalenze categoriali tra le seguenti categorie:
\begin{enumerate}[(i)]
  \item  pre-bireticoli intrecciati ($\IPBL$) e la categoria prodotto $\lat \times \lat$, i cui oggetti sono coppie di reticoli,
  \item pre-bireticoli distributivi ($\DPBL$) e la categoria prodotto $\dl \times \dl$,  i cui oggetti sono coppie di reticoli distributivi,
  \item bireticoli intrecciati  ($\IBL$) e reticoli ($\lat$),
  \item bireticoli  distributivi  ($\DBL$) e reticoli  distributivi ($\dl $),
  \item bireticoli intrecciati commutativi con conflazione ($\IBLC$) e reticoli con involuzione ($\IL$),
  \item bireticoli distributivi commutativi con conflazione ($\DBLC$) e reticoli di De Morgan ($\DML$),
    \item bireticoli  di Kleene con conflazione  ($\KB$, una sottovarietà di $\DBLC$)
    e reticoli di Kleene ($\KL $),
        \item bireticoli  classici con conflazione ($\CB$, una sottovarietà di $\KB$)
         e reticoli booleani ($\BOL $),
\item bireticoli implicativi ($\ib$) e reticoli implicativi classici ($\cl $).
\end{enumerate}

Terminiamo la sezione e il capitolo dimostrando che è inoltre possibile definire funtori 
 $F $ e $G$
tra la categoria corrispondente alle $I$-algebre e quella corrispondente ai bireticoli implicativi in maniera tale che $\la F, G \ra$ forma un'aggiunzione.




\end{document}